\documentclass[a4paper]{article}

\usepackage{latexsym,amsthm,amscd}
\usepackage[all]{xy}
\usepackage{rotating}
\usepackage[utf8]{inputenc}
\usepackage[T1]{fontenc}
\usepackage[english]{babel}
\usepackage{amsmath}
\usepackage{amssymb}
\usepackage{amsfonts}
\usepackage[svgnames]{xcolor}
\usepackage{geometry}
\usepackage{fancyhdr}
\usepackage{lastpage}
\usepackage{array}
\usepackage{stmaryrd}
\usepackage{hyperref}
\usepackage{tikz}
\usetikzlibrary{patterns}
\usepgflibrary{decorations.pathmorphing}
\usetikzlibrary{decorations.pathmorphing}
\usetikzlibrary{calc,shapes.arrows,decorations.pathreplacing}
\usetikzlibrary{shadows}
\usetikzlibrary{positioning}

\title{Simulation of
minimal effective dynamical systems on the Cantor sets 
by minimal $\Z^3$-SFT}

\author{Silvère Gangloff, Mathieu Sablik}

\newcounter{quest}
\newcounter{squest}
        {\begin{list}
        {\bfseries #1.\arabic{quest}.}
        {\usecounter{quest}
        \setlength{\listparindent}{0cm}
        \setlength{\labelwidth}{1.5cm}
        \setlength{\leftmargin}{1cm}
        \setlength{\parsep}{0.5ex plus0.2ex minus0.1ex}
        \setlength{\itemsep}{0.7ex plus0.2ex minus0.1ex}}
        }{\end{list}}
       {\begin{list}
        {\bfseries \alph{squest}.}
        {\usecounter{squest}}
        }{\end{list}}

\usetikzlibrary{arrows}

\newtheorem{remark}{Remark}

\def\N{\mathbb N}

\def\Z{\mathbb Z}

\newcommand{\A}{\mathcal{A}}
\newcommand{\s}{\sigma}

\renewcommand{\vec}[1]{\textbf{#1}}
\newcommand{\U}{\mathbb{U}}

\newcommand{\define}[1]{\textbf{#1}}
\newcommand{\supp}[1]{\textrm{supp}\left(#1\right)}

\newtheorem*{rep@theorem}{\rep@title}
\newcommand{\newreptheorem}[2]{%
\newenvironment{rep#1}[1]{%
 \def\rep@title{#2 \ref{##1}}%
 \begin{rep@theorem}}%
 {\end{rep@theorem}}}

\newtheorem{theorem}{Theorem}
\newtheorem{proposition}{Proposition}
\newtheorem{definition}{Definition}

\newtheorem{lemma}{Lemma}
\newtheorem{example}{Example}

\newreptheorem{theorem}{Theorem}
\newreptheorem{lemma}{Lemma}

\newcommand{\robionegauche}[2]{
\draw (#1,#2) rectangle (#1+2,#2+2);
\draw [-latex] (#1+2,#2+1) -- (#1+0,#2+1);
\draw [-latex] (#1+1,#2+0) -- (#1+1,#2+1); 
\draw [-latex] (#1+1,#2+2) -- (#1+1,#2+1);}
\newcommand{\robionedroite}[2]{
\draw (#1,#2) rectangle (#1+2,#2+2);
\draw [-latex] (#1+0,#2+1) -- (#1+2,#2+1);
\draw [-latex] (#1+1,#2+0) -- (#1+1,#2+1); 
\draw [-latex] (#1+1,#2+2) -- (#1+1,#2+1);}
\newcommand{\robionebas}[2]{
\draw (#1,#2) rectangle (#1+2,#2+2);
\draw [-latex] (#1+1,#2+0) -- (#1+1,#2+2);
\draw [-latex] (#1+0,#2+1) -- (#1+1,#2+1); 
\draw [-latex] (#1+2,#2+1) -- (#1+1,#2+1);}
\newcommand{\robionehaut}[2]{
\draw (#1,#2) rectangle (#1+2,#2+2);
\draw [-latex] (#1+1,#2+2) -- (#1+1,#2+0);
\draw [-latex] (#1+0,#2+1) -- (#1+1,#2+1); 
\draw [-latex] (#1+2,#2+1) -- (#1+1,#2+1);}

\newcommand{\robitwobas}[2]{
\draw (#1,#2) rectangle (#1+2,#2+2);
\draw [-latex] (#1+1,#2+2) -- (#1+1,#2+0) ;
\draw [-latex] (#1+0,#2+1) -- (#1+1,#2+1) ; 
\draw [-latex] (#1+0,#2+0.5) -- (#1+1,#2+0.5) ; 
\draw [-latex] (#1+2,#2+1) -- (#1+1,#2+1) ;
\draw [-latex] (#1+2,#2+0.5) -- (#1+1,#2+0.5) ;}
\newcommand{\robitwohaut}[2]{
\draw (#1,#2) rectangle (#1+2,#2+2);
\draw [-latex] (#1+1,#2+0) -- (#1+1,#2+2) ;
\draw [-latex] (#1+0,#2+1) -- (#1+1,#2+1) ; 
\draw [-latex] (#1+0,#2+1.5) -- (#1+1,#2+1.5) ; 
\draw [-latex] (#1+2,#2+1) -- (#1+1,#2+1) ;
\draw [-latex] (#1+2,#2+1.5) -- (#1+1,#2+1.5) ;}
\newcommand{\robitwodroite}[2]{
\draw (#1,#2) rectangle (#1+2,#2+2);
\draw [-latex] (#1+0,#2+1) -- (#1+2,#2+1) ;
\draw [-latex] (#1+1,#2+0) -- (#1+1,#2+1) ; 
\draw [-latex] (#1+1.5,#2+0) -- (#1+1.5,#2+1) ; 
\draw [-latex] (#1+1,#2+2) -- (#1+1,#2+1) ;
\draw [-latex] (#1+1.5,#2+2) -- (#1+1.5,#2+1) ;}
\newcommand{\robitwogauche}[2]{
\draw (#1,#2) rectangle (#1+2,#2+2);
\draw [-latex] (#1+2,#2+1) -- (#1+0,#2+1) ;
\draw [-latex] (#1+1,#2+0) -- (#1+1,#2+1) ; 
\draw [-latex] (#1+0.5,#2+0) -- (#1+0.5,#2+1) ; 
\draw [-latex] (#1+1,#2+2) -- (#1+1,#2+1) ;
\draw [-latex] (#1+0.5,#2+2) -- (#1+0.5,#2+1) ;}

\newcommand{\robithreehaut}[2]{
\draw (#1,#2) rectangle (#1+2,#2+2) ;
\draw [-latex] (#1+1,#2+0) -- (#1+1,#2+2) ;
\draw [-latex] (#1+0.5,#2+0) -- (#1+0.5,#2+2) ; 
\draw [-latex] (#1+0,#2+1) -- (#1+0.5,#2+1) ; 
\draw [-latex] (#1+2,#2+1) -- (#1+1,#2+1) ;}

\newcommand{\robithreedroite}[2]{
\draw (#1,#2) rectangle (#1+2,#2+2) ;
\draw [-latex] (#1+0,#2+1) -- (#1+2,#2+1) ;
\draw [-latex] (#1+0,#2+0.5) -- (#1+2,#2+0.5) ; 
\draw [-latex] (#1+1,#2+0) -- (#1+1,#2+0.5) ; 
\draw [-latex] (#1+1,#2+2) -- (#1+1,#2+1) ;}

\newcommand{\robifourhaut}[2]{
\draw (#1,#2) rectangle (#1+2,#2+2) ;
\draw [-latex] (#1+1,#2+0) -- (#1+1,#2+2) ;
\draw [-latex] (#1+1.5,#2+0) -- (#1+1.5,#2+2) ; 
\draw [-latex] (#1+0,#2+1) -- (#1+1,#2+1) ; 
\draw [-latex] (#1+2,#2+1) -- (#1+1.5,#2+1) ;}

\newcommand{\robisixbas}[2]{
\draw (#1,#2) rectangle (#1+2,#2+2) ;
\draw [-latex] (#1+1,#2+2) -- (#1+1,#2+0) ;
\draw [-latex] (#1+0.5,#2+2) -- (#1+0.5,#2+0) ; 
\draw [-latex] (#1+0,#2+1) -- (#1+0.5,#2+1) ; 
\draw [-latex] (#1+2,#2+1) -- (#1+1,#2+1) ;
\draw [-latex] (#1+0,#2+0.5) -- (#1+0.5,#2+0.5) ; 
\draw [-latex] (#1+2,#2+0.5) -- (#1+1,#2+0.5) ;}
\newcommand{\robisixgauche}[2]{
\draw (#1,#2) rectangle (#1+2,#2+2) ;
\draw [-latex] (#1+2,#2+1) -- (#1+0,#2+1) ;
\draw [-latex] (#1+2,#2+0.5) -- (#1+0,#2+0.5) ; 
\draw [-latex] (#1+1,#2+0) -- (#1+1,#2+0.5) ; 
\draw [-latex] (#1+1,#2+2) -- (#1+1,#2+1) ;
\draw [-latex] (#1+0.5,#2+0) -- (#1+0.5,#2+0.5) ; 
\draw [-latex] (#1+0.5,#2+2) -- (#1+0.5,#2+1) ;}
\newcommand{\robisixhaut}[2]{
\draw (#1,#2) rectangle (#1+2,#2+2) ;
\draw [-latex] (#1+1,#2+0) -- (#1+1,#2+2) ;
\draw [-latex] (#1+0.5,#2+0) -- (#1+0.5,#2+2) ; 
\draw [-latex] (#1+0,#2+1) -- (#1+0.5,#2+1) ; 
\draw [-latex] (#1+2,#2+1) -- (#1+1,#2+1) ;
\draw [-latex] (#1+0,#2+1.5) -- (#1+0.5,#2+1.5) ; 
\draw [-latex] (#1+2,#2+1.5) -- (#1+1,#2+1.5) ;}
\newcommand{\robisixdroite}[2]{
\draw (#1,#2) rectangle (#1+2,#2+2) ;
\draw [-latex] (#1+0,#2+1) -- (#1+2,#2+1) ;
\draw [-latex] (#1+0,#2+0.5) -- (#1+2,#2+0.5) ; 
\draw [-latex] (#1+1,#2+0) -- (#1+1,#2+0.5) ; 
\draw [-latex] (#1+1,#2+2) -- (#1+1,#2+1) ;
\draw [-latex] (#1+1.5,#2+0) -- (#1+1.5,#2+0.5) ; 
\draw [-latex] (#1+1.5,#2+2) -- (#1+1.5,#2+1) ;}

\newcommand{\robisevenbas}[2]{
\draw (#1,#2) rectangle (#1+2,#2+2) ;
\draw [-latex] (#1+1,#2+2) -- (#1+1,#2+0) ;
\draw [-latex] (#1+1.5,#2+2) -- (#1+1.5,#2+0) ; 
\draw [-latex] (#1+0,#2+1) -- (#1+1,#2+1) ; 
\draw [-latex] (#1+2,#2+1) -- (#1+1.5,#2+1) ;
\draw [-latex] (#1+0,#2+0.5) -- (#1+1,#2+0.5) ; 
\draw [-latex] (#1+2,#2+0.5) -- (#1+1.5,#2+0.5) ;}
\newcommand{\robisevengauche}[2]{
\draw (#1,#2) rectangle (#1+2,#2+2) ;
\draw [-latex] (#1+2,#2+1) -- (#1+0,#2+1) ;
\draw [-latex] (#1+2,#2+1.5) -- (#1+0,#2+1.5) ; 
\draw [-latex] (#1+1,#2+0) -- (#1+1,#2+1) ; 
\draw [-latex] (#1+1,#2+2) -- (#1+1,#2+1.5) ;
\draw [-latex] (#1+0.5,#2+0) -- (#1+0.5,#2+1) ; 
\draw [-latex] (#1+0.5,#2+2) -- (#1+0.5,#2+1.5) ;}
\newcommand{\robisevenhaut}[2]{
\draw (#1,#2) rectangle (#1+2,#2+2) ;
\draw [-latex] (#1+1,#2+0) -- (#1+1,#2+2) ;
\draw [-latex] (#1+1.5,#2+0) -- (#1+1.5,#2+2) ; 
\draw [-latex] (#1+0,#2+1) -- (#1+1,#2+1) ; 
\draw [-latex] (#1+2,#2+1) -- (#1+1.5,#2+1) ;
\draw [-latex] (#1+0,#2+1.5) -- (#1+1,#2+1.5) ; 
\draw [-latex] (#1+2,#2+1.5) -- (#1+1.5,#2+1.5) ;}
\newcommand{\robisevendroite}[2]{
\draw (#1,#2) rectangle (#1+2,#2+2) ;
\draw [-latex] (#1+0,#2+1) -- (#1+2,#2+1) ;
\draw [-latex] (#1+0,#2+1.5) -- (#1+2,#2+1.5) ; 
\draw [-latex] (#1+1,#2+0) -- (#1+1,#2+1) ; 
\draw [-latex] (#1+1,#2+2) -- (#1+1,#2+1.5) ;
\draw [-latex] (#1+1.5,#2+0) -- (#1+1.5,#2+1) ; 
\draw [-latex] (#1+1.5,#2+2) -- (#1+1.5,#2+1.5) ;}

\newcommand{\robibluebastgauche}[2]{
\fill[blue!40] (#1+0.5,#2+0.5) rectangle (#1+1,#2+2) ;
\fill[blue!40] (#1+0.5,#2+0.5) rectangle (#1+2,#2+1) ;
\draw (#1,#2) rectangle (#1+2,#2+2) ;
\draw [-latex] (#1+0.5,#2+0.5) -- (#1+0.5,#2+2) ; 
\draw [-latex] (#1+0.5,#2+0.5) -- (#1+2,#2+0.5) ;
\draw [-latex] (#1+1,#2+1) -- (#1+1,#2+2) ; 
\draw [-latex] (#1+1,#2+1) -- (#1+2,#2+1) ; 
\draw [-latex] (#1+0.5,#2+1) -- (#1+0,#2+1) ; 
\draw [-latex] (#1+1,#2+0.5) -- (#1+1,#2+0) ;}
\newcommand{\robibluebastdroite}[2]{
\fill[blue!40] (#1+1.5,#2+0.5) rectangle (#1+1,#2+2) ;
\fill[blue!40] (#1+1.5,#2+0.5) rectangle (#1+0,#2+1) ;
\draw (#1,#2) rectangle (#1+2,#2+2) ;
\draw [-latex] (#1+1.5,#2+0.5) -- (#1+1.5,#2+2) ; 
\draw [-latex] (#1+1.5,#2+0.5) -- (#1+0,#2+0.5) ;
\draw [-latex] (#1+1,#2+1) -- (#1+1,#2+2) ; 
\draw [-latex] (#1+1,#2+1) -- (#1+0,#2+1) ; 
\draw [-latex] (#1+1.5,#2+1) -- (#1+2,#2+1) ; 
\draw [-latex] (#1+1,#2+0.5) -- (#1+1,#2+0) ;}
\newcommand{\robibluehauttgauche}[2]{
\fill[blue!40] (#1+2,#2+1) rectangle (#1+0.5,#2+1.5) ;
\fill[blue!40] (#1+1,#2+0) rectangle (#1+0.5,#2+1.5) ;
\draw (#1,#2) rectangle (#1+2,#2+2) ;
\draw [-latex] (#1+0.5,#2+1.5) -- (#1+0.5,#2+0) ; 
\draw [-latex] (#1+0.5,#2+1.5) -- (#1+2,#2+1.5) ;
\draw [-latex] (#1+1,#2+1) -- (#1+1,#2+0) ; 
\draw [-latex] (#1+1,#2+1) -- (#1+2,#2+1) ; 
\draw [-latex] (#1+0.5,#2+1) -- (#1+0,#2+1) ; 
\draw [-latex] (#1+1,#2+1.5) -- (#1+1,#2+2) ;}
\newcommand{\robibluehauttdroite}[2]{
\fill[blue!40] (#1+0,#2+1) rectangle (#1+1.5,#2+1.5) ;
\fill[blue!40] (#1+1,#2+0) rectangle (#1+1.5,#2+1.5) ;
\draw (#1,#2) rectangle (#1+2,#2+2) ;
\draw [-latex] (#1+1.5,#2+1.5) -- (#1+1.5,#2+0) ; 
\draw [-latex] (#1+1.5,#2+1.5) -- (#1+0,#2+1.5) ;
\draw [-latex] (#1+1,#2+1) -- (#1+1,#2+0) ; 
\draw [-latex] (#1+1,#2+1) -- (#1+0,#2+1) ; 
\draw [-latex] (#1+1.5,#2+1) -- (#1+2,#2+1) ; 
\draw [-latex] (#1+1,#2+1.5) -- (#1+1,#2+2) ;}

\newcommand{\robiredbasgauche}[2]{
\fill[red!40] (#1+0.5,#2+0.5) rectangle (#1+1,#2+2) ;
\fill[red!40] (#1+0.5,#2+0.5) rectangle (#1+2,#2+1) ;
\draw (#1,#2) rectangle (#1+2,#2+2) ;
\draw [-latex] (#1+0.5,#2+0.5) -- (#1+0.5,#2+2) ; 
\draw [-latex] (#1+0.5,#2+0.5) -- (#1+2,#2+0.5) ;
\draw [-latex] (#1+1,#2+1) -- (#1+1,#2+2) ; 
\draw [-latex] (#1+1,#2+1) -- (#1+2,#2+1) ; 
\draw [-latex] (#1+0.5,#2+1) -- (#1+0,#2+1) ; 
\draw [-latex] (#1+1,#2+0.5) -- (#1+1,#2+0) ;}
\newcommand{\robiredbasdroite}[2]{
\fill[red!40] (#1+1.5,#2+0.5) rectangle (#1+1,#2+2) ;
\fill[red!40] (#1+1.5,#2+0.5) rectangle (#1+0,#2+1) ;
\draw (#1,#2) rectangle (#1+2,#2+2) ;
\draw [-latex] (#1+1.5,#2+0.5) -- (#1+1.5,#2+2) ; 
\draw [-latex] (#1+1.5,#2+0.5) -- (#1+0,#2+0.5) ;
\draw [-latex] (#1+1,#2+1) -- (#1+1,#2+2) ; 
\draw [-latex] (#1+1,#2+1) -- (#1+0,#2+1) ; 
\draw [-latex] (#1+1.5,#2+1) -- (#1+2,#2+1) ; 
\draw [-latex] (#1+1,#2+0.5) -- (#1+1,#2+0) ;}
\newcommand{\robiredhautgauche}[2]{
\fill[red!40] (#1+2,#2+1) rectangle (#1+0.5,#2+1.5) ;
\fill[red!40] (#1+1,#2+0) rectangle (#1+0.5,#2+1.5) ;
\draw (#1,#2) rectangle (#1+2,#2+2) ;
\draw [-latex] (#1+0.5,#2+1.5) -- (#1+0.5,#2+0) ; 
\draw [-latex] (#1+0.5,#2+1.5) -- (#1+2,#2+1.5) ;
\draw [-latex] (#1+1,#2+1) -- (#1+1,#2+0) ; 
\draw [-latex] (#1+1,#2+1) -- (#1+2,#2+1) ; 
\draw [-latex] (#1+0.5,#2+1) -- (#1+0,#2+1) ; 
\draw [-latex] (#1+1,#2+1.5) -- (#1+1,#2+2) ;}
\newcommand{\robiredhautdroite}[2]{
\fill[red!40] (#1+0,#2+1) rectangle (#1+1.5,#2+1.5) ;
\fill[red!40] (#1+1,#2+0) rectangle (#1+1.5,#2+1.5) ;
\draw (#1,#2) rectangle (#1+2,#2+2) ;
\draw [-latex] (#1+1.5,#2+1.5) -- (#1+1.5,#2+0) ; 
\draw [-latex] (#1+1.5,#2+1.5) -- (#1+0,#2+1.5) ;
\draw [-latex] (#1+1,#2+1) -- (#1+1,#2+0) ; 
\draw [-latex] (#1+1,#2+1) -- (#1+0,#2+1) ; 
\draw [-latex] (#1+1.5,#2+1) -- (#1+2,#2+1) ; 
\draw [-latex] (#1+1,#2+1.5) -- (#1+1,#2+2) ;}

\begin{document}
\maketitle
\begin{abstract}
In this text, we prove then that any minimal effective dynamical 
system on a Cantor set $\mathcal{A}^{\N}$ can be simulated by a minimal 
$\Z^3$-SFT, in a sense that we explicit here. This notion 
is a generalization of simulation by factor and sub-action 
defined~\cite{Hochman-2009}.
\end{abstract} 

\section{Introduction}

\subsection{Embedding computations into multidimensional SFT}

A subshift of finite type (SFT) 
is a discrete 
dynamical system defined as the shift action on a set of 
layouts of symbols in a finite set on an infinite regular grid ($\Z^d$) 
and local rules on symbols.
Subshifts of finite type are central objects in the symbolic dynamics 
field and are studied through various dynamical properties: possible values of 
topological invariants, existence of periodic orbits, possible periods, etc.
When $d=1$ their properties 
are related to algebra. Many results provided evidence that the properties 
of multidimensional SFT ($d \ge 2$) are related to computability theory.
This evidence comes from characterization results of topological invariants 
using embedding Turing computations in multidimensional SFT~\cite{Hochman-Meyerovitch-2010}~\cite{Meyerovitch2011}, 
using methods developed by R. Berger~\cite{Berger66} and R. Robinson~\cite{R71} in order to prove undecidability results.

A notable result by M. Hochman~\cite{Hochman-2009} states the shift actions 
obtained by factor and one-dimensional projective sub-action of tridimensional SFT 
are exactly the effective $\Z$-subshifts. This result was proved to be still true 
for bidimensional SFT independently by N. Aubrun and M. Sablik~\cite{Aubrun-Sablik-2010} and 
B. Durand and A. Romashchenko~\cite{DRS}. Intuitively, multidimensional subshifts can simulate 
any effective $\Z$-subshift, by factor and sub-action. 

\subsection{Effect of dynamical constraints}

This result was proved more recently to be robust under the constraint of minimality by 
B. Durand and A. Romashchenko~\cite{DR17}: they proved that any minimal effective $\Z$-subshift 
can be simulated in the same way by minimal multidimensional SFT. This is not however a characterization, since 
the sub-action of a minimal SFT is 
not forced to be minimal.
This result follows a recent trend in symbolic dynamics which consists in understanding 
to which extent the constructions embedding Turing computations in multidimensional SFT are robust 
under dynamical constraints. This question originates in~\cite{Hochman-Meyerovitch-2010}, 
where the authors wonder if their result is still true under transitivity constraint. The effect 
of a stronger constraint, block gluing, 
was notably studied by R. Pavlov and M. Schraudner~\cite{Pavlov-Schraudner-2014}, who provided 
a realization of a sub-class of the computable numbers as entropy of block gluing multidimensional subshift. 
This constraint imposes that any two patterns in the language of the subshift appear on any couple 
of positions in some configuration, provided that the space between the two patterns is greater 
than a constant. 
The characterization of these numbers is still an open problem. 
On the other hand, the authors of the present text provided some method in order to prove 
the robustness of the characterization in~\cite{Hochman-Meyerovitch-2010} 
to a linear version (where the minimal distance imposed 
between the pattern depends linearly on the size of these patterns) of the block gluing~\cite{GS17}.

\subsection{Extended abstract and statement of the result}

In this text, we define a notion of simulation of an effective dynamical 
system on the Cantor sets $\mathcal{A}^{\N}$ by a multidimensional SFT in order 
to study the possible effective dynamical systems (not only subshifts) that multidimensional SFT can 
simulate. This definition consists in the existence of a computable map which commutes 
with the shift action in only one direction. It was 
proved in~\cite{Hochman-2009} that tridimensional subshifts 
simulate effective dynamical systems on the 
Cantor sets $\mathcal{A}^{\N}$. We prove 
here that this statement 
is robust to the minimality property: 

\begin{theorem}
\label{theorem.main.introduction}
Any minimal effective dynamical system $(Z,f)$
on a Cantor set $\mathcal{A}^{\N}$, where $\mathcal{A}$ is 
a finite set, can be simulated 
by a minimal $\Z^3$-SFT. 
\end{theorem}

\subsection{Specific information processing phenomena}

The construction involved in the proof of this theorem relies on R. Robinson's and M. Hochman, 
T. Meyerovitch's constructions. 
However, it exhibits specific information processing 
phenomena.

Some of them can be observed 
already in the literature. In 
particular, the computing units in which are embedded Turing 
computations are decomposed in sub-units used for specific functions, as in B. Durand and A. Romashchenko's constructions. Thus the construction exhibits functional 
specialisation of the computing units. However, this functional specialisation
does not rely on a universal machine. 

Some others are specific to the construction presented 
here:

\begin{enumerate}
\item The natural way to support effective dynamical systems implies specific degenerated behaviors 
which consist of the presence of a full shift 
supported by some infinite computing units. A significant part of the construction 
is devoted to the simulation of these degenerated behaviors. 
\item As in the authors' construction in~\cite{GS17}, the density of information exchanges is higher than in 
the constructions of B. Durand 
and A. Romashchenko. This density imposes that the information exchanges between the sub-units can not follow 
a tight ''$U$'' shape, in order to avoid breaking the minimality property. 
\item For the same reason, some information characterizing the behavior of the computing units are determined by a hierarchical signaling process which is 
\textit{imposed to avoid some regions} of the construction. 
\item We also provide to some computing sub-units the 
choice of direction of transport 
of information. In~\cite{GS17ED},
this phenomenon appears for the 
propagation of an error signal triggered 
by the machines. Here it is used for 
the attribution of a value to some 
bits related to the simulation of the 
system. This information has to be transported through 
''random'' channels.
\item Moreover, we use, as in~\cite{GS17ED}, Fermat numbers in order 
to desynchronize some counters involved in the construction. This time 
we have to include counters having 
different functions.
\end{enumerate}

\subsection{Comparison with existing constructions}

It is noteworthy that the construction of~\cite{GS17ED} and 
the construction of this present text are really different, 
although using common mechanisms for the machines and the linear 
counters ruling the machines. 
Thus, they exhibit a strong difference with 
constructions of~\cite{Aubrun-Sablik-2010} and \cite{DRS},
in which the characterization of the possible values 
of the entropy derives from the simulation of effective 
subshifts by factor and sub-action.
The reason is that these deal with different degenerated 
behaviors. We don't know if there is a method from which derive 
the simulation theorem and the characterization
of the entropy dimensions under constraint of minimality.

Moreover, the specificity of the proof 
of Theorem~\ref{theorem.main.introduction} is the simulation 
of any effective system on a Cantor set (not only 
subshifts), with the cost of increasing the 
dimension. However, it seems reasonable that the same principles 
of simulation of degenerated behaviors and 
suspended counters can be applied on Hochman's 
reformulation of B. Durand, A. Romashchenko's 
and N. Aubrun, M. Sablik's result in~\cite{Hochmanwords}
to recover the result of~\cite{DR17}. At least it is direct 
to recover the simulation of minimal effective one-dimensional subshifts 
by minimal $\Z^3$-SFT by factor and projective sub-action (instead of 
$\Z^2$-SFT).

\begin{remark}
Another consequence of Theorem~\ref{theorem.main.introduction}, is that one can not decide 
if a $\Z^3$-SFT is minimal or not. 
\end{remark}

\section{Definitions}\label{sec.definition}

In this section we recall some definitions on $\Z^d$-subshifts dynamics. 
Denote $\vec{e}^1 , ... , \vec{e}^d$ the canonical generators 
of $\Z^d$.

\subsection{Subshifts as dynamical systems}

Let $d\ge 1$.
Let $\A$ be a finite set (alphabet). 
A \define{configuration} $x$ is an element of $\A^{\Z^d}$. 
The space $\A^{\Z^d}$ is endowed with the 
product topology derived from the discrete topology on $\A$. 
For this topology,  $\A^{\Z^d}$ is a compact space. 
This \define{shift action} of $\Z^d$ on $\A^{\Z^d}$ is defined, for 
all $\vec{u} \in \Z^d$ and $x \in \A^{\Z^d}$, by 
\[\s^{\vec{u}} (x)_{\vec{v}} = x_{\vec{u}+\vec{v}}.\]
For all $N \ge 1$, the 
\define{$N$th subaction of the shift} 
is the action $\s_N$ defined for all 
$\vec{u} \in \Z^d$ and $x \in \A^{\Z^d}$, by 
\[\s_N ^{\vec{u}} (x)_{\vec{v}} = 
x_{N.\vec{u}+\vec{v}}.\]

For any finite subset $\U$ of $\Z^d$, we denote 
$x_{\U}$ the \define{restriction} of $x\in\A^{\Z^d}$ to $\U$. 
A \define{pattern} $P$ on \define{support} $\U$ is an element of $\A^{\U}$.
The support of the pattern $P$ is denoted $\supp{P}$.
A block $P$ \textbf{appears} in a configuration $x$ when 
there exists some $\vec{u} \in \Z^d$ such that $x_{\vec{u}+\supp{P}}=P$.
Let us denote, for all $n$, $\U_n^d =\llbracket 
0;n-1 \rrbracket ^d$. A pattern on support $\U_n^d$ is a called 
a \define{n-block}. 
A $n$-block $P$ appears at position $\vec{u} \in\Z^d$ in a 
configuration $x\in\A^{\Z^d}$ 
when $x_{\vec{u}+\U_n^d}=P$. We denote this $P\sqsubset x$.  
A pattern $P$ on support 
$\mathbb{U}$ is a \define{sub-pattern} 
of a pattern $Q$ on support $\mathbb{V}$ when $\mathbb{U} \subset \mathbb{V}$ 
and $Q_{\mathbb{U}}=P$.

A \define{$d$-dimensional subshift} (or $\Z^d$-subshift) 
$X$ is a closed subset of $\A^{\Z^d}$ which 
is invariant under the action of the shift. This means that  
$\sigma (X) \subset X$. The couple $(X,\sigma)$ is a dynamical system. 
Any subshift $X$ can be defined by a set of \define{forbidden 
patterns}: this means that $X$ is the set of configurations 
where none of the forbidden patterns appears. 
Formally, there exists $\mathcal{F}$ a set of patterns 
such that :
$$X=X_{\mathcal{F}}:=\left\{x\in\A^{\Z^d}: \textrm{ for 
all $P\in\mathcal{F}$, $P \not \sqsubset x$}\right\}.$$
We also use the term \define{local rule} (or rule) the act of forbidding a pattern. 
Sometimes we forbid patterns by imposing the rule that 
patterns on a given support are in a restricted set of patterns on this support. 

If the subshift can be defined by a finite set of forbidden patterns, it is called 
a \define{subshift of finite type} (SFT for short). The \textbf{order} of a 
SFT is the smallest $r$ such 
that it can be defined by forbidden $r$-blocks.

A pattern \define{appears} in a subshift $X$ if there is a configuration of $X$ in which 
it appears. The set of patterns which appear in $X$ 
is called the \define{language} of $X$, denoted $\mathcal{L}(X)$. 
As well, the alphabet of the subshift $X$ is often denoted $\A_X$. \bigskip

A subshift is said to be \define{minimal} if 
every pattern in the language of $X$ appears in 
every configuration of $X$. This means that the 
trajectory of any point of the subshift under the 
shift action is dense in the subshift. \bigskip

In this article, we construct subshifts on some alphabet $\A$ which is 
a product of alphabets $\A = \A_1 \times ... \times \A_k$. 
We call informally the $i^{\textrm{th
}}$ layer of this subshift the space of the projections of a configuration written on the $i^{
\textrm{th}}$ alphabet $\A_i$, and we describe the subshifts layer by layer, 
giving the symbols, the rules that are internal to the layer, 
and the rules making interact the layer with the previous ones. Sometimes, for clarity 
of the construction, a layer is described as a product of layers : we call it a sublayer 
of the product subshift.

We say that a \textbf{symbol} on a position $\vec{u} \in \Z^d$ 
is transmitted through a rule to a neighbor when the rule impose that 
the symbol on this neighbor is the same.

\subsection{Simulation of effective dynamical systems 
on the Cantor set}

A Cantor set is a set $\A^{\N}$, where $\A$ is some finite set, 
that we consider here as a topological space 
with the power of the discrete topology.
A \textbf{cylinder} is some set 
\[\{x \in \A^{\N} \ : \ \forall i \le n-1, \
x_i = p_i\},\]
where $n\ge 1$ and $p$ is some length $n$ 
word on the alphabet $\mathcal{A}$. This set is denoted $[p]$.

\begin{definition}
Let $Z$ be a closed subset of $\A^{\N}$ and $f$ a continuous function $f: Z \rightarrow Z$. 
The dynamical system $(Z,f)$ is said to be 
\textbf{minimal} 
when for any 
\begin{enumerate}
\item configuration $x \in Z$, 
\item $[q]$ a cylinder such that 
$[q] \cap Z \neq \emptyset$, 
\end{enumerate} 
there exists some $n \ge 0$ such that 
\[f^{n} (x) \cap [q] \neq \emptyset.\]
\end{definition}

\begin{definition}
Let $Z$ be a closed subset of $\A^{\N}$. This set is said to be 
\textbf{effective} 
when there is a Turing machine which on input $n$ outputs 
some word $p_n$ on alphabet $\A$ such that 
\[Z^c = \bigcup_n [p_n].\]
\end{definition}

\begin{definition}
Let $Z$ be an effective closed subset of $\A^{\N}$ and 
$f$ be a continuous function $Z \rightarrow Z$. The 
dynamical system $(Z,f)$ is said to be 
\textbf{effective} when there exists a Turing machine 
which on input $n$ outputs some word $p_n$ such that 
\[\{(x_i,f(x)_i)_i \ : \ x \in Z\}^c = \bigcup_n [p_n].\]
\end{definition}

Let us choose for the following an effective encodings 
of patterns on an alphabet $\mathcal{A}$ on $\Z^d$ into $\N$. Any algorithm 
manipulating patterns (when taken as input or output) manipulates their 
representation in $\N$.

\begin{definition}
Let $\varphi$ be a function $X \rightarrow X'$, where $X,X'$ are 
two effective subshifts respectively on $\Z^d$ and $\Z^{d'}$. We say that $\varphi$ is \textbf{computable}
when there exists an algorithm that taking as input an integer $k$ 
computes some $\psi (k)$ and outputs $\varphi (x)_{\llbracket -k , k \rrbracket ^{d'}}$ 
out of knowledge of the pattern $x_{\llbracket -\psi (k), \psi (k) \rrbracket ^d}$.
\end{definition}

\begin{definition}
Let $(Z,f)$ be
an effective dynamical system such that $f$ is onto,
and $X$ be a $\Z^d$-SFT. We say that $X$ \define{simulates} 
the dynamical system $(Z,f)$ if there exists some computable onto function 
$\varphi : X \rightarrow \A^{\N}$ such that the following diagram commutes : 
\[\xymatrix{
  X \ar[d]_{\varphi} \ar[r]^{\sigma^{{\vec{e}}^d}}  & X \ar[d]^{\varphi} \\
  Z \ar[r]_{f} & Z }
\]
The function $\varphi$ is called the \define{simulation function}.
\end{definition}

\begin{example}
A \textit{factor map} between two subshifts is a map which commutes with 
the shift action. The \textit{projective sub-action} on $\Z^{d'}$, $d'\le d$ of 
a $\Z^d$-subshift $X$ is the projection of $X$ on the subshift $Z$ of $\mathcal{A}^{\Z^{d'}}$, where 
$\mathcal{A}$ is the alphabet of $X$, and  
\[Z= \left\{z \in \mathcal{A}^{\Z^{d'}} : \exists x \in X : \forall i \le d', z_i = x_i\right\}\]
Both factor and projective sub-action are simulation functions for $f$ the shift map.
\end{example}

As a consequence, one can formulate the following theorem:

\begin{theorem}[\cite{Hochman-2009},\cite{Aubrun-Sablik-2010},\cite{DRS}]
Let $d \ge 2$. The subshifts that can be simulated by a $\Z^d$-SFT by factor and projective sub-action are 
the effective subshifts on $\Z^k$ for some $k \le d-1$.
\end{theorem}

\section{Simulation 
of minimal effective dynamical 
systems on the Cantor sets by minimal $\Z^3$-SFT.}

The aim of this section is to give 
an outline of the proof of 
Theorem~\ref{theorem.main.introduction}.

\subsection{Statement and outline}

\begin{reptheorem}{theorem.main.introduction}
Let $\A$ be some finite set, and 
$f : \A^{\N} \rightarrow \A^{\N}$ any effective 
function, such that the dynamical system $(f,\A^{\N})$ is minimal. 
Then there exists some minimal $\Z^3$-SFT $X$ 
which simulates this system.
\end{reptheorem}

The structure of the subshifts constructed for the proof is an infinite 
stack, in direction $\vec{e}^3$, of a rigid version of the Robinson subshift presented earlier in~\cite{GS17}.
Most of the mechanisms are described in sections of $\Z^3$ 
orthogonal to this direction, since this is the direction 
of time for the simulated system.

The idea is to encode configurations 
of $\mathcal{A}^{\N}$ into the hierarchical structures of
the Robinson subshift in each section of $\Z^3$ orthogonal 
to $\vec{e}^3$: each level $n$ of the hierarchy supports 
the $n$th bit of the configuration. We implement 
Turing computations in each of the copies in order to ensure 
that the configuration $y$ in any copy is equal to $f(x)$, 
where $x$ is the configuration in the copy just below in 
the direction $\sigma^{\vec{e}^d}$. Moreover, it verifies 
that $x$ is in $Z$. Both these tasks are possible since the 
dynamical system $(Z,f)$ is effective.

There are two obstacles to the minimality property in this scheme. 
The first one 
comes from the implementation of the machines.
This is solved 
as in~\cite{GS17ED}: we use counters that alternate all the possible 
computations and use error signals in order to take into account 
only well initialized ones. The main difference with construction of~\cite{GS17ED} 
is that the counters and machines 
are implemented in a two dimensional section. 
Since we can not superimpose the counters and computations of the machines 
without breaking the minimality, we subdivide the computing units into a finite 
set of sub-units. Each of the sub-units has a specific function: information 
transport, incrementation of the counter, computations, etc. This decomposition 
is illustrated on Figure~\ref{figure.functional.diagram.outline}.

The main other obstacle, which is specific to this construction, 
comes from the encoding
of configurations in $\mathcal{A}^{\N}$. The bits supported by 
infinite computing units are not under control 
of the computing machines. This implies that an infinite stack of these computing units 
supports a full shift on the alphabet $\mathcal{A}$. We use the idea of simulation 
of degenerated behaviors
used previously in~\cite{GS17} and~\cite{GS17ED} and underlying the construction 
of~\cite{DR17} and simulate the one dimensional full shift on 
a larger alphabet $\mathcal{E}$ on the computing units having odd 
levels.
The simulated effective dynamical system is supported by the even ones. 
Since the information of simulation can not meet the bits generating 
the full shift or the effective dynamical system, we use an avoiding hierarchical 
signaling process in order to give access to each of the computing units 
to the information of the parity of its level. Moreover, 
a sub-unit of each computing unit supports a counter which is incremented 
in direction $\vec{e}^3$ and synchronized in the other directions. We choose 
the cardinality of $\mathcal{E}$ so that this 
counter has a Fermat number has period in order 
to keep the minimality property, in the same spirit as in~\cite{GS17ED}. 
This counter is used to determine the system bits in odd levels, 
through transport of information to the location of the system bits. 
We need two possible channels for this transport, only 
one of them being used, so that over infinite 
computing units, 
even if no information flux is detected locally through the border, 
it is still possible to see this area as part of a computing unit where 
the bits are simulated by the counter. The choice of the channel 
is done according to the class of the level modulo $4$, which 
is transported to the counter area in the same way as the parity.
Moreover, this counter is synchronized over each section 
of $\Z^3$ orthogonal to $\vec{e}^3$. The same channels are used for 
this purpose, in order to synchronize these bits, which is sufficient 
to ensure the synchronization of the counters.
We do not synchronize these counters by transporting their values 
since this is not possible to guarantee 
that the values of this counter along direction $\vec{e}^3$ are coherent along 
the transport. 

Another noteworthy aspect of this construction 
is that the counters used to simulate 
the degenerated behaviors that appear 
here are incremented in a more compact 
way than it was in~\cite{GS17ED}. In 
this setting, the counter is not robust 
to the minimality, because of the suspension 
mechanism. We thus naturally interpret 
this version of the mechanism as \textit{dominant} 
with respect to the other one, since 
it can replace it in 
the construction of~\cite{GS17ED}.

\subsection{Description of the layers}

Let $\mathcal{A}$ be a finite set and $(Z,f)$ a minimal 
effective dynamical system on $\mathcal{A}^{\N}$.
Let us construct a minimal $\Z^3$-SFT $X_{(Z,f)}$ 
which simulates this system.

This subshift is constructed as a product of various layers, 
as follows :

\begin{itemize}
\item \textbf{\textit{Structure layer [\ref{section.structure}]:}}
This layer has  a first sublayer having alphabet ${\mathcal{A}}_{\texttt{R}}$, the alphabet 
of $X_{\texttt{R}}$, the rigid version of the Robinson subshift, whose description 
can be found in Annex~\ref{sec.valued.robinson}.

The symbols are transported along $\vec{e}^3$ and follow the rules of the subshift $X_{\texttt{R}}$ 
in the other directions. 

When not specified, the configurations in 
the other layers are described in a section 
\[\Z_c^2:=\Z.\vec{e}^1 + \Z.\vec{e}^2+ c.\vec{e}^3\]
of $\Z^3$, with $c \in \Z$, and identified on $\Z_{c+1}^2$ and $\Z_{c}^2$ 
for all $c \in \Z$, using local rules.

Moreover, in all the sections, 
the order $n \ge 3$ cells are 
subdivided into $64$ parts, using 
a hierarchical signaling process, in a second sublayer. These subcells (or organites) are
coded by two elements of $\Z/8\Z$, thought as projective coordinates: 
the first element of this pair
corresponds to the position 
of the sub-unit in the unit in direction $\vec{e}^1$. The other 
one corresponds to the position in direction $\vec{e}^2$. Each 
of the sub-units is attributed with specific functions amongst the following ones:
\begin{itemize} 
\item computations, 
\item incrementation of the linear counter,
\item information transport,
\item deviation of the information transfer, or 
extraction of some information (demultiplexers),
\item support the system counter. 
\end{itemize}
See Figure~\ref{figure.functional.diagram.outline} for 
an illustration.
The linear counter is the counter which codes for the behaviors 
of the computing machines. Its mechanism is similar to the linear 
counter in~\cite{GS17ED}. The system counter is the one 
which is used to simulate the full shift on $\mathcal{A}$ on odd 
levels. The symbols in this layer are identified in the 
direction $\vec{e}^3$.

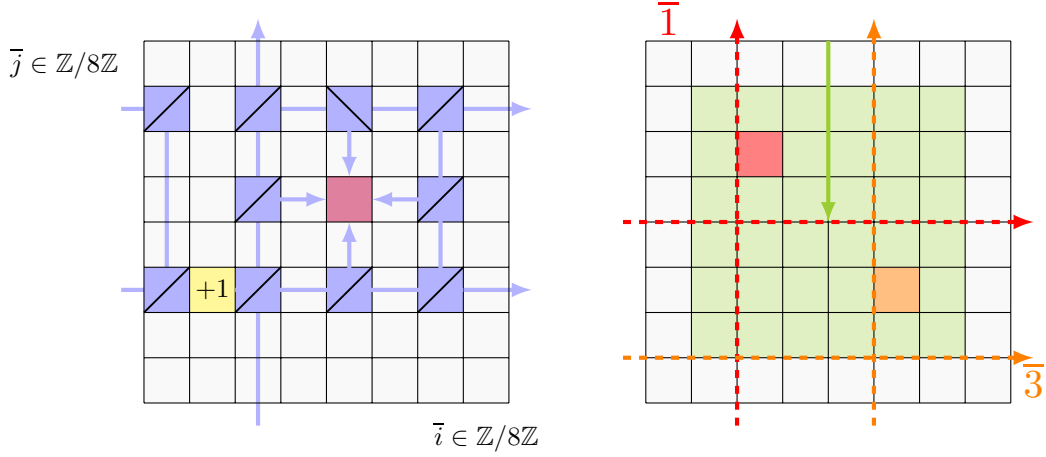
\begin{figure}[ht]
\[\begin{tikzpicture}[scale=0.6]
\begin{scope}
\fill[gray!5] (0,0) rectangle (8,8);
\fill[blue!30] (0,2) rectangle (1,3);
\fill[yellow!50] (1,2) rectangle (2,3);
\fill[purple!50] (4,4) rectangle (5,5);
\fill[blue!30] (2,2) rectangle (3,3);
\fill[blue!30] (4,6) rectangle (5,7);
\fill[blue!30] (4,2) rectangle (5,3);
\fill[blue!30] (2,6) rectangle (3,7);
\fill[blue!30] (2,4) rectangle (3,5);
\fill[blue!30] (6,6) rectangle (7,7);
\fill[blue!30] (6,4) rectangle (7,5);
\fill[blue!30] (6,2) rectangle (7,3);

\fill[blue!30] (0,6) rectangle (1,7);

\draw[line width = 0.55mm,color=blue!30] (2.5,2.5) -- (2.5,6.5);
\draw[line width = 0.55mm,color=blue!30] (-0.5,6.5) -- (0.5,6.5);

\draw[line width = 0.55mm,color=blue!30,-latex] (0.5,6.5) -- (0.5,2.5);

\draw[line width = 0.55mm,color=blue!30,-latex] (6.75,2.5) -- (8.5,2.5);
\draw[line width = 0.55mm,color=blue!30,-latex] (6.75,6.5) -- (8.5,6.5);

\draw[line width = 0.55mm,color=blue!30] (-0.5,2.5) -- (1,2.5);
\draw[-latex,line width = 0.55mm,color=blue!30] 
(2.5,6.5) -- (2.5,8.5);
\draw[-latex,line width = 0.55mm,color=blue!30] (2.5,-0.5) 
-- (2.5,2.5);
\draw[line width = 0.55mm,color=blue!30] (2.5,6.5) -- (7,6.5) ;
\node at (1.5,2.5) {$+1$};
\draw (0,0) grid (8,8);

\draw[line width = 0.55mm,color=blue!30] (2,2.5) -- (7,2.5);
\draw[-latex,line width = 0.55mm,color=blue!30] (4.5,2.5) -- (4.5,4);
\draw[-latex,line width = 0.55mm,color=blue!30] (4.5,6.5) -- (4.5,5);
\draw[-latex,line width = 0.55mm,color=blue!30] (2.5,4.5) -- (4,4.5);

\draw[line width = 0.55mm,color=blue!30] (6.5,6.5) -- (6.5,2.5);
\draw[line width = 0.55mm,color=blue!30,-latex] (6.5,4.5) -- (5,4.5);

\draw[line width = 0.25mm] (2,2) -- (3,3);
\draw[line width = 0.25mm] (4,2) -- (5,3);
\draw[line width = 0.25mm] (2,6) -- (3,7);
\draw[line width = 0.25mm] (4,7) -- (5,6);
\draw[line width = 0.25mm] (0,2) -- (1,3);
\draw[line width = 0.25mm] (2,4) -- (3,5);
\draw[line width = 0.25mm] (0,6) -- (1,7);
\draw[line width = 0.25mm] (6,2) -- (7,3);
\draw[line width = 0.25mm] (6,4) -- (7,5);
\draw[line width = 0.25mm] (6,6) -- (7,7);

\node at (7.5,-0.75) {$\overline{i} \in \Z/8\Z$};
\node at (-1.75,7.5) {$\overline{j} \in \Z/8\Z$};
\end{scope}

\begin{scope}[xshift= 11cm]

\fill[gray!5] (0,0) rectangle (8,8);
\fill[YellowGreen!30] (1,1) rectangle (7,7);
\fill[red!50] (2,5) rectangle (3,6);
\fill[orange!50] (5,2) rectangle (6,3);

\draw (0,0) grid (8,8);

\draw[YellowGreen, line width=0.55mm,-latex] (4,8) -- (4,4);

\draw[line width = 0.55mm,color=red,dashed,-latex] (2,-0.5) -- (2,8.5);
\draw[line width = 0.55mm,color=red,dashed,-latex] (-0.5,4) -- (8.5,4);
\draw[line width = 0.55mm,color=orange,dashed,-latex] (-0.5,1) -- (8.5,1);
\draw[line width = 0.55mm,color=orange,dashed,-latex] (5,-0.5) -- (5,8.5);

\node[red,scale=1.5] at (0.5,8.5) {$\overline{1}$};
\node[orange,scale=1.5] at (8.5,0.5) {$\overline{3}$};

\end{scope}

\end{tikzpicture}\]
\caption{\label{figure.functional.diagram.outline} Illustration 
of the subdivision of a computing unit into sub-units and 
preview of the functional diagram. The arrows designate 
information transports: the blue ones 
correspond to the linear counter, the orange ones 
to the system counters in order to synchronize them, 
and the green ones to the transport of the modularity mark.
The colored sub-units 
execute specific functions: the yellow 
one $((\overline{i},
\overline{j})=(\overline{2},\overline{3})$) supports 
the incrementation of the linear 
counter, the purple one $((\overline{i},
\overline{j})=(\overline{6},\overline{5})$) supports 
computations. The blue ones are demultiplexers (extracting 
information). On the right, the orange area (sub-units $(\overline{2},\overline{5}), (\overline{5},\overline{2})$) is the support 
of the incrementation of the system counter. 
The dashed arrows are potential transports of information. The green area 
is the localisation of the modularity mark inside the cell.}
\end{figure}

\item \textbf{\textit{System bits layer [\ref{section.system.bits}]:}}
This layer supports the system bits, which generate 
the full shift or the simulated effective dynamical system. These bits
are synchronized over a level. We use only two sides of the computing 
units in order to transport these bits for the synchronization, as illustrated 
on Figure~\ref{figure.synchronization.system.bits.outline}.

\begin{figure}[ht]

\[\begin{tikzpicture}[scale=0.3]

\fill[gray!20] (0,0) rectangle (6,6);
\draw (0,0) rectangle (6,6);
\draw (0.75,0.75) rectangle (2.25,2.25);
\draw (3.75,3.75) rectangle (5.25,5.25);
\draw (0.75,3.75) rectangle (2.25,5.25);
\draw (3.75,0.75) rectangle (5.25,2.25);

\draw[dashed, line width=0.5mm] (0,-2) -- (0,8);
 \draw[dashed, line width=0.5mm] (-2,0) -- (8,0);
\end{tikzpicture}\]
\caption{\label{figure.synchronization.system.bits.outline} Preview 
of the synchronization mechanism of system bits. The gray square 
designates a computing units and the dashed lines the locations of 
the system bits.}
\end{figure}
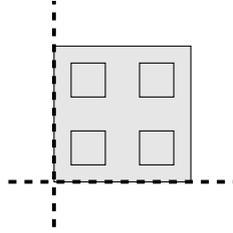

\item \textbf{\textit{Functional areas layer [\ref{section.functional.areas}]:}} 
In this layer, we attribute a 
function to the proper positions 
of the computing units: transport 
of information, step of 
computation. This allows 
the localization of the various 
information processes in the cells, 
outside the machine sub-unit.
For this 
purpose, we
use a hierarchical signaling process.
Another process, similar to the 
one used in~\cite{R71}, is used 
to separate inside the machine sub-unit 
the columns and lines into two sets: active and inactive ones. 
Only the active ones are allowed to 
transport information. We use 
this in order to ensure the minimality.

\item \textbf{\textit{Modularity marks [\ref{section.parity.marks}]:}}
In this layer, we superimpose marks on some parts 
of the border of the computing units. This mark 
is $\overline{n} [4]$, the arithmetic class of $n$ 
modulo $4$, where $n$ is the level of the computing unit. 
The hierarchical signaling process which allows this 
marking is different from the one used for the 
functional areas since the information transport 
is permitted only from some unit to the one 
just above in the hierarchy, amongst those who share 
this one as parent. See an illustration 
on Figure~\ref{figure.hierarchical.process.parity.marks.outline}.
Moreover, we impose the 
mark $\overline{0}$ on order $0$ units. All the other marks are 
determined by the hierarchical process. 
The modularity mark serves to impose the presence of the system counter when 
the level is odd. It prevents this counter when the level is even. When the 
level is odd, it also serves to determine the channels by which the value of the counter 
is transported.

\begin{figure}[ht]

\[\begin{tikzpicture}[scale=0.3]

\fill[gray!20] (0,0) rectangle (6,6);
\draw (0,0) rectangle (6,6);
\draw (0.75,0.75) rectangle (2.25,2.25);
\draw (3.75,3.75) rectangle (5.25,5.25);
\draw (0.75,3.75) rectangle (2.25,5.25);
\draw (3.75,0.75) rectangle (5.25,2.25);

\draw (6.75,6.75) rectangle (8.25,8.25);
\draw (-2.25,-2.25) rectangle (-0.75,-0.75);

\draw (6.75,3.75) rectangle (8.25,5.25);
\draw (6.75,0.75) rectangle (8.25,2.25);
\draw (6.75,-2.25) rectangle (8.25,-0.75);

\draw (3.75,-2.25) rectangle (5.25,-0.75);
\draw (0.75,-2.25) rectangle (2.25,-0.75);

\draw (3.75,6.75) rectangle (5.25,8.25);
\draw (0.75,6.75) rectangle (2.25,8.25);
\draw (-2.25,6.75) rectangle (-0.75,8.25);

\draw (-2.25,3.75) rectangle (-0.75,5.25);
\draw (-2.25,0.75) rectangle (-0.75,2.25);

\draw[dashed,-latex] (4.5,5.25) -- (4.5,6);
\draw[dashed,-latex] (5.25,4.5) -- (6,4.5);

\draw[line width =0.4mm] (6,6) -- (3,6);
\draw[line width =0.4mm] (6,6) -- (6,3);

\draw[line width =0.4mm] (8.25,8.25) -- (7.5,8.25);
\draw[line width =0.4mm] (8.25,8.25) -- (8.25,7.5);
\draw[line width =0.4mm] (8.25,5.25) -- (7.5,5.25);
\draw[line width =0.4mm] (8.25,5.25) -- (8.25,4.5);
\draw[line width =0.4mm] (8.25,2.25) -- (7.5,2.25);
\draw[line width =0.4mm] (8.25,2.25) -- (8.25,1.5);
\draw[line width =0.4mm] (8.25,-0.75) -- (7.5,-0.75);
\draw[line width =0.4mm] (8.25,-0.75) -- (8.25,-1.5);

\draw[line width =0.4mm] (5.25,8.25) -- (4.5,8.25);
\draw[line width =0.4mm] (5.25,8.25) -- (5.25,7.5);
\draw[line width =0.4mm] (5.25,5.25) -- (4.5,5.25);
\draw[line width =0.4mm] (5.25,5.25) -- (5.25,4.5);
\draw[line width =0.4mm] (5.25,2.25) -- (4.5,2.25);
\draw[line width =0.4mm] (5.25,2.25) -- (5.25,1.5);
\draw[line width =0.4mm] (5.25,-0.75) -- (4.5,-0.75);
\draw[line width =0.4mm] (5.25,-0.75) -- (5.25,-1.5);

\draw[line width =0.4mm] (2.25,8.25) -- (1.5,8.25);
\draw[line width =0.4mm] (2.25,8.25) -- (2.25,7.5);
\draw[line width =0.4mm] (2.25,5.25) -- (1.5,5.25);
\draw[line width =0.4mm] (2.25,5.25) -- (2.25,4.5);
\draw[line width =0.4mm] (2.25,2.25) -- (1.5,2.25);
\draw[line width =0.4mm] (2.25,2.25) -- (2.25,1.5);
\draw[line width =0.4mm] (2.25,-0.75) -- (1.5,-0.75);
\draw[line width =0.4mm] (2.25,-0.75) -- (2.25,-1.5);

\draw[line width =0.4mm] (-0.75,8.25) -- (-1.5,8.25);
\draw[line width =0.4mm] (-0.75,8.25) -- (-0.75,7.5);
\draw[line width =0.4mm] (-0.75,5.25) -- (-1.5,5.25);
\draw[line width =0.4mm] (-0.75,5.25) -- (-0.75,4.5);
\draw[line width =0.4mm] (-0.75,2.25) -- (-1.5,2.25);
\draw[line width =0.4mm] (-0.75,2.25) -- (-0.75,1.5);
\draw[line width =0.4mm] (-0.75,-0.75) -- (-1.5,-0.75);
\draw[line width =0.4mm] (-0.75,-0.75) -- (-0.75,-1.5);

\draw[dashed] (3,3) -- (3,9); 
\draw[dashed] (3,3) -- (9,3);
\draw[-latex] (3,6) -- (3,9);
\draw[-latex] (6,3) -- (9,3);
\end{tikzpicture}\]
\caption{\label{figure.hierarchical.process.parity.marks.outline} Preview 
of the hierarchical signaling process 
which allows modularity marking. The thick lines
designates the location of the marking and the 
arrows designate the direction of communication.}
\end{figure}
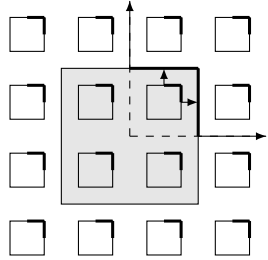

\item \textbf{\textit{Linear counter layer [\ref{section.linear.counter}]:}}
This layer supports the incrementation mechanism of the linear 
counter. The value of this counter is a word on 
the alphabet 
\[\mathcal{A}_c = \mathcal{A}' \times \mathcal{Q}^3 \times \{\rightarrow,\leftarrow\}
\times \mathcal{D},\] 
where $\mathcal{A}$' is the alphabet of the Turing machine used in 
the construction, after completing this alphabet so that its cardinality 
is $2^{2^l}$ for some integer $l$. The set $\mathcal{Q}$ is 
the state set of the machine after similar completion, so that its 
cardinality is $2^{2^l}$. We complete these two sets so that the states 
and letters of the initial machine can not be transformed into 
the additional letters and states. We need this so that the machine 
behaves as the initial one when it is initialized in the initial state and 
with empty tape. The arrows designate the propagation direction of 
an error signal triggered by the machine when it has a head ending 
computation in error state. The last set has cardinality 
$2^{4.2^{l}-2}$. This set is included in the product 
so that $\mathcal{A}_c$ has cardinality a double 
power of $2$: $2^{8.2^{l}}=2^{2^{l+3}}$.

The incrementation mechanism of this counter 
consists of an adding machine. This incrementation 
goes in direction $\vec{e}^1$, and the counter 
is synchronized in direction $\vec{e}^2$, $\vec{e}^3$.
It is suspended for one step when the counter reaches 
its maximal value. As a consequence the period of 
the counter attached to level $n \ge 3$ computing units
is 
\[2^{2^{l+3+n+1-3}}+1 = 
2^{2^{l+n+1}}+1.\] 
Thus these periods are 
different Fermat numbers. We use this fact in 
order to ensure the minimality property.

\item \textbf{\textit{System counter layer 
[\ref{section.system.counter}]:}}
This layer supports the incrementation mechanism of 
the system counter, 
supported by 
the system counter areas 
when the level is odd. When 
the modularity mark is $\overline{1}$ 
the sub-unit is $(\overline{2},\overline{5})$ 
and it is $(\overline{5},\overline{2})$ when the mark is $\overline{3}$. 

The value 
of this 
counter is the product of
two words on an alphabet
$\mathcal{E} ^2$, where $\mathcal{E}$ 
has cardinality $2^{2^{m}}$ and is such that
such that $\mathcal{A} \subset \mathcal{E}$ 
and $m \ge 0$.

In direction $\vec{e}^3$ from 
each $\Z_c^2$ to $\Z_{c+1}^2$, the first word is incremented 
by an adding machine. The second word 
is considered as a word on a discrete torus which consists of the concatenation 
of two words on $\mathcal{E}$, and is at each step rotated as on Figure~\ref{figure.rotation.preview}.

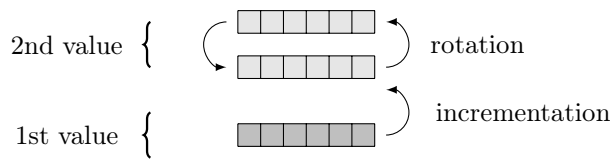
\begin{figure}[ht]
\[\begin{tikzpicture}[scale=0.3]
\fill[gray!20] (0,0) rectangle (6,1);
\draw (0,0) grid (6,1);

\fill[gray!20] (0,-2) rectangle (6,-1);
\draw (0,-2) grid (6,-1);

\fill[gray!50] (0,-5) rectangle (6,-4);
\draw (0,-5) grid (6,-4);

\draw[-latex] (6.5,-1.5) arc (-90:90:1);
\draw[-latex] (6.5,-4.5) arc (-90:90:1);
\draw[-latex] (-0.5,0.5) arc (90:270:1);
\node at (10.5,-0.5) {rotation};
\node at (12.5,-3.5) {incrementation};
\node at (-7.5,-0.5) {2nd value};
\draw[decorate,decoration={brace,raise=0.4cm},line width=0.3mm] 
(-2.5,-1.5) -- (-2.5,0.5);
\node at (-7.5,-4.5) {1st value};
\draw[decorate,decoration={brace,raise=0.4cm},line width=0.3mm] 
(-2.5,-5.5) -- (-2.5,-3.5);
\end{tikzpicture}\]
\caption{\label{figure.rotation.preview} Preview of the system counter mechanism.}
\end{figure}

Moreover, it is incremented as the concatenation of the two words on a segment 
each time the first value reaches its maximal value. 
When the two values are maximal at the same time, the rotation and the incrementation 
mechanisms are suspended for one step. 

The values of this counter are synchronized over a level, using channels 
which depend on the value of the modularity mark. 
They are designated by orange arrows on 
Figure~\ref{figure.functional.diagram.outline}. These channels transport 
the first bit of the second value only.

Like the linear counters, the period of these counters are different Fermat numbers.

Moreover, the first bit of the second value is identified to the system bit 
after transport through another random channel. As a consequence, 
between times when the first value reaches its maximal value in direction $\vec{e}^3$, 
the bits of the second value are listed in this direction along the direction 
of rotation. This means that the sequence of system bits over odd levels consists 
of a concatenation of blocks of the same word, where these words form a complete lists 
of all the words on alphabet $\mathcal{E}$ and having length $2^n$ where $n$ is the level. 

The construction of these counters 
implies that any sequence of 
bits in $\mathcal{E}$ appears 
in the sequence of system bits 
of any odd and sufficiently great 
level. This fact is used to prove 
the minimality property.

\item \textbf{\textit{Machines layer [\ref{section.machines}]:}}
This layer supports the computations of the machines. 
Here we use a model of computing machines which includes 
multiple heads~\cite{GS17ED}: the machine starts computation 
with multiple heads on its tape and any machine head 
can enter on the two sides of the tape and at any time 
in any state. We manage collisions between the machine heads 
by fusion of the heads into a head in error state. 
If there are machine heads in error state, an error signal 
is triggered in the direction specified in the 
value of the linear counter. This signal propagates 
towards one end of the initial tape, where signaling 
mechanisms verify if the tape was well initialized and 
that no head enters in one of the two sides of the tape 
for all time. When this is the case, we forbid 
the error signal to reach its destination. 
In this case, the machine have the aimed behaviors, 
which means that the sequence of systems bits 
of \textbf{even levels} to be in $Z$ and that the 
two sequence in a section $\Z_{c+1}^2$ is the 
image by $f$ of the sequence in the section $\Z_c^2$.

\item \textbf{\textit{Information transport layer [\ref{section.information.transport}]:}}
This layer supports all the information transports: 
\begin{enumerate}
\item transport of the linear 
counter value between incrementation 
regions;
\item extraction of information from the linear 
counter to the machine area; 
\item transport of the first bit of the rotating part of the system counter
\item and of information between cells having the same level.
\end{enumerate}
These transports follow fixed roads through the sub-units, 
except for the fourth one, which is done according to the
modularity mark.
\end{itemize}

The main arguments in the proof 
of the minimality property of this subshift 
are the following ones: 

\begin{enumerate}
\item Any pattern $P$ can be completed 
into a pattern $P'$ which is a stack of 
two-dimensional cells in direction $\vec{e}^3$. Hence it is sufficient 
to prove 
that any such pattern appears in any configuration. 
Such a pattern
is characterized by the values of 
the linear counters of a sequence of two-dimensional cells 
included in its support, intersecting 
all the intermediate levels, the sequences of simulated system 
bits in direction $\vec{e}^3$ for the same sequence of 
two-dimensional cells, and 
the sequence of non-simulated ones. 
\item Considering any configuration $x$ of 
the SFT, one can find back any sequence of 
values for the linear counters starting 
from any stack of two-dimensional cells, by jumping from 
one stack to the next one having the same level in direction 
$\vec{e}^1$. This is possible since the periods of the counters 
are co-prime. The system bits are preserved. 
\item For the same reason, by jumping in direction $\vec{e}^3$, one can 
find back any sequences 
of simulated system bits. At this point the sequence of non simulated 
system bits are changed, but the linear counter are preserved. 
\item Moreover, from the fact that 
a factor of a minimal dynamical system is also minimal, and 
using Lemma~\ref{lemma.minimality.one.dimensional} with 
integer $N$ equal to the product of Fermat numbers of the system counters, 
one can find back the sequences of non simulated system bits without 
changing the other bits.
\end{enumerate}
This is illustrated on Figure~\ref{fig.minimality.proof.intro}. 
On this figure, $\mathcal{T}$ is the shift to the next stack of cells having the same level 
in direction $\vec{e}^1$. \bigskip

In the following, we describe the 
mechanisms specific to this construction. 
Some mechanisms are well known or already described 
in other articles written by the authors. Since 
these mechanisms are although not standard, they 
are described in annexes for completeness.

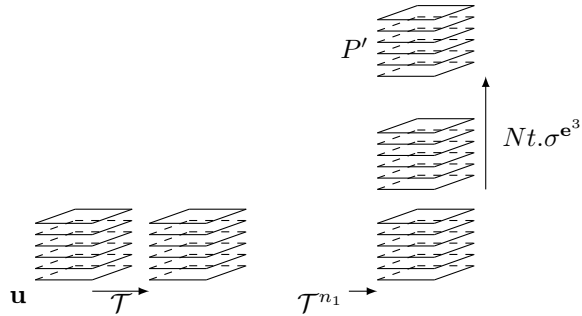
\begin{figure}[ht]
\[\begin{tikzpicture}[scale=0.15]
\begin{scope}
\draw (0,0) -- (5,0) ;
\draw[dashed] (0,0) -- (3.5,1.25) ;
\draw[dashed] (3.5,1.25) -- (8.5,1.25);
\draw (8.5,1.25) -- (5,0);

\draw (0,1) -- (5,1) ;
\draw[dashed] (0,1) -- (3.5,2.25) ;
\draw[dashed] (3.5,2.25) -- (8.5,2.25);
\draw (8.5,2.25) -- (5,1);

\draw (0,2) -- (5,2);
\draw[dashed] (0,2) -- (3.5,3.25) ;
\draw[dashed] (3.5,3.25) -- (8.5,3.25);
\draw (8.5,3.25) -- (5,2);

\draw (0,3) -- (5,3);
\draw[dashed] (0,3) -- (3.5,4.25) ;
\draw[dashed] (3.5,4.25) -- (8.5,4.25);
\draw (8.5,4.25) -- (5,3);

\draw (0,4) -- (5,4) ;
\draw[dashed] (0,4) -- (3.5,5.25) ;
\draw[dashed] (3.5,5.25) -- (8.5,5.25);
\draw (8.5,5.25) -- (5,4);

\draw (0,5) -- (5,5) ;
\draw (0,5) -- (3.5,6.25) ;
\draw(3.5,6.25) -- (8.5,6.25);
\draw (8.5,6.25) -- (5,5);

\end{scope}

\begin{scope}[yshift=8cm]
\draw (0,0) -- (5,0) ;
\draw[dashed] (0,0) -- (3.5,1.25) ;
\draw[dashed] (3.5,1.25) -- (8.5,1.25);
\draw (8.5,1.25) -- (5,0);

\draw (0,1) -- (5,1) ;
\draw[dashed] (0,1) -- (3.5,2.25) ;
\draw[dashed] (3.5,2.25) -- (8.5,2.25);
\draw (8.5,2.25) -- (5,1);

\draw (0,2) -- (5,2);
\draw[dashed] (0,2) -- (3.5,3.25) ;
\draw[dashed] (3.5,3.25) -- (8.5,3.25);
\draw (8.5,3.25) -- (5,2);

\draw (0,3) -- (5,3);
\draw[dashed] (0,3) -- (3.5,4.25) ;
\draw[dashed] (3.5,4.25) -- (8.5,4.25);
\draw (8.5,4.25) -- (5,3);

\draw (0,4) -- (5,4) ;
\draw[dashed] (0,4) -- (3.5,5.25) ;
\draw[dashed] (3.5,5.25) -- (8.5,5.25);
\draw (8.5,5.25) -- (5,4);

\draw (0,5) -- (5,5) ;
\draw (0,5) -- (3.5,6.25) ;
\draw (3.5,6.25) -- (8.5,6.25);
\draw (8.5,6.25) -- (5,5);

\draw[-latex] (9.5,0) -- (9.5,10);
\node at (14.5,5) {$Nt.\sigma^{\vec{e}^3}$};

\end{scope}

\begin{scope}[yshift=18cm]

\node at (-2,2.5) {$P'$};
\draw (0,0) -- (5,0) ;
\draw[dashed] (0,0) -- (3.5,1.25) ;
\draw[dashed] (3.5,1.25) -- (8.5,1.25);
\draw (8.5,1.25) -- (5,0);

\draw (0,1) -- (5,1) ;
\draw[dashed] (0,1) -- (3.5,2.25) ;
\draw[dashed] (3.5,2.25) -- (8.5,2.25);
\draw (8.5,2.25) -- (5,1);

\draw (0,2) -- (5,2);
\draw[dashed] (0,2) -- (3.5,3.25) ;
\draw[dashed] (3.5,3.25) -- (8.5,3.25);
\draw (8.5,3.25) -- (5,2);

\draw (0,3) -- (5,3);
\draw[dashed] (0,3) -- (3.5,4.25) ;
\draw[dashed] (3.5,4.25) -- (8.5,4.25);
\draw (8.5,4.25) -- (5,3);

\draw (0,4) -- (5,4) ;
\draw[dashed] (0,4) -- (3.5,5.25) ;
\draw[dashed] (3.5,5.25) -- (8.5,5.25);
\draw (8.5,5.25) -- (5,4);

\draw (0,5) -- (5,5) ;
\draw (0,5) -- (3.5,6.25) ;
\draw (3.5,6.25) -- (8.5,6.25);
\draw (8.5,6.25) -- (5,5);
\end{scope}

\begin{scope}[xshift=-20cm]
\draw (0,0) -- (5,0) ;
\draw[dashed] (0,0) -- (3.5,1.25) ;
\draw[dashed] (3.5,1.25) -- (8.5,1.25);
\draw (8.5,1.25) -- (5,0);

\draw (0,1) -- (5,1) ;
\draw[dashed] (0,1) -- (3.5,2.25) ;
\draw[dashed] (3.5,2.25) -- (8.5,2.25);
\draw (8.5,2.25) -- (5,1);

\draw (0,2) -- (5,2);
\draw[dashed] (0,2) -- (3.5,3.25) ;
\draw[dashed] (3.5,3.25) -- (8.5,3.25);
\draw (8.5,3.25) -- (5,2);

\draw (0,3) -- (5,3);
\draw[dashed] (0,3) -- (3.5,4.25) ;
\draw[dashed] (3.5,4.25) -- (8.5,4.25);
\draw (8.5,4.25) -- (5,3);

\draw (0,4) -- (5,4) ;
\draw[dashed] (0,4) -- (3.5,5.25) ;
\draw[dashed] (3.5,5.25) -- (8.5,5.25);
\draw (8.5,5.25) -- (5,4);

\draw (0,5) -- (5,5) ;
\draw (0,5) -- (3.5,6.25) ;
\draw (3.5,6.25) -- (8.5,6.25);
\draw (8.5,6.25) -- (5,5);
\draw[-latex] (17.5,-1) -- (20,-1);
\node at (15,-2) {${\mathcal{T}}^{n_1}$};
\end{scope}

\begin{scope}[xshift=-30cm]
\draw (0,0) -- (5,0) ;
\draw[dashed] (0,0) -- (3.5,1.25) ;
\draw[dashed] (3.5,1.25) -- (8.5,1.25);
\draw (8.5,1.25) -- (5,0);

\draw (0,1) -- (5,1) ;
\draw[dashed] (0,1) -- (3.5,2.25) ;
\draw[dashed] (3.5,2.25) -- (8.5,2.25);
\draw (8.5,2.25) -- (5,1);

\draw (0,2) -- (5,2);
\draw[dashed] (0,2) -- (3.5,3.25) ;
\draw[dashed] (3.5,3.25) -- (8.5,3.25);
\draw (8.5,3.25) -- (5,2);

\draw (0,3) -- (5,3);
\draw[dashed] (0,3) -- (3.5,4.25) ;
\draw[dashed] (3.5,4.25) -- (8.5,4.25);
\draw (8.5,4.25) -- (5,3);

\draw (0,4) -- (5,4) ;
\draw[dashed] (0,4) -- (3.5,5.25) ;
\draw[dashed] (3.5,5.25) -- (8.5,5.25);
\draw (8.5,5.25) -- (5,4);

\draw (0,5) -- (5,5) ;
\draw (0,5) -- (3.5,6.25) ;
\draw (3.5,6.25) -- (8.5,6.25);
\draw (8.5,6.25) -- (5,5);
\node at (-1.5,-1.5) {$\vec{u}$};
\draw[-latex] (5,-1) -- (10,-1);
\node at (7.5,-2) {$\mathcal{T}$};
\end{scope}

\end{tikzpicture}\]
\caption{\label{fig.minimality.proof.intro} Schema of the 
proof for the minimality property of $X_{(Z,f)}$.}
\end{figure}

\section{\label{section.structure} Structure layer}

This layer is composed of two sublayers. The first one 
is the three-dimensional rigid version of the Robinson subshift 
presented briefly in the abstract of the construction. 
With the second one, we subdivide the
cells in the copies of the Robinson subshift orthogonal 
to $\vec{e}^3$ into $64$ similar 
parts, called \textbf{organites}, by analogy with 
the functional division of living cells. This sublayer is described as follows.

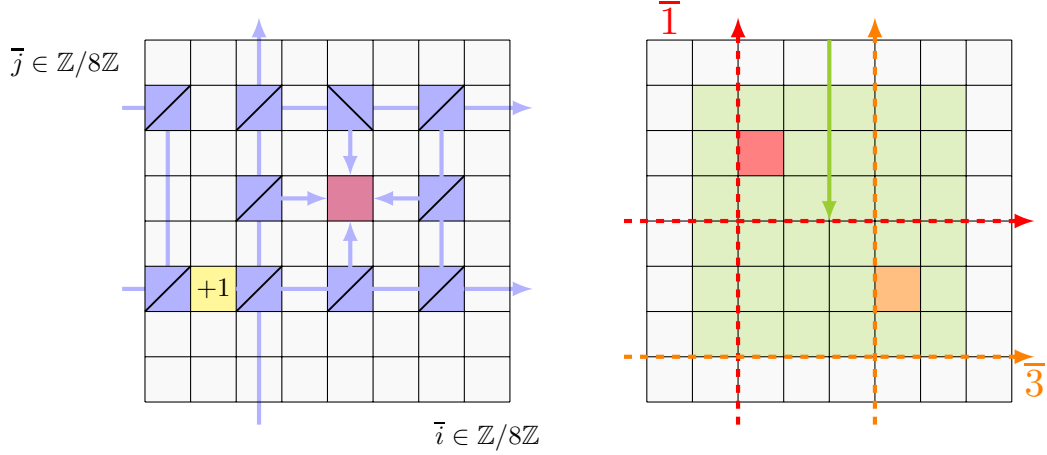
\begin{figure}[ht]
\[\begin{tikzpicture}[scale=0.6]
\begin{scope}
\fill[gray!5] (0,0) rectangle (8,8);
\fill[blue!30] (0,2) rectangle (1,3);
\fill[yellow!50] (1,2) rectangle (2,3);
\fill[purple!50] (4,4) rectangle (5,5);
\fill[blue!30] (2,2) rectangle (3,3);
\fill[blue!30] (4,6) rectangle (5,7);
\fill[blue!30] (4,2) rectangle (5,3);
\fill[blue!30] (2,6) rectangle (3,7);
\fill[blue!30] (2,4) rectangle (3,5);
\fill[blue!30] (6,6) rectangle (7,7);
\fill[blue!30] (6,4) rectangle (7,5);
\fill[blue!30] (6,2) rectangle (7,3);

\fill[blue!30] (0,6) rectangle (1,7);

\draw[line width = 0.55mm,color=blue!30] (2.5,2.5) -- (2.5,6.5);
\draw[line width = 0.55mm,color=blue!30] (-0.5,6.5) -- (0.5,6.5);

\draw[line width = 0.55mm,color=blue!30,-latex] (0.5,6.5) -- (0.5,2.5);

\draw[line width = 0.55mm,color=blue!30,-latex] (6.75,2.5) -- (8.5,2.5);
\draw[line width = 0.55mm,color=blue!30,-latex] (6.75,6.5) -- (8.5,6.5);

\draw[line width = 0.55mm,color=blue!30] (-0.5,2.5) -- (1,2.5);
\draw[-latex,line width = 0.55mm,color=blue!30] 
(2.5,6.5) -- (2.5,8.5);
\draw[-latex,line width = 0.55mm,color=blue!30] (2.5,-0.5) 
-- (2.5,2.5);
\draw[line width = 0.55mm,color=blue!30] (2.5,6.5) -- (7,6.5) ;
\node at (1.5,2.5) {$+1$};
\draw (0,0) grid (8,8);

\draw[line width = 0.55mm,color=blue!30] (2,2.5) -- (7,2.5);
\draw[-latex,line width = 0.55mm,color=blue!30] (4.5,2.5) -- (4.5,4);
\draw[-latex,line width = 0.55mm,color=blue!30] (4.5,6.5) -- (4.5,5);
\draw[-latex,line width = 0.55mm,color=blue!30] (2.5,4.5) -- (4,4.5);

\draw[line width = 0.55mm,color=blue!30] (6.5,6.5) -- (6.5,2.5);
\draw[line width = 0.55mm,color=blue!30,-latex] (6.5,4.5) -- (5,4.5);

\draw[line width = 0.25mm] (2,2) -- (3,3);
\draw[line width = 0.25mm] (4,2) -- (5,3);
\draw[line width = 0.25mm] (2,6) -- (3,7);
\draw[line width = 0.25mm] (4,7) -- (5,6);
\draw[line width = 0.25mm] (0,2) -- (1,3);
\draw[line width = 0.25mm] (2,4) -- (3,5);
\draw[line width = 0.25mm] (0,6) -- (1,7);
\draw[line width = 0.25mm] (6,2) -- (7,3);
\draw[line width = 0.25mm] (6,4) -- (7,5);
\draw[line width = 0.25mm] (6,6) -- (7,7);

\node at (7.5,-0.75) {$\overline{i} \in \Z/8\Z$};
\node at (-1.75,7.5) {$\overline{j} \in \Z/8\Z$};
\end{scope}

\begin{scope}[xshift= 11cm]

\fill[gray!5] (0,0) rectangle (8,8);
\fill[YellowGreen!30] (1,1) rectangle (7,7);
\fill[red!50] (2,5) rectangle (3,6);
\fill[orange!50] (5,2) rectangle (6,3);

\draw (0,0) grid (8,8);

\draw[YellowGreen, line width=0.55mm,-latex] (4,8) -- (4,4);

\draw[line width = 0.55mm,color=red,dashed,-latex] (2,-0.5) -- (2,8.5);
\draw[line width = 0.55mm,color=red,dashed,-latex] (-0.5,4) -- (8.5,4);
\draw[line width = 0.55mm,color=orange,dashed,-latex] (-0.5,1) -- (8.5,1);
\draw[line width = 0.55mm,color=orange,dashed,-latex] (5,-0.5) -- (5,8.5);

\node[red,scale=1.5] at (0.5,8.5) {$\overline{1}$};
\node[orange,scale=1.5] at (8.5,0.5) {$\overline{3}$};

\end{scope}

\end{tikzpicture}\]
\caption{\label{figure.functional.diagram} Functional diagram of 
odd level $n \ge 3$ cells.}
\end{figure}

This sublayer creates the subdivision
of the cells. \bigskip

\noindent \textbf{\textit{Symbols:}}

Elements of $\Z/4\Z$, $\Z/4\Z ^2$, $\Z / 4 \Z ^3$, and 
symbols $\begin{tikzpicture}[scale=0.3]
\fill[gray!50] (0,0) rectangle (1,1);
\draw (0,0) rectangle (1,1);
\end{tikzpicture}$, $\begin{tikzpicture}[scale=0.3]
\draw (0,0) rectangle (1,1);
\end{tikzpicture}$. \bigskip

\noindent \textbf{\textit{Local rules:}}

\begin{itemize}
\item \textbf{Localization:} 
the petals are superimposed with non blank symbols, 
and other positions with blank one.
\item \textbf{Transmission:}
the symbols are transmitted through the petals 
except on \textbf{transformation positions}, 
defined to be the positions where a support petal 
intersects the transmission petal just above in the hierarchy.
\item \textbf{Transformation:} 

\begin{itemize}
\item On 
a transformation position which is not in an order $\le 2$ cell,
the symbol on the position is equal to the positions in the 
same support petal in the neighborhood. 
Moreover, if the symbol in the first sublayer is 
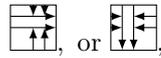
\[\begin{tikzpicture}[scale=0.3]
\draw (0,0) rectangle (2,2) ;
\draw [-latex] (0,1) -- (2,1) ;
\draw [-latex] (0,1.5) -- (2,1.5) ; 
\draw [-latex] (1,0) -- (1,1) ; 
\draw [-latex] (1,2) -- (1,1.5) ;
\draw [-latex] (1.5,0) -- (1.5,1) ; 
\draw [-latex] (1.5,2) -- (1.5,1.5) ;
\end{tikzpicture}, \ \text{or} \ \begin{tikzpicture}[scale=0.3]
\draw (0,0) rectangle (2,2) ;
\draw [-latex] (1,2) -- (1,0) ;
\draw [latex-] (0.5,0) -- (0.5,2) ; 
\draw [-latex] (0,1) -- (0.5,1) ; 
\draw [-latex] (2,1) -- (1,1) ;
\draw [-latex] (0,1.5) -- (0.5,1.5) ; 
\draw [-latex] (2,1.5) -- (1,1.5) ;
\end{tikzpicture},\]
meaning that the position is in the north west part of the cell,
then color of the transmission petal positions in the 
neighborhood is according to the following rules : 
\begin{enumerate}
\item if the orientation symbol is $\begin{tikzpicture}[scale=0.2] 
\fill[black] (0.5,0) rectangle (1,1.5);
\fill[black] (1,1.5) rectangle (2,1);
\draw (0,0) rectangle (2,2);
\end{tikzpicture}$, then the symbol of this layer is transformed into
$\overline{0}$.
\item if the orientation symbol is   $\begin{tikzpicture}[scale=0.2] 
\fill[black] (1.5,0) rectangle (1,1.5);
\fill[black] (1,1.5) rectangle (0,1);
\draw (0,0) rectangle (2,2);
\end{tikzpicture}$, then the transformed symbol is 
\begin{tikzpicture}[scale=0.3]
\fill[gray!50] (0,0) rectangle (1,1);
\draw (0,0) rectangle (1,1);
\end{tikzpicture}.
\item if the orientation symbol is  
$\begin{tikzpicture}[scale=0.2] 
\fill[black] (1.5,2) rectangle (1,0.5);
\fill[black] (1,0.5) rectangle (0,1);
\draw (0,0) rectangle (2,2);
\end{tikzpicture}$, and the symbol on the position is \begin{tikzpicture}[scale=0.3]
\fill[gray!50] (0,0) rectangle (1,1);
\draw (0,0) rectangle (1,1);
\end{tikzpicture}, then it is transformed into \begin{tikzpicture}[scale=0.3]
\fill[gray!50] (0,0) rectangle (1,1);
\draw (0,0) rectangle (1,1);
\end{tikzpicture}. If the symbol is $i \in \Z/4\Z$, 
then it is transformed into $(i,\overline{0}) \in \Z/4\Z ^2$.
If the symbol is $(i_1,i_2) \in \Z/4\Z ^2$, it is transformed 
into $(i_1,i_2,\overline{0}) \in \Z/4\Z ^3$ (see 
Figure~\ref{figure.split.transformation.1} for these rules). 
If in $\Z/4\Z^3$, the symbol is transmitted (see Figure~\ref{figure.split.transformation.2}).
\item if the orientation symbol is 
$\begin{tikzpicture}[scale=0.2] 
\fill[black] (0.5,2) rectangle (1,0.5);
\fill[black] (1,0.5) rectangle (2,1);
\draw (0,0) rectangle (2,2);
\end{tikzpicture}$, then the symbol 
is transformed into $\begin{tikzpicture}[scale=0.3]
\fill[gray!50] (0,0) rectangle (1,1);
\draw (0,0) rectangle (1,1);
\end{tikzpicture}$.
\end{enumerate}
Similar rules are imposed for the other types of transformation positions.

\item When the transformation position is 
in an order $\le 2$ cell, then we change the 
last set of rules by that \begin{tikzpicture}[scale=0.3]
\fill[gray!50] (0,0) rectangle (1,1); 
\draw (0,0) rectangle (1,1);
\end{tikzpicture} can not be transformed, even when re-entering 
inside the cell.
\end{itemize}
\end{itemize}

\begin{figure}[h]
\[\begin{tikzpicture}[scale=0.075]
\fill[gray!90] (16,16) rectangle (48,48); 
\fill[white] (16.5,16.5) rectangle (47.5,47.5);
\fill[gray!20] (32,0) rectangle (32.5,32.5);
\fill[gray!20] (32,32.5) rectangle (64,32);

\fill[gray!20] (8,8) rectangle (8.5,24); 
\fill[gray!20] (8,8) rectangle (24,8.5);
\fill[gray!20] (8,23.5) rectangle (24,24);  
\fill[gray!20] (23.5,8) rectangle (24,24);

\fill[gray!20] (40,8) rectangle (40.5,24); 
\fill[gray!20] (40,8) rectangle (56,8.5);
\fill[gray!20] (40,23.5) rectangle (56,24);  
\fill[gray!20] (55.5,8) rectangle (56,24);

\fill[gray!20] (8,40) rectangle (8.5,56); 
\fill[gray!20] (8,40) rectangle (24,40.5);
\fill[gray!20] (8,55.5) rectangle (24,56);  
\fill[gray!20] (23.5,40) rectangle (24,56);

\fill[gray!20] (40,40) rectangle (40.5,56); 
\fill[gray!20] (40,40) rectangle (56,40.5);
\fill[gray!20] (40,55.5) rectangle (56,56);  
\fill[gray!20] (55.5,40) rectangle (56,56);

\fill[gray!90] (4,4) rectangle (4.5,12); 
\fill[gray!90] (4,4) rectangle (12,4.5); 
\fill[gray!90] (4.5,11.5) rectangle (12,12); 
\fill[gray!90] (11.5,4.5) rectangle (12,12); 

\fill[gray!90] (4,20) rectangle (4.5,28); 
\fill[gray!90] (4,20) rectangle (12,20.5); 
\fill[gray!90] (4.5,27.5) rectangle (12,28); 
\fill[gray!90] (11.5,20.5) rectangle (12,28); 

\fill[gray!90] (4,36) rectangle (4.5,44); 
\fill[gray!90] (4,36) rectangle (12,36.5); 
\fill[gray!90] (4.5,43.5) rectangle (12,44); 
\fill[gray!90] (11.5,36.5) rectangle (12,44);

\fill[gray!90] (4,52) rectangle (4.5,60); 
\fill[gray!90] (4,52) rectangle (12,52.5); 
\fill[gray!90] (4.5,59.5) rectangle (12,60); 
\fill[gray!90] (11.5,52.5) rectangle (12,60);

\fill[gray!90] (20,4) rectangle (20.5,12); 
\fill[gray!90] (20,4) rectangle (28,4.5); 
\fill[gray!90] (20.5,11.5) rectangle (28,12); 
\fill[gray!90] (27.5,4.5) rectangle (28,12); 

\fill[gray!90] (20,20) rectangle (20.5,28); 
\fill[gray!90] (20,20) rectangle (28,20.5); 
\fill[gray!90] (20.5,27.5) rectangle (28,28); 
\fill[gray!90] (27.5,20.5) rectangle (28,28); 

\fill[gray!90] (20,36) rectangle (20.5,44); 
\fill[gray!90] (20,36) rectangle (28,36.5); 
\fill[gray!90] (20.5,43.5) rectangle (28,44); 
\fill[gray!90] (27.5,36.5) rectangle (28,44);

\fill[gray!90] (20,52) rectangle (20.5,60); 
\fill[gray!90] (20,52) rectangle (28,52.5); 
\fill[gray!90] (20.5,59.5) rectangle (28,60); 
\fill[gray!90] (27.5,52.5) rectangle (28,60);

\fill[gray!90] (36,4) rectangle (36.5,12); 
\fill[gray!90] (36,4) rectangle (44,4.5); 
\fill[gray!90] (36.5,11.5) rectangle (44,12); 
\fill[gray!90] (43.5,4.5) rectangle (44,12); 

\fill[gray!90] (36,20) rectangle (36.5,28); 
\fill[gray!90] (36,20) rectangle (44,20.5); 
\fill[gray!90] (36.5,27.5) rectangle (44,28); 
\fill[gray!90] (43.5,20.5) rectangle (44,28); 

\fill[gray!90] (36,36) rectangle (36.5,44); 
\fill[gray!90] (36,36) rectangle (44,36.5); 
\fill[gray!90] (36.5,43.5) rectangle (44,44); 
\fill[gray!90] (43.5,36.5) rectangle (44,44);

\fill[gray!90] (36,52) rectangle (36.5,60); 
\fill[gray!90] (36,52) rectangle (44,52.5); 
\fill[gray!90] (36.5,59.5) rectangle (44,60); 
\fill[gray!90] (43.5,52.5) rectangle (44,60);

\fill[gray!90] (52,4) rectangle (52.5,12); 
\fill[gray!90] (52,4) rectangle (60,4.5); 
\fill[gray!90] (52.5,11.5) rectangle (60,12); 
\fill[gray!90] (59.5,4.5) rectangle (60,12); 

\fill[gray!90] (52,20) rectangle (52.5,28); 
\fill[gray!90] (52,20) rectangle (60,20.5); 
\fill[gray!90] (52.5,27.5) rectangle (60,28); 
\fill[gray!90] (59.5,20.5) rectangle (60,28); 

\fill[gray!90] (52,36) rectangle (52.5,44); 
\fill[gray!90] (52,36) rectangle (60,36.5); 
\fill[gray!90] (52.5,43.5) rectangle (60,44); 
\fill[gray!90] (59.5,36.5) rectangle (60,44);

\fill[gray!90] (52,52) rectangle (52.5,60); 
\fill[gray!90] (52,52) rectangle (60,52.5); 
\fill[gray!90] (52.5,59.5) rectangle (60,60); 
\fill[gray!90] (59.5,52.5) rectangle (60,60);
\node at (24,40) {$\overline{0}$};
\node at (24,24) {$\overline{3}$};
\node at (40,40) {$\overline{1}$};
\node at (40,24) {$\overline{2}$};
\draw[->] (66,32) -- (48,32);
\node at (88,32) {$\vec{i} \in \Z/4\Z \cup \Z/4\Z^2$};
\node at (8,8) {$\vec{i},\overline{3}$};
\node at (56,56) {$\vec{i},\overline{1}$};
\node at (8,56) {$\vec{i},\overline{0}$};
\node at (56,8) {$\vec{i},\overline{2}$};
\end{tikzpicture}\]
\caption{\label{figure.split.transformation.1} Schematic illustration 
of the transformation rules when the symbol is 
$\vec{i} \in \Z/4\Z$ or $\vec{i} \in \Z/4\Z^2$.}
\end{figure}
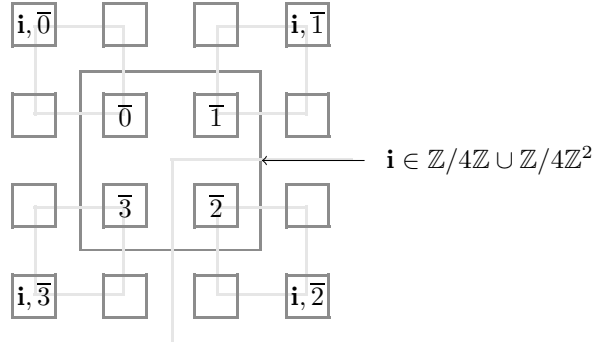

\begin{figure}[ht]
\[\begin{tikzpicture}[scale=0.075]
\fill[gray!90] (16,16) rectangle (48,48); 
\fill[white] (16.5,16.5) rectangle (47.5,47.5);
\fill[gray!20] (32,0) rectangle (32.5,32.5);
\fill[gray!20] (32,32.5) rectangle (64,32);

\fill[gray!20] (8,8) rectangle (8.5,24); 
\fill[gray!20] (8,8) rectangle (24,8.5);
\fill[gray!20] (8,23.5) rectangle (24,24);  
\fill[gray!20] (23.5,8) rectangle (24,24);

\fill[gray!20] (40,8) rectangle (40.5,24); 
\fill[gray!20] (40,8) rectangle (56,8.5);
\fill[gray!20] (40,23.5) rectangle (56,24);  
\fill[gray!20] (55.5,8) rectangle (56,24);

\fill[gray!20] (8,40) rectangle (8.5,56); 
\fill[gray!20] (8,40) rectangle (24,40.5);
\fill[gray!20] (8,55.5) rectangle (24,56);  
\fill[gray!20] (23.5,40) rectangle (24,56);

\fill[gray!20] (40,40) rectangle (40.5,56); 
\fill[gray!20] (40,40) rectangle (56,40.5);
\fill[gray!20] (40,55.5) rectangle (56,56);  
\fill[gray!20] (55.5,40) rectangle (56,56);

\fill[gray!90] (4,4) rectangle (4.5,12); 
\fill[gray!90] (4,4) rectangle (12,4.5); 
\fill[gray!90] (4.5,11.5) rectangle (12,12); 
\fill[gray!90] (11.5,4.5) rectangle (12,12); 

\fill[gray!90] (4,20) rectangle (4.5,28); 
\fill[gray!90] (4,20) rectangle (12,20.5); 
\fill[gray!90] (4.5,27.5) rectangle (12,28); 
\fill[gray!90] (11.5,20.5) rectangle (12,28); 

\fill[gray!90] (4,36) rectangle (4.5,44); 
\fill[gray!90] (4,36) rectangle (12,36.5); 
\fill[gray!90] (4.5,43.5) rectangle (12,44); 
\fill[gray!90] (11.5,36.5) rectangle (12,44);

\fill[gray!90] (4,52) rectangle (4.5,60); 
\fill[gray!90] (4,52) rectangle (12,52.5); 
\fill[gray!90] (4.5,59.5) rectangle (12,60); 
\fill[gray!90] (11.5,52.5) rectangle (12,60);

\fill[gray!90] (20,4) rectangle (20.5,12); 
\fill[gray!90] (20,4) rectangle (28,4.5); 
\fill[gray!90] (20.5,11.5) rectangle (28,12); 
\fill[gray!90] (27.5,4.5) rectangle (28,12); 

\fill[gray!90] (20,20) rectangle (20.5,28); 
\fill[gray!90] (20,20) rectangle (28,20.5); 
\fill[gray!90] (20.5,27.5) rectangle (28,28); 
\fill[gray!90] (27.5,20.5) rectangle (28,28); 

\fill[gray!90] (20,36) rectangle (20.5,44); 
\fill[gray!90] (20,36) rectangle (28,36.5); 
\fill[gray!90] (20.5,43.5) rectangle (28,44); 
\fill[gray!90] (27.5,36.5) rectangle (28,44);

\fill[gray!90] (20,52) rectangle (20.5,60); 
\fill[gray!90] (20,52) rectangle (28,52.5); 
\fill[gray!90] (20.5,59.5) rectangle (28,60); 
\fill[gray!90] (27.5,52.5) rectangle (28,60);

\fill[gray!90] (36,4) rectangle (36.5,12); 
\fill[gray!90] (36,4) rectangle (44,4.5); 
\fill[gray!90] (36.5,11.5) rectangle (44,12); 
\fill[gray!90] (43.5,4.5) rectangle (44,12); 

\fill[gray!90] (36,20) rectangle (36.5,28); 
\fill[gray!90] (36,20) rectangle (44,20.5); 
\fill[gray!90] (36.5,27.5) rectangle (44,28); 
\fill[gray!90] (43.5,20.5) rectangle (44,28); 

\fill[gray!90] (36,36) rectangle (36.5,44); 
\fill[gray!90] (36,36) rectangle (44,36.5); 
\fill[gray!90] (36.5,43.5) rectangle (44,44); 
\fill[gray!90] (43.5,36.5) rectangle (44,44);

\fill[gray!90] (36,52) rectangle (36.5,60); 
\fill[gray!90] (36,52) rectangle (44,52.5); 
\fill[gray!90] (36.5,59.5) rectangle (44,60); 
\fill[gray!90] (43.5,52.5) rectangle (44,60);

\fill[gray!90] (52,4) rectangle (52.5,12); 
\fill[gray!90] (52,4) rectangle (60,4.5); 
\fill[gray!90] (52.5,11.5) rectangle (60,12); 
\fill[gray!90] (59.5,4.5) rectangle (60,12); 

\fill[gray!90] (52,20) rectangle (52.5,28); 
\fill[gray!90] (52,20) rectangle (60,20.5); 
\fill[gray!90] (52.5,27.5) rectangle (60,28); 
\fill[gray!90] (59.5,20.5) rectangle (60,28); 

\fill[gray!90] (52,36) rectangle (52.5,44); 
\fill[gray!90] (52,36) rectangle (60,36.5); 
\fill[gray!90] (52.5,43.5) rectangle (60,44); 
\fill[gray!90] (59.5,36.5) rectangle (60,44);

\fill[gray!90] (52,52) rectangle (52.5,60); 
\fill[gray!90] (52,52) rectangle (60,52.5); 
\fill[gray!90] (52.5,59.5) rectangle (60,60); 
\fill[gray!90] (59.5,52.5) rectangle (60,60);
\node at (24,40) {$\overline{0}$};
\node at (24,24) {$\overline{3}$};
\node at (40,40) {$\overline{1}$};
\node at (40,24) {$\overline{2}$};
\draw[->] (66,32) -- (48,32);
\node at (77,32) {$\vec{i} \in \Z/4\Z^3$};
\node at (8,8) {$\vec{i}$};
\node at (56,56) {$\vec{i}$};
\node at (8,56) {$\vec{i}$};
\node at (56,8) {$\vec{i}$};
\end{tikzpicture}\]
\caption{\label{figure.split.transformation.2} Schematic illustration 
of the transformation rules when the symbol is 
$\vec{i} \in \Z/4\Z^3$.}
\end{figure}
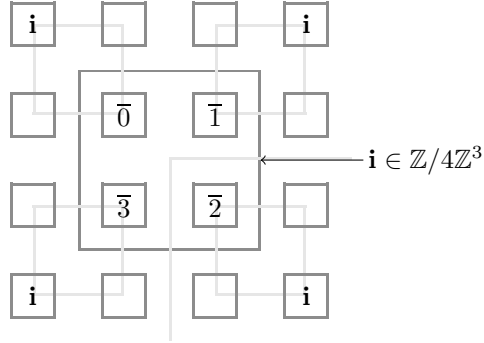

\noindent \textbf{\textit{Global behavior:}} \bigskip

As a consequence of the local rules, each position with 
a blue corner is in an order $ \ge 3$ cell is 
marked with an element of $\Z/4\Z ^3$ or with \begin{tikzpicture}[scale=0.3]
\fill[gray!50] (0,0) rectangle (1,1); 
\draw (0,0) rectangle (1,1);
\end{tikzpicture}. The positions with an element of $\Z/4\Z ^3$ are exactly the functional 
positions, presented later in this text. 
In an order $\le 2$ cell, all the positions with a blue corner in the first sublayer 
are marked with \begin{tikzpicture}[scale=0.3]
\fill[gray!50] (0,0) rectangle (1,1); 
\draw (0,0) rectangle (1,1);
\end{tikzpicture} (we won't use these cells since they are too small 
and can not be subdivided into $64$ parts). Moreover, the areas where a triplet 
can appear in an order $\ge 3$ cell are shown on 
Figure~\ref{figure.subdivision.hierarchical}. For simplicity we 
recode these marks into pairs $(i,j) \in \Z/8\Z ^2$, where $i$ corresponds 
to the horizontal coordinate and $j$ to the vertical one. These two coordinates 
are oriented from west to east and from south to north.  

\begin{figure}[ht]
\[\begin{tikzpicture}[scale=0.3]
\fill[color=gray!5] (0,0) rectangle (16,16);
\draw[step=2] (0,0) grid (16,16);
\draw[line width=0.3mm,step=4] (0,0) grid (16,16);
\draw[line width=0.4mm,step=8] (0,0) grid (16,16);
\node[scale=0.8] at (1,1) {$\overline{333}$};
\node[scale=0.8] at (3,1) {$\overline{332}$};
\node[scale=0.8] at (1,3) {$\overline{330}$};
\node[scale=0.8] at (3,3) {$\overline{331}$};
\node[scale=0.8] at (5,1) {$\overline{323}$};
\node[scale=0.8] at (7,1) {$\overline{322}$};
\node[scale=0.8] at (5,3) {$\overline{320}$};
\node[scale=0.8] at (7,3) {$\overline{321}$};
\node[scale=0.8] at (1,5) {$\overline{303}$};
\node[scale=0.8] at (3,5) {$\overline{302}$};
\node[scale=0.8] at (1,7) {$\overline{300}$};
\node[scale=0.8] at (3,7) {$\overline{301}$};
\node[scale=0.8] at (5,5) {$\overline{313}$};
\node[scale=0.8] at (7,5) {$\overline{312}$};
\node[scale=0.8] at (5,7) {$\overline{310}$};
\node[scale=0.8] at (7,7) {$\overline{311}$};

\node[scale=0.8] at (9,1) {$\overline{233}$};
\node[scale=0.8] at (11,1) {$\overline{232}$};
\node[scale=0.8] at (9,3) {$\overline{230}$};
\node[scale=0.8] at (11,3) {$\overline{231}$};
\node[scale=0.8] at (13,1) {$\overline{223}$};
\node[scale=0.8] at (15,1) {$\overline{222}$};
\node[scale=0.8] at (13,3) {$\overline{220}$};
\node[scale=0.8] at (15,3) {$\overline{221}$};
\node[scale=0.8] at (9,5) {$\overline{203}$};
\node[scale=0.8] at (11,5) {$\overline{202}$};
\node[scale=0.8] at (9,7) {$\overline{200}$};
\node[scale=0.8] at (11,7) {$\overline{201}$};
\node[scale=0.8] at (13,5) {$\overline{213}$};
\node[scale=0.8] at (15,5) {$\overline{212}$};
\node[scale=0.8] at (13,7) {$\overline{210}$};
\node[scale=0.8] at (15,7) {$\overline{211}$};

\node[scale=0.8] at (1,9) {$\overline{033}$};
\node[scale=0.8] at (3,9) {$\overline{032}$};
\node[scale=0.8] at (1,11) {$\overline{030}$};
\node[scale=0.8] at (3,11) {$\overline{031}$};
\node[scale=0.8] at (5,9) {$\overline{023}$};
\node[scale=0.8] at (7,9) {$\overline{022}$};
\node[scale=0.8] at (5,11) {$\overline{020}$};
\node[scale=0.8] at (7,11) {$\overline{021}$};
\node[scale=0.8] at (1,13) {$\overline{003}$};
\node[scale=0.8] at (3,13) {$\overline{002}$};
\node[scale=0.8] at (1,15) {$\overline{000}$};
\node[scale=0.8] at (3,15) {$\overline{001}$};
\node[scale=0.8] at (5,13) {$\overline{013}$};
\node[scale=0.8] at (7,13) {$\overline{012}$};
\node[scale=0.8] at (5,15) {$\overline{010}$};
\node[scale=0.8] at (7,15) {$\overline{011}$};

\node[scale=0.8] at (9,9) {$\overline{133}$};
\node[scale=0.8] at (11,9) {$\overline{132}$};
\node[scale=0.8] at (9,11) {$\overline{130}$};
\node[scale=0.8] at (11,11) {$\overline{131}$};
\node[scale=0.8] at (13,9) {$\overline{123}$};
\node[scale=0.8] at (15,9) {$\overline{122}$};
\node[scale=0.8] at (13,11) {$\overline{120}$};
\node[scale=0.8] at (15,11) {$\overline{121}$};
\node[scale=0.8] at (9,13) {$\overline{103}$};
\node[scale=0.8] at (11,13) {$\overline{102}$};
\node[scale=0.8] at (9,15) {$\overline{100}$};
\node[scale=0.8] at (11,15) {$\overline{101}$};
\node[scale=0.8] at (13,13) {$\overline{113}$};
\node[scale=0.8] at (15,13) {$\overline{112}$};
\node[scale=0.8] at (13,15) {$\overline{110}$};
\node[scale=0.8] at (15,15) {$\overline{111}$};

\node[scale=0.8] at (1,-1) {$\overline{0}$};
\node[scale=0.8] at (3,-1) {$\overline{1}$};
\node[scale=0.8] at (5,-1) {$\overline{2}$};
\node[scale=0.8] at (7,-1) {$\overline{3}$};
\node[scale=0.8] at (9,-1) {$\overline{4}$};
\node[scale=0.8] at (11,-1) {$\overline{5}$};
\node[scale=0.8] at (13,-1) {$\overline{6}$};
\node[scale=0.8] at (15,-1) {$\overline{7}$};

\node[scale=0.8] at (-1,1) {$\overline{0}$};
\node[scale=0.8] at (-1,3) {$\overline{1}$};
\node[scale=0.8] at (-1,5) {$\overline{2}$};
\node[scale=0.8] at (-1,7) {$\overline{3}$};
\node[scale=0.8] at (-1,9) {$\overline{4}$};
\node[scale=0.8] at (-1,11) {$\overline{5}$};
\node[scale=0.8] at (-1,13) {$\overline{6}$};
\node[scale=0.8] at (-1,15) {$\overline{7}$};
\end{tikzpicture}\]
\caption{\label{figure.subdivision.hierarchical} Schema of the correspondence 
between areas in a cell and triplets and recoding the 
addresses of the areas.}
\end{figure}

To each sub-unit correspond functions, 
as follows: 

\begin{enumerate}
\item The sub-unit having coordinates $(\overline{6},\overline{5})$ 
supports computations of the machines.
\item The sub-units $(\overline{1},\overline{3})$, $(\overline{3},\overline{3})$, 
$(\overline{6},\overline{3})$, $(\overline{3},\overline{7})$, 
$(\overline{6},\overline{7})$ are demultiplexers. This means that 
they allow the extraction of information contained in the linear counter. 
\item The sub-unit $(\overline{2},\overline{3})$ supports the incrementation 
of the linear counter.
\item The sub-units $(\overline{3},\overline{3})$, $(\overline{3},\overline{4})$, $(\overline{4},\overline{3})$ and $(\overline{4},\overline{4})$ support the incrementation 
of the system counter.
\item The ones specified by arrows on Figure~\ref{figure.functional.diagram} 
are used for information transport. 
\item The other ones have no function. 
\end{enumerate} 

A schema of these functions 
is shown on Figure~\ref{figure.functional.diagram}.

\section{\label{section.system.bits} System bits layer} 

\noindent \textbf{\textit{Symbols:}}
elements of $\left(\mathcal{A} \cup 
\{\begin{tikzpicture}[scale=0.3]
\draw (0,0) rectangle (1,1);
\end{tikzpicture}\}\right)^2$. \bigskip

\noindent \textbf{\textit{Local rules:}}

\begin{itemize}
\item \textbf{Transmission:} 
\begin{enumerate}
\item The 
first coordinate of the ordered pair 
is transmitted horizontally. This means 
that the symbols of positions $\vec{u}$ 
and $\vec{u}+\vec{e}^1$ have the same first 
coordinate for all $\vec{u}\in \Z^3$.
\item The second coordinate is transmitted 
vertically.
\end{enumerate}
\item \textbf{Synchronization:} on 
a corner symbol in the structure layer, 
the two coordinates are equal.
\item \textbf{Localization:} when 
the corner has $0,1$-counter value
equal to $0$, then the two coordinates 
are blank. When it is equal to $1$, 
the two coordinates are non blank.
\end{itemize}

\noindent \textbf{\textit{Global behavior:}}

In each of the sections $\Z_c^2$ and 
each $n$, the order $n$ cells and the 
lines connecting their corners
are attached with the same bit, called 
system bit. The sequence of these bits, 
denoted $(s^{(n)}_c)_n$ 
depends on the section $\Z_c^2$. 
The transmission petals 
are superimposed with the blank symbol.

See Figure~\ref{fig.basis3}.

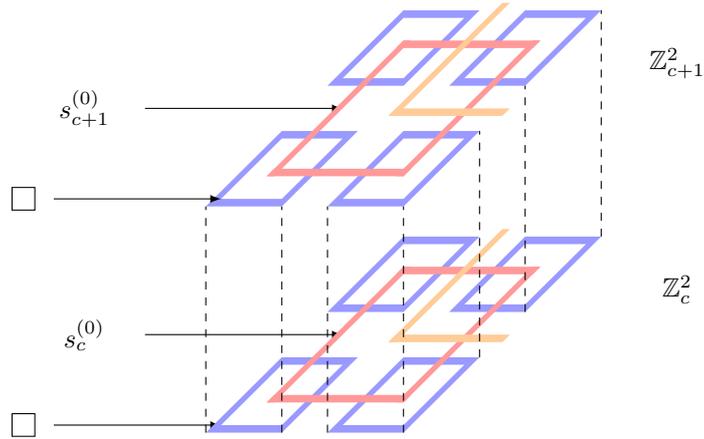
\begin{figure}[h]
\[\begin{tikzpicture}[every node/.style={minimum size=1cm},on grid, scale=0.2]
\begin{scope}[every node/.append style={xslant=1},xslant=1]
\fill[blue!40] (8.5,0.5) rectangle (13.5,5.5);
\fill[blue!40] (8.5,8.5) rectangle (13.5,13.5);
\fill[blue!40] (0.5,8.5) rectangle (5.5,13.5);
\fill[white] (9,1) rectangle (13,5);
\fill[white] (9,9) rectangle (13,13);
\fill[white] (1,9) rectangle (5,13);

\draw[-latex] (-10,7) -- (3,7);
\draw[-latex] (-10,1) -- (1,1);

\fill[blue!40] (0.5,0.5) rectangle (5.5,5.5);
\fill[white] (1,1) rectangle (5,5);

\fill[red!40] (2.5,2.5) rectangle (3,11.5);
\fill[red!40] (2.5,2.5) rectangle (11.5,3);
\fill[red!40] (2.5,11) rectangle (11.5,11.5);
\fill[red!40] (11,2.5) rectangle (11.5,11.5);

\fill[orange!40] (6.5,6.5) rectangle (7,14); 
\fill[orange!40] (6.5,6.5) rectangle (14,7);
\end{scope}

\begin{scope}[yshift=15cm,every node/.append style={xslant=1},xslant=1]
\fill[blue!40] (8.5,0.5) rectangle (13.5,5.5);
\fill[blue!40] (8.5,8.5) rectangle (13.5,13.5);
\fill[blue!40] (0.5,8.5) rectangle (5.5,13.5);
\fill[white] (9,1) rectangle (13,5);
\fill[white] (9,9) rectangle (13,13);
\fill[white] (1,9) rectangle (5,13);

\fill[blue!40] (0.5,0.5) rectangle (5.5,5.5);
\fill[white] (1,1) rectangle (5,5);

\draw[-latex] (-10,7) -- (3,7);
\draw[-latex] (-10,1) -- (1,1);

\fill[red!40] (2.5,2.5) rectangle (3,11.5);
\fill[red!40] (2.5,2.5) rectangle (11.5,3);
\fill[red!40] (2.5,11) rectangle (11.5,11.5);
\fill[red!40] (11,2.5) rectangle (11.5,11.5);

\fill[orange!40] (6.5,6.5) rectangle (7,14); 
\fill[orange!40] (6.5,6.5) rectangle (14,7);
\end{scope}

\begin{scope}
\node at (-11,1) {$\begin{tikzpicture}[scale=0.3] \draw (0,0) rectangle (1,1);
\end{tikzpicture}$};
\node at (-11,16) {$\begin{tikzpicture}[scale=0.3] \draw (0,0) rectangle (1,1);
\end{tikzpicture}$};
\node at (-7,7) {$s^{(0)}_c$};
\node at (-7,22) {$s^{(0)}_{c+1}$};
\end{scope}

\begin{scope}
\draw[dashed] (1,0.5) -- (1,15.5);
\draw[dashed] (9,0.5) -- (9,15.5);
\draw[dashed] (6,0.5) -- (6,15.5);
\draw[dashed] (14,0.5) -- (14,15.5);
\draw[dashed] (19,5.5) -- (19,20.5);
\draw[dashed] (22,8.5) -- (22,23.5);
\draw[dashed] (27,13.5) -- (27,28.5);
\end{scope}

\begin{scope}[xshift=37cm]
\draw node at (-5,25) {$\Z_{c+1}^2$};
\draw node at (-5,10) {$\Z_c^2$};
\end{scope}
\end{tikzpicture}\]
\caption{\label{fig.basis3} Illustration of the basis layer rules}
\end{figure}

\section{\label{section.parity.marks} Modularity marks}

In this section, we describe how to give access to each cell 
to the class of its level modulo $4$ so that the road followed to transport 
this information never meet the location of system bits. This layer 
has two sub-layers. The first one is described 
in the annexes, in Section~\ref{section.hierarchy.orientation}. Let us describe 
the second one.

\begin{figure}[ht]

\[\begin{tikzpicture}[scale=0.4]

\fill[gray!20] (0,0) rectangle (6,6);
\draw (0,0) rectangle (6,6);
\draw (0.75,0.75) rectangle (2.25,2.25);
\draw (3.75,3.75) rectangle (5.25,5.25);
\draw (0.75,3.75) rectangle (2.25,5.25);
\draw (3.75,0.75) rectangle (5.25,2.25);

\draw (6.75,6.75) rectangle (8.25,8.25);
\draw (-2.25,-2.25) rectangle (-0.75,-0.75);

\draw (6.75,3.75) rectangle (8.25,5.25);
\draw (6.75,0.75) rectangle (8.25,2.25);
\draw (6.75,-2.25) rectangle (8.25,-0.75);

\draw (3.75,-2.25) rectangle (5.25,-0.75);
\draw (0.75,-2.25) rectangle (2.25,-0.75);

\draw (3.75,6.75) rectangle (5.25,8.25);
\draw (0.75,6.75) rectangle (2.25,8.25);
\draw (-2.25,6.75) rectangle (-0.75,8.25);

\draw (-2.25,3.75) rectangle (-0.75,5.25);
\draw (-2.25,0.75) rectangle (-0.75,2.25);

\draw[line width =0.5mm,color=blue!50] (4.5,5.25) -- (4.5,6);
\draw[line width =0.5mm,color=blue!50] (5.25,4.5) -- (6,4.5);

\draw[line width =0.5mm,color=red] (6,6) -- (3,6);
\draw[line width =0.5mm,color=red] (6,6) -- (6,3);

\draw[line width =0.5mm,color=red] (8.25,8.25) -- (7.5,8.25);
\draw[line width =0.5mm,color=red] (8.25,8.25) -- (8.25,7.5);
\draw[line width =0.5mm,color=red] (8.25,5.25) -- (7.5,5.25);
\draw[line width =0.5mm,color=red] (8.25,5.25) -- (8.25,4.5);
\draw[line width =0.5mm,color=red] (8.25,2.25) -- (7.5,2.25);
\draw[line width =0.5mm,color=red] (8.25,2.25) -- (8.25,1.5);
\draw[line width =0.5mm,color=red] (8.25,-0.75) -- (7.5,-0.75);
\draw[line width =0.5mm,color=red] (8.25,-0.75) -- (8.25,-1.5);

\draw[line width =0.5mm,color=red] (5.25,8.25) -- (4.5,8.25);
\draw[line width =0.5mm,color=red] (5.25,8.25) -- (5.25,7.5);
\draw[line width =0.5mm,color=red] (5.25,5.25) -- (4.5,5.25);
\draw[line width =0.5mm,color=red] (5.25,5.25) -- (5.25,4.5);
\draw[line width =0.5mm,color=red] (5.25,2.25) -- (4.5,2.25);
\draw[line width =0.5mm,color=red] (5.25,2.25) -- (5.25,1.5);
\draw[line width =0.5mm,color=red] (5.25,-0.75) -- (4.5,-0.75);
\draw[line width =0.5mm,color=red] (5.25,-0.75) -- (5.25,-1.5);

\draw[line width =0.5mm,color=red] (2.25,8.25) -- (1.5,8.25);
\draw[line width =0.5mm,color=red] (2.25,8.25) -- (2.25,7.5);
\draw[line width =0.5mm,color=red] (2.25,5.25) -- (1.5,5.25);
\draw[line width =0.5mm,color=red] (2.25,5.25) -- (2.25,4.5);
\draw[line width =0.5mm,color=red] (2.25,2.25) -- (1.5,2.25);
\draw[line width =0.5mm,color=red] (2.25,2.25) -- (2.25,1.5);
\draw[line width =0.5mm,color=red] (2.25,-0.75) -- (1.5,-0.75);
\draw[line width =0.5mm,color=red] (2.25,-0.75) -- (2.25,-1.5);

\draw[line width =0.5mm,color=red] (-0.75,8.25) -- (-1.5,8.25);
\draw[line width =0.5mm,color=red] (-0.75,8.25) -- (-0.75,7.5);
\draw[line width =0.5mm,color=red] (-0.75,5.25) -- (-1.5,5.25);
\draw[line width =0.5mm,color=red] (-0.75,5.25) -- (-0.75,4.5);
\draw[line width =0.5mm,color=red] (-0.75,2.25) -- (-1.5,2.25);
\draw[line width =0.5mm,color=red] (-0.75,2.25) -- (-0.75,1.5);
\draw[line width =0.5mm,color=red] (-0.75,-0.75) -- (-1.5,-0.75);
\draw[line width =0.5mm,color=red] (-0.75,-0.75) -- (-0.75,-1.5);

\draw[line width =0.5mm,color=blue!50] (3,3) -- (3,9); 
\draw[line width =0.5mm,color=blue!50] (3,3) -- (9,3);
\draw[line width =0.5mm,color=blue!50] (3,6) -- (3,9);
\draw[line width =0.5mm,color=blue!50] (6,3) -- (9,3);

\draw[line width =0.5mm,color=YellowGreen,dashed] (0,-3) -- (0,9);
\draw[line width =0.5mm,color=YellowGreen,dashed] (-3,0) -- (9,0);

\fill[gray!90] (2.8,5.8) rectangle (3.2,6.2);
\fill[gray!90] (5.8,2.8) rectangle (6.2,3.2);

\fill[gray!90] (4.35,5.1) rectangle (4.65,5.4);
\fill[gray!90] (5.1,4.35) rectangle (5.4,4.65);
\end{tikzpicture}\]
\caption{\label{figure.hierarchical.process.parity.marks} Localization 
of the modularity marks. Here level $n$ cells 
and the level $n+1$ cell above are represented. 
The locations of level $n+1$ system bits are colored green and the locations 
of parity marks are colored blue and red. Observation: the system bits 
never cross parity marks for the same level $n+1$.}
\end{figure}
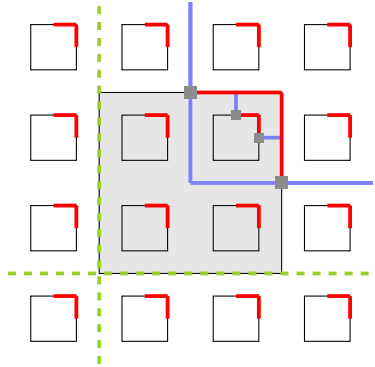

\noindent \textbf{\textit{Symbols:}} 
elements of $\Z/4\Z$, $\{(\overline{n},\overline{n}+\overline{1}), n \in \Z/4\Z\}$ and a blank symbol. \bigskip

\noindent \textbf{\textit{Local rules:}} 

\begin{enumerate}
\item \textbf{Localization:} 
the non-blank symbols are superimposed on and only on the east half of 
the north side of a cell, the north half of the east side of a cell, and the segments of 
transmission petals connecting them. See an illustration on 
Figure~\ref{figure.hierarchical.process.parity.marks}. Formally, this 
corresponds to positions having the following symbols in the structure layer: 
\begin{enumerate}
\item the symbols 
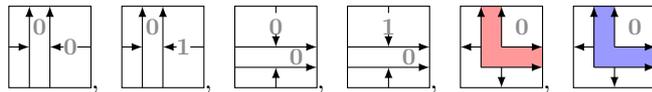
\[\begin{tikzpicture}[scale=0.9]
\draw (0,0) rectangle (1.2,1.2) ;
\draw [latex-] (0.6,1.2) -- (0.6,0) ;
\draw [latex-] (0.3,1.2) -- (0.3,0) ; 
\draw [-latex] (0,0.6) -- (0.3,0.6) ; 
\draw [-latex] (0.8,0.6) -- (0.6,0.6) ;
\draw (1.2,0.6) -- (1,0.6);
\node[scale=1,color=gray!90,scale=0.9] at (0.45,0.9) {\textbf{0}};
\node[scale=1,color=gray!90,scale=0.9] at (0.9,0.6) {\textbf{0}};
\end{tikzpicture}, \ \ \begin{tikzpicture}[scale=0.9]
\draw (0,0) rectangle (1.2,1.2) ;
\draw [latex-] (0.6,1.2) -- (0.6,0) ;
\draw [latex-] (0.3,1.2) -- (0.3,0) ; 
\draw [-latex] (0,0.6) -- (0.3,0.6) ; 
\draw [-latex] (0.8,0.6) -- (0.6,0.6) ;
\draw (1.2,0.6) -- (1,0.6);
\node[scale=1,color=gray!90,scale=0.9] at (0.45,0.9) {\textbf{0}};
\node[scale=1,color=gray!90,scale=0.9] at (0.9,0.6) {\textbf{1}};
\end{tikzpicture}, \ \ \begin{tikzpicture}[scale=0.9]
\draw (0,0) rectangle (1.2,1.2) ;
\draw [latex-] (1.2,0.6) -- (0,0.6) ;
\draw [latex-] (1.2,0.3) -- (0,0.3) ; 
\draw [-latex] (0.6,0) -- (0.6,0.3) ; 
\draw [-latex] (0.6,0.8) -- (0.6,0.6) ;
\draw (0.6,1.2) -- (0.6,1.1);
\node[scale=1,color=gray!90,scale=0.9] at (0.9,0.45) {\textbf{0}};
\node[scale=1,color=gray!90,scale=0.9] at (0.6,0.9) {\textbf{0}};
\end{tikzpicture}, \ \ \begin{tikzpicture}[scale=0.9]
\draw (0,0) rectangle (1.2,1.2) ;
\draw [latex-] (1.2,0.6) -- (0,0.6) ;
\draw [latex-] (1.2,0.3) -- (0,0.3) ; 
\draw [-latex] (0.6,0) -- (0.6,0.3) ; 
\draw [-latex] (0.6,0.8) -- (0.6,0.6) ;
\draw (0.6,1.2) -- (0.6,1.1);
\node[scale=1,color=gray!90,scale=0.9] at (0.9,0.45) {\textbf{0}};
\node[scale=1,color=gray!90,scale=0.9] at (0.6,0.9) {\textbf{1}};
\end{tikzpicture}, \ \ \begin{tikzpicture}[scale=0.9]
\fill[red!40] (0.3,0.3) rectangle (0.6,1.2) ;
\fill[red!40] (0.3,0.3) rectangle (1.2,0.6) ;
\draw (0,0) rectangle (1.2,1.2) ;
\draw [-latex] (0.3,0.3) -- (0.3,1.2) ; 
\draw [-latex] (0.3,0.3) -- (1.2,0.3) ;
\draw [-latex] (0.6,0.6) -- (0.6,1.2) ; 
\draw [-latex] (0.6,0.6) -- (1.2,0.6) ; 
\draw [-latex] (0.3,0.6) -- (0,0.6) ; 
\draw [-latex] (0.6,0.3) -- (0.6,0) ;
\node[scale=1,color=gray!90,scale=0.9] at (0.9,0.9) {\textbf{0}};
\end{tikzpicture}, \ \ \begin{tikzpicture}[scale=0.9]
\fill[blue!40] (0.3,0.3) rectangle (0.6,1.2) ;
\fill[blue!40] (0.3,0.3) rectangle (1.2,0.6) ;
\draw (0,0) rectangle (1.2,1.2) ;
\draw [-latex] (0.3,0.3) -- (0.3,1.2) ; 
\draw [-latex] (0.3,0.3) -- (1.2,0.3) ;
\draw [-latex] (0.6,0.6) -- (0.6,1.2) ; 
\draw [-latex] (0.6,0.6) -- (1.2,0.6) ; 
\draw [-latex] (0.3,0.6) -- (0,0.6) ; 
\draw [-latex] (0.6,0.3) -- (0.6,0) ;
\node[scale=1,color=gray!90,scale=0.9] at (0.9,0.9) {\textbf{0}};
\end{tikzpicture}\] superimposed with the orientation symbol \begin{tikzpicture}[scale=0.2] 
\fill[black] (1.5,0) rectangle (1,1.5);
\fill[black] (1,1.5) rectangle (0,1);
\draw (0,0) rectangle (2,2);
\end{tikzpicture}. These positions are colored blue on Figure~\ref{figure.hierarchical.process.parity.marks} 
and are superimposed with an element of $\Z/4\Z$ in the present layer.
\item the symbols 
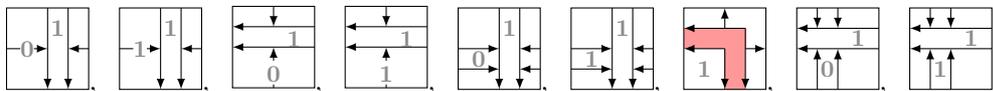
\[\begin{tikzpicture}[scale=0.9]
\draw (0,0) rectangle (1.2,1.2) ;
\draw [-latex] (0.6,1.2) -- (0.6,0) ;
\draw [-latex] (0.9,1.2) -- (0.9,0) ; 
\draw (0,0.6) -- (0.2,0.6);
\draw [-latex] (0.4,0.6) -- (0.6,0.6) ; 
\draw [-latex] (1.2,0.6) -- (0.9,0.6) ;
\node[scale=1,color=gray!90,scale=0.9] at (0.75,0.9) {\textbf{1}};
\node[scale=1,color=gray!90,scale=0.9] at (0.3,0.6) {\textbf{0}};
\end{tikzpicture}, \ \  \begin{tikzpicture}[scale=0.9]
\draw (0,0) rectangle (1.2,1.2) ;
\draw [-latex] (0.6,1.2) -- (0.6,0) ;
\draw [-latex] (0.9,1.2) -- (0.9,0) ; 
\draw (0,0.6) -- (0.2,0.6);
\draw [-latex] (0.4,0.6) -- (0.6,0.6) ; 
\draw [-latex] (1.2,0.6) -- (0.9,0.6) ;
\node[scale=1,color=gray!90,scale=0.9] at (0.75,0.9) {\textbf{1}};
\node[scale=1,color=gray!90,scale=0.9] at (0.3,0.6) {\textbf{1}};
\end{tikzpicture}, \ \  \begin{tikzpicture}[scale=0.9]
\draw (0,0) rectangle (1.2,1.2) ;
\draw [-latex] (1.2,0.6) -- (0,0.6) ;
\draw [-latex] (1.2,0.9) -- (0,0.9) ; 
\draw (0.6,0) -- (0.6,0.05);
\draw [-latex] (0.6,0.4) -- (0.6,0.6) ; 
\draw [-latex] (0.6,1.2) -- (0.6,0.9) ;
\node[scale=1,color=gray!90,scale=0.9] at (0.9,0.725) {\textbf{1}};
\node[scale=1,color=gray!90,scale=0.9] at (0.6,0.2) {\textbf{0}};
\end{tikzpicture}, \ \  \begin{tikzpicture}[scale=0.9]
\draw (0,0) rectangle (1.2,1.2) ;
\draw [-latex] (1.2,0.6) -- (0,0.6) ;
\draw [-latex] (1.2,0.9) -- (0,0.9) ; 
\draw (0.6,0) -- (0.6,0.05);
\draw [-latex] (0.6,0.4) -- (0.6,0.6) ; 
\draw [-latex] (0.6,1.2) -- (0.6,0.9) ;
\node[scale=1,color=gray!90,scale=0.9] at (0.9,0.725) {\textbf{1}};
\node[scale=1,color=gray!90,scale=0.9] at (0.6,0.2) {\textbf{1}};
\end{tikzpicture}, \ \ \begin{tikzpicture}[scale=0.9]
\draw (0,0) rectangle (1.2,1.2) ;
\draw [-latex] (0.6,1.2) -- (0.6,0) ;
\draw [-latex] (0.9,1.2) -- (0.9,0) ; 
\draw [-latex] (0,0.6) -- (0.6,0.6) ; 
\draw [-latex] (1.2,0.6) -- (0.9,0.6) ;
\draw [-latex] (0,0.3) -- (0.6,0.3) ; 
\draw [-latex] (1.2,0.3) -- (0.9,0.3) ;
\node[scale=1,color=gray!90,scale=0.9] at (0.3,0.45) {\textbf{0}};
\node[scale=1,color=gray!90,scale=0.9] at (0.75,0.9) {\textbf{1}};
\end{tikzpicture}, \ \ \begin{tikzpicture}[scale=0.9]
\draw (0,0) rectangle (1.2,1.2) ;
\draw [-latex] (0.6,1.2) -- (0.6,0) ;
\draw [-latex] (0.9,1.2) -- (0.9,0) ; 
\draw [-latex] (0,0.6) -- (0.6,0.6) ; 
\draw [-latex] (1.2,0.6) -- (0.9,0.6) ;
\draw [-latex] (0,0.3) -- (0.6,0.3) ; 
\draw [-latex] (1.2,0.3) -- (0.9,0.3) ;
\node[scale=1,color=gray!90,scale=0.9] at (0.3,0.45) {\textbf{1}};
\node[scale=1,color=gray!90,scale=0.9] at (0.75,0.9) {\textbf{1}};
\end{tikzpicture},\ \ \begin{tikzpicture}[scale=0.9]
\fill[red!40] (0.9,0.9) rectangle (0.6,0) ;
\fill[red!40] (0.9,0.9) rectangle (0,0.6) ;
\draw (0,0) rectangle (1.2,1.2) ;
\draw [-latex] (0.9,0.9) -- (0.9,0) ; 
\draw [-latex] (0.9,0.9) -- (0,0.9) ;
\draw [-latex] (0.6,0.6) -- (0.6,0) ; 
\draw [-latex] (0.6,0.6) -- (0,0.6) ;
\draw [-latex] (0.9,0.6) -- (1.2,0.6) ; 
\draw [-latex] (0.6,0.9) -- (0.6,1.2) ;
\node[scale=1,color=gray!90,scale=0.9] at (0.3,0.3) {\textbf{1}};
\end{tikzpicture}, \ \ \begin{tikzpicture}[scale=0.9]
\draw (0,0) rectangle (1.2,1.2) ;
\draw [-latex] (1.2,0.6) -- (0,0.6) ;
\draw [-latex] (1.2,0.9) -- (0,0.9) ; 
\draw [-latex] (0.6,0) -- (0.6,0.6) ; 
\draw [-latex] (0.6,1.2) -- (0.6,0.9) ;
\draw [-latex] (0.3,0) -- (0.3,0.6) ; 
\draw [-latex] (0.3,1.2) -- (0.3,0.9) ;
\node[scale=1,color=gray!90,scale=0.9] at (0.45,0.3) {\textbf{0}};
\node[scale=1,color=gray!90,scale=0.9] at (0.9,0.75) {\textbf{1}};
\end{tikzpicture}, \ \ \begin{tikzpicture}[scale=0.9]
\draw (0,0) rectangle (1.2,1.2) ;
\draw [-latex] (1.2,0.6) -- (0,0.6) ;
\draw [-latex] (1.2,0.9) -- (0,0.9) ; 
\draw [-latex] (0.6,0) -- (0.6,0.6) ; 
\draw [-latex] (0.6,1.2) -- (0.6,0.9) ;
\draw [-latex] (0.3,0) -- (0.3,0.6) ; 
\draw [-latex] (0.3,1.2) -- (0.3,0.9) ;
\node[scale=1,color=gray!90,scale=0.9] at (0.45,0.3) {\textbf{1}};
\node[scale=1,color=gray!90,scale=0.9] at (0.9,0.75) {\textbf{1}};
\end{tikzpicture} \] on which are superimposed symbols in $\Z/4\Z$. These 
positions are colored red on Figure~\ref{figure.hierarchical.process.parity.marks}.
\item the symbols
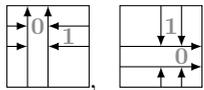
\[\begin{tikzpicture}[scale=0.9]
\draw (0,0) rectangle (1.2,1.2) ;
\draw [latex-] (0.6,1.2) -- (0.6,0) ;
\draw [latex-] (0.3,1.2) -- (0.3,0) ; 
\draw [-latex] (0,0.6) -- (0.3,0.6) ;
\draw[-latex] (0,0.9) -- (0.3,0.9); 
\draw [-latex] (1.2,0.6) -- (0.6,0.6) ;
\draw [-latex] (1.2,0.9) -- (0.6,0.9) ;
\node[scale=1,color=gray!90,scale=0.9] at (0.45,0.9) {\textbf{0}};
\node[scale=1,color=gray!90,scale=0.9] at (0.9,0.75) {\textbf{1}};
\end{tikzpicture}, \ \ \begin{tikzpicture}[scale=0.9]
\draw (0,0) rectangle (1.2,1.2) ;
\draw [latex-] (1.2,0.6) -- (0,0.6) ;
\draw [latex-] (1.2,0.3) -- (0,0.3) ; 
\draw [-latex] (0.6,0) -- (0.6,0.3) ;
\draw[-latex] (0.9,0) -- (0.9,0.3); 
\draw [-latex] (0.6,1.2) -- (0.6,0.6) ;
\draw [-latex] (0.9,1.2) -- (0.9,0.6) ;
\node[scale=1,color=gray!90,scale=0.9] at (0.9,0.45) {\textbf{0}};
\node[scale=1,color=gray!90,scale=0.9] at (0.75,0.9) {\textbf{1}};
\end{tikzpicture}.\] These ones are the \textbf{transformation positions}, and 
are superimposed with an element of $\{(\overline{n},\overline{n}+\overline{1}): \overline{n} \in \Z/4\Z\}$. They 
are designated by dark gray squares on Figure~\ref{figure.hierarchical.process.parity.marks}.
\end{enumerate}
\item \textbf{Initialization of the process:}
\item \textbf{Transformation:}
On the transformation positions, the first symbol of the ordered pair 
is equal to the symbol of the position below (respectively on the 
left) when the structure symbol 
is 
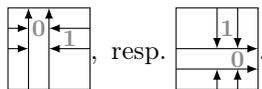
\[\begin{tikzpicture}[scale=0.9,baseline=4mm]
\draw (0,0) rectangle (1.2,1.2) ;
\draw [latex-] (0.6,1.2) -- (0.6,0) ;
\draw [latex-] (0.3,1.2) -- (0.3,0) ; 
\draw [-latex] (0,0.6) -- (0.3,0.6) ;
\draw[-latex] (0,0.9) -- (0.3,0.9); 
\draw [-latex] (1.2,0.6) -- (0.6,0.6) ;
\draw [-latex] (1.2,0.9) -- (0.6,0.9) ;
\node[scale=1,color=gray!90,scale=0.9] at (0.45,0.9) {\textbf{0}};
\node[scale=1,color=gray!90,scale=0.9] at (0.9,0.75) {\textbf{1}};
\end{tikzpicture}, \ \text{resp.} \ \begin{tikzpicture}[scale=0.9,baseline=4mm]
\draw (0,0) rectangle (1.2,1.2) ;
\draw [latex-] (1.2,0.6) -- (0,0.6) ;
\draw [latex-] (1.2,0.3) -- (0,0.3) ; 
\draw [-latex] (0.6,0) -- (0.6,0.3) ;
\draw[-latex] (0.9,0) -- (0.9,0.3); 
\draw [-latex] (0.6,1.2) -- (0.6,0.6) ;
\draw [-latex] (0.9,1.2) -- (0.9,0.6) ;
\node[scale=1,color=gray!90,scale=0.9] at (0.9,0.45) {\textbf{0}};
\node[scale=1,color=gray!90,scale=0.9] at (0.75,0.9) {\textbf{1}};
\end{tikzpicture}.\]
The second symbol is equal to the symbol of the position 
on the right (resp. above). 
\end{enumerate}

\noindent \textbf{\textit{Global behavior:}} \bigskip

A hierarchical signaling process is implemented which 
allows the coloration of north east quarter of 
each cell's border with the class of its level modulo $4$. In order to avoid this information to 
cross the system bit of the same level, the communication is allowed only from the north 
east order $n$ cell inside an order $n+1$ cell to this cell. The signal is transmitted 
only through the south west quarter of the transmission petal connecting them.

\section{\label{section.functional.areas} Functional areas}

In the construction, we use a functional division of 
the cytoplasm of each of the computing units. 
We use two different mechanisms. The first one works 
on the whole cells and is a hierarchical signaling process 
similar to a 
substitution. This mechanism is described in Section~\ref{subsection.functional.areas.whole.cell}. 
It allows the recognition 
of the free columns and free lines of the cytoplasm 
in order to localize the information and its transports 
outside the machine areas, so that there is no possible error 
in the process. The second one works on the machine areas 
and allows the localization of information transport 
inside them. This mechanism is similar to the one 
used in~\cite{R71} and uses signals 
from border
to border of the area. Errors are allowed in the process, 
that trigger error signals along the border of the area. 
When these error occur, the computations of the machine are not 
taken into account. The freedom in the constitution 
of the functional areas is used to ensure the minimality property. In this section, we describe 
the second mechanism. \bigskip

\noindent \textbf{\textit{Symbols:}} \bigskip

Elements of $\{\texttt{on},\texttt{off}\}^2$, 
of $\{\texttt{on},\texttt{off}\}$ 
and a blank symbol. \bigskip

\noindent \textbf{\textit{Local rules:}} \bigskip

\begin{itemize}
\item \textbf{Localization rules:}

\begin{itemize}
\item 
the non-blank symbols are superimposed 
on the positions having a blue symbol 
in the Robinson layer, and a blue or arrow symbol in 
the first sublayer of the present layer, and 
are in the $(\overline{6},\overline{5})$ sub-unit, 
corresponding to the machines.
\item the ordered pairs are superimposed on 
the positions having a blue symbol in the first sublayer.
\end{itemize}
\item \textbf{Transmission rule:}
the symbol is transmitted 
along lines/columns. On 
the intersections the second symbol 
is equal to the symbol on the column. 
The first one is equal 
to the symbol on the line.
\end{itemize}

\noindent \textbf{\textit{Global 
behavior:}} \bigskip

In the cytoplasm of each computing 
unit in the machine area, the 
free columns and free lines are colored with 
a symbol in $\{\texttt{on},\texttt{off}\}$. 
We call 
columns (resp.lines) colored 
with $\texttt{on}$ \textbf{computation-active}
columns (resp. lines). 

\section{\label{section.system.counter} System counter}

In this section, we explain how the system 
counter works. The corresponding 
layer is subdivided into two
sublayers. The first one supports the 
transport of the modularity mark to the 
system counter area~\ref{subsection.parity.mark.transport}.
The second one supports the incrementation of the system 
counter~\ref{subsection.incrementation.system.counter}.

\subsection{\label{subsection.parity.mark.transport} 
Transport of the modularity mark}

\noindent \textbf{\textit{Symbols:}} \bigskip

The symbols of this first sublayer 
are the elements of $\Z/4\Z$ and a blank symbol. \bigskip

\noindent \textbf{\textit{Local rules:}} 

\begin{itemize}
\item \textbf{Localization:} 
the non-blank symbols are superimposed 
on and only on the cytoplasm positions 
in the sub-units that are not on the border, 
meaning the sub-units $(\overline{k},\overline{l})$ with 
$k,l \in \llbracket 1, 6 \rrbracket$, and 
$(\overline{4},\overline{7})$.
\item \textbf{Transport of information:}
two adjacent such positions have the same symbol.
\item \textbf{Nature of the information:}
on the top line of the
$(\overline{4},\overline{7})$ sub-unit, the symbol of this layer 
is equal to the modularity mark on 
the line just above.
\end{itemize}

\noindent \textbf{\textit{Global behavior:}} \bigskip

The consequence of local rules 
is that the information of the modularity
mark is transported through a path 
avoiding the system bit location 
to the system counter area to the surrounding of the system counter area. \bigskip

\subsection{\label{subsection.incrementation.system.counter}
Incrementation of the system counter}

Recall that $\mathcal{A}$ is the alphabet of the simulated 
system. Let us complete this alphabet into 
an alphabet $\mathcal{E}$ having cardinality
$2^{2^{m}}$ for some $m \ge 0$.

Let us fix $t$ some cyclic permutation 
of the set $\mathcal{E}$, and 
$e_{\texttt{max}}$ some 
element of $\mathcal{E}$.
As well, we consider $\overline{t}$ 
a permutation of $\mathcal{E}^2$ 
and $\overline{e}_{\texttt{max}}$ 
an element of $\mathcal{E}^2$. \bigskip

\noindent \textbf{\textit{Symbols:}} \bigskip

The elements of the sets
\[({\mathcal{E}}^2 \times \{0,1\} \times \{\rightarrow\} ) \times ({\mathcal{E}} \times \{0,1\} \times \{\rightarrow\} ) \times ({\mathcal{E}} \times \{0,1\} \times \{\leftarrow\} ) 
\times \left( \left\{\begin{tikzpicture}[scale=0.3] 
\fill[Salmon] (0,0) rectangle (1,1);
\draw (0,0) rectangle (1,1);
\end{tikzpicture}, \begin{tikzpicture}[scale=0.3]
\fill[YellowGreen] (0,0) rectangle (1,1);
\draw (0,0) rectangle (1,1);
\end{tikzpicture}\right\} \right)^3,\]
\[({\mathcal{E}}^2 \times \{0,1\} \times \{\rightarrow\} ) \times ({\mathcal{E}} \times \{0,1\} \times \{\rightarrow\} ) \times ({\mathcal{E}} \times \{0,1\} \times \{\leftarrow\} ) 
\times \left( \left\{\begin{tikzpicture}[scale=0.3] 
\fill[Salmon] (0,0) rectangle (1,1);
\draw (0,0) rectangle (1,1);
\end{tikzpicture}, \begin{tikzpicture}[scale=0.3]
\fill[YellowGreen] (0,0) rectangle (1,1);
\draw (0,0) rectangle (1,1);
\end{tikzpicture}\right\} \right)^3 \times \left\{\begin{tikzpicture}[scale=0.3] 
\fill[gray!10] (0,0) rectangle (1,1);
\draw (0,0) rectangle (1,1);
\end{tikzpicture}, \begin{tikzpicture}[scale=0.3]
\fill[purple!80] (0,0) rectangle (1,1);
\draw (0,0) rectangle (1,1);
\end{tikzpicture}\right\}\]
and 
\[\{0,1\}^3 \times  \left( \left\{\begin{tikzpicture}[scale=0.3] 
\fill[Salmon] (0,0) rectangle (1,1);
\draw (0,0) rectangle (1,1);
\end{tikzpicture}, \begin{tikzpicture}[scale=0.3]
\fill[YellowGreen] (0,0) rectangle (1,1);
\draw (0,0) rectangle (1,1);
\end{tikzpicture}\right\} \right)^3.\]

The elements of $({\mathcal{E}} \times \{0,1\} \times \{\rightarrow\})$
are thought as the following tiles. 
The first symbol represents the south 
symbol in the tile, and the second one 
representing the west symbol in the tile. The arrow represent the 
direction of propagation of the increment signal. 
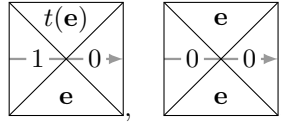
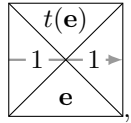
\[\begin{tikzpicture}[scale=1.5]
\draw[gray!80,line width=0.3mm] (0.35,0.5) -- (0.65,0.5);
\draw[gray!80,line width=0.3mm] (0,0.5) -- (0.15,0.5);
\draw[gray!80,line width=0.3mm,-latex] (0.85,0.5) -- (1,0.5);
\draw (0,0) rectangle (1,1);
\draw (0,0) -- (1,1);
\draw (1,0) -- (0,1);
\node at (0.5,0.15) {$\vec{e}$};
\node at (0.5,0.85) {$t(\vec{e})$};
\node at (0.75,0.5) {$0$};
\node at (0.25,0.5) {$1$};
\end{tikzpicture}, \quad
\begin{tikzpicture}[scale=1.5]
\draw[gray!80,line width=0.3mm] (0.35,0.5) -- (0.65,0.5);
\draw[gray!80,line width=0.3mm] (0,0.5) -- (0.15,0.5);
\draw[gray!80,line width=0.3mm,-latex] (0.85,0.5) -- (1,0.5);
\draw (0,0) rectangle (1,1);
\draw (0,0) -- (1,1);
\draw (1,0) -- (0,1);
\node at (0.5,0.15) {$\vec{e}$};
\node at (0.5,0.85) {$\vec{e}$};
\node at (0.25,0.5) {$0$};
\node at (0.75,0.5) {$0$};
\end{tikzpicture},\]
for $\vec{e} \neq \vec{e}_{\texttt{max}}$, 
and 
\[\begin{tikzpicture}[scale=1.5]
\draw[gray!80,line width=0.3mm] (0.35,0.5) -- (0.65,0.5);
\draw[gray!80,line width=0.3mm] (0,0.5) -- (0.15,0.5);
\draw[gray!80,line width=0.3mm,-latex] (0.85,0.5) -- (1,0.5);
\draw (0,0) rectangle (1,1);
\draw (0,0) -- (1,1);
\draw (1,0) -- (0,1);
\node at (0.5,0.15) {$\vec{e}$};
\node at (0.5,0.85) {$t(\vec{e})$};
\node at (0.25,0.5) {$1$};
\node at (0.75,0.5) {$1$};
\end{tikzpicture},\]
for $\vec{e} = \vec{e}_{\texttt{max}}$. \bigskip

The elements of ${\mathcal{E}} \times \{0,1\} \times \{\leftarrow\}$ 
are thought as similar tiles, with reversed arrow and pairs of west and east symbols: 

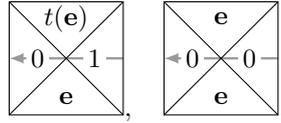
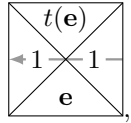
\[\begin{tikzpicture}[scale=1.5]
\draw[gray!80,line width=0.3mm] (0.35,0.5) -- (0.65,0.5);
\draw[gray!80,line width=0.3mm,latex-] (0,0.5) -- (0.15,0.5);
\draw[gray!80,line width=0.3mm] (0.85,0.5) -- (1,0.5);
\draw (0,0) rectangle (1,1);
\draw (0,0) -- (1,1);
\draw (1,0) -- (0,1);
\node at (0.5,0.15) {$\vec{e}$};
\node at (0.5,0.85) {$t(\vec{e})$};
\node at (0.75,0.5) {$1$};
\node at (0.25,0.5) {$0$};
\end{tikzpicture}, \quad
\begin{tikzpicture}[scale=1.5]
\draw[gray!80,line width=0.3mm] (0.35,0.5) -- (0.65,0.5);
\draw[gray!80,line width=0.3mm,latex-] (0,0.5) -- (0.15,0.5);
\draw[gray!80,line width=0.3mm] (0.85,0.5) -- (1,0.5);
\draw (0,0) rectangle (1,1);
\draw (0,0) -- (1,1);
\draw (1,0) -- (0,1);
\node at (0.5,0.15) {$\vec{e}$};
\node at (0.5,0.85) {$\vec{e}$};
\node at (0.25,0.5) {$0$};
\node at (0.75,0.5) {$0$};
\end{tikzpicture},\]
for $\vec{e} \neq \vec{e}_{\texttt{max}}$, 
and 
\[\begin{tikzpicture}[scale=1.5]
\draw[gray!80,line width=0.3mm] (0.35,0.5) -- (0.65,0.5);
\draw[gray!80,line width=0.3mm,latex-] (0,0.5) -- (0.15,0.5);
\draw[gray!80,line width=0.3mm] (0.85,0.5) -- (1,0.5);
\draw (0,0) rectangle (1,1);
\draw (0,0) -- (1,1);
\draw (1,0) -- (0,1);
\node at (0.5,0.15) {$\vec{e}$};
\node at (0.5,0.85) {$t(\vec{e})$};
\node at (0.25,0.5) {$1$};
\node at (0.75,0.5) {$1$};
\end{tikzpicture},\]
for $\vec{e} = \vec{e}_{\texttt{max}}$. \bigskip

The elements of $({\mathcal{E}}^2 \times \{0,1\} \times \{\rightarrow\} )$ are thought as similar tiles, except that elements of 
$\mathcal{E}$ are replaced by elements 
of $\mathcal{E}^2$, $e_{\texttt{max}}$ 
by $\overline{e}_{\texttt{max}}$ and the 
permutation $t$ by $\overline{t}$. \bigskip

Let us note that the arrow is, as the tile, not necessary but 
aimed to make the construction more readable.
The elements of $\left\{\begin{tikzpicture}[scale=0.3] 
\fill[Salmon] (0,0) rectangle (1,1);
\draw (0,0) rectangle (1,1);
\end{tikzpicture}, \begin{tikzpicture}[scale=0.3]
\fill[YellowGreen] (0,0) rectangle (1,1);
\draw (0,0) rectangle (1,1);
\end{tikzpicture}\right\}$ are thought as the symbols of a detecting signal, 
aimed to detect when all the south symbols of the tiles are $e_{\texttt{max}}$ 
or $\overline{e}_{\texttt{max}}$.
The symbols in $\left\{\begin{tikzpicture}[scale=0.3] 
\fill[gray!10] (0,0) rectangle (1,1);
\draw (0,0) rectangle (1,1);
\end{tikzpicture}, \begin{tikzpicture}[scale=0.3]
\fill[purple!80] (0,0) rectangle (1,1);
\draw (0,0) rectangle (1,1);
\end{tikzpicture}\right\}$ are called the {\bf{freezing symbols}}.
The freezing symbol serves to suspend the incrementation mechanism for one step. \bigskip

\noindent \textbf{\textit{Local rules:}} \bigskip

\begin{itemize}
\item \textbf{Localization rules:}

\begin{itemize}
\item The non-blank symbols are 
superimposed on positions of bottom line of the sub-unit $(\overline{2},\overline{5})$
of order $\ge 3$ cells when the 
modularity mark is $\overline{1}$. 
\item When the modularity mark is $\overline{3}$ the support sub-unit is $(\overline{2},\overline{5})$.
\item For the other modularity marks, the symbols over these lines are blank.

\item The 
elements of the set 
\[({\mathcal{E}}^2 \times \{0,1\} \times \{\rightarrow\} ) \times (\mathcal{E}  \times \{0,1\} \times \{\rightarrow\}) \times (\mathcal{E}  \times \{0,1\} \times \{\leftarrow\})
\times \{\begin{tikzpicture}[scale=0.3]
\fill[Salmon] (0,0) rectangle (1,1); 
\draw (0,0) rectangle (1,1); \end{tikzpicture}, 
\begin{tikzpicture}[scale=0.3] 
\fill[YellowGreen] (0,0) rectangle (1,1);
\draw (0,0) rectangle (1,1); \end{tikzpicture}\}^3\] appear on 
the computation positions of the line except the leftmost one.
On this one, the symbol is in the set 
\[({\mathcal{E}}^2 \times \{0,1\} \times \{\rightarrow\} ) \times (\mathcal{E}  \times \{0,1\} \times \{\rightarrow\}) \times (\mathcal{E}  \times \{0,1\} \times \{\leftarrow\})
\times \{\begin{tikzpicture}[scale=0.3]
\fill[Salmon] (0,0) rectangle (1,1); 
\draw (0,0) rectangle (1,1); \end{tikzpicture}, 
\begin{tikzpicture}[scale=0.3] 
\fill[YellowGreen] (0,0) rectangle (1,1);
\draw (0,0) rectangle (1,1); \end{tikzpicture}\}^3 \times \{\begin{tikzpicture}[scale=0.3]
\fill[gray!10] (0,0) rectangle (1,1); 
\draw (0,0) rectangle (1,1); \end{tikzpicture}, 
\begin{tikzpicture}[scale=0.3] 
\fill[purple!80] (0,0) rectangle (1,1);
\draw (0,0) rectangle (1,1); \end{tikzpicture}\}.\]

\item The other positions 
have a symbol in the set 
\[\{0,1\}^3
\times \{\begin{tikzpicture}[scale=0.3]
\fill[Salmon] (0,0) rectangle (1,1); 
\draw (0,0) rectangle (1,1); \end{tikzpicture}, 
\begin{tikzpicture}[scale=0.3] 
\fill[YellowGreen] (0,0) rectangle (1,1);
\draw (0,0) rectangle (1,1); \end{tikzpicture}\}^3.\]
See a schema of these rules on Figure~\ref{figure.schema.localization.system.counter}.
\end{itemize}

\begin{figure}[ht]
\[\begin{tikzpicture}[scale=0.3]
\fill[gray!20] (0,0) rectangle (24,8);
\draw (0,6) -- (0,0) -- (24,0) -- (24,6);
\draw[dashed] (0,6) -- (0,7);
\draw[dashed] (24,6) -- (24,7);
\node at (12,4) {$(\overline{2},\overline{5})$ or $(\overline{5},\overline{2}) $};
\draw[decorate,decoration={brace,raise=0.4cm},line width=0.3mm] 
(27,-3) -- (27,4);

\fill[blue!30] (0,0) rectangle (1,1); 
\fill[blue!30] (3,0) rectangle (4,1); 
\fill[blue!30] (8,0) rectangle (9,1); 
\fill[blue!30] (11,0) rectangle (12,1); 
\fill[blue!30] (12,0) rectangle (13,1); 
\fill[blue!30] (15,0) rectangle (16,1); 
\fill[blue!30] (20,0) rectangle (21,1); 
\fill[blue!30] (23,0) rectangle (24,1);

\draw (0,0) grid (24,1);

\fill[blue!30] (28,1) rectangle (29,2); 
\fill[blue!30] (31,1) rectangle (32,2); 
\fill[blue!30] (36,1) rectangle (37,2); 
\fill[blue!30] (39,1) rectangle (40,2); 
\fill[blue!30] (40,1) rectangle (41,2); 
\fill[blue!30] (43,1) rectangle (44,2); 
\fill[blue!30] (48,1) rectangle (49,2); 
\fill[blue!30] (51,1) rectangle (52,2); 

\draw[-latex] (28,1.5) -- (29,1.5);
\draw[-latex] (31,1.5) -- (32,1.5);
\draw[-latex] (36,1.5) -- (37,1.5); 
\draw[-latex] (39,1.5) -- (40,1.5); 
\draw[-latex] (40,1.5) -- (41,1.5);
\draw[-latex] (43,1.5) -- (44,1.5); 
\draw[-latex]  (48,1.5) -- (49,1.5);
\draw[-latex] (51,1.5) -- (52,1.5); 

\draw (28,1) grid (52,2);

\begin{scope}[yshift=2cm]

\fill[blue!30] (28,1) rectangle (29,2); 
\fill[blue!30] (31,1) rectangle (32,2); 
\fill[blue!30] (36,1) rectangle (37,2); 
\fill[blue!30] (39,1) rectangle (40,2); 
\fill[blue!30] (40,1) rectangle (41,2); 
\fill[blue!30] (43,1) rectangle (44,2); 
\fill[blue!30] (48,1) rectangle (49,2); 
\fill[blue!30] (51,1) rectangle (52,2); 

\draw[latex-] (28,1.5) -- (29,1.5);
\draw[latex-] (31,1.5) -- (32,1.5);
\draw[latex-] (36,1.5) -- (37,1.5); 
\draw[latex-] (39,1.5) -- (40,1.5); 
\draw[latex-] (40,1.5) -- (41,1.5);
\draw[latex-] (43,1.5) -- (44,1.5); 
\draw[latex-]  (48,1.5) -- (49,1.5);
\draw[latex-] (51,1.5) -- (52,1.5); 

\draw (28,1) grid (52,2);
\end{scope}

\begin{scope}[yshift=-2cm]

\fill[blue!30] (28,1) rectangle (29,2); 
\fill[blue!30] (31,1) rectangle (32,2); 
\fill[blue!30] (36,1) rectangle (37,2); 
\fill[blue!30] (39,1) rectangle (40,2); 
\fill[blue!30] (40,1) rectangle (41,2); 
\fill[blue!30] (43,1) rectangle (44,2); 
\fill[blue!30] (48,1) rectangle (49,2); 
\fill[blue!30] (51,1) rectangle (52,2); 

\draw[-latex] (28,1.5) -- (29,1.5);
\draw[-latex] (31,1.5) -- (32,1.5);
\draw[-latex] (36,1.5) -- (37,1.5); 
\draw[-latex] (39,1.5) -- (40,1.5); 
\draw[-latex] (40,1.5) -- (41,1.5);
\draw[-latex] (43,1.5) -- (44,1.5); 
\draw[-latex]  (48,1.5) -- (49,1.5);
\draw[-latex] (51,1.5) -- (52,1.5); 

\draw (28,1) grid (52,2);

\end{scope}

\draw (26.5,1.5) -- (26.5,6);

\draw[-latex] (31.5,-4) -- (31.5,-1);

\node at (33.5,-5) {$\mathcal{E}^2 \times 
\{0,1\} \times \{\begin{tikzpicture}[scale=0.3]
\fill[Salmon] (0,0) rectangle (1,1);
\draw (0,0) rectangle (1,1);
\end{tikzpicture}, \begin{tikzpicture}[scale=0.3]
\fill[YellowGreen] (0,0) rectangle (1,1);
\draw (0,0) rectangle (1,1);
\end{tikzpicture}\}$};

\draw[-latex] (26.5,1.5) -- (27.5,1.5);
\draw[-latex] (26.5,3.5) -- (27.5,3.5);
\draw[-latex] (31.5,6) -- (31.5,4.5);
\node at (31,7) {$\mathcal{E} \times \{0,1\} \times \{\begin{tikzpicture}[scale=0.3]
\fill[Salmon] (0,0) rectangle (1,1);
\draw (0,0) rectangle (1,1);
\end{tikzpicture}, \begin{tikzpicture}[scale=0.3]
\fill[YellowGreen] (0,0) rectangle (1,1);
\draw (0,0) rectangle (1,1);
\end{tikzpicture}\}$};

\draw[-latex] (42.5,-2.5) -- (42.5,-0.5);
\draw[-latex] (44.5,-2.5) -- (44.5,1.5);
\draw[-latex] (45.5,-2.5) -- (45.5,3.5);
\node at (44,-4) {$\{0,1\} \times \{\begin{tikzpicture}[scale=0.3]
\fill[Salmon] (0,0) rectangle (1,1);
\draw (0,0) rectangle (1,1);
\end{tikzpicture}, \begin{tikzpicture}[scale=0.3]
\fill[YellowGreen] (0,0) rectangle (1,1);
\draw (0,0) rectangle (1,1);
\end{tikzpicture}\}$};

\fill[gray!10] (28.75,-3) rectangle (29.75,-2);
\draw (28.75,-3) rectangle (29.75,-2);

\fill[purple!80] (27.25,-3) rectangle (28.25,-2);
\draw (27.25,-3) rectangle (28.25,-2);

\end{tikzpicture}\]
\caption{\label{figure.schema.localization.system.counter}
Schema of the localization rules for the system counter. The 
pattern on the bottom line of the sub-unit is considered as the superposition 
of three words and a symbol on the leftmost position.}
\end{figure}
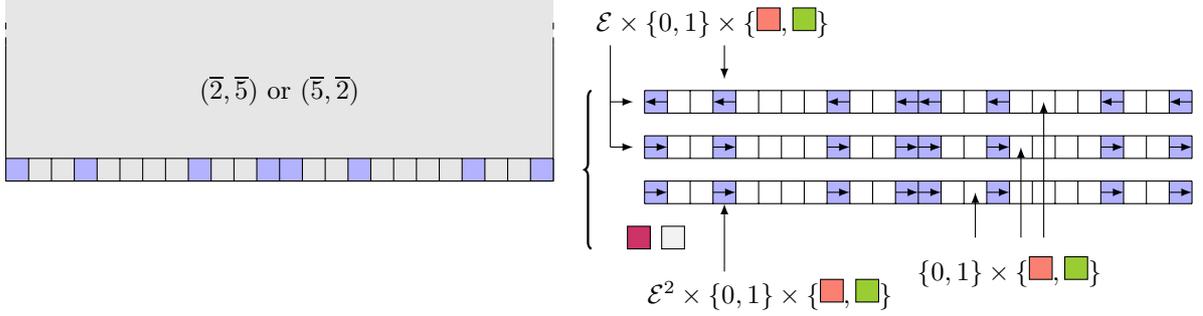

A \textbf{value} of the counter is a possible 
sequence of symbols of the 
alphabet $\mathcal{E}^4$ that can appear 
on these lines, on the computation positions. A value is thought as 
the superposition of three words and a letter in $\{\begin{tikzpicture}[scale=0.3]
\fill[gray!10] (0,0) rectangle (1,1);
\draw (0,0) rectangle (1,1);
\end{tikzpicture}, \begin{tikzpicture}[scale=0.3]
\fill[purple!80] (0,0) rectangle (1,1);
\draw (0,0) rectangle (1,1);
\end{tikzpicture}\}$
 on the leftmost position.
The first word has alphabet
\[\mathcal{E} \times \{0,1\} \times \{\leftarrow\}\times \{\begin{tikzpicture}[scale=0.3]
\fill[Salmon] (0,0) rectangle (1,1);
\draw (0,0) rectangle (1,1);
\end{tikzpicture}, \begin{tikzpicture}[scale=0.3]
\fill[YellowGreen] (0,0) rectangle (1,1);
\draw (0,0) rectangle (1,1);
\end{tikzpicture}\}\]
on the computation positions and alphabet 
$\{0,1\} \times \{\begin{tikzpicture}[scale=0.3]
\fill[Salmon] (0,0) rectangle (1,1);
\draw (0,0) rectangle (1,1);
\end{tikzpicture}, \begin{tikzpicture}[scale=0.3]
\fill[YellowGreen] (0,0) rectangle (1,1);
\draw (0,0) rectangle (1,1);
\end{tikzpicture}\}$ on the others. The other two words are similar, except for 
the orientation of the arrows, and 
the alphabet $\mathcal{E}$ is replaced 
by $\mathcal{E}^2$ in the third one.

\item \textbf{Detection signals:}

\begin{itemize}
\item 
\begin{enumerate}
\item On the third word, on the rightmost position, 
the detecting symbol is $\begin{tikzpicture}[scale=0.3] 
\fill[YellowGreen] (0,0) rectangle (1,1);
\draw (0,0) rectangle (1,1);
\end{tikzpicture}$ if and only if the south symbol of the tile 
is $\overline{e}_{\texttt{max}}$. 
\item The symbol $\begin{tikzpicture}[scale=0.3] 
\fill[YellowGreen] (0,0) rectangle (1,1);
\draw (0,0) rectangle (1,1);
\end{tikzpicture}$ propagates to the right while on the support line
and propagates to the left until meeting a tile with 
south symbol different from $\overline{e}_{\texttt{max}}$.
In this case, the symbol on the left is $\begin{tikzpicture}[scale=0.3] 
\fill[Salmon] (0,0) rectangle (1,1);
\draw (0,0) rectangle (1,1);
\end{tikzpicture}$. Moreover, the symbol $\begin{tikzpicture}[scale=0.3] 
\fill[YellowGreen] (0,0) rectangle (1,1);
\draw (0,0) rectangle (1,1);
\end{tikzpicture}$ can appear on a tile only if the south symbol 
is $\overline{e}_{\texttt{max}}$.
\item The symbol $\begin{tikzpicture}[scale=0.3] 
\fill[Salmon] (0,0) rectangle (1,1);
\draw (0,0) rectangle (1,1);
\end{tikzpicture}$ propagates to the left.
\end{enumerate}
\item On the second word, the rules are similar, 
except that the orientation of the line 
is reversed.
\item On the first word, the orientation 
is not reversed, but the first rule 
of the set of rules is changed: the rightmost symbol 
is $\begin{tikzpicture}[scale=0.3] 
\fill[Salmon] (0,0) rectangle (1,1);
\draw (0,0) rectangle (1,1);
\end{tikzpicture}$ if the rightmost symbol of the second word is also $\begin{tikzpicture}[scale=0.3] 
\fill[Salmon] (0,0) rectangle (1,1);
\draw (0,0) rectangle (1,1);
\end{tikzpicture}$ or if the south symbol on the corresponding tile is not $e_{\texttt{max}}$. 
Else it is $\begin{tikzpicture}[scale=0.3] 
\fill[YellowGreen] (0,0) rectangle (1,1);
\draw (0,0) rectangle (1,1);
\end{tikzpicture}$. As a consequence, one can think the two first words, 
from the point of view of the detecting signal, as a unique word. This word is the concatenation 
of the two words on the right side.
\item \textbf{Coherence rule:} 
If the symbol on a position $\vec{u}$ of 
the line in the first word is $\begin{tikzpicture}[scale=0.3] 
\fill[YellowGreen] (0,0) rectangle (1,1);
\draw (0,0) rectangle (1,1);
\end{tikzpicture}$, then the corresponding 
symbol in the second word is also 
$\begin{tikzpicture}[scale=0.3] 
\fill[YellowGreen] (0,0) rectangle (1,1);
\draw (0,0) rectangle (1,1);
\end{tikzpicture}$.
\end{itemize}

Here is a typical configuration of the detecting signals:

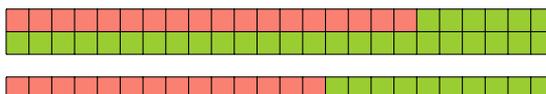
\[\begin{tikzpicture}[scale=0.3]
\fill[YellowGreen] (0,1) rectangle (24,2);
\fill[YellowGreen] (24,2) rectangle (18,3);
\fill[Salmon] (18,2) rectangle (0,3);
\draw (0,1) grid (24,3);
\fill[Salmon] (0,-1) rectangle (14,0);
\fill[YellowGreen] (14,-1) rectangle (24,0);
\draw (0,-1) grid (24,0);
\end{tikzpicture}\]

\item \textbf{Freezing signal:}

Recall that the freezing symbol is localized on 
the leftmost position of the line, denoted 
in this rule $\vec{u}$. 
\begin{itemize}
\item If the 
detecting signal on this position 
is $\begin{tikzpicture}[scale=0.3] 
\fill[YellowGreen] (0,0) rectangle (1,1);
\draw (0,0) rectangle (1,1);
\end{tikzpicture}$ for the three words, 
then the freezing symbol on positions 
$\vec{u}$ and $\vec{u}+\vec{e}^3$ 
are different, meaning 
that if the symbol 
on $\vec{u}$ is $\begin{tikzpicture}[scale=0.3] 
\fill[gray!10] (0,0) rectangle (1,1);
\draw (0,0) rectangle (1,1);
\end{tikzpicture}$, then it is changed 
into $\begin{tikzpicture}[scale=0.3] 
\fill[purple!80] (0,0) rectangle (1,1);
\draw (0,0) rectangle (1,1);
\end{tikzpicture}$ on the next position 
in direction $\vec{e}^3$, 
and the symbol $\begin{tikzpicture}[scale=0.3] 
\fill[purple!80] (0,0) rectangle (1,1);
\draw (0,0) rectangle (1,1);
\end{tikzpicture}$ is changed into 
$\begin{tikzpicture}[scale=0.3] 
\fill[gray!10] (0,0) rectangle (1,1);
\draw (0,0) rectangle (1,1);
\end{tikzpicture}$
\item In the other case, meaning 
when one of the detecting signals is 
$\begin{tikzpicture}[scale=0.3] 
\fill[Salmon] (0,0) rectangle (1,1);
\draw (0,0) rectangle (1,1);
\end{tikzpicture}$, then the freezing 
symbol is not changed.
\item \textbf{Coherence rule:} 
The symbol freezing 
symbol can be $\begin{tikzpicture}[scale=0.3] 
\fill[purple!80] (0,0) rectangle (1,1);
\draw (0,0) rectangle (1,1);
\end{tikzpicture}$ only if the three 
detecting signals are $\begin{tikzpicture}[scale=0.3] 
\fill[YellowGreen] (0,0) rectangle (1,1);
\draw (0,0) rectangle (1,1);
\end{tikzpicture}$.
\end{itemize}

\item \textbf{Incrementation of the counter:} 

\begin{itemize}
\item \textbf{Triggering the 
incrementation:} 
\begin{itemize}
\item \textbf{For the third word:}
On the leftmost position of the line, 
the west symbol of the 
tile in the third word is $0$ if and only if the 
freezing signal is $\begin{tikzpicture}[scale=0.3]
\fill[gray!10] (0,0) rectangle (1,1);
\draw (0,0) rectangle (1,1);
\end{tikzpicture}$ and 
the detecting signals are all $\begin{tikzpicture}[scale=0.3] 
\fill[YellowGreen] (0,0) rectangle (1,1);
\draw (0,0) rectangle (1,1);
\end{tikzpicture}$.
\item \textbf{For the other two words:}
On the leftmost position of the line, 
the west symbol of the tile in the second 
word is $1$ only when 
the detecting signal of the third word is 
$\begin{tikzpicture}[scale=0.3] 
\fill[YellowGreen] (0,0) rectangle (1,1);
\draw (0,0) rectangle (1,1);
\end{tikzpicture}$ and the one on the 
first one is $\begin{tikzpicture}[scale=0.3] 
\fill[Salmon] (0,0) rectangle (1,1);
\draw (0,0) rectangle (1,1);
\end{tikzpicture}$, or 
the freezing signal is 
$\begin{tikzpicture}[scale=0.3] 
\fill[purple!80] (0,0) rectangle (1,1);
\draw (0,0) rectangle (1,1);
\end{tikzpicture}$. This rule 
means that the value which consists 
of the two first words is incremented 
exactly when the third word 
reaches its maximal value, except 
when the whole counter is suspended.
\end{itemize}

\item \textbf{Transmission of increment:} 
\begin{itemize}
\item On a computation position 
except the leftmost, the west symbol 
of the tile is equal to the symbol in $\{0,1\}$ 
of the position on the left. A similar rule 
is true for the east symbol. 
\item Between two computation positions, 
the symbol in $\{0,1\}$ is 
transported.
\item On the rightmost position of the line, 
the east symbols of the two tiles of the first 
and second words are equal. This means 
that the increment is transmitted from the 
second to the first word.
\end{itemize}

\item \textbf{Transfer the transformed value 
in direction $\vec{e}^3$:} 
On each of the positions $\vec{u}$ in 
the line, for the
third word, the north (resp. south) symbol 
is equal to the south (resp. north) symbol 
of the position $\vec{u}+\vec{e}^3$ (resp. $\vec{u}-\vec{e}^3$) of the same word. 
\end{itemize}
\item \textbf{Rotation mechanism:}
\begin{itemize}
\item 
For all $\vec{u}$ in the line, except the 
rightmost (resp. leftmost) one, the south symbol 
of the tile in the first (resp. second) word 
on position $\vec{u}+\vec{e}^3$ is equal 
to the north symbol of the first (resp. second) word 
on position $\vec{u}+\vec{e}^1$ (resp. $\vec{u}-\vec{e}^1$).
\item When $\vec{u}$ is the rightmost (resp. leftmost) position of the line, the south symbol of the 
tile in the first (resp. second) word is 
equal to the north symbol of the tile 
in the second (resp. first) word on position 
$\vec{u}$. 
\end{itemize}
\end{itemize}

\noindent \textbf{\textit{Global behavior:}} \bigskip

In some particular sub-units of the odd 
order cells is supported a counter called 
system counter. The localization 
of this sub-unit depends on 
the class of the level modulo $4$: 
if this class is $\overline{1}$, 
the sub-unit is $(\overline{2},\overline{5})$
and $(\overline{5},\overline{2})$ 
when the class is $\overline{3}$. 

The value of the counter is the product of
three words on the same alphabet 
$\mathcal{E}$. The third word 
is incremented in direction 
$\vec{e}^3$ in each of the sections (except 
for the suspension step) and 
when it reaches the maximal value, 
it triggers the incrementation of the second 
value, which consists in the concatenation 
of the two other words. 
Moreover, this value is rotated between 
each couple of sections $\Z_c^2$ 
and $\Z_{c+1}^2$.
Since the number of column in each of the sub-units 
is $2^{n-3}$, 
the period of the system counter 
is 
\[2^{4. 2^m.2^{n-3}} +1 = 2^{2^{m+n-1}}+1.\]

\section{\label{section.information.transport} Information transports}

In this section we describe the various information transfers in this construction. 
In Section~\ref{section.diagonals} we describe how to color 
the diagonal of some of the sub-units, in order to change 
the direction of information transport. This mechanism will be used in the 
following subsections. In Section~\ref{subsection.transport.information.system.counter} we 
describe how the system counter information is transported to the walls of each cell.
Section~\ref{section.transport.linear.counter} is devoted to the description 
of information transfers relative to the linear counter inside the cells and 
Section~\ref{section.transport.intercellular} to information transfers between 
cells having the same level.

\subsection{\label{section.diagonals} Structure for direction changes 
of information transport}

\noindent \textbf{\textit{Symbols:}}

The symbols of this sublayer are 
$\begin{tikzpicture}[scale=0.4,baseline = 1mm]
\draw[line width=0.4mm] (0,0) -- (1,1);
\draw (0,0) rectangle (1,1);
\end{tikzpicture}$ and $\begin{tikzpicture}[scale=0.4,baseline = 1mm]
\draw (0,0) rectangle (1,1);
\end{tikzpicture}$. \bigskip

\noindent \textbf{\textit{Local rules:}}

\begin{itemize}
\item \textbf{Localization:} 
the petals are superimposed with non blank symbols, 
and other positions with blank one.
\item \textbf{Transmission:}
the symbols are transmitted through the petals 
except on \textbf{transformation positions}, 
defined to be the positions where a support petal 
intersects the transmission petal just above in the hierarchy.
\item \textbf{Transformation:} 

\begin{itemize}
\item When the symbol in the functional specialization sublayer is in $\Z/4\Z$ or $\Z/4\Z^2$, 
the symbol in the present layer is $\begin{tikzpicture}[scale=0.4,baseline = 1mm]
\draw[line width=0.4mm] (0,0) -- (1,1);
\draw (0,0) rectangle (1,1);
\end{tikzpicture}$. When it is $\begin{tikzpicture}[scale=0.3]
\fill[gray!90] (0,0) rectangle (1,1);
\draw (0,0) rectangle (1,1); \end{tikzpicture}$, the symbol 
in this layer is $\begin{tikzpicture}[scale=0.4,baseline = 1mm]
\draw (0,0) rectangle (1,1);
\end{tikzpicture}$.
\item When the symbol is in $\Z/4\Z^3$, the symbol can 
be $\begin{tikzpicture}[scale=0.4,baseline = 1mm]
\draw[line width=0.4mm] (0,0) -- (1,1);
\draw (0,0) rectangle (1,1);
\end{tikzpicture}$ or $\begin{tikzpicture}[scale=0.4,baseline = 1mm]
\draw (0,0) rectangle (1,1);
\end{tikzpicture}$. 
On a transformation position,
the symbol on the position is equal to the positions in the 
same support petal in the neighborhood. If both the symbol 
on this position and on the neighbors 
positions in the transmission petal in the functional specialization sublayer is in 
 $\Z/4\Z^3$, then the symbol on the transformation position is as follows: 

\begin{enumerate}
\item If the symbol in the Robinson layer is 
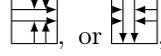
\[\begin{tikzpicture}[scale=0.3]
\draw (0,0) rectangle (2,2) ;
\draw [-latex] (0,1) -- (2,1) ;
\draw [-latex] (0,1.5) -- (2,1.5) ; 
\draw [-latex] (1,0) -- (1,1) ; 
\draw [-latex] (1,2) -- (1,1.5) ;
\draw [-latex] (1.5,0) -- (1.5,1) ; 
\draw [-latex] (1.5,2) -- (1.5,1.5) ;
\end{tikzpicture}, \ \text{or} \ \begin{tikzpicture}[scale=0.3]
\draw (0,0) rectangle (2,2) ;
\draw [-latex] (1,2) -- (1,0) ;
\draw [latex-] (0.5,0) -- (0.5,2) ; 
\draw [-latex] (0,1) -- (0.5,1) ; 
\draw [-latex] (2,1) -- (1,1) ;
\draw [-latex] (0,1.5) -- (0.5,1.5) ; 
\draw [-latex] (2,1.5) -- (1,1.5) ;
\end{tikzpicture},\]
meaning that the transformation position is in the north west 
part of the support petal,
and the orientation symbol is $\begin{tikzpicture}[scale=0.2] 
\fill[black] (1.5,2) rectangle (1,0.5);
\fill[black] (1,0.5) rectangle (0,1);
\draw (0,0) rectangle (2,2);
\end{tikzpicture}$, then the symbol on the transmission petal 
is $\begin{tikzpicture}[scale=0.4,baseline = 1mm]
\draw (0,0) rectangle (1,1);
\end{tikzpicture}$. For the orientations, the transformation is already determined 
by the first rule.
\item When the transformation position is in the south east part 
and the orientation is $\begin{tikzpicture}[scale=0.2] 
\fill[black] (0.5,0) rectangle (1,1.5);
\fill[black] (1,1.5) rectangle (2,1);
\draw (0,0) rectangle (2,2);
\end{tikzpicture}$, the symbol on the transmission petal is also 
$\begin{tikzpicture}[scale=0.4,baseline = 1mm]
\draw (0,0) rectangle (1,1);
\end{tikzpicture}$.
\item When the transformation position is in the south west part and 
the orientation is $\begin{tikzpicture}[scale=0.2] 
\fill[black] (1.5,0) rectangle (1,1.5);
\fill[black] (1,1.5) rectangle (0,1);
\draw (0,0) rectangle (2,2);
\end{tikzpicture}$ (or north east and $\begin{tikzpicture}[scale=0.2] 
\fill[black] (0.5,2) rectangle (1,0.5);
\fill[black] (1,0.5) rectangle (2,1);
\draw (0,0) rectangle (2,2);
\end{tikzpicture}$) then the symbol on the transmission petal is 
$\begin{tikzpicture}[scale=0.4,baseline = 1mm]
\draw[line width=0.4mm] (0,0) -- (1,1);
\draw (0,0) rectangle (1,1);
\end{tikzpicture}$. See an illustration of these rules on Figure~\ref{figure.structure.information.transport}.
\end{enumerate}
\end{itemize}
\end{itemize}

\begin{figure}[h]
\[\begin{tikzpicture}[scale=0.075]
\fill[gray!90] (16,16) rectangle (48,48); 
\fill[white] (16.5,16.5) rectangle (47.5,47.5);
\fill[gray!20] (32,0) rectangle (32.5,32.5);
\fill[gray!20] (32,32.5) rectangle (64,32);

\fill[gray!20] (8,8) rectangle (8.5,24); 
\fill[gray!20] (8,8) rectangle (24,8.5);
\fill[gray!20] (8,23.5) rectangle (24,24);  
\fill[gray!20] (23.5,8) rectangle (24,24);

\fill[gray!20] (40,8) rectangle (40.5,24); 
\fill[gray!20] (40,8) rectangle (56,8.5);
\fill[gray!20] (40,23.5) rectangle (56,24);  
\fill[gray!20] (55.5,8) rectangle (56,24);

\fill[gray!20] (8,40) rectangle (8.5,56); 
\fill[gray!20] (8,40) rectangle (24,40.5);
\fill[gray!20] (8,55.5) rectangle (24,56);  
\fill[gray!20] (23.5,40) rectangle (24,56);

\fill[gray!20] (40,40) rectangle (40.5,56); 
\fill[gray!20] (40,40) rectangle (56,40.5);
\fill[gray!20] (40,55.5) rectangle (56,56);  
\fill[gray!20] (55.5,40) rectangle (56,56);

\fill[gray!90] (4,4) rectangle (4.5,12); 
\fill[gray!90] (4,4) rectangle (12,4.5); 
\fill[gray!90] (4.5,11.5) rectangle (12,12); 
\fill[gray!90] (11.5,4.5) rectangle (12,12); 

\fill[gray!90] (4,20) rectangle (4.5,28); 
\fill[gray!90] (4,20) rectangle (12,20.5); 
\fill[gray!90] (4.5,27.5) rectangle (12,28); 
\fill[gray!90] (11.5,20.5) rectangle (12,28); 

\fill[gray!90] (4,36) rectangle (4.5,44); 
\fill[gray!90] (4,36) rectangle (12,36.5); 
\fill[gray!90] (4.5,43.5) rectangle (12,44); 
\fill[gray!90] (11.5,36.5) rectangle (12,44);

\fill[gray!90] (4,52) rectangle (4.5,60); 
\fill[gray!90] (4,52) rectangle (12,52.5); 
\fill[gray!90] (4.5,59.5) rectangle (12,60); 
\fill[gray!90] (11.5,52.5) rectangle (12,60);

\fill[gray!90] (20,4) rectangle (20.5,12); 
\fill[gray!90] (20,4) rectangle (28,4.5); 
\fill[gray!90] (20.5,11.5) rectangle (28,12); 
\fill[gray!90] (27.5,4.5) rectangle (28,12); 

\fill[gray!90] (20,20) rectangle (20.5,28); 
\fill[gray!90] (20,20) rectangle (28,20.5); 
\fill[gray!90] (20.5,27.5) rectangle (28,28); 
\fill[gray!90] (27.5,20.5) rectangle (28,28); 

\fill[gray!90] (20,36) rectangle (20.5,44); 
\fill[gray!90] (20,36) rectangle (28,36.5); 
\fill[gray!90] (20.5,43.5) rectangle (28,44); 
\fill[gray!90] (27.5,36.5) rectangle (28,44);

\fill[gray!90] (20,52) rectangle (20.5,60); 
\fill[gray!90] (20,52) rectangle (28,52.5); 
\fill[gray!90] (20.5,59.5) rectangle (28,60); 
\fill[gray!90] (27.5,52.5) rectangle (28,60);

\fill[gray!90] (36,4) rectangle (36.5,12); 
\fill[gray!90] (36,4) rectangle (44,4.5); 
\fill[gray!90] (36.5,11.5) rectangle (44,12); 
\fill[gray!90] (43.5,4.5) rectangle (44,12); 

\fill[gray!90] (36,20) rectangle (36.5,28); 
\fill[gray!90] (36,20) rectangle (44,20.5); 
\fill[gray!90] (36.5,27.5) rectangle (44,28); 
\fill[gray!90] (43.5,20.5) rectangle (44,28); 

\fill[gray!90] (36,36) rectangle (36.5,44); 
\fill[gray!90] (36,36) rectangle (44,36.5); 
\fill[gray!90] (36.5,43.5) rectangle (44,44); 
\fill[gray!90] (43.5,36.5) rectangle (44,44);

\fill[gray!90] (36,52) rectangle (36.5,60); 
\fill[gray!90] (36,52) rectangle (44,52.5); 
\fill[gray!90] (36.5,59.5) rectangle (44,60); 
\fill[gray!90] (43.5,52.5) rectangle (44,60);

\fill[gray!90] (52,4) rectangle (52.5,12); 
\fill[gray!90] (52,4) rectangle (60,4.5); 
\fill[gray!90] (52.5,11.5) rectangle (60,12); 
\fill[gray!90] (59.5,4.5) rectangle (60,12); 

\fill[gray!90] (52,20) rectangle (52.5,28); 
\fill[gray!90] (52,20) rectangle (60,20.5); 
\fill[gray!90] (52.5,27.5) rectangle (60,28); 
\fill[gray!90] (59.5,20.5) rectangle (60,28); 

\fill[gray!90] (52,36) rectangle (52.5,44); 
\fill[gray!90] (52,36) rectangle (60,36.5); 
\fill[gray!90] (52.5,43.5) rectangle (60,44); 
\fill[gray!90] (59.5,36.5) rectangle (60,44);

\fill[gray!90] (52,52) rectangle (52.5,60); 
\fill[gray!90] (52,52) rectangle (60,52.5); 
\fill[gray!90] (52.5,59.5) rectangle (60,60); 
\fill[gray!90] (59.5,52.5) rectangle (60,60);
\draw[->] (66,32) -- (48,32);
\node at (84,32) {$\vec{i} \in \Z/4\Z^3$};
\node at (8,8) {$\vec{i}$};
\node at (56,56) {$\vec{i}$};
\node at (8,56) {$\vec{i}$};
\node at (56,8) {$\vec{i}$};
\node at (24,24) {$\overline{3}$};
\node at (40,24) {$\overline{2}$};
\node at (40,40) {$\overline{1}$};
\node at (24,40) {$\overline{0}$};

\begin{scope}[xshift=-70cm]

\fill[gray!90] (16,16) rectangle (48,48); 
\fill[white] (16.5,16.5) rectangle (47.5,47.5);
\fill[gray!20] (32,0) rectangle (32.5,32.5);
\fill[gray!20] (32,32.5) rectangle (64,32);

\fill[gray!20] (8,8) rectangle (8.5,24); 
\fill[gray!20] (8,8) rectangle (24,8.5);
\fill[gray!20] (8,23.5) rectangle (24,24);  
\fill[gray!20] (23.5,8) rectangle (24,24);

\fill[gray!20] (40,8) rectangle (40.5,24); 
\fill[gray!20] (40,8) rectangle (56,8.5);
\fill[gray!20] (40,23.5) rectangle (56,24);  
\fill[gray!20] (55.5,8) rectangle (56,24);

\fill[gray!20] (8,40) rectangle (8.5,56); 
\fill[gray!20] (8,40) rectangle (24,40.5);
\fill[gray!20] (8,55.5) rectangle (24,56);  
\fill[gray!20] (23.5,40) rectangle (24,56);

\fill[gray!20] (40,40) rectangle (40.5,56); 
\fill[gray!20] (40,40) rectangle (56,40.5);
\fill[gray!20] (40,55.5) rectangle (56,56);  
\fill[gray!20] (55.5,40) rectangle (56,56);

\fill[gray!90] (4,4) rectangle (4.5,12); 
\fill[gray!90] (4,4) rectangle (12,4.5); 
\fill[gray!90] (4.5,11.5) rectangle (12,12); 
\fill[gray!90] (11.5,4.5) rectangle (12,12); 

\fill[gray!90] (4,20) rectangle (4.5,28); 
\fill[gray!90] (4,20) rectangle (12,20.5); 
\fill[gray!90] (4.5,27.5) rectangle (12,28); 
\fill[gray!90] (11.5,20.5) rectangle (12,28); 

\fill[gray!90] (4,36) rectangle (4.5,44); 
\fill[gray!90] (4,36) rectangle (12,36.5); 
\fill[gray!90] (4.5,43.5) rectangle (12,44); 
\fill[gray!90] (11.5,36.5) rectangle (12,44);

\fill[gray!90] (4,52) rectangle (4.5,60); 
\fill[gray!90] (4,52) rectangle (12,52.5); 
\fill[gray!90] (4.5,59.5) rectangle (12,60); 
\fill[gray!90] (11.5,52.5) rectangle (12,60);

\fill[gray!90] (20,4) rectangle (20.5,12); 
\fill[gray!90] (20,4) rectangle (28,4.5); 
\fill[gray!90] (20.5,11.5) rectangle (28,12); 
\fill[gray!90] (27.5,4.5) rectangle (28,12); 

\fill[gray!90] (20,20) rectangle (20.5,28); 
\fill[gray!90] (20,20) rectangle (28,20.5); 
\fill[gray!90] (20.5,27.5) rectangle (28,28); 
\fill[gray!90] (27.5,20.5) rectangle (28,28); 

\fill[gray!90] (20,36) rectangle (20.5,44); 
\fill[gray!90] (20,36) rectangle (28,36.5); 
\fill[gray!90] (20.5,43.5) rectangle (28,44); 
\fill[gray!90] (27.5,36.5) rectangle (28,44);

\fill[gray!90] (20,52) rectangle (20.5,60); 
\fill[gray!90] (20,52) rectangle (28,52.5); 
\fill[gray!90] (20.5,59.5) rectangle (28,60); 
\fill[gray!90] (27.5,52.5) rectangle (28,60);

\fill[gray!90] (36,4) rectangle (36.5,12); 
\fill[gray!90] (36,4) rectangle (44,4.5); 
\fill[gray!90] (36.5,11.5) rectangle (44,12); 
\fill[gray!90] (43.5,4.5) rectangle (44,12); 

\fill[gray!90] (36,20) rectangle (36.5,28); 
\fill[gray!90] (36,20) rectangle (44,20.5); 
\fill[gray!90] (36.5,27.5) rectangle (44,28); 
\fill[gray!90] (43.5,20.5) rectangle (44,28); 

\fill[gray!90] (36,36) rectangle (36.5,44); 
\fill[gray!90] (36,36) rectangle (44,36.5); 
\fill[gray!90] (36.5,43.5) rectangle (44,44); 
\fill[gray!90] (43.5,36.5) rectangle (44,44);

\fill[gray!90] (36,52) rectangle (36.5,60); 
\fill[gray!90] (36,52) rectangle (44,52.5); 
\fill[gray!90] (36.5,59.5) rectangle (44,60); 
\fill[gray!90] (43.5,52.5) rectangle (44,60);

\fill[gray!90] (52,4) rectangle (52.5,12); 
\fill[gray!90] (52,4) rectangle (60,4.5); 
\fill[gray!90] (52.5,11.5) rectangle (60,12); 
\fill[gray!90] (59.5,4.5) rectangle (60,12); 

\fill[gray!90] (52,20) rectangle (52.5,28); 
\fill[gray!90] (52,20) rectangle (60,20.5); 
\fill[gray!90] (52.5,27.5) rectangle (60,28); 
\fill[gray!90] (59.5,20.5) rectangle (60,28); 

\fill[gray!90] (52,36) rectangle (52.5,44); 
\fill[gray!90] (52,36) rectangle (60,36.5); 
\fill[gray!90] (52.5,43.5) rectangle (60,44); 
\fill[gray!90] (59.5,36.5) rectangle (60,44);

\fill[gray!90] (52,52) rectangle (52.5,60); 
\fill[gray!90] (52,52) rectangle (60,52.5); 
\fill[gray!90] (52.5,59.5) rectangle (60,60); 
\fill[gray!90] (59.5,52.5) rectangle (60,60);

\node at (8,8) {$\begin{tikzpicture}[scale=0.3,baseline = 1mm]
\draw[line width=0.4mm] (0,0) -- (1,1);
\draw (0,0) rectangle (1,1);
\end{tikzpicture}$};
\node at (56,56) {$\begin{tikzpicture}[scale=0.3,baseline = 1mm]
\draw[line width=0.4mm] (0,0) -- (1,1);
\draw (0,0) rectangle (1,1);
\end{tikzpicture}$};
\node at (8,56) {$\begin{tikzpicture}[scale=0.3,baseline = 1mm]
\draw (0,0) rectangle (1,1);
\end{tikzpicture}$};
\node at (56,8) {$\begin{tikzpicture}[scale=0.3,baseline = 1mm]
\draw (0,0) rectangle (1,1);
\end{tikzpicture}$};

\node at (24,24) {$\begin{tikzpicture}[scale=0.3,baseline = 1mm]
\draw[line width=0.4mm] (0,0) -- (1,1);
\draw (0,0) rectangle (1,1);
\end{tikzpicture}$};

\node at (40,40) {$\begin{tikzpicture}[scale=0.3,baseline = 1mm]
\draw[line width=0.4mm] (0,0) -- (1,1);
\draw (0,0) rectangle (1,1);
\end{tikzpicture}$};

\node at (24,40) {$\begin{tikzpicture}[scale=0.3,baseline = 1mm]
\draw[line width=0.4mm] (0,0) -- (1,1);
\draw (0,0) rectangle (1,1);
\end{tikzpicture}$};

\node at (40,24) {$\begin{tikzpicture}[scale=0.3,baseline = 1mm]
\draw[line width=0.4mm] (0,0) -- (1,1);
\draw (0,0) rectangle (1,1);
\end{tikzpicture}$};

\end{scope}

\end{tikzpicture}\]
\caption{\label{figure.structure.information.transport} Schematic illustration 
of special case of the transformation rules for the structures of direction 
changes when the order of the central cell is $\ge 3$.}
\end{figure}
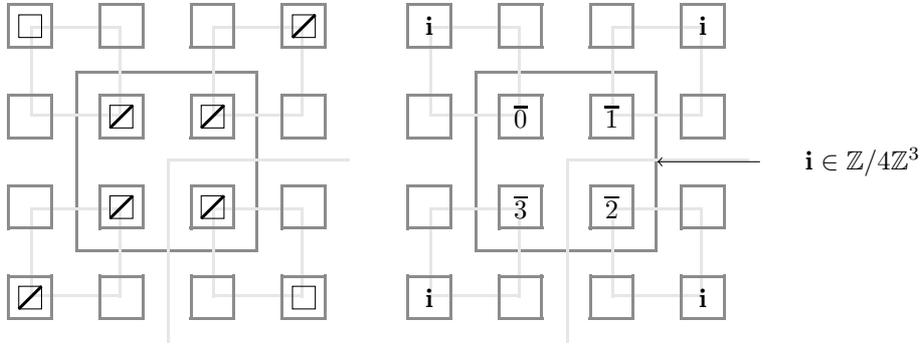

\noindent \textbf{\textit{Global behavior:}} \bigskip

As a consequence of the hierarchical signaling process 
described by the local rules, each position 
superimposed with a blue corner in the Robinson layer 
is superimposed with $\begin{tikzpicture}[scale=0.4,baseline = 1mm]
\draw[line width=0.4mm] (0,0) -- (1,1);
\draw (0,0) rectangle (1,1);
\end{tikzpicture}$ or $\begin{tikzpicture}[scale=0.4,baseline = 1mm]
\draw (0,0) rectangle (1,1);
\end{tikzpicture}$. The positions having the first symbol 
are on the diagonal of a sub-unit of an order $\ge 3$ cell.

\subsection{\label{subsection.transport.information.system.counter}
Random channels}

In this section we give details on the transport of the system counter symbols 
through channels which depend on the modularity of the level of cells.

We describe in Section~\ref{section.modularity.mark.one.transport} how the information is 
transported to the border of the green area on Figure~\ref{figure.functional.diagram}.
In 
Section~\ref{section.modularity.mark.three.transport} we describe the transport to 
the border of the cell. This is where the transport depends on the modularity mark.

\subsubsection{\label{section.modularity.mark.one.transport} 
Transport of the simulating bit to the border of the modularity mark area}

\noindent \textbf{\textit{Symbols:}} \bigskip

This sublayer has symbols in $\mathcal{E}$ and 
a blank symbol.

\bigskip

\noindent \textbf{\textit{Local rules:}}

\begin{itemize}
\item \textbf{Localization:} 
\begin{itemize}
\item The non-blank symbols are superimposed on and only on sub-units $(\overline{k},\overline{4})$ 
for $k \in \llbracket 1 , 6 \rrbracket$ and $(\overline{2},\overline{k})$ for 
$k \in \llbracket 1,6\rrbracket$ when 
the modularity mark is $\overline{1}$. 
When the modularity mark 
is $\overline{3}$, these sub-units 
are $(\overline{k},\overline{1})$ 
for $k \in \llbracket 1 , 6 \rrbracket$ and $(\overline{5},\overline{k})$ for 
$k \in \llbracket 1,6\rrbracket$.
When the modularity mark is $\overline{0}$ or $\overline{2}$, all 
the symbols are blank.

\item On the sub-units $(\overline{k},\overline{4})$ or 
$(\overline{k},\overline{1})$ the non-blank symbols are on 
the leftmost column. On sub-units $(\overline{2},\overline{k})$ or 
$(\overline{5},\overline{k})$ they are on the bottommost row.
\end{itemize}
\item \textbf{Transmission:}
The symbol is transmitted through sub-units $(\overline{k},\overline{4})$ or 
$(\overline{k},\overline{1})$ and through sub-units $(\overline{2},\overline{k})$ or 
$(\overline{5},\overline{k})$.
\item \textbf{Synchronization:}
On the sub-unit $(\overline{5},\overline{2})$ or $(\overline{2},\overline{5})$, on 
the leftmost and bottommost position, the symbol is equal to the symbols 
on the top, bottom, left and right in this layer, and to the first bit 
of the second value of the system counter. 
\end{itemize}

\subsubsection{\label{section.modularity.mark.three.transport}
Transport to the border of the cell}

In this layer we describe how this simulating bit is transported to the 
border of the cells. This transport use channels which depend on the modularity mark, 
so that it is 'not known' locally, on the location of the system bits, if these bits are simulated 
by the system counter, since while the counter is not 'seen' locally it is still possible that the 
counter undertook the other channel.

\noindent \textbf{\textit{Symbols:}} \bigskip

This sublayer has symbols in $\mathcal{E}$ and 
a blank symbol.

\bigskip

\noindent \textbf{\textit{Local rules:}}

\begin{itemize}
\item \textbf{Localization:} 
The non-blank symbols can be superimposed only on the bottom line 
of sub-units $(\overline{0},\overline{1})$, $(\overline{0},\overline{4})$, 
$(\overline{7},\overline{1})$, $(\overline{7},\overline{4})$ and the 
leftmost column of sub-units $(\overline{2},\overline{0})$, 
$(\overline{2},\overline{7})$, $(\overline{5},\overline{0})$, $(\overline{5},\overline{7})$.
\item \textbf{Transmission:}
Through the border of these sub-units with the 
modularity mark area, the symbol is transmitted.
\item \textbf{Evaluation:}
Through the border with the cell, the symbol 
is equal to the system bit if not blank.
\end{itemize}

\subsubsection{Global behavior}

\begin{figure}[ht]
\[\begin{tikzpicture}[scale=0.4]

\fill[YellowGreen!30] (1,1) rectangle (7,7);
\fill[orange!50] (2,5) rectangle (3,6);

\draw (0,0) grid (8,8);

\draw[orange,-latex,line width=0.4mm] 
(2,-0.5) -- (2,8.5);
\draw[orange,-latex,line width=0.4mm] 
(-0.5,4) -- (8.5,4);

\node at (4,-2) {Modularity mark $\overline{1}$};

\begin{scope}[xshift=12cm]

\fill[YellowGreen!30] (1,1) rectangle (7,7);
\fill[orange!50] (5,2) rectangle (6,3);

\draw (0,0) grid (8,8);

\draw[orange,-latex,line width=0.4mm] 
(-0.5,1) -- (8.5,1);
\draw[orange,-latex,line width=0.4mm] 
(5,-0.5) -- (5,8.5);

\node at (4,-2) {Modularity mark $\overline{3}$};

\end{scope}
\end{tikzpicture}\]
\caption{\label{figure.information.transfer.to.the.border} Random channels 
and communication of the system counter information to the border of the cell}
\end{figure}
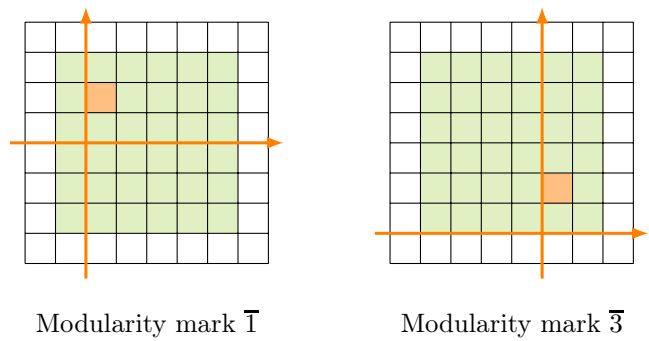

The consequence of the local rules is that the symbols of the system counter layer are transported through 
channels depending on the modularity mark, as on Figure~\ref{figure.information.transfer.to.the.border}.
Moreover, when crossing the location of the system bit, meaning the south or the west border of the cell, 
the system bit is determined by the counter value, according to the specified position in the counter.

\subsection{\label{section.transport.linear.counter} Transport of the linear counter}

In this section we describe the transport of information relative to the linear counter:
simple transport in Section~\ref{section.transport.linear.counter} 
and extraction towards the machine area in Section~\ref{section.extraction.linear.counter}.

\subsubsection{\label{section.transport.linear.counter} Transmission of the counter information}

\noindent \textbf{\textit{Symbols:}} \bigskip

The symbols in this layer are elements of the sets $\mathcal{A}_c \times \{\begin{tikzpicture}[scale=0.3]
\fill[Salmon] (0,0) rectangle (1,1);
\draw (0,0) rectangle (1,1);
\end{tikzpicture}, \begin{tikzpicture}[scale=0.3]
\draw (0,0) rectangle (1,1);
\end{tikzpicture}\}$, $\mathcal{A}_c ^2 \times \{\begin{tikzpicture}[scale=0.3]
\fill[Salmon] (0,0) rectangle (1,1);
\draw (0,0) rectangle (1,1);
\end{tikzpicture}, \begin{tikzpicture}[scale=0.3]
\draw (0,0) rectangle (1,1);
\end{tikzpicture}\}$ 
and $\{\begin{tikzpicture}[scale=0.3]
\fill[Salmon] (0,0) rectangle (1,1);
\draw (0,0) rectangle (1,1);
\end{tikzpicture}, \begin{tikzpicture}[scale=0.3]
\draw (0,0) rectangle (1,1);
\end{tikzpicture}\}$ and a blank symbol. The set $\mathcal{A}_c$ is described in Annex~\ref{section.linear.counter}.
\bigskip

\noindent \textbf{\textit{Local rules:}}

\begin{itemize}
\item \textbf{Localization:}
\begin{itemize}
\item The non-blank symbols are superimposed on sub-units $(\overline{0},\overline{3})$, 
$(\overline{0},\overline{7})$, $(\overline{1},\overline{3})$, $(\overline{1},\overline{4})$, 
$(\overline{1},\overline{5})$, $(\overline{1},\overline{6})$, $(\overline{1},\overline{7})$, 
$(\overline{3},\overline{k})$ for all $k \in \llbracket 0,7\rrbracket$, $(\overline{4},\overline{3})$, 
$(\overline{5},\overline{3})$, $(\overline{6},\overline{3})$, $(\overline{7},\overline{3})$, 
$(\overline{7},\overline{4})$, $(\overline{7},\overline{5})$, $(\overline{7},\overline{6})$, 
$(\overline{7},\overline{7})$, $(\overline{6},\overline{7})$, $(\overline{5},\overline{7})$ and 
$(\overline{4},\overline{7})$.
\item On these sub-units, all the positions of the cytoplasm are superimposed with a color in 
$\{\begin{tikzpicture}[scale=0.3]
\fill[Salmon] (0,0) rectangle (1,1);
\draw (0,0) rectangle (1,1);
\end{tikzpicture}, \begin{tikzpicture}[scale=0.3]
\draw (0,0) rectangle (1,1);
\end{tikzpicture}\}$.
\item On sub-units $(\overline{0},\overline{3})$, 
$(\overline{0},\overline{7})$, $(\overline{4},\overline{3})$, 
$(\overline{4},\overline{7})$, $(\overline{6},\overline{3})$, 
$(\overline{6},\overline{7})$,
the information transfer lines are superimposed also with an element 
of $\mathcal{A}_c$.
\item On sub-units $(\overline{1},\overline{3})$, 
$(\overline{1},\overline{7})$, $(\overline{3},\overline{3})$, $(\overline{3},\overline{5})$
$(\overline{3},\overline{7})$, $(\overline{5},\overline{3})$, 
$(\overline{5},\overline{7})$, $(\overline{7},\overline{3})$, $(\overline{7},\overline{5})$
$(\overline{7},\overline{7})$ the computation positions are superimposed with an 
element of $\mathcal{A}_c ^2$ and the other information transfer lines and columns are superimposed 
with an element of $\mathcal{A}_c$.
\item On the other sub-units in the first localization rule, the information transfer columns are superimposed 
with an element of $\mathcal{A}_c$.
\end{itemize}
\item \textbf{Transmission:}
\begin{itemize}
\item The symbols in $\mathcal{A}_c$ are transmitted along the information transfer lines 
and columns. In sub-units of the third localization rule, the propagation is stopped 
at computation positions. 
\item On these positions, the first symbol of $\mathcal{A}_c ^2$ is transmitted vertically and 
the other one horizontally.
\item Across the border of two horizontally (resp. vertically) adjacent of these sub-units, the 
second (resp. first symbol) or the unique symbol in $\mathcal{A}_c$ is transmitted. 
\end{itemize}
\item \textbf{Deviation:}
On the positions marked with \begin{tikzpicture}[scale=0.4,baseline = 0.5mm]
\draw[line width = 0.3mm] (0,0) -- (1,1);
\draw (0,0) rectangle (1,1);
\end{tikzpicture}, on the sub-units of the third localization 
rule, the two elements of the pair $\mathcal{A}_c^2$ are equal.
\end{itemize}

\noindent \textbf{\textit{Global behavior:}}

The information of the linear counter is transported through a circuit 
in all the cells, which is represented by blue arrows on Figure~\ref{figure.functional.diagram}.

\subsubsection{\label{section.extraction.linear.counter} Extraction mechanism}

\noindent \textbf{\textit{Symbols:}} \bigskip

Elements of $\mathcal{A}' \times \mathcal{Q} \times \{\texttt{on},\texttt{off}\}$, of 
$\mathcal{Q} \times \{\texttt{on},\texttt{off}\}$, of $\mathcal{Q}$, of $\{\rightarrow,\leftarrow\}$ and a blank symbol.

\bigskip

\noindent \textbf{\textit{Local rules:}} 

\begin{itemize}
\item \textbf{Localization:} the non-blank symbols are superimposed on 
information transfer lines of sub-units $(\overline{4},\overline{5})$ (elements 
of $\mathcal{Q} \times \{\texttt{on},\texttt{off}\}$) and $(\overline{6},\overline{5})$ (elements of $\mathcal{Q}$)
and on the information transfer columns of sub-units $(\overline{5},\overline{4})$ (elements 
of $\mathcal{A}' \times \mathcal{Q} \times \{\texttt{on},\texttt{off}\}$) 
and $(\overline{5},\overline{6})$ (elements of $\{\rightarrow,\leftarrow\}$).
\item \textbf{Transmission:} 
Across the border of these sub-units with sub-units of the last sub-layer, the symbol 
is equal to the corresponding part of $\mathcal{A}_c$ in the first (resp. second) symbol of the pair for 
sub-units $(\overline{5},\overline{4})$ and $(\overline{5},\overline{6})$ (resp. $(\overline{4},\overline{5})$
and $(\overline{6},\overline{5})$).  
\end{itemize}

\subsection{\label{section.transport.intercellular} 
Intercellular transport}

The last transport mechanism to describe 
is the intercellular transport, which allows the 
counters values to be synchronized between cells 
having the same level and in the same section $\Z_c^2$. \bigskip

\noindent \textbf{\textit{Symbols:}} 

The symbols of this layer are symbols in $\mathcal{G}$ or $\mathcal{G}^2$, 
where symbols of $\mathcal{G}$ are pairs of a symbol in the alphabet of the linear counter 
together with symbol in the simulating bit layer. We add also a blank symbol. \bigskip

\noindent \textbf{\textit{Local rules:}} 

\begin{itemize}
\item \textbf{Localization:}
The non blank symbols are superimposed on and only on positions 
with a blue corner in the structure layer and a non-blue symbol 
in the functional areas layer [Annex~\ref{subsection.functional.areas.whole.cell}], meaning 
non-computation positions. On the positions with an arrow, the symbol 
is in $\mathcal{G}$, and on the position with a light gray symbol, 
\begin{tikzpicture}[scale=0.3]
\fill[gray!20] (0,0) rectangle (1,1);
\draw (0,0) rectangle (1,1);
\end{tikzpicture}, the symbol is in $\mathcal{G}^2$.
\item \textbf{Transmission:} 
\begin{itemize}
\item On the positions with 
\begin{tikzpicture}[scale=0.3]
\fill[gray!20] (0,0) rectangle (1,1);
\draw (0,0) rectangle (1,1);
\end{tikzpicture}, the first (resp. second) symbol is transmitted to the next 
functional positions vertically (resp. horizontally).
\item On positions with arrow symbol, the symbol is transmitted to the next 
position in direction orthogonal to the arrow.
\end{itemize}
\end{itemize}

\noindent \textbf{\textit{Global behavior:}} \bigskip

The consequence of the local rules is that the information 
of the linear counter of each cell , as well as 
the simulating bit if any, is transmitted to the next 
cells having the same order in directions $\vec{e}^1$ and $\vec{e}^2$. 
Thus, these informations are synchronized for cells having the same level 
over a section.

\section{Proof of Theorem~\ref{theorem.main.introduction}}

\subsection{Simulation}

In this section we prove that for any 
of the dynamical systems $(Z,f)$, the subshift $X_{(Z,f)}$ simulates 
$(Z,f)$, by defining a function $\varphi$. \bigskip

Given a configuration $x$ of $X_{(Z,f)}$, $\varphi(x)$ is the sequence 
$(\epsilon_n)_n$ such that for all $n$, $\epsilon_n$ is the system bit 
written on order $2n$ two-dimensional cells in section $\Z^2_{0}$. 

This function is \textbf{computable}: in order to compute $\varphi(x)$, 
consider successively the patterns $x_{\llbracket -N,N\rrbracket ^3 \cap \Z^2_{0}}$ 
for all $N$ and for all $n \ge N$ search for a complete order $2n$ two-dimensional cell. 
If there is one, then $\varphi(x)_n$ is the system bit with which the border of the cell 
is colored. 

Moreover, this function is \textbf{onto}. Indeed, from any configuration $z$ in $Z$, 
one can construct some configuration $x$ in $X_{(Z,f)}$ whose image by $\varphi$ is $z$, 
by writing the sequences $f^{c} (z)$ in sections $\Z^2_c$. This is possible since $f$ is 
onto. 

\subsection{Minimality}

\subsubsection{Completing patterns}

\begin{proposition}
\label{proposition.completing.into.stacks}
Any pattern in the language of $X_{(Z,f)}$ is sub-pattern of 
a pattern on some $\llbracket n_1 , m_1 \rrbracket \times \llbracket n_2 , m_2 \rrbracket  \times 
\llbracket n_3 , m_3 \rrbracket $ in the language of $X_{(Z,f)}$
whose projection on a section $\Z^2_c$ is a cell 
of the Robinson subshift.
\end{proposition}

Let $P$ some $n$ block in the language of $X_{(Z,f)}$ which appears in some 
configuration $x$.
Let us prove that it can be completed into a stack of two-dimensional cells.
We follow some order in the layers for the completion. 
First we complete the pattern in the structure layer, 
then in the functional areas layer, the linear counter and machine layers 
and then the system counter.

\paragraph{Completion of the structure and functional areas}

When the projection of the pattern $P$ on a section $\Z^2_c$ 
and on the Robinson sublayer is a sub-pattern of an infinite supertile of $x$
in $\Z^2_c$, 
this is clear that the projection of $P$ into 
the structure layer can be completed into a finite supertile in 
the configuration $x$.

When this projection crosses the separating area 
between the supports of the infinite supertiles, as 
in the proof of Proposition~\ref{prop.complete.rob}, 
we can still complete 
the projection of $P$ over the structure layer into great enough supertile, 
this time the supertile does not appear directly in the configuration $x$.

This supertile can be completed then into a cell. The order of this cell is 
denoted $n_0$. When the part of the cell containing the modularity mark of 
this cell is known, we choose $n_0$ according to this mark (there is no contradiction 
with other information since the system bits are not known in this part). 
When this is not the case, we choose $n_0$ to be odd, and the modularity 
is chosen according to the location of the random channel if there is any 
in the pattern $P$. 

Then we complete the functional areas layer according to the hierarchical signaling process.

At this point, we have a completion of the pattern $P$ in the structure layer 
and functional areas layer as a stack of cells.

\paragraph{Completion of the counters and machines computations}

In the following paragraphs, we tell 
how to complete the pattern over this stack of cells
in the linear counter, system counter and machines layers. 
For the order $n < n_0$ cell, these layers can be completed according 
to the configuration $x$. The difficulty comes from completing 
these layers over the proper positions of the order $n_0$ cell.

The projection of the pattern $P$ over a section $\Z^2_c$ intersects 
at most four sub-units. We have to explain how to complete the pattern 
in a section according to the cases when it intersects (with increasing difficulty) information transport 
sub-units, demultiplexers, counters incrementation sub-units, machine sub-units. 
We can consider these intersections independently since when the pattern intersects two 
adjacent sub-units, the connection rule between the sub-units is verified inside the 
pattern $P$ and we complete the two sub-units so that the rules are verified.

\paragraph{Intersection with information transfer areas}

In the case when the pattern intersects information transfer areas, the only restriction 
is that the added symbols in the direction of transfer agree with the symbols 
in $P$, which is always possible.

\paragraph{Intersection with a demultiplexer}

When the pattern intersects a demultiplexer, there are two possible cases: the 
pattern intersects the synchronizing diagonal or not. Let us consider the second case first. 
In this case, we can consider without loss of generality that 
the known part of the demultiplexer is the south east part. In this case, we have 
to complete the leftmost columns and the topmost rows according to the pattern $P$. 
This is possible since $P$ does not intersect these columns and rows. All the other 
symbols can be chosen freely. In the first case, one the only restriction 
is that the added columns and rows agree on the diagonal, and again this is possible 
since the pattern $P$ does not intersect them.

\paragraph{Intersection with the linear counter incrementation sub-unit} 

Considering the intersection with the linear counter incrementation sub-unit, 
the completion depends on where the pattern intersects it.

\begin{enumerate}
\item when the pattern intersects the east part or 
only the inside of the area, the completion 
is as for information transfer areas;
\item when knowing a part 
of the west, the symbols are added according to the freezing signal 
and the incrementation signal. 
If the freezing signal on the two sides of the leftmost column 
is different, the freezing signal is imposed to be \begin{tikzpicture}[scale=0.3]
\fill[Salmon] (0,0) rectangle (1,1);
\draw (0,0) rectangle (1,1);
\end{tikzpicture} all along the column and the symbols 
are forced to be $c_{\texttt{max}}$. When this is not 
the case, the symbols 
can be chosen freely except that one 
of them has to be different from $c_{\texttt{max}}$. On the bottom however, if there is 
a freezing signal \begin{tikzpicture}[scale=0.3]
\fill[Salmon] (0,0) rectangle (1,1);
\draw (0,0) rectangle (1,1);
\end{tikzpicture} in the known part, the symbols 
have to be $c_{\texttt{max}}$, which is coherent 
with the knowledge the pattern has. If not, then these symbols 
can also be chosen freely.
\end{enumerate} 

The completion of the system counter area is similar. The main 
difference is the rotation mechanism which implies no 
difficulty for the completion. 

\paragraph{Intersection with the machine area} 

This part is similar to the corresponding part in~\cite{GS17ED} and 
we refer to Annex~\ref{section.completing.machines}.

\paragraph{Intersection with the system bits}

When the pattern intersects the system bits location, these bits 
are chosen according to the configuration $x$. When this 
is not the case, $n_0$ is chosen odd and we choose the bits 
for the order $n_0$ cell according to the 
value of the system counter. 

\subsubsection{Recovering system bits and counter values}

In this section we prove that the subshift $X_{(Z,f)}$ is minimal. 
This proof relies first on the following lemma: 

\begin{lemma}[Globach's theorem] \label{lm.fermat.numbers}
The numbers $F_i = 2^{2^i}+1$, $n \ge 0$ are coprime. 
\end{lemma}

\begin{proof}
Let $i > j \ge 0$. Then 
\[F_i = 2^{2^i}+1 = (2^{2^j} + 1 - 1)^{2^{i-j}} +1 = \sum_{k=0}^{2^{i-j}} \binom{2^{i-j}}{k} 
(-1)^{-k} F_n^k + 1 \]
\[= F_j \sum_{k=1}^{2^{i-j}} \binom{2^{i-j}}{k} (-1)^{k} F_n^{k-1} + 2 \]
This means that a common divisor of $F_j$ and $F_i$ divide $2$, but $2$ does not divide $F_j$, 
so $F_j$ and $F_i$ are coprime. 
\end{proof}

Consider some block $P$ in the language of 
$X$, and complete it into a pattern $P'$ using Proposition~\ref{proposition.completing.into.stacks}. 
Pick some configuration $x' \in X$ and some stack of order $n_0$ cells 
in this configuration having the same height as $P'$. \bigskip

Let $\mathcal{T}$ be the following application: 
\[\begin{array}{ccccc} \mathcal{T} & : & \Z/p_1 \Z \times ... \times \Z/p_k \Z & \rightarrow & \Z/p_1 \Z \times ... \times \Z/p_k \Z\\
& & (i_1 , ... , i_k) & \mapsto & (i_1 + 4^{(2^k-1)p} , ... , i_k +1) 
\end{array},\] where $p_1$, ... , $p_k$ are 
the periods of $k$ linear counters contained in the 
order $n_0$ cell. This is a minimal application, 
meaning that for all $\vec{i},\vec{j} \in \Z/p_1 \Z \times ... \times \Z/p_k \Z$, 
there exists some $n$ such that $\mathcal{T}^n (\vec{i}) = \vec{j}$.
Indeed, considering some $\vec{i}$, denote $n_1$ the smallest 
positive integer such that $\mathcal{T}^{n_1} (\vec{i}) = \vec{i}$. 
This means that $p_j$ divides $n_1 4^{(2^k-2^j)p}$ for all $j$, 
and because $p_j$ is a Fermat number, it is odd, and 
this implies that $p_j$ divides $n_1$. 
For the numbers $p_j$ are coprime (Lemma~\ref{lm.fermat.numbers}), 
this implies that $p_1 \times ... \times p_j$  
divides $n_1$. Because this number is a period of the 
application $\mathcal{T}$, this means that $p_1 \times ... \times p_j$ is the smallest 
period of every element of $\Z/p_1 \Z \times ... \times \Z/p_k \Z$ under the application $\mathcal{T}$.

As a consequence, for all $\vec{i}$, the finite sequence $(\mathcal{T}^n (\vec{i}))$, for $n$ going 
from $0$ to $p_1 \times ... \times p_k-1$, takes all the possible values 
in $\Z/p_1 \Z \times ... \times \Z/p_k \Z$. \bigskip

By jumping from one order $n_0$ cell to the next one in direction $\vec{e}^2$
multiple time one can find in $x'$ a stack of cells so that 
the values of the linear counters in this pattern are those of $P'$.

By jumping in direction $\vec{e}^3$, and since 
the application \[\begin{array}{ccccc} \mathcal{T}' & : & \Z/q_1 \Z \times ... \times \Z/q_{k'} \Z & \rightarrow & \Z/q_1 \Z \times ... \times \Z/q_{k'} \Z\\
& & (i_1 , ... , i_{k'}) & \mapsto & (i_1 + 1 , ... , i_{k'} +1) 
\end{array},\]
where $q_1, .. , q_{k'}$ are the periods of the system counters contained in 
the pattern $P'$
is also minimal, one can find a stack of $n_0$ order cells 
with the same values for the system counters as in $P'$. Since 
the linear counters are synchronized in this direction, 
the values of the linear counters are also those of $P'$. \bigskip

Now we make use of the following lemma:

\begin{lemma}[\cite{DR17}]
\label{lemma.minimality.one.dimensional}
Let $W$ be a minimal $\Z$-subshift, $w$ one of
its elements and $p$ some pattern that appears in $w$ on position 
$u \in \Z$. For all integer $N>0$, there exists some $t$ integer 
such that $p$ appears in $x$ on position $u+Nt$.
\end{lemma}

We take $N= q_1 ... q_k$, and denote $h$ the height of $P'$, 
and $r$ the number of even level of cells 
included in $P'$. Let us consider $W$ the subshift 
on $(\mathcal{A}^r)^h$ such that for a position $\vec{u} \in \Z$, 
the symbol  on $\vec{u}+1$ has its $h-1$ first $r$-uplets 
equal to the $h-1$ last ones of the one on position $\vec{u}$.
Moreover, two consecutive of $r$-uplets have to be possibly completed 
into elements of $Z$ such that the second one 
is the image by $f$ of the first one. In other 
words, $W$ is obtained from $(Z,f)$ by factoring onto 
the $r$ first bits and then by contatenating $h$ consecutive of 
sequences of $r$ bits.

The subshift $W$ is minimal. By applying Lemma~\ref{lemma.minimality.one.dimensional}
to the sequence in $W$ obtained by restricting to the non simulated system bits 
in the columns of $x'$ containing the pattern $P'$, such that 
the position $\vec{0} \in \Z$ corresponds to the bottom section of $P'$, 
one can recover the non-simulated system bits of $P'$ by shifting $Nt$ 
times in direction $\vec{e}^3$, where $t$ the integer given by the lemma. 
This does not change the linear counter values or the system counter values, 
since $N$ is a multiple of all the periods 
of the system counters contained in the stack of cells.

As a consequence, $P'$ appears in $x'$. This proves the minimality.

\begin{figure}[ht]
\[\begin{tikzpicture}[scale=0.15]
\begin{scope}
\draw (0,0) -- (5,0) ;
\draw[dashed] (0,0) -- (3.5,1.25) ;
\draw[dashed] (3.5,1.25) -- (8.5,1.25);
\draw (8.5,1.25) -- (5,0);

\draw (0,1) -- (5,1) ;
\draw[dashed] (0,1) -- (3.5,2.25) ;
\draw[dashed] (3.5,2.25) -- (8.5,2.25);
\draw (8.5,2.25) -- (5,1);

\draw (0,2) -- (5,2);
\draw[dashed] (0,2) -- (3.5,3.25) ;
\draw[dashed] (3.5,3.25) -- (8.5,3.25);
\draw (8.5,3.25) -- (5,2);

\draw (0,3) -- (5,3);
\draw[dashed] (0,3) -- (3.5,4.25) ;
\draw[dashed] (3.5,4.25) -- (8.5,4.25);
\draw (8.5,4.25) -- (5,3);

\draw (0,4) -- (5,4) ;
\draw[dashed] (0,4) -- (3.5,5.25) ;
\draw[dashed] (3.5,5.25) -- (8.5,5.25);
\draw (8.5,5.25) -- (5,4);

\draw (0,5) -- (5,5) ;
\draw (0,5) -- (3.5,6.25) ;
\draw(3.5,6.25) -- (8.5,6.25);
\draw (8.5,6.25) -- (5,5);

\end{scope}

\begin{scope}[yshift=8cm]
\draw (0,0) -- (5,0) ;
\draw[dashed] (0,0) -- (3.5,1.25) ;
\draw[dashed] (3.5,1.25) -- (8.5,1.25);
\draw (8.5,1.25) -- (5,0);

\draw (0,1) -- (5,1) ;
\draw[dashed] (0,1) -- (3.5,2.25) ;
\draw[dashed] (3.5,2.25) -- (8.5,2.25);
\draw (8.5,2.25) -- (5,1);

\draw (0,2) -- (5,2);
\draw[dashed] (0,2) -- (3.5,3.25) ;
\draw[dashed] (3.5,3.25) -- (8.5,3.25);
\draw (8.5,3.25) -- (5,2);

\draw (0,3) -- (5,3);
\draw[dashed] (0,3) -- (3.5,4.25) ;
\draw[dashed] (3.5,4.25) -- (8.5,4.25);
\draw (8.5,4.25) -- (5,3);

\draw (0,4) -- (5,4) ;
\draw[dashed] (0,4) -- (3.5,5.25) ;
\draw[dashed] (3.5,5.25) -- (8.5,5.25);
\draw (8.5,5.25) -- (5,4);

\draw (0,5) -- (5,5) ;
\draw (0,5) -- (3.5,6.25) ;
\draw (3.5,6.25) -- (8.5,6.25);
\draw (8.5,6.25) -- (5,5);

\draw[-latex] (9.5,0) -- (9.5,10);
\node at (14.5,5) {$Nt.\sigma^{\vec{e}^3}$};

\end{scope}

\begin{scope}[yshift=18cm]

\node at (-2,2.5) {$P'$};
\draw (0,0) -- (5,0) ;
\draw[dashed] (0,0) -- (3.5,1.25) ;
\draw[dashed] (3.5,1.25) -- (8.5,1.25);
\draw (8.5,1.25) -- (5,0);

\draw (0,1) -- (5,1) ;
\draw[dashed] (0,1) -- (3.5,2.25) ;
\draw[dashed] (3.5,2.25) -- (8.5,2.25);
\draw (8.5,2.25) -- (5,1);

\draw (0,2) -- (5,2);
\draw[dashed] (0,2) -- (3.5,3.25) ;
\draw[dashed] (3.5,3.25) -- (8.5,3.25);
\draw (8.5,3.25) -- (5,2);

\draw (0,3) -- (5,3);
\draw[dashed] (0,3) -- (3.5,4.25) ;
\draw[dashed] (3.5,4.25) -- (8.5,4.25);
\draw (8.5,4.25) -- (5,3);

\draw (0,4) -- (5,4) ;
\draw[dashed] (0,4) -- (3.5,5.25) ;
\draw[dashed] (3.5,5.25) -- (8.5,5.25);
\draw (8.5,5.25) -- (5,4);

\draw (0,5) -- (5,5) ;
\draw (0,5) -- (3.5,6.25) ;
\draw (3.5,6.25) -- (8.5,6.25);
\draw (8.5,6.25) -- (5,5);
\end{scope}

\begin{scope}[xshift=-20cm]
\draw (0,0) -- (5,0) ;
\draw[dashed] (0,0) -- (3.5,1.25) ;
\draw[dashed] (3.5,1.25) -- (8.5,1.25);
\draw (8.5,1.25) -- (5,0);

\draw (0,1) -- (5,1) ;
\draw[dashed] (0,1) -- (3.5,2.25) ;
\draw[dashed] (3.5,2.25) -- (8.5,2.25);
\draw (8.5,2.25) -- (5,1);

\draw (0,2) -- (5,2);
\draw[dashed] (0,2) -- (3.5,3.25) ;
\draw[dashed] (3.5,3.25) -- (8.5,3.25);
\draw (8.5,3.25) -- (5,2);

\draw (0,3) -- (5,3);
\draw[dashed] (0,3) -- (3.5,4.25) ;
\draw[dashed] (3.5,4.25) -- (8.5,4.25);
\draw (8.5,4.25) -- (5,3);

\draw (0,4) -- (5,4) ;
\draw[dashed] (0,4) -- (3.5,5.25) ;
\draw[dashed] (3.5,5.25) -- (8.5,5.25);
\draw (8.5,5.25) -- (5,4);

\draw (0,5) -- (5,5) ;
\draw (0,5) -- (3.5,6.25) ;
\draw (3.5,6.25) -- (8.5,6.25);
\draw (8.5,6.25) -- (5,5);
\draw[-latex] (17.5,-1) -- (20,-1);
\node at (15,-2) {${\mathcal{T}}^{n_1}$};
\end{scope}

\begin{scope}[xshift=-30cm]
\draw (0,0) -- (5,0) ;
\draw[dashed] (0,0) -- (3.5,1.25) ;
\draw[dashed] (3.5,1.25) -- (8.5,1.25);
\draw (8.5,1.25) -- (5,0);

\draw (0,1) -- (5,1) ;
\draw[dashed] (0,1) -- (3.5,2.25) ;
\draw[dashed] (3.5,2.25) -- (8.5,2.25);
\draw (8.5,2.25) -- (5,1);

\draw (0,2) -- (5,2);
\draw[dashed] (0,2) -- (3.5,3.25) ;
\draw[dashed] (3.5,3.25) -- (8.5,3.25);
\draw (8.5,3.25) -- (5,2);

\draw (0,3) -- (5,3);
\draw[dashed] (0,3) -- (3.5,4.25) ;
\draw[dashed] (3.5,4.25) -- (8.5,4.25);
\draw (8.5,4.25) -- (5,3);

\draw (0,4) -- (5,4) ;
\draw[dashed] (0,4) -- (3.5,5.25) ;
\draw[dashed] (3.5,5.25) -- (8.5,5.25);
\draw (8.5,5.25) -- (5,4);

\draw (0,5) -- (5,5) ;
\draw (0,5) -- (3.5,6.25) ;
\draw (3.5,6.25) -- (8.5,6.25);
\draw (8.5,6.25) -- (5,5);
\node at (-1.5,-1.5) {$\vec{u}$};
\draw[-latex] (5,-1) -- (10,-1);
\node at (7.5,-2) {$\mathcal{T}$};
\end{scope}

\end{tikzpicture}\]
\caption{\label{fig.minimality.proof.intro} Schema of the 
proof for the minimality property of $X_{(Z,f)}$.}
\end{figure}
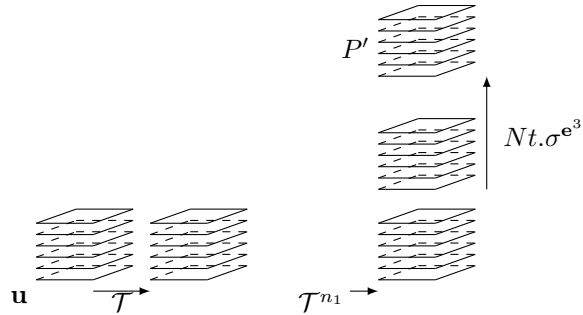

 \appendix

\section{\label{section.rigid.robinson} A rigid three-dimensional version of the  Robinson subshift}

The Robinson subshift was constructed by R. Robinson~\cite{R71} in order to prove 
undecidability results. It has been used by in other constructions 
of subshifts of finite type as a structure layer~\cite{Pavlov-Schraudner-2014}.

In this section, we present a 
version of this subshift which is 
adapted to constructions under the 
dynamical constraints 
that we consider. In order to understand this 
section, it is preferable to read before the description 
of the Robinson subshift done in~\cite{R71}. 
The results of this section are well known. 
We refer to~\cite{R71} and \cite{GS17ED}. 

Let us denote $X_{\texttt{R}}$ this subshift, which is constructed as the product of two layers.
We present the first layer in Section~\ref{sec.valued.robinson}, and then 
describe some hierarchical structures appearing in this layer in Section~\ref{sec.hierarchical.structures}.
In Section~\ref{sec.alignment.positioning}, we describe the second layer. This layer allows 
adding rigidity to 
the first layer, in order 
to enforce dynamical properties.

\subsection{\label{sec.valued.robinson} Robinson subshift}

The first layer has the following {\textit{symbols}}, and their transformation by 
rotations by $\frac{\pi}{2}$, $\pi$ or $\frac{3\pi}{2}$: 

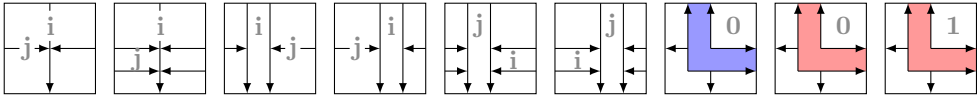
\[\begin{tikzpicture}
\draw (0,0) rectangle (1.2,1.2) ;
\draw (0.6,1.2) -- (0.6,1.1);
\draw [-latex] (0.6,0.7) -- (0.6,0) ;
\draw (0,0.6) -- (0.2,0.6);
\draw [-latex] (0.4,0.6) -- (0.6,0.6) ; 
\draw [-latex] (1.2,0.6) -- (0.6,0.6) ;
\node[scale=1,color=gray!90] at (0.6125,0.9) {\textbf{i}};
\node[scale=1,color=gray!90] at (0.3,0.6) {\textbf{j}};
\end{tikzpicture} \ \
\begin{tikzpicture}
\draw (0,0) rectangle (1.2,1.2) ;
\draw (0.6,1.2) -- (0.6,1.1);
\draw [-latex] (0.6,0.7) -- (0.6,0) ;
\draw [-latex] (0,0.6) -- (0.6,0.6) ; 
\draw [-latex] (0,0.3) -- (0.6,0.3) ; 
\draw [-latex] (1.2,0.6) -- (0.6,0.6) ;
\draw [-latex] (1.2,0.3) -- (0.6,0.3) ;
\node[scale=1,color=gray!90] at (0.6125,0.9) {\textbf{i}};
\node[scale=1,color=gray!90] at (0.3,0.45) {\textbf{j}};
\end{tikzpicture} \ \ \begin{tikzpicture}
\draw (0,0) rectangle (1.2,1.2) ;
\draw [-latex] (0.6,1.2) -- (0.6,0) ;
\draw [-latex] (0.3,1.2) -- (0.3,0) ; 
\draw [-latex] (0,0.6) -- (0.3,0.6) ; 
\draw [-latex] (0.8,0.6) -- (0.6,0.6) ;
\draw (1.2,0.6) -- (1,0.6);
\node[scale=1,color=gray!90] at (0.45,0.9) {\textbf{i}};
\node[scale=1,color=gray!90] at (0.9,0.6) {\textbf{j}};
\end{tikzpicture} \ \  \begin{tikzpicture}
\draw (0,0) rectangle (1.2,1.2) ;
\draw [-latex] (0.6,1.2) -- (0.6,0) ;
\draw [-latex] (0.9,1.2) -- (0.9,0) ; 
\draw (0,0.6) -- (0.2,0.6);
\draw [-latex] (0.4,0.6) -- (0.6,0.6) ; 
\draw [-latex] (1.2,0.6) -- (0.9,0.6) ;
\node[scale=1,color=gray!90] at (0.75,0.9) {\textbf{i}};
\node[scale=1,color=gray!90] at (0.3,0.6) {\textbf{j}};
\end{tikzpicture} \ \ 
\begin{tikzpicture}
\draw (0,0) rectangle (1.2,1.2) ;
\draw [-latex] (0.6,1.2) -- (0.6,0) ;
\draw [-latex] (0.3,1.2) -- (0.3,0) ; 
\draw [-latex] (0,0.6) -- (0.3,0.6) ; 
\draw [-latex] (1.2,0.6) -- (0.6,0.6) ;
\draw [-latex] (0,0.3) -- (0.3,0.3) ; 
\draw [-latex] (1.2,0.3) -- (0.6,0.3) ;
\node[scale=1,color=gray!90] at (0.9,0.45) {\textbf{i}};
\node[scale=1,color=gray!90] at (0.45,0.9) {\textbf{j}};
\end{tikzpicture} \ \
\begin{tikzpicture}
\draw (0,0) rectangle (1.2,1.2) ;
\draw [-latex] (0.6,1.2) -- (0.6,0) ;
\draw [-latex] (0.9,1.2) -- (0.9,0) ; 
\draw [-latex] (0,0.6) -- (0.6,0.6) ; 
\draw [-latex] (1.2,0.6) -- (0.9,0.6) ;
\draw [-latex] (0,0.3) -- (0.6,0.3) ; 
\draw [-latex] (1.2,0.3) -- (0.9,0.3) ;
\node[scale=1,color=gray!90] at (0.3,0.45) {\textbf{i}};
\node[scale=1,color=gray!90] at (0.75,0.9) {\textbf{j}};
\end{tikzpicture} \ \ \begin{tikzpicture}
\fill[blue!40] (0.3,0.3) rectangle (0.6,1.2) ;
\fill[blue!40] (0.3,0.3) rectangle (1.2,0.6) ;
\draw (0,0) rectangle (1.2,1.2) ;
\draw [-latex] (0.3,0.3) -- (0.3,1.2) ; 
\draw [-latex] (0.3,0.3) -- (1.2,0.3) ;
\draw [-latex] (0.6,0.6) -- (0.6,1.2) ; 
\draw [-latex] (0.6,0.6) -- (1.2,0.6) ; 
\draw [-latex] (0.3,0.6) -- (0,0.6) ; 
\draw [-latex] (0.6,0.3) -- (0.6,0) ;
\node[scale=1,color=gray!90] at (0.9,0.9) {\textbf{0}};
\end{tikzpicture} \ \ \begin{tikzpicture}
\fill[red!40] (0.3,0.3) rectangle (0.6,1.2) ;
\fill[red!40] (0.3,0.3) rectangle (1.2,0.6) ;
\draw (0,0) rectangle (1.2,1.2) ;
\draw [-latex] (0.3,0.3) -- (0.3,1.2) ; 
\draw [-latex] (0.3,0.3) -- (1.2,0.3) ;
\draw [-latex] (0.6,0.6) -- (0.6,1.2) ; 
\draw [-latex] (0.6,0.6) -- (1.2,0.6) ; 
\draw [-latex] (0.3,0.6) -- (0,0.6) ; 
\draw [-latex] (0.6,0.3) -- (0.6,0) ;
\node[scale=1,color=gray!90] at (0.9,0.9) {\textbf{0}};
\end{tikzpicture} \ \ \begin{tikzpicture}
\fill[red!40] (0.3,0.3) rectangle (0.6,1.2) ;
\fill[red!40] (0.3,0.3) rectangle (1.2,0.6) ;
\draw (0,0) rectangle (1.2,1.2) ;
\draw [-latex] (0.3,0.3) -- (0.3,1.2) ; 
\draw [-latex] (0.3,0.3) -- (1.2,0.3) ;
\draw [-latex] (0.6,0.6) -- (0.6,1.2) ; 
\draw [-latex] (0.6,0.6) -- (1.2,0.6) ; 
\draw [-latex] (0.3,0.6) -- (0,0.6) ; 
\draw [-latex] (0.6,0.3) -- (0.6,0) ;
\node[scale=1,color=gray!90] at (0.9,0.9) {\textbf{1}};
\end{tikzpicture}\]

The symbols $i$ and $j$ can have 
value $0,1$ and 
are attached respectively to vertical and 
horizontal arrows. In the text, we refer to this value as 
the value of the \textbf{$0,1$-counter}. In 
order to simplify the representations, these 
values will often be omitted on the figures. \bigskip

In the text we will often designate as \textbf{corners} the two last symbols.
The other ones are called \textbf{arrows symbols} and are 
specified by the number of arrows in the symbol. 
For instance a six arrows symbols are the images by rotation of the fifth and sixth symbols. \bigskip

The {\textit{rules}} 
are the following ones: \begin{enumerate} \item the outgoing arrows and incoming ones 
correspond for two adjacent symbols. For instance, 
the pattern 
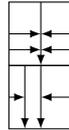
\[\begin{tikzpicture}[scale=0.7]
\draw (0,1.2) rectangle (1.2,2.4) ;
\draw [-latex] (0.6,2.4) -- (0.6,1.2) ;
\draw [-latex] (0,1.8) -- (0.6,1.8) ; 
\draw [-latex] (0,1.5) -- (0.6,1.5) ; 
\draw [-latex] (1.2,1.8) -- (0.6,1.8) ;
\draw [-latex] (1.2,1.5) -- (0.6,1.5) ;
\draw (0,0) rectangle (1.2,1.2) ;
\draw [-latex] (0.6,1.2) -- (0.6,0) ;
\draw [-latex] (0.3,1.2) -- (0.3,0) ; 
\draw [-latex] (0,0.6) -- (0.3,0.6) ; 
\draw [-latex] (1.2,0.6) -- (0.6,0.6) ;
\end{tikzpicture}\]
is forbidden, but the pattern 
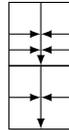
\[\begin{tikzpicture}[scale=0.7]
\draw (0,0) rectangle (1.2,1.2) ;
\draw [-latex] (0.6,1.2) -- (0.6,0) ;
\draw [-latex] (0,0.6) -- (0.6,0.6) ; 
\draw [-latex] (1.2,0.6) -- (0.6,0.6) ;
\draw (0,1.2) rectangle (1.2,2.4) ;
\draw [-latex] (0.6,2.4) -- (0.6,1.2) ;
\draw [-latex] (0,1.8) -- (0.6,1.8) ; 
\draw [-latex] (0,1.5) -- (0.6,1.5) ; 
\draw [-latex] (1.2,1.8) -- (0.6,1.8) ;
\draw [-latex] (1.2,1.5) -- (0.6,1.5) ;
\end{tikzpicture}\]
is allowed. 
\item in every $2 \times 2$ square there is a blue 
symbol and the presence of a blue symbol in position $\vec{u} \in \Z^2$ 
forces the presence of a blue symbol in the positions $\vec{u}+(0,2), \vec{u}-(0,2), 
\vec{u}+(2,0)$ and $\vec{u}-(2,0)$.
\item on a position having 
mark $(i,j)$, the first coordinate is transmitted 
to the horizontally adjacent positions and the second one is transmitted 
to the vertically adjacent positions.
\item on a six arrows symbol, like 
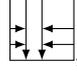
\[\begin{tikzpicture}[scale=0.7]
\draw (0,0) rectangle (1.2,1.2) ;
\draw [-latex] (0.6,1.2) -- (0.6,0) ;
\draw [-latex] (0.3,1.2) -- (0.3,0) ; 
\draw [-latex] (0,0.6) -- (0.3,0.6) ; 
\draw [-latex] (1.2,0.6) -- (0.6,0.6) ;
\draw [-latex] (0,0.3) -- (0.3,0.3) ; 
\draw [-latex] (1.2,0.3) -- (0.6,0.3) ;
\end{tikzpicture},\] 
or a five arrow symbol, 
like 
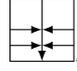
\[\begin{tikzpicture}[scale=0.7]
\draw (0,0) rectangle (1.2,1.2) ;
\draw [-latex] (0.6,1.2) -- (0.6,0) ;
\draw [-latex] (0,0.6) -- (0.6,0.6) ; 
\draw [-latex] (0,0.3) -- (0.6,0.3) ; 
\draw [-latex] (1.2,0.6) -- (0.6,0.6) ;
\draw [-latex] (1.2,0.3) -- (0.6,0.3) ;
\end{tikzpicture},\]
the marks $i$ and $j$ are different.
\end{enumerate} \bigskip

The Figure~\ref{figure.order2supertile} 
shows some pattern in the language of 
this layer. The subshift on this alphabet 
and generated by these rules is denoted $X_R$: this 
is the Robinson subshift.\bigskip

In the following, we state some properties of 
this subshift. The proofs of these properties 
can also be found in \cite{R71}.

\subsection{\label{sec.hierarchical.structures} 
Hierarchical structures}

In this section we describe some observable hierarchical 
structures in the elements of the Robinson subshift.

Let us recall that for 
all $d \ge 1$ and $k \ge 1$, 
we denote $\mathbb{U}^{(d)}_{k}$
the set $\llbracket 0,k-1 \rrbracket^d$.

\subsubsection{Finite supertiles}

\begin{figure}[h!]
\[\begin{tikzpicture}[scale=0.4]

\begin{scope}
\robibluebastgauche{0}{0};
\robibluebastgauche{8}{0};
\robibluebastgauche{0}{8};
\robibluebastgauche{8}{8};
\robiredbasgauche{2}{2};
\robiredbasgauche{6}{6};
\robibluehauttgauche{0}{4};
\robibluehauttgauche{8}{4};
\robibluehauttgauche{0}{12};
\robibluehauttgauche{8}{12};
\robiredhautgauche{2}{10};
\robibluebastdroite{4}{0};
\robibluebastdroite{12}{0};
\robibluebastdroite{4}{8};
\robibluebastdroite{12}{8};
\robiredbasdroite{10}{2};
\robibluehauttdroite{4}{4};
\robibluehauttdroite{12}{4};
\robibluehauttdroite{4}{12};
\robibluehauttdroite{12}{12};
\robiredhautdroite{10}{10};
\robitwobas{2}{0}
\robitwobas{10}{0}
\robitwobas{6}{2}
\robitwogauche{0}{2}
\robitwogauche{2}{6}
\robitwogauche{0}{10}
\robitwodroite{12}{2}
\robitwodroite{12}{10}
\robitwohaut{2}{12}
\robitwohaut{10}{12}
\robionehaut{6}{0}
\robionehaut{6}{4}
\robionegauche{0}{6}
\robionegauche{4}{6}
\robisixbas{2}{8}
\robisixdroite{4}{2}
\robisixhaut{2}{4}
\robisevendroite{4}{10}
\robithreehaut{6}{8}
\robisixhaut{6}{10}
\robithreehaut{6}{12}
\robisixgauche{8}{2}
\robisevengauche{8}{10}
\robisevenbas{10}{8}
\robisevenhaut{10}{4}
\robisixdroite{10}{6}
\robithreedroite{8}{6}
\robithreedroite{12}{6}
\draw (0,0) rectangle (14,14);
\foreach \x in {1,...,6} \draw (0,2*\x) -- (14,2*\x);
\foreach \x in {1,...,6} \draw (2*\x,0) -- (2*\x,14);
\end{scope}

\begin{scope}[xshift=17.5cm]
\fill[blue!40] (8.5,0.5) rectangle (13.5,5.5);
\fill[blue!40] (8.5,8.5) rectangle (13.5,13.5);
\fill[blue!40] (0.5,8.5) rectangle (5.5,13.5);
\fill[white] (9,1) rectangle (13,5);
\fill[white] (9,9) rectangle (13,13);
\fill[white] (1,9) rectangle (5,13);

\fill[blue!40] (0.5,0.5) rectangle (5.5,5.5);
\fill[white] (1,1) rectangle (5,5);

\fill[red!40] (2.5,2.5) rectangle (3,11.5);
\fill[red!40] (2.5,2.5) rectangle (11.5,3);
\fill[red!40] (2.5,11) rectangle (11.5,11.5);
\fill[red!40] (11,2.5) rectangle (11.5,11.5);

\fill[orange!40] (6.5,6.5) rectangle (7,14); 
\fill[orange!40] (6.5,6.5) rectangle (14,7);

\draw[step=2] (0,0) grid (14,14);
\end{scope}
\end{tikzpicture}\]
\caption{The south west order $2$ supertile denoted $St_{sw}(2)$ 
and petals intersecting it.}\label{figure.order2supertile}
\end{figure}
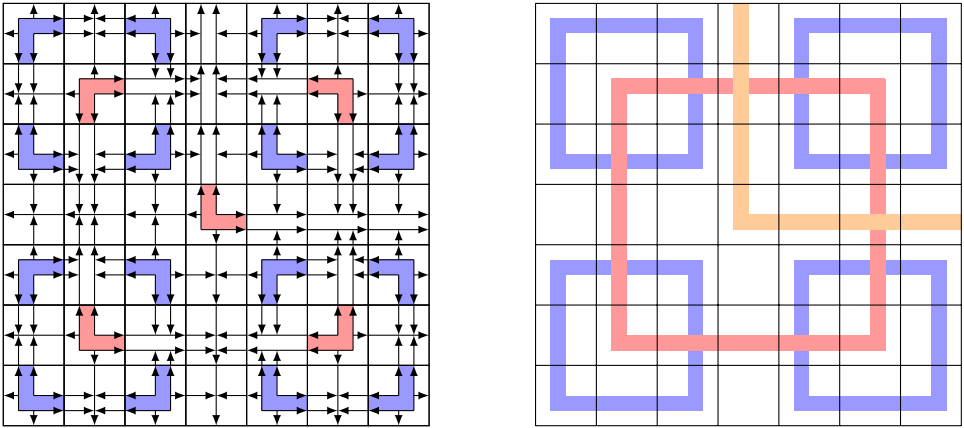 \bigskip

Let us define by induction the south west 
(resp. south east, north west, north east) 
\define{supertile of order $n\in\N$}. For $n=0$, one has
\[
St_{sw} (0)=\vbox to 14pt{\hbox{
\begin{tikzpicture}[scale=0.4]
 \robibluebastgauche{0}{0};
\end{tikzpicture}}},\quad
St_{se} (0)=\vbox to 14pt{\hbox{
\begin{tikzpicture}[scale=0.4]
 \robibluebastdroite{0}{0};
\end{tikzpicture}}},\quad
St_{nw} (0)=\vbox to 14pt{\hbox{
\begin{tikzpicture}[scale=0.4]
 \robibluehauttgauche{0}{0};
\end{tikzpicture}}},\quad
St_{ne} (0)=\vbox to 14pt{\hbox{
\begin{tikzpicture}[scale=0.4]
 \robibluehauttdroite{0}{0};
\end{tikzpicture}}}.
\] 

For $n\in\N$, the support of the supertile $St_{sw} (n+1)$ 
(resp. $St_{se} (n+1)$, $St_{nw} (n+1)$, $St_{ne} (n+1)$) is 
$\U^{(2)}_{2^{n+2}-1}$. 
On position $\vec{u}=(2^{n+1}-1,2^{n+1}-1)$ write
\[
St_{sw} (n+1) _{\vec{u}}=\vbox to 14pt{\hbox{
\begin{tikzpicture}[scale=0.4]
 \robiredbasgauche{0}{0};
\end{tikzpicture}}},\quad
St_{se} (n+1) _{\vec{u}}=\vbox to 14pt{\hbox{
\begin{tikzpicture}[scale=0.4]
 \robiredbasdroite{0}{0};
\end{tikzpicture}}},\quad
St_{nw} (n+1) _{\vec{u}}=\vbox to 14pt{\hbox{
\begin{tikzpicture}[scale=0.4]
 \robiredhautgauche{0}{0};
\end{tikzpicture}}},\quad
St_{ne} (n+1) _{\vec{u}}=\vbox to 14pt{\hbox{
\begin{tikzpicture}[scale=0.4]
 \robiredhautdroite{0}{0};
\end{tikzpicture}}}.
\]

Then complete the supertile such that the restriction 
to $\mathbb{U}^{(2)}_{2^{n+1}-1}$ (resp. 
$(2^{n+1},0)+\mathbb{U}^{(2)}_{2^{n+1}-1}$, 
 $(0,2^{n+1})+\mathbb{U}^{(2)}_{2^{n+1}-1}$,
$(2^{n+1},2^{n+1})+\mathbb{U}^{(2)}_{2^{n+1}-1}$)
is 
$St_{sw}(n)$ (resp. $St_{se} (n)$, $St_{nw} (n)$,
$St_{ne} (n)$).

Then complete the cross with 
the symbol 
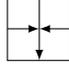
\[\begin{tikzpicture}[scale=0.7]
\draw (0,0) rectangle (1.2,1.2) ;
\draw [-latex] (0.6,1.2) -- (0.6,0) ;
\draw [-latex] (0,0.6) -- (0.6,0.6) ; 
\draw [-latex] (1.2,0.6) -- (0.6,0.6) ;
\end{tikzpicture}\]
or with the symbol
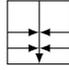
\[\begin{tikzpicture}[scale=0.7]
\draw (0,0) rectangle (1.2,1.2) ;
\draw [-latex] (0.6,1.2) -- (0.6,0) ;
\draw [-latex] (0,0.6) -- (0.6,0.6) ; 
\draw [-latex] (0,0.3) -- (0.6,0.3) ; 
\draw [-latex] (1.2,0.6) -- (0.6,0.6) ;
\draw [-latex] (1.2,0.3) -- (0.6,0.3) ;
\end{tikzpicture}\] in the south vertical arm with the first symbol when there is one incoming arrow, and the second when there are two. 
The other arms are completed in a similar way. For instance, Figure \ref{figure.order2supertile}
shows the south west supertile of order two.

\begin{proposition}[\cite{R71}]
For all configuration $x$, if an order $n$ supertile 
appears in this configuration, 
then there is an order $n+1$ supertile, having 
this order $n$ supertile as sub-pattern, which appears 
in the same configuration.
\end{proposition}

\subsubsection{Infinite supertiles}

\begin{figure}[ht]
\[\begin{tikzpicture}[scale=0.3]
\fill[gray!20] (0,0) rectangle (20,20); 
\fill[white] (1,1) rectangle (19,19);
\fill[gray!20] (9.5,9.5) rectangle (10.5,22.5);
\fill[gray!20] (9.5,9.5) rectangle (22.5,10.5);
\fill[red!40] (9.5,9.5) rectangle (10.5,11.5);
\fill[red!40] (9.5,9.5) rectangle (11.5,10.5);
\fill[red!40] (0,0) rectangle (1,2);
\fill[red!40] (0,0) rectangle (2,1);
\draw (0,2) -- (1,2);
\draw (2,0) -- (2,1);

\fill[red!40] (0,20) rectangle (1,18);
\fill[red!40] (0,20) rectangle (2,19);
\draw (0,18) -- (1,18);
\draw (2,19) -- (2,20);

\fill[red!40] (20,20) rectangle (18,19);
\fill[red!40] (20,20) rectangle (19,18);
\draw (19,18) -- (20,18);
\draw (18,19) -- (18,20);

\fill[red!40] (20,0) rectangle (18,1);
\fill[red!40] (20,0) rectangle (19,2);
\draw (19,2) -- (20,2);
\draw (18,1) -- (18,0);

\draw[-latex] (9.5,10.5) -- (5,10.5);
\draw[-latex] (5,10.5) -- (-2.5,10.5);
\draw[-latex] (10.5,9.5) -- (10.5,5);
\draw[-latex] (10.5,5) -- (10.5,-2.5);
\draw[-latex] (10.5,10.5) -- (10.5,15);
\draw[-latex] (9.5,10.5) -- (9.5,15);
\draw[-latex] (10.5,10.5) -- (15,10.5);
\draw[-latex] (10.5,9.5) -- (15,9.5);
\draw[-latex] (9.5,15) -- (9.5,19);
\draw[-latex] (10.5,15) -- (10.5,19);
\draw[-latex] (9.5,20) -- (9.5,22.5);
\draw[-latex] (10.5,20) -- (10.5,22.5);
\draw[-latex] (8.5,19) -- (9.5,19);
\draw[-latex] (8.5,20) -- (9.5,20);
\draw[-latex] (11.5,19) -- (10.5,19);
\draw[-latex] (11.5,20) -- (10.5,20);
\draw[-latex] (10.5,5) -- (10.5,1);
\draw[-latex] (9.5,1) -- (10.5,1);
\draw[-latex] (9.5,0) -- (10.5,0);
\draw[-latex] (11.5,1) -- (10.5,1);
\draw[-latex] (11.5,0) -- (10.5,0);
\draw[-latex] (0,10.5) -- (0,15);
\draw[-latex] (1,10.5) -- (1,15);
\draw (9.5,11.5) -- (10.5,11.5);
\draw (11.5,9.5) -- (11.5,10.5);
\draw (0,0) rectangle (20,20);
\draw (1,1) rectangle (19,19);
\draw (9.5,9.5) -- (22.5,9.5);
\draw (9.5,9.5) -- (9.5,22.5);
\draw (10.5,10.5) -- (10.5,22.5);
\draw (10.5,10.5) -- (22.5,10.5);
\draw[dashed] (8,8) rectangle (13,13);
\draw[dashed] (8,17.5) rectangle (13,22.5);
\draw[dashed] (8,-2.5) rectangle (13,2.5);
\draw[dashed] (-1.5,12.5) rectangle (3.5,17.5);
\draw[dashed] (8,2.75) rectangle (13,7.75);
\node at (15,5.25) {$(ii),(1)$};
\node at (14,23.5) {$(iii),(2)$};
\node at (6,8) {$(iii),(1)$};
\node at (14,-3.5) {$(iii),(3)$};
\node at (-2.5,11.5) {$(ii),(2)$};
\end{tikzpicture}\]
\caption{\label{fig.infinite.supertiles} Correspondence 
between infinite supertiles and sub-patterns of order $n$ supertiles. The whole picture 
represents a schema of some finite order supertile.}
\end{figure}
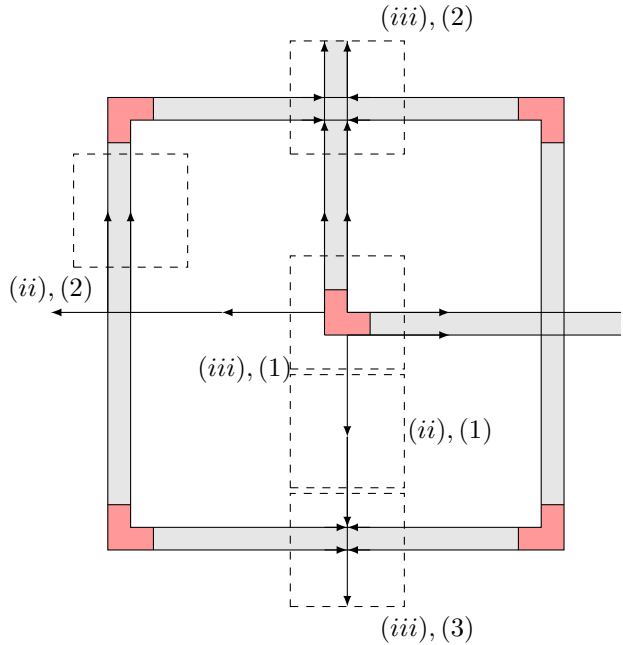

Let $x$ be a configuration in the first layer and consider 
the equivalence relation $\sim_x$ on $\Z^2$ defined by 
$\vec{i} \sim_x \vec{j}$ if there is a finite
supertile in $x$ which contains $\vec{i}$ and $\vec{j}$.
An \define{infinite order} supertile is an infinite 
pattern over an equivalence 
class of this relation. Each configuration is amongst the 
following types (with types corresponding with 
types numbers on Figure~\ref{fig.infinite.supertiles}): 
\begin{itemize}
\item[(i)] A unique infinite order supertile which covers $\Z^2$.
\item[(ii)] Two infinite order supertiles separated by a line or a column with only three-arrows symbols
(1) or only four arrows symbols (2).  
In such a configuration, the order $n$
finite supertiles appearing in the two 
infinite supertiles 
are not necessary aligned, whereas this 
is the case in a type (i) or (iii) configuration.
\item[(iii)] Four infinite order supertiles, separated by a cross, whose center is superimposed with:
\begin{itemize}
\item a red symbol, and arms are filled with arrows symbols induced by the red one. (1)
\item a six arrows symbol, and arms are filled with double arrow symbols 
induced by this one. (2)
\item a five arrow symbol, and arms are filled with double arrow symbols and simple 
arrow symbols induced 
by this one. (3)
\end{itemize}
\end{itemize} 

Informally, the types of infinite 
supertiles correspond to configurations 
that are limits (for type $(ii)$ infinite supertiles 
this will be true after alignment 
[Section~\ref{sec.alignment.positioning}]) of 
a sequence of configurations centered on 
particular sub-patterns of 
finite supertiles of order $n$. 
This correspondence is illustrated on 
Figure~\ref{fig.infinite.supertiles}.
We notice this fact so that it helps 
to understand how 
patterns in configurations having multiple 
infinite supertiles are sub-patterns of 
finite supertiles. \bigskip

We say that a 
pattern $p$ on support $\mathbb{U}$ appears 
periodically in the horizontal (resp. vertical) direction 
in a configuration $x$ of a subshift $X$
when there exists some $T>0$ and 
$\vec{u}_0 \in \Z^2$ 
such that 
for all $k \in \Z$, 
\[x_{\vec{u}_0+\mathbb{U}+kT(1,0)}=p\]
(resp. $x_{\vec{u}_0+\mathbb{U}+kT(0,1)}=p$).
The number $T$ is called the period of this periodic appearance.

\begin{lemma}[\cite{R71}] 
\label{lem.repetition.supertiles}
For all $n$ and $m$ integers 
such that $n \ge m$, any order $m$ supertile 
appears periodically, horizontally and vertically, in 
any supertile of order $n \ge m$ with period $2^{m+2}$.
This is also true inside any infinite supertile. 
\end{lemma}

\subsubsection{Petals}

For a configuration 
$x$ of the Robinson subshift some finite subset of $\Z^2$ 
which has the following properties is called 
a \textbf{petal}.

\begin{itemize}
\item this set is minimal with respect to the inclusion, 
\item it contains some 
symbol with more than three arrows, 
\item  if a position is in the petal, 
the next position in the direction, or the opposite one, 
of the double arrows, is also in it,
\item and in the case of a six arrows 
symbol, the previous 
property is true only for one pair of arrows.
\end{itemize}

These sets are represented on 
the figures as squares joining 
four corners when these corners 
have the right orientations.
\bigskip

Petals containing blue symbols are called 
order $0$ petals. Each one intersect 
a unique greater order petal.
The other ones intersect 
four smaller petals and a greater 
one: if the intermediate petal is of order $n \ge 1$, 
then the four smaller are of order $n-1$ and the greatest 
one is of order $n+1$. Hence they 
form a hierarchy, and 
we refer to this in the text as 
the \textbf{petal hierarchy} (or hierarchy).

We usually call 
the petals valued with $1$ 
\textbf{support petals}, 
and the other ones 
are called 
\textbf{transmission petals}.

\begin{lemma}[\cite{R71}]
For all $n$, an order $n$ petal has size $2^{n+1}+1$.
\end{lemma}

We call order $n$ \textbf{two dimensional cell} the part 
of $\Z^2$ which is enclosed in an order 
$2n+1$ petal, for $n \ge 0$.
We also sometimes refer 
to the order $2n+1$ petals 
as the cells borders.

In particular, order $n \ge 0$ 
two-dimensional cells 
have size $4^{n+1}+1$ and 
repeat periodically with 
period $4^{n+2}$, vertically 
and horizontally, 
in every cell or supertile 
having greater order. 

See an illustration on 
Figure~\ref{figure.order2supertile}.

\subsection{ \label{sec.alignment.positioning} 
Alignment positioning
} 

If a configuration of the first layer has two infinite order 
supertiles, then the two sides of the column or line which separates them are 
non dependent. The two infinite order supertiles of this configuration 
can be shifted vertically (resp. horizontally) one from each 
other, while the configuration obtained stays 
an element of the subshift.
This is an obstacle to dynamical properties such as 
minimality or transitivity, 
since a pattern which crosses the separating line can not 
appear in the other configurations.
In this section, we describe additional layers 
that allow aligning all the supertiles 
having the same order and 
eliminate this phenomenon. \bigskip

Here is a description of the second layer: \bigskip

\textit{Symbols:} $nw, ne, sw, se$, 
and a blank symbol. \bigskip

The {\textit{rules}} are 
the following ones: 

\begin{itemize}
\item \textbf{Localization:} 
the symbols $nw$, $ne$ ,$sw$ and $se$
are superimposed only on three arrows and five arrows 
symbols in the Robinson layer.
\item \textbf{Induction 
of the orientation:} on a position with 
a three arrows symbol such 
that the long arrow 
originate in a corner
is superimposed a symbol 
corresponding to the orientation 
of the corner.
\item \textbf{Transmission rule:} 
on a three or five arrows symbol position, 
the symbol in this layer is 
transmitted to the position in the direction pointed 
by the long arrow when the Robinson symbol 
 is a three or five arrows symbol 
with long arrow pointing in the same direction. 
\item \textbf{Synchronization rule:} 
On the pattern 
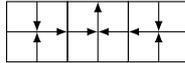
\[\begin{tikzpicture}[scale=0.4]
\robionedroite{0}{0}
\robionebas{2}{0}
\robionegauche{4}{0}
\end{tikzpicture}\]
or 
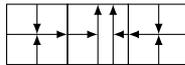
\[\begin{tikzpicture}[scale=0.4]
\robionedroite{0}{0}
\robifourhaut{2}{0}
\robionegauche{4}{0}
\end{tikzpicture}\]
in the Robinson layer, if 
the symbol on the left side is $ne$
(resp. $se$), 
then the symbol on the right side is $nw$
(resp. $sw$).
On the images by rotation of these patterns, 
we impose similar rules.
\item \textbf{Coherence rule:} 
the other pairs of symbols are forbidden 
on these patterns.
\end{itemize}
 \bigskip

\textit{Global behavior:} the 
symbols 
$ne,nw,sw,se$ designate 
orientations: 
north east, north west, south 
west and south east. We will 
re-use this symbolisation 
in the following. The 
localization rule implies 
that these symbols 
are superimposed on and only on 
straight paths connecting 
the corners of adjacent order $n$ 
cells for some integer $n$.

The effect of transmission and synchronization rules 
is stated by the following lemma: 

\begin{lemma}[\cite{GS17ED}]
In any configuration $x$ of the subshift $X_{\texttt{R}}$, 
any order $n$ supertile appears 
periodically in the whole configuration, 
with period $2^{n+2}$, horizontally and vertically.
\end{lemma}

\subsection{Completing blocks}

Let $\chi : \N^{*} \rightarrow \N^{*}$ such 
that for all $n \ge 1$, 
\[\chi(n) = \Bigl\lceil \log_2 (n) \Bigr\rceil + 4.\]
Let us also denote $\chi'$ the function 
such that for all $n \ge 1$,
\[\chi'(n) =  \Bigl\lceil\frac{\Bigl\lceil \log_2 (n) 
\Bigr\rceil}{2} \Bigr\rceil + 2.\]

The following lemma will be extensively used 
in the following of this text, in order to prove 
dynamical properties of the constructed subshifts: 

\begin{lemma}[\cite{GS17ED}] \label{prop.complete.rob} 
For all $n \ge 1$, any $n$-block in the language of 
$X_{\texttt{R}}$ is sub-pattern of some order $\chi(n)$ 
supertile, and is sub-pattern of some order $\chi'(n)$ 
order cell.
\end{lemma}

\section{Some hierarchical signaling processes}

In this section, we present a layer whose purpose is to 
attribute to the blue corners in the Robinson layer a function 
relative to the counters and machines.

\subsection{\label{section.hierarchy.orientation} Orientation in the hierarchy}

The purpose of the first sublayer is to give access, to 
the support petal of each colored face,
to the orientation of this petal relatively to the 
support petal just above in the hierarchy. \bigskip

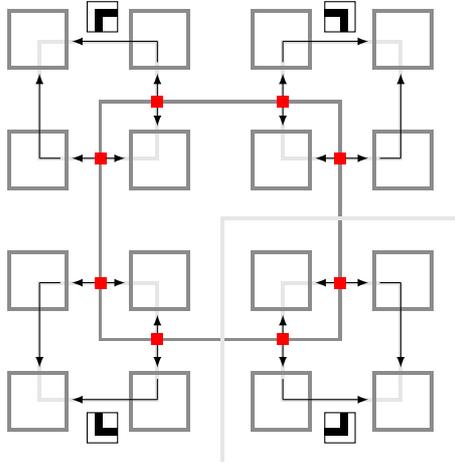
\begin{figure}[ht]
\[\begin{tikzpicture}[scale=0.10]
\fill[gray!90] (16,16) rectangle (48,48); 
\fill[white] (16.5,16.5) rectangle (47.5,47.5);
\fill[gray!20] (32,0) rectangle (32.5,32.5);
\fill[gray!20] (32,32.5) rectangle (64,32);

\fill[gray!20] (8,8) rectangle (8.5,24); 
\fill[gray!20] (8,8) rectangle (24,8.5);
\fill[gray!20] (8,23.5) rectangle (24,24);  
\fill[gray!20] (23.5,8) rectangle (24,24);

\fill[gray!20] (40,8) rectangle (40.5,24); 
\fill[gray!20] (40,8) rectangle (56,8.5);
\fill[gray!20] (40,23.5) rectangle (56,24);  
\fill[gray!20] (55.5,8) rectangle (56,24);

\fill[gray!20] (8,40) rectangle (8.5,56); 
\fill[gray!20] (8,40) rectangle (24,40.5);
\fill[gray!20] (8,55.5) rectangle (24,56);  
\fill[gray!20] (23.5,40) rectangle (24,56);

\fill[gray!20] (40,40) rectangle (40.5,56); 
\fill[gray!20] (40,40) rectangle (56,40.5);
\fill[gray!20] (40,55.5) rectangle (56,56);  
\fill[gray!20] (55.5,40) rectangle (56,56);

\fill[gray!90] (4,4) rectangle (4.5,12); 
\fill[gray!90] (4,4) rectangle (12,4.5); 
\fill[gray!90] (4.5,11.5) rectangle (12,12); 
\fill[gray!90] (11.5,4.5) rectangle (12,12); 

\fill[gray!90] (4,20) rectangle (4.5,28); 
\fill[gray!90] (4,20) rectangle (12,20.5); 
\fill[gray!90] (4.5,27.5) rectangle (12,28); 
\fill[gray!90] (11.5,20.5) rectangle (12,28); 

\fill[gray!90] (4,36) rectangle (4.5,44); 
\fill[gray!90] (4,36) rectangle (12,36.5); 
\fill[gray!90] (4.5,43.5) rectangle (12,44); 
\fill[gray!90] (11.5,36.5) rectangle (12,44);

\fill[gray!90] (4,52) rectangle (4.5,60); 
\fill[gray!90] (4,52) rectangle (12,52.5); 
\fill[gray!90] (4.5,59.5) rectangle (12,60); 
\fill[gray!90] (11.5,52.5) rectangle (12,60);

\fill[gray!90] (20,4) rectangle (20.5,12); 
\fill[gray!90] (20,4) rectangle (28,4.5); 
\fill[gray!90] (20.5,11.5) rectangle (28,12); 
\fill[gray!90] (27.5,4.5) rectangle (28,12); 

\fill[gray!90] (20,20) rectangle (20.5,28); 
\fill[gray!90] (20,20) rectangle (28,20.5); 
\fill[gray!90] (20.5,27.5) rectangle (28,28); 
\fill[gray!90] (27.5,20.5) rectangle (28,28); 

\fill[gray!90] (20,36) rectangle (20.5,44); 
\fill[gray!90] (20,36) rectangle (28,36.5); 
\fill[gray!90] (20.5,43.5) rectangle (28,44); 
\fill[gray!90] (27.5,36.5) rectangle (28,44);

\fill[gray!90] (20,52) rectangle (20.5,60); 
\fill[gray!90] (20,52) rectangle (28,52.5); 
\fill[gray!90] (20.5,59.5) rectangle (28,60); 
\fill[gray!90] (27.5,52.5) rectangle (28,60);

\fill[gray!90] (36,4) rectangle (36.5,12); 
\fill[gray!90] (36,4) rectangle (44,4.5); 
\fill[gray!90] (36.5,11.5) rectangle (44,12); 
\fill[gray!90] (43.5,4.5) rectangle (44,12); 

\fill[gray!90] (36,20) rectangle (36.5,28); 
\fill[gray!90] (36,20) rectangle (44,20.5); 
\fill[gray!90] (36.5,27.5) rectangle (44,28); 
\fill[gray!90] (43.5,20.5) rectangle (44,28); 

\fill[gray!90] (36,36) rectangle (36.5,44); 
\fill[gray!90] (36,36) rectangle (44,36.5); 
\fill[gray!90] (36.5,43.5) rectangle (44,44); 
\fill[gray!90] (43.5,36.5) rectangle (44,44);

\fill[gray!90] (36,52) rectangle (36.5,60); 
\fill[gray!90] (36,52) rectangle (44,52.5); 
\fill[gray!90] (36.5,59.5) rectangle (44,60); 
\fill[gray!90] (43.5,52.5) rectangle (44,60);

\fill[gray!90] (52,4) rectangle (52.5,12); 
\fill[gray!90] (52,4) rectangle (60,4.5); 
\fill[gray!90] (52.5,11.5) rectangle (60,12); 
\fill[gray!90] (59.5,4.5) rectangle (60,12); 

\fill[gray!90] (52,20) rectangle (52.5,28); 
\fill[gray!90] (52,20) rectangle (60,20.5); 
\fill[gray!90] (52.5,27.5) rectangle (60,28); 
\fill[gray!90] (59.5,20.5) rectangle (60,28); 

\fill[gray!90] (52,36) rectangle (52.5,44); 
\fill[gray!90] (52,36) rectangle (60,36.5); 
\fill[gray!90] (52.5,43.5) rectangle (60,44); 
\fill[gray!90] (59.5,36.5) rectangle (60,44);

\fill[gray!90] (52,52) rectangle (52.5,60); 
\fill[gray!90] (52,52) rectangle (60,52.5); 
\fill[gray!90] (52.5,59.5) rectangle (60,60); 
\fill[gray!90] (59.5,52.5) rectangle (60,60);


\draw[-latex] (15.5,23.75) -- (12.5,23.75);
\draw[-latex] (17,23.75) -- (19.5,23.75);

\draw (11,23.75) -- (8.25,23.75); 
\draw[-latex] (8.25,23.75) -- (8.25,12.5);
\draw (23.75,11) -- (23.75,8.25); 
\draw[-latex] (23.75,8.25) -- (12.5,8.25);

\draw[-latex] (23.75,15.5) -- (23.75,12.5);

\draw[-latex] (23.75,17) -- (23.75,19.5);

\node at (16.5,4.5) {
\begin{tikzpicture}[scale=0.2] 
\fill[black] (0.5,2) rectangle (1,0.5);
\fill[black] (1,0.5) rectangle (2,1);
\draw (0,0) rectangle (2,2);
\end{tikzpicture}};

\fill[red] (15.5,39.5) rectangle (17,41);
\fill[red] (23,47) rectangle (24.5,48.5);

\fill[red] (15.5+31.5,39.5) rectangle (17+31.5,41);
\fill[red] (23+16.5,47) rectangle (24.5+16.5,48.5);

\fill[red] (15.5,39.5-16.5) rectangle (17,41-16.5);
\fill[red] (15.5+31.5,39.5-16.5) rectangle (17+31.5,41-16.5);

\fill[red] (23,47-31.5) rectangle (24.5,48.5-31.5);
\fill[red] (23+16.5,47-31.5) rectangle (24.5+16.5,48.5-31.5);


\draw[-latex] (23.75,48.5) -- (23.75,51.5);
\draw[-latex] (23.75,47) -- (23.75,44.5);
\draw (23.75,53) -- (23.75,55.75);
\draw[-latex] (23.75,55.75) -- (12.5,55.75);
\draw[-latex] (15.5,40.25) -- (12.5,40.25);
\draw[-latex] (17,40.25) -- (19.5,40.25);

\draw (11,40.25) -- (8.25,40.25);
\draw[-latex] (8.25,40.25) -- (8.25,51.5);
\node at (16.5,59) {
\begin{tikzpicture}[scale=0.2] 
\fill[black] (0.5,0) rectangle (1,1.5);
\fill[black] (1,1.5) rectangle (2,1);
\draw (0,0) rectangle (2,2);
\end{tikzpicture}};


\draw[-latex] (48.5,23.75) -- (51.5,23.75);
\draw[-latex] (47,23.75) -- (44.5,23.75);
\draw (53,23.75) -- (55.75,23.75);
\draw[-latex] (55.75,23.75) -- (55.75,12.5);
\draw[-latex] (40.25,15.5) -- (40.25,12.5);
\draw[-latex] (40.25,17) -- (40.25,19.5);

\draw (40.25,11) -- (40.25,8.25);
\draw[-latex] (40.25,8.25) -- (51.5,8.25);
\node at (47.75,4.5) {
\begin{tikzpicture}[scale=0.2] 
\fill[black] (1.5,2) rectangle (1,0.5);
\fill[black] (1,0.5) rectangle (0,1);
\draw (0,0) rectangle (2,2);
\end{tikzpicture}};


\draw[-latex] (48.5,40.25) -- (51.5,40.25);
\draw[-latex] (47,40.25) -- (44.5,40.25);
\draw[-latex] (40.25,48.5) -- (40.25,51.5);
\draw[-latex] (40.25,47) -- (40.25,44.5);
\draw (40.25,53) -- (40.25,55.75);
\draw[-latex] (40.25,55.75) -- (51.5,55.75);
\draw (53,40.25) -- (55.75,40.25);
\draw[-latex] (55.75,40.25) -- (55.75,51.5);
\node at (47.75,59) {
\begin{tikzpicture}[scale=0.2] 
\fill[black] (1.5,0) rectangle (1,1.5);
\fill[black] (1,1.5) rectangle (0,1);
\draw (0,0) rectangle (2,2);
\end{tikzpicture}};
\end{tikzpicture}\]
\caption{\label{fig.orientation} Schematic illustration of the
orientation rules, showing a support petal and the 
support petals just under this one in 
the hierarchy, 
all of them colored dark 
gray. The transmission 
petals connecting 
them are colored light gray.
Transformation positions are colored with a red square. The arrows give the natural interpretation 
of the propagation direction 
of the signal transmitting 
the orientation information.}
\end{figure}

\noindent \textbf{\textit{Symbols:}} \bigskip

The symbols are 
elements of \[\left\{\begin{tikzpicture}[scale=0.2] 
\draw (0,0) rectangle (2,2);
\end{tikzpicture}, \begin{tikzpicture}[scale=0.2] 
\fill[black] (0.5,0) rectangle (1,1.5);
\fill[black] (1,1.5) rectangle (2,1);
\draw (0,0) rectangle (2,2);
\end{tikzpicture}, \begin{tikzpicture}[scale=0.2] 
\fill[black] (1.5,0) rectangle (1,1.5);
\fill[black] (1,1.5) rectangle (0,1);
\draw (0,0) rectangle (2,2);
\end{tikzpicture}, \begin{tikzpicture}[scale=0.2] 
\fill[black] (1.5,2) rectangle (1,0.5);
\fill[black] (1,0.5) rectangle (0,1);
\draw (0,0) rectangle (2,2);
\end{tikzpicture}, \begin{tikzpicture}[scale=0.2] 
\fill[black] (0.5,2) rectangle (1,0.5);
\fill[black] (1,0.5) rectangle (2,1);
\draw (0,0) rectangle (2,2);
\end{tikzpicture} \right\},\] \[\left\{\begin{tikzpicture}[scale=0.2] 
\draw (0,0) rectangle (2,2);
\end{tikzpicture}, \begin{tikzpicture}[scale=0.2] 
\fill[black] (0.5,0) rectangle (1,1.5);
\fill[black] (1,1.5) rectangle (2,1);
\draw (0,0) rectangle (2,2);
\end{tikzpicture}, \begin{tikzpicture}[scale=0.2] 
\fill[black] (1.5,0) rectangle (1,1.5);
\fill[black] (1,1.5) rectangle (0,1);
\draw (0,0) rectangle (2,2);
\end{tikzpicture}, \begin{tikzpicture}[scale=0.2] 
\fill[black] (1.5,2) rectangle (1,0.5);
\fill[black] (1,0.5) rectangle (0,1);
\draw (0,0) rectangle (2,2);
\end{tikzpicture}, \begin{tikzpicture}[scale=0.2] 
\fill[black] (0.5,2) rectangle (1,0.5);
\fill[black] (1,0.5) rectangle (2,1);
\draw (0,0) rectangle (2,2);
\end{tikzpicture} \right\} \times \left\{\begin{tikzpicture}[scale=0.2] 
\fill[black] (0.5,0) rectangle (1,1.5);
\fill[black] (1,1.5) rectangle (2,1);
\draw (0,0) rectangle (2,2);
\end{tikzpicture}, \begin{tikzpicture}[scale=0.2] 
\fill[black] (1.5,0) rectangle (1,1.5);
\fill[black] (1,1.5) rectangle (0,1);
\draw (0,0) rectangle (2,2);
\end{tikzpicture}, \begin{tikzpicture}[scale=0.2] 
\fill[black] (1.5,2) rectangle (1,0.5);
\fill[black] (1,0.5) rectangle (0,1);
\draw (0,0) rectangle (2,2);
\end{tikzpicture}, \begin{tikzpicture}[scale=0.2] 
\fill[black] (0.5,2) rectangle (1,0.5);
\fill[black] (1,0.5) rectangle (2,1);
\draw (0,0) rectangle (2,2);
\end{tikzpicture} \right\},\] and a blank symbol. \bigskip

\noindent \textbf{\textit{Local rules:}}

\begin{itemize}
\item \textbf{Localization:} the non blank symbols are superimposed 
on and only on positions with petal symbols of the gray faces.
The symbol is transmitted through the petals, except on 
transformation positions, defined just below.
\item \textbf{Transformation positions:} the 
\textbf{transformation positions} are the 
positions where the transformation rule occurs 
(meaning that the signal is transformed). 
These positions depend on the (sub)layer.
In this sublayer, these are the positions 
where a support petal intersects 
a transmission petal just 
under in the hierarchy. 
On these positions is superimposed a \textbf{couple 
of symbols}, in 
\[\left\{\begin{tikzpicture}[scale=0.2] 
\draw (0,0) rectangle (2,2);
\end{tikzpicture}, \begin{tikzpicture}[scale=0.2] 
\fill[black] (0.5,0) rectangle (1,1.5);
\fill[black] (1,1.5) rectangle (2,1);
\draw (0,0) rectangle (2,2);
\end{tikzpicture}, \begin{tikzpicture}[scale=0.2] 
\fill[black] (1.5,0) rectangle (1,1.5);
\fill[black] (1,1.5) rectangle (0,1);
\draw (0,0) rectangle (2,2);
\end{tikzpicture}, \begin{tikzpicture}[scale=0.2] 
\fill[black] (1.5,2) rectangle (1,0.5);
\fill[black] (1,0.5) rectangle (0,1);
\draw (0,0) rectangle (2,2);
\end{tikzpicture}, \begin{tikzpicture}[scale=0.2] 
\fill[black] (0.5,2) rectangle (1,0.5);
\fill[black] (1,0.5) rectangle (2,1);
\draw (0,0) rectangle (2,2);
\end{tikzpicture} \right\} \times \left\{\begin{tikzpicture}[scale=0.2] 
\fill[black] (0.5,0) rectangle (1,1.5);
\fill[black] (1,1.5) rectangle (2,1);
\draw (0,0) rectangle (2,2);
\end{tikzpicture}, \begin{tikzpicture}[scale=0.2] 
\fill[black] (1.5,0) rectangle (1,1.5);
\fill[black] (1,1.5) rectangle (0,1);
\draw (0,0) rectangle (2,2);
\end{tikzpicture}, \begin{tikzpicture}[scale=0.2] 
\fill[black] (1.5,2) rectangle (1,0.5);
\fill[black] (1,0.5) rectangle (0,1);
\draw (0,0) rectangle (2,2);
\end{tikzpicture}, \begin{tikzpicture}[scale=0.2] 
\fill[black] (0.5,2) rectangle (1,0.5);
\fill[black] (1,0.5) rectangle (2,1);
\draw (0,0) rectangle (2,2);
\end{tikzpicture} \right\},\] 
while the other petal positions are superimposed 
with a \textbf{simple symbol}, in 
\[\left\{\begin{tikzpicture}[scale=0.2] 
\draw (0,0) rectangle (2,2);
\end{tikzpicture}, \begin{tikzpicture}[scale=0.2] 
\fill[black] (0.5,0) rectangle (1,1.5);
\fill[black] (1,1.5) rectangle (2,1);
\draw (0,0) rectangle (2,2);
\end{tikzpicture}, \begin{tikzpicture}[scale=0.2] 
\fill[black] (1.5,0) rectangle (1,1.5);
\fill[black] (1,1.5) rectangle (0,1);
\draw (0,0) rectangle (2,2);
\end{tikzpicture}, \begin{tikzpicture}[scale=0.2] 
\fill[black] (1.5,2) rectangle (1,0.5);
\fill[black] (1,0.5) rectangle (0,1);
\draw (0,0) rectangle (2,2);
\end{tikzpicture}, \begin{tikzpicture}[scale=0.2] 
\fill[black] (0.5,2) rectangle (1,0.5);
\fill[black] (1,0.5) rectangle (2,1);
\draw (0,0) rectangle (2,2);
\end{tikzpicture} \right\}.\]
On the transmission positions, 
the first symbol corresponds 
to outgoing arrows (this is 
the direction from which the signal 
comes, that is to say the support petal), 
and the second one to incoming arrows 
(this is the direction where the signal is 
transmitted, after transformation).
\item \textbf{Transformation:} on a transformation
position, the second bit of the
couple is $\begin{tikzpicture}[scale=0.2] 
\fill[black] (0.5,0) rectangle (1,1.5);
\fill[black] (1,1.5) rectangle (2,1);
\draw (0,0) rectangle (2,2);
\end{tikzpicture}$ if the six arrows symbol is 
\[\begin{tikzpicture}[scale=0.3]
\draw (0,0) rectangle (2,2) ;
\draw [-latex] (0,1) -- (2,1) ;
\draw [-latex] (0,1.5) -- (2,1.5) ; 
\draw [-latex] (1,0) -- (1,1) ; 
\draw [-latex] (1,2) -- (1,1.5) ;
\draw [-latex] (1.5,0) -- (1.5,1) ; 
\draw [-latex] (1.5,2) -- (1.5,1.5) ;
\end{tikzpicture}, \ \text{or} \ \begin{tikzpicture}[scale=0.3]
\draw (0,0) rectangle (2,2) ;
\draw [-latex] (1,2) -- (1,0) ;
\draw [latex-] (0.5,0) -- (0.5,2) ; 
\draw [-latex] (0,1) -- (0.5,1) ; 
\draw [-latex] (2,1) -- (1,1) ;
\draw [-latex] (0,1.5) -- (0.5,1.5) ; 
\draw [-latex] (2,1.5) -- (1,1.5) ;
\end{tikzpicture}.\]
These positions correspond to the 
intersection of an 
order $n+1$ support petal  
with the \textbf{north west} order $n$ 
transmission petal just under in the hierarchy. 
\textit{There are similar rules for the other orientations}.
\item \textbf{Border rule:} on the border of a 
gray face, the symbol is blank if not on a transmission 
position. On a transmission position, the
first symbol of the couple is blank.
\end{itemize}

\noindent \textbf{\textit{Global behavior:}} \bigskip

This layer supports a signal that propagates
through the petal 
hierarchy on the colored faces.
This signal is transmitted through 
the petals except on the intersections 
of a support petal and a 
transmission petal just above 
in the hierarchy. On these positions, 
the symbol transmitted by the 
signal is transformed into a symbol 
representing the orientation of the transmission 
petal with respect to the support petal. 

As a consequence, the support petals 
just under the transmission petal 
are colored with this orientation symbol.
See the schema on Figure~\ref{fig.orientation}.

\subsection{\label{subsection.functional.areas.whole.cell} 
Hierarchical constitution of functional areas}

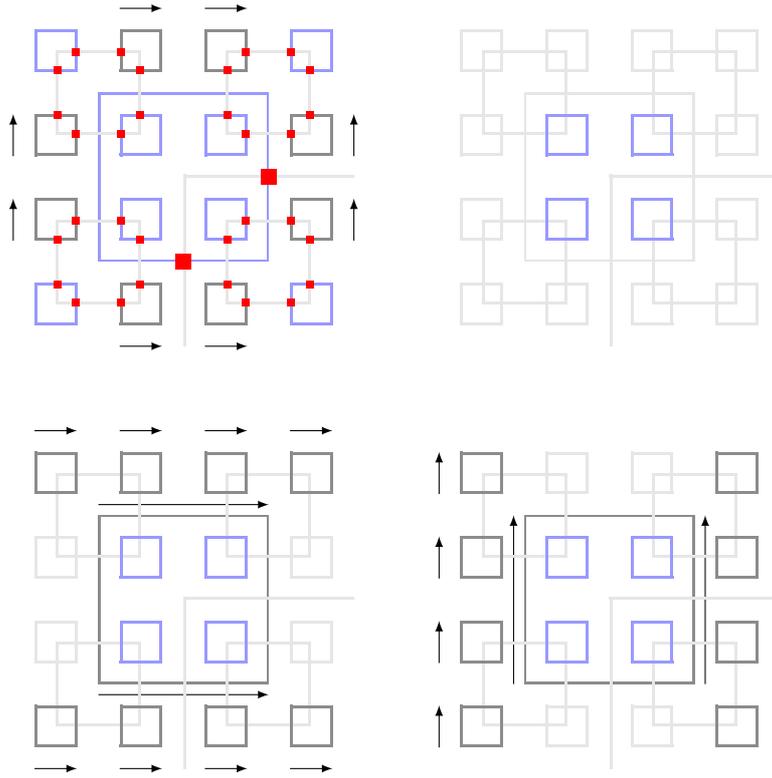
\begin{figure}[ht]
\[\begin{tikzpicture}[scale=0.07]
\begin{scope}[xshift=0cm]
\fill[blue!40] (16,16) rectangle (48,48); 
\fill[white] (16.5,16.5) rectangle (47.5,47.5);
\fill[gray!20] (32,0) rectangle (32.5,32.5);
\fill[gray!20] (32,32.5) rectangle (64,32);

\fill[gray!20] (8,8) rectangle (8.5,24); 
\fill[gray!20] (8,8) rectangle (24,8.5);
\fill[gray!20] (8,23.5) rectangle (24,24);  
\fill[gray!20] (23.5,8) rectangle (24,24);

\fill[gray!20] (40,8) rectangle (40.5,24); 
\fill[gray!20] (40,8) rectangle (56,8.5);
\fill[gray!20] (40,23.5) rectangle (56,24);  
\fill[gray!20] (55.5,8) rectangle (56,24);

\fill[gray!20] (8,40) rectangle (8.5,56); 
\fill[gray!20] (8,40) rectangle (24,40.5);
\fill[gray!20] (8,55.5) rectangle (24,56);  
\fill[gray!20] (23.5,40) rectangle (24,56);

\fill[gray!20] (40,40) rectangle (40.5,56); 
\fill[gray!20] (40,40) rectangle (56,40.5);
\fill[gray!20] (40,55.5) rectangle (56,56);  
\fill[gray!20] (55.5,40) rectangle (56,56);

\fill[blue!40] (4,4) rectangle (4.5,12); 
\fill[blue!40] (4,4) rectangle (12,4.5); 
\fill[blue!40] (4.5,11.5) rectangle (12,12); 
\fill[blue!40] (11.5,4.5) rectangle (12,12); 

\fill[gray!90] (4,20) rectangle (4.5,28); 
\fill[gray!90] (4,20) rectangle (12,20.5); 
\fill[gray!90] (4.5,27.5) rectangle (12,28); 
\fill[gray!90] (11.5,20.5) rectangle (12,28); 

\fill[gray!90] (4,36) rectangle (4.5,44); 
\fill[gray!90] (4,36) rectangle (12,36.5); 
\fill[gray!90] (4.5,43.5) rectangle (12,44); 
\fill[gray!90] (11.5,36.5) rectangle (12,44);

\fill[blue!40] (4,52) rectangle (4.5,60); 
\fill[blue!40] (4,52) rectangle (12,52.5); 
\fill[blue!40] (4.5,59.5) rectangle (12,60); 
\fill[blue!40] (11.5,52.5) rectangle (12,60);

\fill[gray!90] (20,4) rectangle (20.5,12); 
\fill[gray!90] (20,4) rectangle (28,4.5); 
\fill[gray!90] (20.5,11.5) rectangle (28,12); 
\fill[gray!90] (27.5,4.5) rectangle (28,12); 

\fill[blue!40] (20,20) rectangle (20.5,28); 
\fill[blue!40] (20,20) rectangle (28,20.5); 
\fill[blue!40] (20.5,27.5) rectangle (28,28); 
\fill[blue!40] (27.5,20.5) rectangle (28,28); 

\fill[blue!40] (20,36) rectangle (20.5,44); 
\fill[blue!40] (20,36) rectangle (28,36.5); 
\fill[blue!40] (20.5,43.5) rectangle (28,44); 
\fill[blue!40] (27.5,36.5) rectangle (28,44);

\fill[gray!90] (20,52) rectangle (20.5,60); 
\fill[gray!90] (20,52) rectangle (28,52.5); 
\fill[gray!90] (20.5,59.5) rectangle (28,60); 
\fill[gray!90] (27.5,52.5) rectangle (28,60);

\fill[gray!90] (36,4) rectangle (36.5,12); 
\fill[gray!90] (36,4) rectangle (44,4.5); 
\fill[gray!90] (36.5,11.5) rectangle (44,12); 
\fill[gray!90] (43.5,4.5) rectangle (44,12); 

\fill[blue!40] (36,20) rectangle (36.5,28); 
\fill[blue!40] (36,20) rectangle (44,20.5); 
\fill[blue!40] (36.5,27.5) rectangle (44,28); 
\fill[blue!40] (43.5,20.5) rectangle (44,28); 

\fill[blue!40] (36,36) rectangle (36.5,44); 
\fill[blue!40] (36,36) rectangle (44,36.5); 
\fill[blue!40] (36.5,43.5) rectangle (44,44); 
\fill[blue!40] (43.5,36.5) rectangle (44,44);

\fill[gray!90] (36,52) rectangle (36.5,60); 
\fill[gray!90] (36,52) rectangle (44,52.5); 
\fill[gray!90] (36.5,59.5) rectangle (44,60); 
\fill[gray!90] (43.5,52.5) rectangle (44,60);

\fill[blue!40] (52,4) rectangle (52.5,12); 
\fill[blue!40] (52,4) rectangle (60,4.5); 
\fill[blue!40] (52.5,11.5) rectangle (60,12); 
\fill[blue!40] (59.5,4.5) rectangle (60,12); 

\fill[gray!90] (52,20) rectangle (52.5,28); 
\fill[gray!90] (52,20) rectangle (60,20.5); 
\fill[gray!90] (52.5,27.5) rectangle (60,28); 
\fill[gray!90] (59.5,20.5) rectangle (60,28); 

\fill[gray!90] (52,36) rectangle (52.5,44); 
\fill[gray!90] (52,36) rectangle (60,36.5); 
\fill[gray!90] (52.5,43.5) rectangle (60,44); 
\fill[gray!90] (59.5,36.5) rectangle (60,44);

\fill[blue!40] (52,52) rectangle (52.5,60); 
\fill[blue!40] (52,52) rectangle (60,52.5); 
\fill[blue!40] (52.5,59.5) rectangle (60,60); 
\fill[blue!40] (59.5,52.5) rectangle (60,60);

\fill[red] (11,7.5) rectangle (12.5,9);
\fill[red] (7.5,11) rectangle (9,12.5);

\fill[red] (19.5,7.5) 
rectangle (21,9);
\fill[red] (23,11) 
rectangle (24.5,12.5);

\fill[red] (43,7.5) rectangle (44.5,9);
\fill[red] (39.5,11) rectangle (41,12.5);

\fill[red] (51.5,7.5) 
rectangle (53,9);
\fill[red] (55,11) 
rectangle (56.5,12.5);

\fill[red] (11,39.5) rectangle (12.5,41);
\fill[red] (7.5,43) rectangle (9,44.5);

\fill[red] (19.5,39.5) 
rectangle (21,41);
\fill[red] (23,43) 
rectangle (24.5,44.5);

\fill[red] (43,39.5) rectangle (44.5,41);
\fill[red] (39.5,43) rectangle (41,44.5);

\fill[red] (51.5,39.5) 
rectangle (53,41);
\fill[red] (55,43) 
rectangle (56.5,44.5);


\fill[red] (11,23) rectangle (12.5,24.5);
\fill[red] (7.5,19.5) rectangle (9,21);

\fill[red] (19.5,23) rectangle (21,24.5);
\fill[red] (23,19.5) rectangle (24.5,21);

\fill[red] (43,23) rectangle (44.5,24.5);
\fill[red] (39.5,19.5) rectangle (41,21);

\fill[red] (51.5,23) rectangle (53,24.5);
\fill[red] (55,19.5) rectangle (56.5,21);

\fill[red] (11,55) rectangle (12.5,56.5);
\fill[red] (7.5,51.5) rectangle (9,53);

\fill[red] (19.5,55) rectangle (21,56.5);
\fill[red] (23,51.5) rectangle (24.5,53);

\fill[red] (43,55) rectangle (44.5,56.5);
\fill[red] (39.5,51.5) rectangle (41,53);

\fill[red] (51.5,55) rectangle (53,56.5);
\fill[red] (55,51.5) rectangle (56.5,53);

\fill[red] (30.5,14.5) rectangle 
(33.5,17.5);

\fill[red] (46.5,30.5) rectangle 
(49.5,33.5);

\draw[-latex] (20,64) -- (28,64);
\draw[-latex] (36,64) -- (44,64);
\draw[-latex] (20,0) -- (28,0);
\draw[-latex] (36,0) -- (44,0);

\draw[-latex] (64,20) -- (64,28);
\draw[-latex] (64,36) -- (64,44);
\draw[-latex] (0,20) -- (0,28);
\draw[-latex] (0,36) -- (0,44);
\end{scope}

\begin{scope}[yshift=-80cm]
\fill[gray!90] (16,16) rectangle (48,48); 
\draw[-latex] (16,14) -- (48,14);
\draw[-latex] (16,50) -- (48,50);
\fill[white] (16.5,16.5) rectangle (47.5,47.5);
\fill[gray!20] (32,0) rectangle (32.5,32.5);
\fill[gray!20] (32,32.5) rectangle (64,32);

\fill[gray!20] (8,8) rectangle (8.5,24); 
\fill[gray!20] (8,8) rectangle (24,8.5);
\fill[gray!20] (8,23.5) rectangle (24,24);  
\fill[gray!20] (23.5,8) rectangle (24,24);

\fill[gray!20] (40,8) rectangle (40.5,24); 
\fill[gray!20] (40,8) rectangle (56,8.5);
\fill[gray!20] (40,23.5) rectangle (56,24);  
\fill[gray!20] (55.5,8) rectangle (56,24);

\fill[gray!20] (8,40) rectangle (8.5,56); 
\fill[gray!20] (8,40) rectangle (24,40.5);
\fill[gray!20] (8,55.5) rectangle (24,56);  
\fill[gray!20] (23.5,40) rectangle (24,56);

\fill[gray!20] (40,40) rectangle (40.5,56); 
\fill[gray!20] (40,40) rectangle (56,40.5);
\fill[gray!20] (40,55.5) rectangle (56,56);  
\fill[gray!20] (55.5,40) rectangle (56,56);

\fill[gray!90] (4,4) rectangle (4.5,12); 
\fill[gray!90] (4,4) rectangle (12,4.5); 
\fill[gray!90] (4.5,11.5) rectangle (12,12); 
\fill[gray!90] (11.5,4.5) rectangle (12,12); 

\fill[gray!20] (4,20) rectangle (4.5,28); 
\fill[gray!20] (4,20) rectangle (12,20.5); 
\fill[gray!20] (4.5,27.5) rectangle (12,28); 
\fill[gray!20] (11.5,20.5) rectangle (12,28); 

\fill[gray!20] (4,36) rectangle (4.5,44); 
\fill[gray!20] (4,36) rectangle (12,36.5); 
\fill[gray!20] (4.5,43.5) rectangle (12,44); 
\fill[gray!20] (11.5,36.5) rectangle (12,44);

\fill[gray!90] (4,52) rectangle (4.5,60); 
\fill[gray!90] (4,52) rectangle (12,52.5); 
\fill[gray!90] (4.5,59.5) rectangle (12,60); 
\fill[gray!90] (11.5,52.5) rectangle (12,60);

\fill[gray!90] (20,4) rectangle (20.5,12); 
\fill[gray!90] (20,4) rectangle (28,4.5); 
\fill[gray!90] (20.5,11.5) rectangle (28,12); 
\fill[gray!90] (27.5,4.5) rectangle (28,12); 

\fill[blue!40] (20,20) rectangle (20.5,28); 
\fill[blue!40] (20,20) rectangle (28,20.5); 
\fill[blue!40] (20.5,27.5) rectangle (28,28); 
\fill[blue!40] (27.5,20.5) rectangle (28,28); 

\fill[blue!40] (20,36) rectangle (20.5,44); 
\fill[blue!40] (20,36) rectangle (28,36.5); 
\fill[blue!40] (20.5,43.5) rectangle (28,44); 
\fill[blue!40] (27.5,36.5) rectangle (28,44);

\fill[gray!90] (20,52) rectangle (20.5,60); 
\fill[gray!90] (20,52) rectangle (28,52.5); 
\fill[gray!90] (20.5,59.5) rectangle (28,60); 
\fill[gray!90] (27.5,52.5) rectangle (28,60);

\fill[gray!90] (36,4) rectangle (36.5,12); 
\fill[gray!90] (36,4) rectangle (44,4.5); 
\fill[gray!90] (36.5,11.5) rectangle (44,12); 
\fill[gray!90] (43.5,4.5) rectangle (44,12); 

\fill[blue!40] (36,20) rectangle (36.5,28); 
\fill[blue!40] (36,20) rectangle (44,20.5); 
\fill[blue!40] (36.5,27.5) rectangle (44,28); 
\fill[blue!40] (43.5,20.5) rectangle (44,28); 

\fill[blue!40] (36,36) rectangle (36.5,44); 
\fill[blue!40] (36,36) rectangle (44,36.5); 
\fill[blue!40] (36.5,43.5) rectangle (44,44); 
\fill[blue!40] (43.5,36.5) rectangle (44,44);

\fill[gray!90] (36,52) rectangle (36.5,60); 
\fill[gray!90] (36,52) rectangle (44,52.5); 
\fill[gray!90] (36.5,59.5) rectangle (44,60); 
\fill[gray!90] (43.5,52.5) rectangle (44,60);

\fill[gray!90] (52,4) rectangle (52.5,12); 
\fill[gray!90] (52,4) rectangle (60,4.5); 
\fill[gray!90] (52.5,11.5) rectangle (60,12); 
\fill[gray!90] (59.5,4.5) rectangle (60,12); 

\fill[gray!20] (52,20) rectangle (52.5,28); 
\fill[gray!20] (52,20) rectangle (60,20.5); 
\fill[gray!20] (52.5,27.5) rectangle (60,28); 
\fill[gray!20] (59.5,20.5) rectangle (60,28); 

\fill[gray!20] (52,36) rectangle (52.5,44); 
\fill[gray!20] (52,36) rectangle (60,36.5); 
\fill[gray!20] (52.5,43.5) rectangle (60,44); 
\fill[gray!20] (59.5,36.5) rectangle (60,44);

\fill[gray!90] (52,52) rectangle (52.5,60); 
\fill[gray!90] (52,52) rectangle (60,52.5); 
\fill[gray!90] (52.5,59.5) rectangle (60,60); 
\fill[gray!90] (59.5,52.5) rectangle (60,60);

\draw[-latex] (20,64) -- (28,64);
\draw[-latex] (4,64) -- (12,64);
\draw[-latex] (36,64) -- (44,64);
\draw[-latex] (52,64) -- (60,64);
\draw[-latex] (20,0) -- (28,0);
\draw[-latex] (36,0) -- (44,0);

\draw[-latex] (4,0) -- (12,0);
\draw[-latex] (52,0) -- (60,0);

\end{scope}

\begin{scope}[xshift=80cm]
\fill[gray!20] (16,16) rectangle (48,48); 
\fill[white] (16.5,16.5) rectangle (47.5,47.5);
\fill[gray!20] (32,0) rectangle (32.5,32.5);
\fill[gray!20] (32,32.5) rectangle (64,32);

\fill[gray!20] (8,8) rectangle (8.5,24); 
\fill[gray!20] (8,8) rectangle (24,8.5);
\fill[gray!20] (8,23.5) rectangle (24,24);  
\fill[gray!20] (23.5,8) rectangle (24,24);

\fill[gray!20] (40,8) rectangle (40.5,24); 
\fill[gray!20] (40,8) rectangle (56,8.5);
\fill[gray!20] (40,23.5) rectangle (56,24);  
\fill[gray!20] (55.5,8) rectangle (56,24);

\fill[gray!20] (8,40) rectangle (8.5,56); 
\fill[gray!20] (8,40) rectangle (24,40.5);
\fill[gray!20] (8,55.5) rectangle (24,56);  
\fill[gray!20] (23.5,40) rectangle (24,56);

\fill[gray!20] (40,40) rectangle (40.5,56); 
\fill[gray!20] (40,40) rectangle (56,40.5);
\fill[gray!20] (40,55.5) rectangle (56,56);  
\fill[gray!20] (55.5,40) rectangle (56,56);

\fill[gray!20] (4,4) rectangle (4.5,12); 
\fill[gray!20] (4,4) rectangle (12,4.5); 
\fill[gray!20] (4.5,11.5) rectangle (12,12); 
\fill[gray!20] (11.5,4.5) rectangle (12,12); 

\fill[gray!20] (4,20) rectangle (4.5,28); 
\fill[gray!20] (4,20) rectangle (12,20.5); 
\fill[gray!20] (4.5,27.5) rectangle (12,28); 
\fill[gray!20] (11.5,20.5) rectangle (12,28); 

\fill[gray!20] (4,36) rectangle (4.5,44); 
\fill[gray!20] (4,36) rectangle (12,36.5); 
\fill[gray!20] (4.5,43.5) rectangle (12,44); 
\fill[gray!20] (11.5,36.5) rectangle (12,44);

\fill[gray!20] (4,52) rectangle (4.5,60); 
\fill[gray!20] (4,52) rectangle (12,52.5); 
\fill[gray!20] (4.5,59.5) rectangle (12,60); 
\fill[gray!20] (11.5,52.5) rectangle (12,60);

\fill[gray!20] (20,4) rectangle (20.5,12); 
\fill[gray!20] (20,4) rectangle (28,4.5); 
\fill[gray!20] (20.5,11.5) rectangle (28,12); 
\fill[gray!20] (27.5,4.5) rectangle (28,12); 

\fill[blue!40] (20,20) rectangle (20.5,28); 
\fill[blue!40] (20,20) rectangle (28,20.5); 
\fill[blue!40] (20.5,27.5) rectangle (28,28); 
\fill[blue!40] (27.5,20.5) rectangle (28,28); 

\fill[blue!40] (20,36) rectangle (20.5,44); 
\fill[blue!40] (20,36) rectangle (28,36.5); 
\fill[blue!40] (20.5,43.5) rectangle (28,44); 
\fill[blue!40] (27.5,36.5) rectangle (28,44);

\fill[gray!20] (20,52) rectangle (20.5,60); 
\fill[gray!20] (20,52) rectangle (28,52.5); 
\fill[gray!20] (20.5,59.5) rectangle (28,60); 
\fill[gray!20] (27.5,52.5) rectangle (28,60);

\fill[gray!20] (36,4) rectangle (36.5,12); 
\fill[gray!20] (36,4) rectangle (44,4.5); 
\fill[gray!20] (36.5,11.5) rectangle (44,12); 
\fill[gray!20] (43.5,4.5) rectangle (44,12); 

\fill[blue!40] (36,20) rectangle (36.5,28); 
\fill[blue!40] (36,20) rectangle (44,20.5); 
\fill[blue!40] (36.5,27.5) rectangle (44,28); 
\fill[blue!40] (43.5,20.5) rectangle (44,28); 

\fill[blue!40] (36,36) rectangle (36.5,44); 
\fill[blue!40] (36,36) rectangle (44,36.5); 
\fill[blue!40] (36.5,43.5) rectangle (44,44); 
\fill[blue!40] (43.5,36.5) rectangle (44,44);

\fill[gray!20] (36,52) rectangle (36.5,60); 
\fill[gray!20] (36,52) rectangle (44,52.5); 
\fill[gray!20] (36.5,59.5) rectangle (44,60); 
\fill[gray!20] (43.5,52.5) rectangle (44,60);

\fill[gray!20] (52,4) rectangle (52.5,12); 
\fill[gray!20] (52,4) rectangle (60,4.5); 
\fill[gray!20] (52.5,11.5) rectangle (60,12); 
\fill[gray!20] (59.5,4.5) rectangle (60,12); 

\fill[gray!20] (52,20) rectangle (52.5,28); 
\fill[gray!20] (52,20) rectangle (60,20.5); 
\fill[gray!20] (52.5,27.5) rectangle (60,28); 
\fill[gray!20] (59.5,20.5) rectangle (60,28); 

\fill[gray!20] (52,36) rectangle (52.5,44); 
\fill[gray!20] (52,36) rectangle (60,36.5); 
\fill[gray!20] (52.5,43.5) rectangle (60,44); 
\fill[gray!20] (59.5,36.5) rectangle (60,44);

\fill[gray!20] (52,52) rectangle (52.5,60); 
\fill[gray!20] (52,52) rectangle (60,52.5); 
\fill[gray!20] (52.5,59.5) rectangle (60,60); 
\fill[gray!20] (59.5,52.5) rectangle (60,60);

\end{scope}

\begin{scope}[yshift=-80cm,xshift=80cm]
\fill[gray!90] (16,16) rectangle (48,48); 
\draw[-latex] (14,16) -- (14,48);
\draw[-latex] (50,16) -- (50,48);
\fill[white] (16.5,16.5) rectangle (47.5,47.5);
\fill[gray!20] (32,0) rectangle (32.5,32.5);
\fill[gray!20] (32,32.5) rectangle (64,32);

\fill[gray!20] (8,8) rectangle (8.5,24); 
\fill[gray!20] (8,8) rectangle (24,8.5);
\fill[gray!20] (8,23.5) rectangle (24,24);  
\fill[gray!20] (23.5,8) rectangle (24,24);

\fill[gray!20] (40,8) rectangle (40.5,24); 
\fill[gray!20] (40,8) rectangle (56,8.5);
\fill[gray!20] (40,23.5) rectangle (56,24);  
\fill[gray!20] (55.5,8) rectangle (56,24);

\fill[gray!20] (8,40) rectangle (8.5,56); 
\fill[gray!20] (8,40) rectangle (24,40.5);
\fill[gray!20] (8,55.5) rectangle (24,56);  
\fill[gray!20] (23.5,40) rectangle (24,56);

\fill[gray!20] (40,40) rectangle (40.5,56); 
\fill[gray!20] (40,40) rectangle (56,40.5);
\fill[gray!20] (40,55.5) rectangle (56,56);  
\fill[gray!20] (55.5,40) rectangle (56,56);

\fill[gray!90] (4,4) rectangle (4.5,12); 
\fill[gray!90] (4,4) rectangle (12,4.5); 
\fill[gray!90] (4.5,11.5) rectangle (12,12); 
\fill[gray!90] (11.5,4.5) rectangle (12,12); 

\fill[gray!90] (4,20) rectangle (4.5,28); 
\fill[gray!90] (4,20) rectangle (12,20.5); 
\fill[gray!90] (4.5,27.5) rectangle (12,28); 
\fill[gray!90] (11.5,20.5) rectangle (12,28); 

\fill[gray!90] (4,36) rectangle (4.5,44); 
\fill[gray!90] (4,36) rectangle (12,36.5); 
\fill[gray!90] (4.5,43.5) rectangle (12,44); 
\fill[gray!90] (11.5,36.5) rectangle (12,44);

\fill[gray!90] (4,52) rectangle (4.5,60); 
\fill[gray!90] (4,52) rectangle (12,52.5); 
\fill[gray!90] (4.5,59.5) rectangle (12,60); 
\fill[gray!90] (11.5,52.5) rectangle (12,60);

\fill[gray!20] (20,4) rectangle (20.5,12); 
\fill[gray!20] (20,4) rectangle (28,4.5); 
\fill[gray!20] (20.5,11.5) rectangle (28,12); 
\fill[gray!20] (27.5,4.5) rectangle (28,12); 

\fill[blue!40] (20,20) rectangle (20.5,28); 
\fill[blue!40] (20,20) rectangle (28,20.5); 
\fill[blue!40] (20.5,27.5) rectangle (28,28); 
\fill[blue!40] (27.5,20.5) rectangle (28,28); 

\fill[blue!40] (20,36) rectangle (20.5,44); 
\fill[blue!40] (20,36) rectangle (28,36.5); 
\fill[blue!40] (20.5,43.5) rectangle (28,44); 
\fill[blue!40] (27.5,36.5) rectangle (28,44);

\fill[gray!20] (20,52) rectangle (20.5,60); 
\fill[gray!20] (20,52) rectangle (28,52.5); 
\fill[gray!20] (20.5,59.5) rectangle (28,60); 
\fill[gray!20] (27.5,52.5) rectangle (28,60);

\fill[gray!20] (36,4) rectangle (36.5,12); 
\fill[gray!20] (36,4) rectangle (44,4.5); 
\fill[gray!20] (36.5,11.5) rectangle (44,12); 
\fill[gray!20] (43.5,4.5) rectangle (44,12); 

\fill[blue!40] (36,20) rectangle (36.5,28); 
\fill[blue!40] (36,20) rectangle (44,20.5); 
\fill[blue!40] (36.5,27.5) rectangle (44,28); 
\fill[blue!40] (43.5,20.5) rectangle (44,28); 

\fill[blue!40] (36,36) rectangle (36.5,44); 
\fill[blue!40] (36,36) rectangle (44,36.5); 
\fill[blue!40] (36.5,43.5) rectangle (44,44); 
\fill[blue!40] (43.5,36.5) rectangle (44,44);

\fill[gray!20] (36,52) rectangle (36.5,60); 
\fill[gray!20] (36,52) rectangle (44,52.5); 
\fill[gray!20] (36.5,59.5) rectangle (44,60); 
\fill[gray!20] (43.5,52.5) rectangle (44,60);

\fill[gray!90] (52,4) rectangle (52.5,12); 
\fill[gray!90] (52,4) rectangle (60,4.5); 
\fill[gray!90] (52.5,11.5) rectangle (60,12); 
\fill[gray!90] (59.5,4.5) rectangle (60,12); 

\fill[gray!90] (52,20) rectangle (52.5,28); 
\fill[gray!90] (52,20) rectangle (60,20.5); 
\fill[gray!90] (52.5,27.5) rectangle (60,28); 
\fill[gray!90] (59.5,20.5) rectangle (60,28); 

\fill[gray!90] (52,36) rectangle (52.5,44); 
\fill[gray!90] (52,36) rectangle (60,36.5); 
\fill[gray!90] (52.5,43.5) rectangle (60,44); 
\fill[gray!90] (59.5,36.5) rectangle (60,44);

\fill[gray!90] (52,52) rectangle (52.5,60); 
\fill[gray!90] (52,52) rectangle (60,52.5); 
\fill[gray!90] (52.5,59.5) rectangle (60,60); 
\fill[gray!90] (59.5,52.5) rectangle (60,60);

\draw[-latex] (64,20) -- (64,28);
\draw[-latex] (64,4) -- (64,12);
\draw[-latex] (64,36) -- (64,44);
\draw[-latex] (64,52) -- (64,60);
\draw[-latex] (0,20) -- (0,28);
\draw[-latex] (0,36) -- (0,44);

\draw[-latex] (0,4) -- (0,12);
\draw[-latex] (0,52) -- (0,60);

\end{scope}

\end{tikzpicture}\]
\caption{\label{fig.trans.comp.area} 
Schematic illustration of 
the transmission rules of the functional areas sublayer.
The transformation positions are marked with red squares
on the first schema.}
\end{figure}

We present here the second sublayer.

\noindent \textbf{\textit{Symbols:}} \bigskip

$\begin{tikzpicture}[scale=0.3] 
\fill[blue!40] (0,0) rectangle (1,1);
\draw (0,0) rectangle (1,1);
\end{tikzpicture}$,  
\begin{tikzpicture}[scale=0.3] 
\fill[gray!20] (0,0) rectangle (1,1);
\draw (0,0) rectangle (1,1);
\end{tikzpicture},
 $\left(\begin{tikzpicture}[scale=0.3] 
\fill[gray!90] (0,0) rectangle (1,1);
\draw (0,0) rectangle (1,1);
\end{tikzpicture}, \rightarrow \right)$, $\left(\begin{tikzpicture}[scale=0.3] 
\fill[gray!90] (0,0) rectangle (1,1);
\draw (0,0) rectangle (1,1);
\end{tikzpicture},\uparrow \right)$, 
$\begin{tikzpicture}[scale=0.3] 
\draw (0,0) rectangle (1,1);
\end{tikzpicture}$ and $\left\{\begin{tikzpicture}[scale=0.3] 
\fill[blue!40] (0,0) rectangle (1,1);
\draw (0,0) rectangle (1,1);
\end{tikzpicture}, \begin{tikzpicture}[scale=0.3] 
\fill[gray!20] (0,0) rectangle (1,1);
\draw (0,0) rectangle (1,1);
\end{tikzpicture}, \left(\begin{tikzpicture}[scale=0.3] 
\fill[gray!90] (0,0) rectangle (1,1);
\draw (0,0) rectangle (1,1);
\end{tikzpicture}, \rightarrow \right), \left(\begin{tikzpicture}[scale=0.3] 
\fill[gray!90] (0,0) rectangle (1,1);
\draw (0,0) rectangle (1,1);
\end{tikzpicture},\uparrow \right)\right\}^2$. \bigskip

\noindent \textbf{\textit{Local rules:}} \bigskip
 
\begin{itemize}
\item \textbf{Localization:} 
the non blank symbols are superimposed on
and only on petal positions of the colored 
faces. 
\item the ordered pairs of symbols 
are superimposed 
on six arrows symbols positions 
where the border of a cell 
intersects the petal just 
above in the hierarchy (\textbf{transformation 
positions}).
\item the symbols are transmitted through 
the petals except on transformation 
positions.
\item {\bf{Transformation through hierarchy}}. 
On the transformation positions, 
if the first symbol is \begin{tikzpicture}[scale=0.3]
\fill[gray!20] (0,0) rectangle (1,1);
\draw (0,0) rectangle (1,1);
\end{tikzpicture}, then the second one is equal. 
For the other cases, if the symbol is 
\[\begin{tikzpicture}[scale=0.3]
\draw (0,0) rectangle (2,2) ;
\draw [-latex] (0,1) -- (2,1) ;
\draw [-latex] (0,1.5) -- (2,1.5) ; 
\draw [-latex] (1,0) -- (1,1) ; 
\draw [-latex] (1,2) -- (1,1.5) ;
\draw [-latex] (1.5,0) -- (1.5,1) ; 
\draw [-latex] (1.5,2) -- (1.5,1.5) ;
\end{tikzpicture}, \ \text{or} \ \begin{tikzpicture}[scale=0.3]
\draw (0,0) rectangle (2,2) ;
\draw [-latex] (1,2) -- (1,0) ;
\draw [latex-] (0.5,0) -- (0.5,2) ; 
\draw [-latex] (0,1) -- (0.5,1) ; 
\draw [-latex] (2,1) -- (1,1) ;
\draw [-latex] (0,1.5) -- (0.5,1.5) ; 
\draw [-latex] (2,1.5) -- (1,1.5) ;
\end{tikzpicture},\]
then second color is according to the following rules. The condition 
on the symbol above corresponds
to positions 
where a support 
petal intersects 
the transmission petal 
just above in the hierarchy 
and being oriented 
\textbf{north west} 
with respect to this 
transmission petal 
(examples of 
such positions 
are represented
with large squares on 
Figure~\ref{fig.trans.comp.area}).

\begin{enumerate}
\item if the orientation symbol is $\begin{tikzpicture}[scale=0.2] 
\fill[black] (0.5,0) rectangle (1,1.5);
\fill[black] (1,1.5) rectangle (2,1);
\draw (0,0) rectangle (2,2);
\end{tikzpicture}$, the second symbol is equal to the first one. 
\item if the orientation symbol is   $\begin{tikzpicture}[scale=0.2] 
\fill[black] (1.5,0) rectangle (1,1.5);
\fill[black] (1,1.5) rectangle (0,1);
\draw (0,0) rectangle (2,2);
\end{tikzpicture}$, the second symbol 
is a function of the first one: 
\begin{itemize}
\item if the first symbol of the ordered pair is \begin{tikzpicture}[scale=0.3]
\fill[blue!40] (0,0) rectangle (1,1);
\draw (0,0) rectangle (1,1);
\end{tikzpicture},  then the second symbol is (\begin{tikzpicture}[scale=0.3]
\fill[gray!90] (0,0) rectangle (1,1);
\draw (0,0) rectangle (1,1);
\end{tikzpicture},$\rightarrow$). 
\item if the first symbol is (\begin{tikzpicture}[scale=0.3]
\fill[gray!90] (0,0) rectangle (1,1);
\draw (0,0) rectangle (1,1);
\end{tikzpicture},$\rightarrow$), then the second one is equal.
\item if the first symbol is $(\begin{tikzpicture}[scale=0.3] 
\fill[gray!90] (0,0) rectangle (1,1);
\draw (0,0) rectangle (1,1);
\end{tikzpicture}, \uparrow)$, then the second symbol is 
\begin{tikzpicture}[scale=0.3]
\fill[gray!20] (0,0) rectangle (1,1);
\draw (0,0) rectangle (1,1);
\end{tikzpicture}. 
\end{itemize}
\item if the orientation symbol is  
$\begin{tikzpicture}[scale=0.2] 
\fill[black] (1.5,2) rectangle (1,0.5);
\fill[black] (1,0.5) rectangle (0,1);
\draw (0,0) rectangle (2,2);
\end{tikzpicture}$, the second color is 
\begin{tikzpicture}[scale=0.3]
\fill[blue!40] (0,0) rectangle (1,1);
\draw (0,0) rectangle (1,1);
\end{tikzpicture}.
\item if the orientation symbol is 
$\begin{tikzpicture}[scale=0.2] 
\fill[black] (0.5,2) rectangle (1,0.5);
\fill[black] (1,0.5) rectangle (2,1);
\draw (0,0) rectangle (2,2);
\end{tikzpicture}$, the second symbol 
is a function of the first one: 
\begin{itemize}
\item if the first symbol of the ordered pair is or \begin{tikzpicture}[scale=0.3]
\fill[blue!40] (0,0) rectangle (1,1);
\draw (0,0) rectangle (1,1);
\end{tikzpicture},  then the second symbol is (\begin{tikzpicture}[scale=0.3]
\fill[gray!90] (0,0) rectangle (1,1);
\draw (0,0) rectangle (1,1);
\end{tikzpicture},$\rightarrow$). 
\item if the first symbol is (\begin{tikzpicture}[scale=0.3]
\fill[gray!90] (0,0) rectangle (1,1);
\draw (0,0) rectangle (1,1);
\end{tikzpicture},$\rightarrow$), then the second one is \begin{tikzpicture}[scale=0.3]
\fill[gray!20] (0,0) rectangle (1,1);
\draw (0,0) rectangle (1,1);
\end{tikzpicture}. 
\item if the first symbol is $(\begin{tikzpicture}[scale=0.3] 
\fill[gray!90] (0,0) rectangle (1,1);
\draw (0,0) rectangle (1,1);
\end{tikzpicture}, \uparrow)$, then the second symbol is equal.
\end{itemize}
\end{enumerate}

For the other orientations 
of the support 
petal with respect 
to the transmission 
petal just above, 
the rules are obtained 
by rotation.

See Figure~\ref{fig.trans.comp.area} for an illustration of these rules.
\item \textbf{Coherence rules:} We impose rules that 
allow the infinite areas to be coherent with the finite ones. 
For instance, the nearest blue corner to the corner of a cell and which 
is inside this cell has to be colored blue in this sublayer. The other 
rules impose similar contraints on middles and quarters of the 
cell's walls.
\end{itemize} \bigskip

\noindent \textbf{
\textit{Global behavior:}} \bigskip

The result of the process presented in the local 
rules is 
that the borders of the order zero 
cells are colored 
with a color
which represents a \textbf{function} of 
the blue corners positions - called 
\textbf{functional positions} - 
included in the 
zero order petal just under this petal 
in the petal hierarchy.
These symbols and functions are as follows: 

\begin{itemize}
\item \textbf{blue} if the set of columns and 
the set of lines in which it is included do not
intersect larger order two-dimensional 
cells. The associated function
is to \textbf{support a step of computation} 
(which can be just to transfer the 
information in the case when 
the face support a counter),  
and the corresponding 
positions are called \textbf{computation 
positions}. 
\item an \textbf{horizontal} (resp. 
vertical)
\textbf{arrow} directed to the right (resp. to 
the top)
when the set of columns (resp. lines) containing 
this petal intersects larger order two-dimensional cells but not the set of lines (resp. 
columns) containing it. The associated 
function is to \textbf{transfer information} 
in the direction of the arrow (this 
information can be trivial in the case 
when the face support a counter. This means  
that the symbol transmitted is the blank 
symbol), and the corresponding positions 
are called \textbf{information transfer 
positions}.
\item when 
the two sets intersect larger order 
cells, the petal is colored \textbf{light gray}.
These positions have no function.
\end{itemize}

\section{\label{section.linear.counter} Linear counter layer}

The construction model of M. Hochman and 
T. Meyerovitch~\cite{Hochman-Meyerovitch-2010} implies 
degenerated behaviors of the Turing machines. 
For this reason, in order to preserve the minimality property, 
we use a counter which alternates all the possible 
behaviors of these machines. We describe this 
counter here. 

In 
Section~\ref{sec.notations.linear} 
we give some notations used in the construction 
of this layer. In Section~\ref{sec.incrementation.linear} we describe 
the incrementation mechanism of the linear 
counter. The rules for the information transport 
will be described later.

\subsubsection{\label{sec.notations.linear} 
Alphabet of the linear counter}

Let $l \ge 1$ be some integer, 
and $\mathcal{A}'$, $\mathcal{Q}$ and $\mathcal{D}$ some finite alphabets 
such that $|{\mathcal{A}}'|= |{\mathcal{Q}}|
= 2^{2^l}$, and $|\mathcal{D}|=2^{4.2^l-2}$.  
Denote  $\mathcal{A}_c$ the alphabet
$\mathcal{A}'  \times \mathcal{Q} ^3
\times \mathcal{D} \times \{\leftarrow,\rightarrow\} \times 
\{\texttt{on},\texttt{off}\}^2$. \bigskip

The alphabet $\mathcal{A}'$ will correspond to the alphabet of the working 
tape of the Turing machine after completing it such that is has cardinality 
equal to $2^{2^l}$ (this is possible by adding letters that interact 
trivially with the machine heads, 
and taking $l$ great enough). 
The alphabet $\mathcal{Q}$ will correspond 
to the set of states of the machine (after similar 
completion). The arrows 
will give the direction of the propagation of 
the error signal. The elements of the set 
$\{\texttt{on},\texttt{off}\} ^2$ are a coefficients 
telling which ones of the lines and columns 
are active for computation (which 
has an influence on how each 
computation positions work), 
 and the alphabet 
$\mathcal{D}$ is an artifact so that the 
cardinality of $\mathcal{A}_c$ is $2^{2^{l+3}}$, 
in order for the counter to have a period 
equal to a Fermat number.
\bigskip

Let us fix $s$ some cyclic permutation 
of the set $\mathcal{A}_c$, and 
$\vec{c}_{\texttt{max}}$ some 
element of $\mathcal{A}_c$.

\subsubsection{\label{sec.incrementation.linear}
Incrementation}

\noindent \textbf{\textit{Symbols:}} \bigskip

The elements of 
$({\mathcal{A}_c} \times \{0,1\}) \times \left\{\begin{tikzpicture}[scale=0.3] 
\fill[Salmon] (0,0) rectangle (1,1);
\draw (0,0) rectangle (1,1);
\end{tikzpicture}, \begin{tikzpicture}[scale=0.3]
\draw (0,0) rectangle (1,1);
\end{tikzpicture}\right\}$. 
The elements of $({\mathcal{A}_c} \times \{0,1\})$
are thought as the following tiles. 
The first symbol represents the south 
symbol in the tile, and the second one 
representing the west symbol in the tile: 
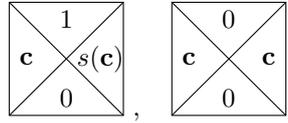
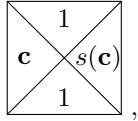
\[\begin{tikzpicture}[scale=1.5]
\draw (0,0) rectangle (1,1);
\draw (0,0) -- (1,1);
\draw (1,0) -- (0,1);
\node at (0.15,0.5) {$\vec{c}$};
\node at (0.8,0.5) {$s(\vec{c})$};
\node at (0.5,0.85) {$1$};
\node at (0.5,0.15) {$0$};
\end{tikzpicture}, \quad
\begin{tikzpicture}[scale=1.5]
\draw (0,0) rectangle (1,1);
\draw (0,0) -- (1,1);
\draw (1,0) -- (0,1);
\node at (0.15,0.5) {$\vec{c}$};
\node at (0.85,0.5) {$\vec{c}$};
\node at (0.5,0.15) {$0$};
\node at (0.5,0.85) {$0$};
\end{tikzpicture},\]
for $\vec{c} \neq \vec{c}_{\texttt{max}}$, 
and 
\[\begin{tikzpicture}[scale=1.5]
\draw (0,0) rectangle (1,1);
\draw (0,0) -- (1,1);
\draw (1,0) -- (0,1);
\node at (0.15,0.5) {$\vec{c}$};
\node at (0.8,0.5) {$s(\vec{c})$};
\node at (0.5,0.15) {$1$};
\node at (0.5,0.85) {$1$};
\end{tikzpicture},\]
for $\vec{c} = \vec{c}_{\texttt{max}}$. \bigskip

The second is called the {\bf{freezing symbol}}.
The other symbols are 
the elements of the following sets: 
$\mathcal{A}_c \times \left\{\begin{tikzpicture}[scale=0.3] 
\fill[Salmon] (0,0) rectangle (1,1);
\draw (0,0) rectangle (1,1);
\end{tikzpicture}, \begin{tikzpicture}[scale=0.3]
\draw (0,0) rectangle (1,1);
\end{tikzpicture}\right\}$, 
$\left\{\begin{tikzpicture}[scale=0.3] 
\fill[Salmon] (0,0) rectangle (1,1);
\draw (0,0) rectangle (1,1);
\end{tikzpicture}, \begin{tikzpicture}[scale=0.3]
\draw (0,0) rectangle (1,1);
\end{tikzpicture}\right\}$,$\mathcal{A}_c$, 
$\{0,1\} \times \{\begin{tikzpicture}[scale=0.3]
\fill[Salmon] (0,0) rectangle (1,1); 
\draw (0,0) rectangle (1,1); \end{tikzpicture}, 
\begin{tikzpicture}[scale=0.3] 
\draw (0,0) rectangle (1,1); \end{tikzpicture}
\}$ \bigskip

\noindent \textbf{\textit{Local rules:}} \bigskip

\begin{itemize}
\item \textbf{Localization rules:}

\begin{itemize}
\item The non-blank symbols are 
superimposed on positions of the cytoplasm in 
the 
$(\overline{2},\overline{3})$ sub-units,
in order $n \ge 4$ cells. 

\item The 
elements of $\mathcal{A}_c  \times \{0,1\} 
\times \{\begin{tikzpicture}[scale=0.3]
\fill[Salmon] (0,0) rectangle (1,1); 
\draw (0,0) rectangle (1,1); \end{tikzpicture}, 
\begin{tikzpicture}[scale=0.3] 
\draw (0,0) rectangle (1,1); \end{tikzpicture}\}$ appear on the leftmost column of the sub-unit 
having a blue symbol in the functional areas 
layer.

\item The elements of $\{0,1\} \times \{\begin{tikzpicture}[scale=0.3]
\fill[Salmon] (0,0) rectangle (1,1); 
\draw (0,0) rectangle (1,1); \end{tikzpicture}, 
\begin{tikzpicture}[scale=0.3] 
\draw (0,0) rectangle (1,1); \end{tikzpicture}\}$ appear on the other positions of the leftmost 
column. 

\item The elements $\mathcal{A}_c \times \{\begin{tikzpicture}[scale=0.3]
\fill[Salmon] (0,0) rectangle (1,1); 
\draw (0,0) rectangle (1,1); \end{tikzpicture}, 
\begin{tikzpicture}[scale=0.3] 
\draw (0,0) rectangle (1,1); \end{tikzpicture}\}$, appear on the positions of the sub-unit 
having a blue symbol or 
an horizontal arrow symbol in the functional areas 
layer and are outside the leftmost column.

\item The elements of 
$\{\begin{tikzpicture}[scale=0.3]
\fill[Salmon] (0,0) rectangle (1,1); 
\draw (0,0) rectangle (1,1); \end{tikzpicture}, 
\begin{tikzpicture}[scale=0.3] 
\draw (0,0) rectangle (1,1); \end{tikzpicture}\}$ 
are superimposed on the other positions of the 
cytoplasm. See an illustration of 
these rules on Figure~\ref{fig.loc.lin.count.face1}.
\end{itemize}

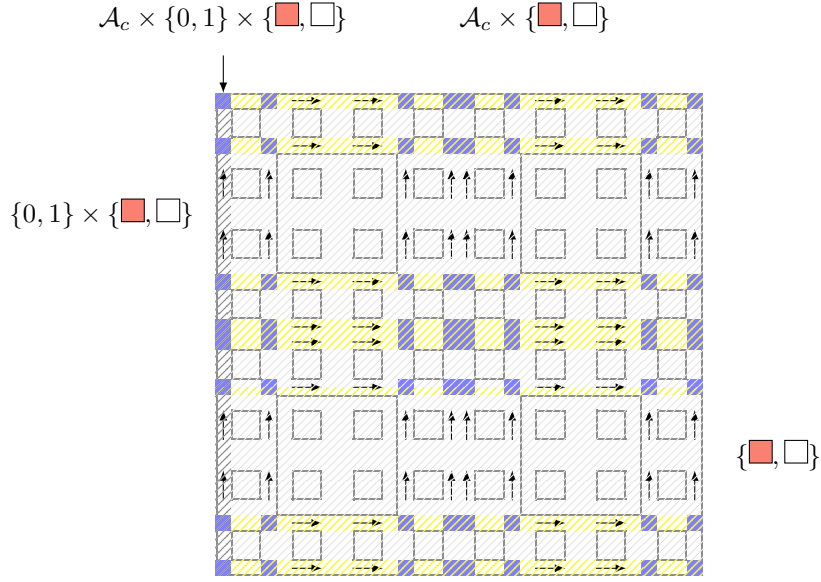
\begin{figure}[ht]
\[\begin{tikzpicture}[scale=0.05]
\fill[gray!90] (0,0) rectangle (128,128);
\fill[white] (0.5,0.5) rectangle (127.5,127.5);
\begin{scope}
\fill[gray!90] (16,16) rectangle (48,48); 
\fill[white] (16.5,16.5) rectangle (47.5,47.5);

\fill[gray!90] (4,4) rectangle (4.5,12); 
\fill[gray!90] (4,4) rectangle (12,4.5); 
\fill[gray!90] (4.5,11.5) rectangle (12,12); 
\fill[gray!90] (11.5,4.5) rectangle (12,12); 

\fill[gray!90] (4,20) rectangle (4.5,28); 
\fill[gray!90] (4,20) rectangle (12,20.5); 
\fill[gray!90] (4.5,27.5) rectangle (12,28); 
\fill[gray!90] (11.5,20.5) rectangle (12,28); 

\fill[gray!90] (4,36) rectangle (4.5,44); 
\fill[gray!90] (4,36) rectangle (12,36.5); 
\fill[gray!90] (4.5,43.5) rectangle (12,44); 
\fill[gray!90] (11.5,36.5) rectangle (12,44);

\fill[gray!90] (4,52) rectangle (4.5,60); 
\fill[gray!90] (4,52) rectangle (12,52.5); 
\fill[gray!90] (4.5,59.5) rectangle (12,60); 
\fill[gray!90] (11.5,52.5) rectangle (12,60); 

\fill[blue!50] (0,0) rectangle (4,4);
\fill[blue!50] (0,0+12) rectangle (4,4+12);
\fill[blue!50] (12,0) rectangle (15+1,4);
\fill[blue!50] (12,12) rectangle (15+1,4+12);

\fill[blue!50] (0,64) rectangle (4,64-4);
\fill[blue!50] (0,64-12) rectangle (4,64-4-12);
\fill[blue!50] (12,64) rectangle (15+1,64-4);
\fill[blue!50] (12,64-12) rectangle (15+1,64-4-12);

\fill[blue!50] (48,64) rectangle (15+1+36,64-4);
\fill[blue!50] (48,64-12) rectangle (15+1+36,64-4-12);
\fill[blue!50] (48,0) rectangle (15+1+36,4);
\fill[blue!50] (48,12) rectangle (15+1+36,4+12);


\fill[blue!50] (64,64) rectangle (60,60);
\fill[blue!50] (64,64-12) rectangle (60,60-12);

\fill[blue!50] (64,0) rectangle (60,4);
\fill[blue!50] (64,12) rectangle (60,4+12);

\fill[gray!90] (20,4) rectangle (20.5,12); 
\fill[gray!90] (20,4) rectangle (28,4.5); 
\fill[gray!90] (20.5,11.5) rectangle (28,12); 
\fill[gray!90] (27.5,4.5) rectangle (28,12); 

\fill[gray!90] (20,20) rectangle (20.5,28); 
\fill[gray!90] (20,20) rectangle (28,20.5); 
\fill[gray!90] (20.5,27.5) rectangle (28,28); 
\fill[gray!90] (27.5,20.5) rectangle (28,28); 

\fill[gray!90] (20,36) rectangle (20.5,44); 
\fill[gray!90] (20,36) rectangle (28,36.5); 
\fill[gray!90] (20.5,43.5) rectangle (28,44); 
\fill[gray!90] (27.5,36.5) rectangle (28,44);

\fill[gray!90] (20,52) rectangle (20.5,60); 
\fill[gray!90] (20,52) rectangle (28,52.5); 
\fill[gray!90] (20.5,59.5) rectangle (28,60); 
\fill[gray!90] (27.5,52.5) rectangle (28,60);

\fill[gray!90] (36,4) rectangle (36.5,12); 
\fill[gray!90] (36,4) rectangle (44,4.5); 
\fill[gray!90] (36.5,11.5) rectangle (44,12); 
\fill[gray!90] (43.5,4.5) rectangle (44,12); 

\fill[gray!90] (36,20) rectangle (36.5,28); 
\fill[gray!90] (36,20) rectangle (44,20.5); 
\fill[gray!90] (36.5,27.5) rectangle (44,28); 
\fill[gray!90] (43.5,20.5) rectangle (44,28); 

\fill[gray!90] (36,36) rectangle (36.5,44); 
\fill[gray!90] (36,36) rectangle (44,36.5); 
\fill[gray!90] (36.5,43.5) rectangle (44,44); 
\fill[gray!90] (43.5,36.5) rectangle (44,44);

\fill[gray!90] (36,52) rectangle (36.5,60); 
\fill[gray!90] (36,52) rectangle (44,52.5); 
\fill[gray!90] (36.5,59.5) rectangle (44,60); 
\fill[gray!90] (43.5,52.5) rectangle (44,60);

\fill[gray!90] (52,4) rectangle (52.5,12); 
\fill[gray!90] (52,4) rectangle (60,4.5); 
\fill[gray!90] (52.5,11.5) rectangle (60,12); 
\fill[gray!90] (59.5,4.5) rectangle (60,12); 

\fill[gray!90] (52,20) rectangle (52.5,28); 
\fill[gray!90] (52,20) rectangle (60,20.5); 
\fill[gray!90] (52.5,27.5) rectangle (60,28); 
\fill[gray!90] (59.5,20.5) rectangle (60,28); 

\fill[gray!90] (52,36) rectangle (52.5,44); 
\fill[gray!90] (52,36) rectangle (60,36.5); 
\fill[gray!90] (52.5,43.5) rectangle (60,44); 
\fill[gray!90] (59.5,36.5) rectangle (60,44);

\fill[gray!90] (52,52) rectangle (52.5,60); 
\fill[gray!90] (52,52) rectangle (60,52.5); 
\fill[gray!90] (52.5,59.5) rectangle (60,60); 
\fill[gray!90] (59.5,52.5) rectangle (60,60);

\draw[-latex] (20,62) -- (28,62);
\draw[-latex] (36,62) -- (44,62);
\draw[-latex] (20,2) -- (28,2);
\draw[-latex] (36,2) -- (44,2);

\draw[-latex] (20,50) -- (28,50);
\draw[-latex] (36,50) -- (44,50);
\draw[-latex] (20,14) -- (28,14);
\draw[-latex] (36,14) -- (44,14);

\draw[-latex] (62,20) -- (62,28);
\draw[-latex] (62,36) -- (62,44);
\draw[-latex] (2,20) -- (2,28);
\draw[-latex] (2,36) -- (2,44);

\draw[-latex] (50,20) -- (50,28);
\draw[-latex] (50,36) -- (50,44);
\draw[-latex] (14,20) -- (14,28);
\draw[-latex] (14,36) -- (14,44);
\end{scope}

\begin{scope}[yshift=64cm]
\fill[gray!90] (16,16) rectangle (48,48); 
\fill[white] (16.5,16.5) rectangle (47.5,47.5);

\fill[gray!90] (4,4) rectangle (4.5,12); 
\fill[gray!90] (4,4) rectangle (12,4.5); 
\fill[gray!90] (4.5,11.5) rectangle (12,12); 
\fill[gray!90] (11.5,4.5) rectangle (12,12); 

\fill[gray!90] (4,20) rectangle (4.5,28); 
\fill[gray!90] (4,20) rectangle (12,20.5); 
\fill[gray!90] (4.5,27.5) rectangle (12,28); 
\fill[gray!90] (11.5,20.5) rectangle (12,28);

\fill[blue!50] (48,64) rectangle (15+1+36,64-4);
\fill[blue!50] (48,64-12) rectangle (15+1+36,64-4-12);
\fill[blue!50] (48,0) rectangle (15+1+36,4);
\fill[blue!50] (48,12) rectangle (15+1+36,4+12);


\fill[gray!90] (4,36) rectangle (4.5,44); 
\fill[gray!90] (4,36) rectangle (12,36.5); 
\fill[gray!90] (4.5,43.5) rectangle (12,44); 
\fill[gray!90] (11.5,36.5) rectangle (12,44);

\fill[gray!90] (4,52) rectangle (4.5,60); 
\fill[gray!90] (4,52) rectangle (12,52.5); 
\fill[gray!90] (4.5,59.5) rectangle (12,60); 
\fill[gray!90] (11.5,52.5) rectangle (12,60); 

\fill[blue!50] (0,0) rectangle (4,4);
\fill[blue!50] (0,0+12) rectangle (4,4+12);
\fill[blue!50] (12,0) rectangle (15+1,4);
\fill[blue!50] (12,12) rectangle (15+1,4+12);

\fill[blue!50] (0,64) rectangle (4,64-4);
\fill[blue!50] (0,64-12) rectangle (4,64-4-12);
\fill[blue!50] (12,64) rectangle (15+1,64-4);
\fill[blue!50] (12,64-12) rectangle (15+1,64-4-12);

\fill[blue!50] (64,64) rectangle (60,60);
\fill[blue!50] (64,64-12) rectangle (60,60-12);

\fill[blue!50] (64,0) rectangle (60,4);
\fill[blue!50] (64,12) rectangle (60,4+12);

\fill[gray!90] (20,4) rectangle (20.5,12); 
\fill[gray!90] (20,4) rectangle (28,4.5); 
\fill[gray!90] (20.5,11.5) rectangle (28,12); 
\fill[gray!90] (27.5,4.5) rectangle (28,12); 

\fill[gray!90] (20,20) rectangle (20.5,28); 
\fill[gray!90] (20,20) rectangle (28,20.5); 
\fill[gray!90] (20.5,27.5) rectangle (28,28); 
\fill[gray!90] (27.5,20.5) rectangle (28,28); 

\fill[gray!90] (20,36) rectangle (20.5,44); 
\fill[gray!90] (20,36) rectangle (28,36.5); 
\fill[gray!90] (20.5,43.5) rectangle (28,44); 
\fill[gray!90] (27.5,36.5) rectangle (28,44);

\fill[gray!90] (20,52) rectangle (20.5,60); 
\fill[gray!90] (20,52) rectangle (28,52.5); 
\fill[gray!90] (20.5,59.5) rectangle (28,60); 
\fill[gray!90] (27.5,52.5) rectangle (28,60);

\fill[gray!90] (36,4) rectangle (36.5,12); 
\fill[gray!90] (36,4) rectangle (44,4.5); 
\fill[gray!90] (36.5,11.5) rectangle (44,12); 
\fill[gray!90] (43.5,4.5) rectangle (44,12); 

\fill[gray!90] (36,20) rectangle (36.5,28); 
\fill[gray!90] (36,20) rectangle (44,20.5); 
\fill[gray!90] (36.5,27.5) rectangle (44,28); 
\fill[gray!90] (43.5,20.5) rectangle (44,28); 

\fill[gray!90] (36,36) rectangle (36.5,44); 
\fill[gray!90] (36,36) rectangle (44,36.5); 
\fill[gray!90] (36.5,43.5) rectangle (44,44); 
\fill[gray!90] (43.5,36.5) rectangle (44,44);

\fill[gray!90] (36,52) rectangle (36.5,60); 
\fill[gray!90] (36,52) rectangle (44,52.5); 
\fill[gray!90] (36.5,59.5) rectangle (44,60); 
\fill[gray!90] (43.5,52.5) rectangle (44,60);

\fill[gray!90] (52,4) rectangle (52.5,12); 
\fill[gray!90] (52,4) rectangle (60,4.5); 
\fill[gray!90] (52.5,11.5) rectangle (60,12); 
\fill[gray!90] (59.5,4.5) rectangle (60,12); 

\fill[gray!90] (52,20) rectangle (52.5,28); 
\fill[gray!90] (52,20) rectangle (60,20.5); 
\fill[gray!90] (52.5,27.5) rectangle (60,28); 
\fill[gray!90] (59.5,20.5) rectangle (60,28); 

\fill[gray!90] (52,36) rectangle (52.5,44); 
\fill[gray!90] (52,36) rectangle (60,36.5); 
\fill[gray!90] (52.5,43.5) rectangle (60,44); 
\fill[gray!90] (59.5,36.5) rectangle (60,44);

\fill[gray!90] (52,52) rectangle (52.5,60); 
\fill[gray!90] (52,52) rectangle (60,52.5); 
\fill[gray!90] (52.5,59.5) rectangle (60,60); 
\fill[gray!90] (59.5,52.5) rectangle (60,60);

\draw[-latex] (20,62) -- (28,62);
\draw[-latex] (36,62) -- (44,62);
\draw[-latex] (20,2) -- (28,2);
\draw[-latex] (36,2) -- (44,2);

\draw[-latex] (20,50) -- (28,50);
\draw[-latex] (36,50) -- (44,50);
\draw[-latex] (20,14) -- (28,14);
\draw[-latex] (36,14) -- (44,14);

\draw[-latex] (62,20) -- (62,28);
\draw[-latex] (62,36) -- (62,44);
\draw[-latex] (2,20) -- (2,28);
\draw[-latex] (2,36) -- (2,44);

\draw[-latex] (50,20) -- (50,28);
\draw[-latex] (50,36) -- (50,44);
\draw[-latex] (14,20) -- (14,28);
\draw[-latex] (14,36) -- (14,44);
\end{scope}

\begin{scope}[xshift=64cm]
\fill[gray!90] (16,16) rectangle (48,48); 
\fill[white] (16.5,16.5) rectangle (47.5,47.5);

\fill[gray!90] (4,4) rectangle (4.5,12); 
\fill[gray!90] (4,4) rectangle (12,4.5); 
\fill[gray!90] (4.5,11.5) rectangle (12,12); 
\fill[gray!90] (11.5,4.5) rectangle (12,12); 

\fill[gray!90] (4,20) rectangle (4.5,28); 
\fill[gray!90] (4,20) rectangle (12,20.5); 
\fill[gray!90] (4.5,27.5) rectangle (12,28); 
\fill[gray!90] (11.5,20.5) rectangle (12,28);

\fill[blue!50] (48,64) rectangle (15+1+36,64-4);
\fill[blue!50] (48,64-12) rectangle (15+1+36,64-4-12);
\fill[blue!50] (48,0) rectangle (15+1+36,4);
\fill[blue!50] (48,12) rectangle (15+1+36,4+12);


\fill[gray!90] (4,36) rectangle (4.5,44); 
\fill[gray!90] (4,36) rectangle (12,36.5); 
\fill[gray!90] (4.5,43.5) rectangle (12,44); 
\fill[gray!90] (11.5,36.5) rectangle (12,44);

\fill[gray!90] (4,52) rectangle (4.5,60); 
\fill[gray!90] (4,52) rectangle (12,52.5); 
\fill[gray!90] (4.5,59.5) rectangle (12,60); 
\fill[gray!90] (11.5,52.5) rectangle (12,60); 

\fill[blue!50] (0,0) rectangle (4,4);
\fill[blue!50] (0,0+12) rectangle (4,4+12);
\fill[blue!50] (12,0) rectangle (15+1,4);
\fill[blue!50] (12,12) rectangle (15+1,4+12);

\fill[blue!50] (0,64) rectangle (4,64-4);
\fill[blue!50] (0,64-12) rectangle (4,64-4-12);
\fill[blue!50] (12,64) rectangle (15+1,64-4);
\fill[blue!50] (12,64-12) rectangle (15+1,64-4-12);

\fill[blue!50] (64,64) rectangle (60,60);
\fill[blue!50] (64,64-12) rectangle (60,60-12);

\fill[blue!50] (64,0) rectangle (60,4);
\fill[blue!50] (64,12) rectangle (60,4+12);

\fill[gray!90] (20,4) rectangle (20.5,12); 
\fill[gray!90] (20,4) rectangle (28,4.5); 
\fill[gray!90] (20.5,11.5) rectangle (28,12); 
\fill[gray!90] (27.5,4.5) rectangle (28,12); 

\fill[gray!90] (20,20) rectangle (20.5,28); 
\fill[gray!90] (20,20) rectangle (28,20.5); 
\fill[gray!90] (20.5,27.5) rectangle (28,28); 
\fill[gray!90] (27.5,20.5) rectangle (28,28); 

\fill[gray!90] (20,36) rectangle (20.5,44); 
\fill[gray!90] (20,36) rectangle (28,36.5); 
\fill[gray!90] (20.5,43.5) rectangle (28,44); 
\fill[gray!90] (27.5,36.5) rectangle (28,44);

\fill[gray!90] (20,52) rectangle (20.5,60); 
\fill[gray!90] (20,52) rectangle (28,52.5); 
\fill[gray!90] (20.5,59.5) rectangle (28,60); 
\fill[gray!90] (27.5,52.5) rectangle (28,60);

\fill[gray!90] (36,4) rectangle (36.5,12); 
\fill[gray!90] (36,4) rectangle (44,4.5); 
\fill[gray!90] (36.5,11.5) rectangle (44,12); 
\fill[gray!90] (43.5,4.5) rectangle (44,12); 

\fill[gray!90] (36,20) rectangle (36.5,28); 
\fill[gray!90] (36,20) rectangle (44,20.5); 
\fill[gray!90] (36.5,27.5) rectangle (44,28); 
\fill[gray!90] (43.5,20.5) rectangle (44,28); 

\fill[gray!90] (36,36) rectangle (36.5,44); 
\fill[gray!90] (36,36) rectangle (44,36.5); 
\fill[gray!90] (36.5,43.5) rectangle (44,44); 
\fill[gray!90] (43.5,36.5) rectangle (44,44);

\fill[gray!90] (36,52) rectangle (36.5,60); 
\fill[gray!90] (36,52) rectangle (44,52.5); 
\fill[gray!90] (36.5,59.5) rectangle (44,60); 
\fill[gray!90] (43.5,52.5) rectangle (44,60);

\fill[gray!90] (52,4) rectangle (52.5,12); 
\fill[gray!90] (52,4) rectangle (60,4.5); 
\fill[gray!90] (52.5,11.5) rectangle (60,12); 
\fill[gray!90] (59.5,4.5) rectangle (60,12); 

\fill[gray!90] (52,20) rectangle (52.5,28); 
\fill[gray!90] (52,20) rectangle (60,20.5); 
\fill[gray!90] (52.5,27.5) rectangle (60,28); 
\fill[gray!90] (59.5,20.5) rectangle (60,28); 

\fill[gray!90] (52,36) rectangle (52.5,44); 
\fill[gray!90] (52,36) rectangle (60,36.5); 
\fill[gray!90] (52.5,43.5) rectangle (60,44); 
\fill[gray!90] (59.5,36.5) rectangle (60,44);

\fill[gray!90] (52,52) rectangle (52.5,60); 
\fill[gray!90] (52,52) rectangle (60,52.5); 
\fill[gray!90] (52.5,59.5) rectangle (60,60); 
\fill[gray!90] (59.5,52.5) rectangle (60,60);

\draw[-latex] (20,62) -- (28,62);
\draw[-latex] (36,62) -- (44,62);
\draw[-latex] (20,2) -- (28,2);
\draw[-latex] (36,2) -- (44,2);

\draw[-latex] (20,50) -- (28,50);
\draw[-latex] (36,50) -- (44,50);
\draw[-latex] (20,14) -- (28,14);
\draw[-latex] (36,14) -- (44,14);

\draw[-latex] (62,20) -- (62,28);
\draw[-latex] (62,36) -- (62,44);
\draw[-latex] (2,20) -- (2,28);
\draw[-latex] (2,36) -- (2,44);

\draw[-latex] (50,20) -- (50,28);
\draw[-latex] (50,36) -- (50,44);
\draw[-latex] (14,20) -- (14,28);
\draw[-latex] (14,36) -- (14,44);
\end{scope}

\begin{scope}[xshift=64cm,yshift=64cm]
\fill[gray!90] (16,16) rectangle (48,48); 
\fill[white] (16.5,16.5) rectangle (47.5,47.5);

\fill[gray!90] (4,4) rectangle (4.5,12); 
\fill[gray!90] (4,4) rectangle (12,4.5); 
\fill[gray!90] (4.5,11.5) rectangle (12,12); 
\fill[gray!90] (11.5,4.5) rectangle (12,12); 

\fill[gray!90] (4,20) rectangle (4.5,28); 
\fill[gray!90] (4,20) rectangle (12,20.5); 
\fill[gray!90] (4.5,27.5) rectangle (12,28); 
\fill[gray!90] (11.5,20.5) rectangle (12,28); 

\fill[gray!90] (4,36) rectangle (4.5,44); 
\fill[gray!90] (4,36) rectangle (12,36.5); 
\fill[gray!90] (4.5,43.5) rectangle (12,44); 
\fill[gray!90] (11.5,36.5) rectangle (12,44);

\fill[blue!50] (48,64) rectangle (15+1+36,64-4);
\fill[blue!50] (48,64-12) rectangle (15+1+36,64-4-12);
\fill[blue!50] (48,0) rectangle (15+1+36,4);
\fill[blue!50] (48,12) rectangle (15+1+36,4+12);


\fill[gray!90] (4,52) rectangle (4.5,60); 
\fill[gray!90] (4,52) rectangle (12,52.5); 
\fill[gray!90] (4.5,59.5) rectangle (12,60); 
\fill[gray!90] (11.5,52.5) rectangle (12,60); 

\fill[blue!50] (0,0) rectangle (4,4);
\fill[blue!50] (0,0+12) rectangle (4,4+12);
\fill[blue!50] (12,0) rectangle (15+1,4);
\fill[blue!50] (12,12) rectangle (15+1,4+12);

\fill[blue!50] (0,64) rectangle (4,64-4);
\fill[blue!50] (0,64-12) rectangle (4,64-4-12);
\fill[blue!50] (12,64) rectangle (15+1,64-4);
\fill[blue!50] (12,64-12) rectangle (15+1,64-4-12);

\fill[blue!50] (64,64) rectangle (60,60);
\fill[blue!50] (64,64-12) rectangle (60,60-12);

\fill[blue!50] (64,0) rectangle (60,4);
\fill[blue!50] (64,12) rectangle (60,4+12);

\fill[gray!90] (20,4) rectangle (20.5,12); 
\fill[gray!90] (20,4) rectangle (28,4.5); 
\fill[gray!90] (20.5,11.5) rectangle (28,12); 
\fill[gray!90] (27.5,4.5) rectangle (28,12); 

\fill[gray!90] (20,20) rectangle (20.5,28); 
\fill[gray!90] (20,20) rectangle (28,20.5); 
\fill[gray!90] (20.5,27.5) rectangle (28,28); 
\fill[gray!90] (27.5,20.5) rectangle (28,28); 

\fill[gray!90] (20,36) rectangle (20.5,44); 
\fill[gray!90] (20,36) rectangle (28,36.5); 
\fill[gray!90] (20.5,43.5) rectangle (28,44); 
\fill[gray!90] (27.5,36.5) rectangle (28,44);

\fill[gray!90] (20,52) rectangle (20.5,60); 
\fill[gray!90] (20,52) rectangle (28,52.5); 
\fill[gray!90] (20.5,59.5) rectangle (28,60); 
\fill[gray!90] (27.5,52.5) rectangle (28,60);

\fill[gray!90] (36,4) rectangle (36.5,12); 
\fill[gray!90] (36,4) rectangle (44,4.5); 
\fill[gray!90] (36.5,11.5) rectangle (44,12); 
\fill[gray!90] (43.5,4.5) rectangle (44,12); 

\fill[gray!90] (36,20) rectangle (36.5,28); 
\fill[gray!90] (36,20) rectangle (44,20.5); 
\fill[gray!90] (36.5,27.5) rectangle (44,28); 
\fill[gray!90] (43.5,20.5) rectangle (44,28); 

\fill[gray!90] (36,36) rectangle (36.5,44); 
\fill[gray!90] (36,36) rectangle (44,36.5); 
\fill[gray!90] (36.5,43.5) rectangle (44,44); 
\fill[gray!90] (43.5,36.5) rectangle (44,44);

\fill[gray!90] (36,52) rectangle (36.5,60); 
\fill[gray!90] (36,52) rectangle (44,52.5); 
\fill[gray!90] (36.5,59.5) rectangle (44,60); 
\fill[gray!90] (43.5,52.5) rectangle (44,60);

\fill[gray!90] (52,4) rectangle (52.5,12); 
\fill[gray!90] (52,4) rectangle (60,4.5); 
\fill[gray!90] (52.5,11.5) rectangle (60,12); 
\fill[gray!90] (59.5,4.5) rectangle (60,12); 

\fill[gray!90] (52,20) rectangle (52.5,28); 
\fill[gray!90] (52,20) rectangle (60,20.5); 
\fill[gray!90] (52.5,27.5) rectangle (60,28); 
\fill[gray!90] (59.5,20.5) rectangle (60,28); 

\fill[gray!90] (52,36) rectangle (52.5,44); 
\fill[gray!90] (52,36) rectangle (60,36.5); 
\fill[gray!90] (52.5,43.5) rectangle (60,44); 
\fill[gray!90] (59.5,36.5) rectangle (60,44);

\fill[gray!90] (52,52) rectangle (52.5,60); 
\fill[gray!90] (52,52) rectangle (60,52.5); 
\fill[gray!90] (52.5,59.5) rectangle (60,60); 
\fill[gray!90] (59.5,52.5) rectangle (60,60);

\draw[-latex] (20,62) -- (28,62);
\draw[-latex] (36,62) -- (44,62);
\draw[-latex] (20,2) -- (28,2);
\draw[-latex] (36,2) -- (44,2);

\draw[-latex] (20,50) -- (28,50);
\draw[-latex] (36,50) -- (44,50);
\draw[-latex] (20,14) -- (28,14);
\draw[-latex] (36,14) -- (44,14);

\draw[-latex] (62,20) -- (62,28);
\draw[-latex] (62,36) -- (62,44);
\draw[-latex] (2,20) -- (2,28);
\draw[-latex] (2,36) -- (2,44);

\draw[-latex] (50,20) -- (50,28);
\draw[-latex] (50,36) -- (50,44);
\draw[-latex] (14,20) -- (14,28);
\draw[-latex] (14,36) -- (14,44);
\end{scope}

\fill[pattern=north east lines,pattern color=gray!20] (0,0) rectangle (128,128);
\fill[pattern=north east lines,pattern color=yellow] (4,0) rectangle (128,4);
\fill[pattern=north east lines,pattern color=yellow] (4,12) rectangle (128,16);
\fill[pattern=north east lines,pattern color=yellow] (4,124) rectangle (128,128);
\fill[pattern=north east lines,pattern color=yellow] (4,112) rectangle (128,116);
\fill[pattern=north east lines,pattern color=yellow] (0,68) rectangle (128,60);
\fill[pattern=north east lines,pattern color=yellow] (0,48) rectangle (128,50);
\fill[pattern=north east lines,pattern color=yellow] (0,76) rectangle (128,80);

\fill[pattern = north east lines, pattern color = gray!90] (0,0) rectangle (4,128);
\node at (2,148) {$\mathcal{A}_c  \times \{0,1\}
 \times \{\begin{tikzpicture}[scale=0.3]
\fill[Salmon] (0,0) rectangle (1,1); 
\draw (0,0) rectangle (1,1); \end{tikzpicture}, 
\begin{tikzpicture}[scale=0.3] 
\draw (0,0) rectangle (1,1); \end{tikzpicture}\}$};

\node at (84,148) {$\mathcal{A}_c \times \{\begin{tikzpicture}[scale=0.3]
\fill[Salmon] (0,0) rectangle (1,1); 
\draw (0,0) rectangle (1,1); \end{tikzpicture}, 
\begin{tikzpicture}[scale=0.3] 
\draw (0,0) rectangle (1,1); \end{tikzpicture}\}$};

\node at (148,32) {$\{\begin{tikzpicture}[scale=0.3]
\fill[Salmon] (0,0) rectangle (1,1); 
\draw (0,0) rectangle (1,1); \end{tikzpicture}, 
\begin{tikzpicture}[scale=0.3] 
\draw (0,0) rectangle (1,1); \end{tikzpicture}\}$};

\draw[-latex] (2,138) -- (2,128);
\node at (-30,96) {$\{0,1\} \times \{\begin{tikzpicture}[scale=0.3]
\fill[Salmon] (0,0) rectangle (1,1); 
\draw (0,0) rectangle (1,1); \end{tikzpicture}, 
\begin{tikzpicture}[scale=0.3] 
\draw (0,0) rectangle (1,1); \end{tikzpicture}\}$};
\end{tikzpicture}\]
\caption{\label{fig.loc.lin.count.face1} Localization 
of the linear counter symbols on face 1.}
\end{figure}

A \textbf{value} of the counter is a possible 
sequence of symbols of the 
alphabet $\mathcal{A}_c$ that can appear 
on a column, on the position having a blue 
symbol in the functional areas layer.

\item \textbf{Freezing signal:}

\begin{itemize}
\item \textbf{On the leftmost column:} 

\begin{enumerate}
\item on the leftmost upmost position, the color 
is \begin{tikzpicture}[scale=0.3] 
\fill[Salmon] (0,0) rectangle (1,1);
\draw (0,0) rectangle (1,1);
\end{tikzpicture} if and 
only if the east symbol in 
the tile is 
$\vec{c}_{\texttt{max}}$.
\item this color propagates towards bottom while the east symbol 
in the tile is $\vec{c}_{\texttt{max}}$.
When this is not true, the color becomes white. 
\item the white symbol propagates to the top.
\end{enumerate}
\item \textbf{Other positions:} 

\begin{enumerate}
\item the color part of the 
symbol propagates on the gray area 
on Figure~\ref{fig.loc.lin.count.face1}.
\item on the bottommost leftmost 
position of this area,
the color is white if the color of 
the position on left is white. 
When this color is salmon (meaning the value of the counter is maximal)
then, if the bottomost rightmost position
of the $(\overline{1},\overline{3})$ sub-unit 
is salmon, then the considered 
position is colored white. If not, 
then it is colored salmon. See Figure~\ref{fig.freezing.signal.configurations} 
for an illustration of possible freezing symbol 
configuration.
\end{enumerate}

\end{itemize} 

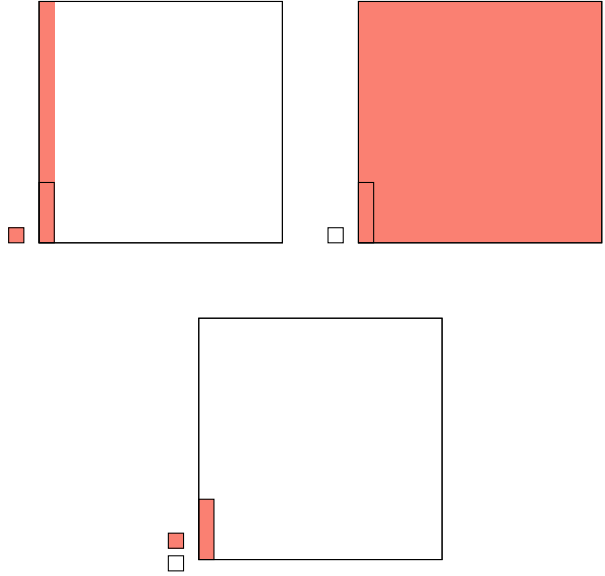
\begin{figure}[ht]
\[\begin{tikzpicture}[scale=0.05]
\begin{scope}
\fill[Salmon] (-8,0) rectangle (-4,4);
\fill[Salmon] (0,0) rectangle (4,64);
\draw (0,0) rectangle (4,16);
\draw (0,0) rectangle (64,64);
\draw (0,0) rectangle (64,64);
\draw (-8,0) rectangle (-4,4);
\end{scope}

\begin{scope}[xshift=84cm]
\fill[Salmon] (0,0) rectangle (64,64);
\draw (0,0) rectangle (4,16);
\draw (0,0) rectangle (64,64);
\draw (0,0) rectangle (64,64);
\draw (-8,0) rectangle (-4,4);
\end{scope}

\begin{scope}[xshift=42cm,yshift=-84cm]
\fill[Salmon] (-8,3) rectangle (-4,7);
\fill[Salmon] (0,0) rectangle (4,16);
\draw (0,0) rectangle (4,16);
\draw (0,0) rectangle (64,64);
\draw (0,0) rectangle (64,64);
\draw (-8,3) rectangle (-4,7);
\draw (-8,-3) rectangle (-4,1);
\end{scope}
\end{tikzpicture}\]
\caption{\label{fig.freezing.signal.configurations} Possible 
freezing symbol configurations on face $1$.}
\end{figure}

\item \textbf{Incrementation of the counter:} 

On the leftmost column 
of the area: 
\begin{itemize}
\item On the topmost position of the leftmost 
column, if the freezing signal 
of the position on the left is white, 
then the north 
symbol in the tile is $1$ (meaning that 
the counter value is incremented in the line). 
Else it is $0$ (meaning this is not incremented).
\item On a computation position of this line, 
the part in $\{0,1\}$ of symbol  
of the position on the bottom is the south 
symbol of the tile. 
The symbol on the top position is the north 
symbol of the tile. The symbol 
on the right position 
is the east symbol of the tile, 
and the symbol 
of the position on the left
is equal to the south symbol of the tile. 
\item between two computation positions, 
the symbol in $\{0,1\}$ is 
transported.
\end{itemize}

\item \textbf{Transfer of state and letter:}
on the positions with a blue symbol 
or an horizontal arrow symbol 
in the functional areas layer, 
the coefficient is transported. 
\end{itemize}

\noindent \textbf{\textit{Global behavior:}} 
\bigskip

On each of the $(\overline{2},\overline{3})$ 
sub-units of the order $\ge 4$ cells, the counter value 
is incremented on the leftmost column 
using an adding machine coded with local 
rules, except 
when the freezing signal on 
the $(\overline{1},\overline{3})$ sub-unit is
\begin{tikzpicture}[scale=0.3] 
\fill[Salmon] (0,0) rectangle (1,1);
\draw (0,0) rectangle (1,1);
\end{tikzpicture}. Then the value is 
transmitted through this sub-unit 
in direction $\vec{e}^1$. 
As a consequence of the 
information transport rules, presented 
after, the counter value is 
incremented cyclically in direction 
$\vec{e}^1$ each time going through 
a cell.
The freezing signal happens each time 
that the counter reaches 
its maximal value and stops the incrementation 
for one step, since during this step, 
the freezing signal is changed into 
\begin{tikzpicture}[scale=0.3] 
\draw (0,0) rectangle (1,1);
\end{tikzpicture}. 
Since the number of lines 
in this sub-unit is $2^{n-4}$ and the alphabet
$\mathcal{A}_c$ has cardinal ${2^{2^{l+2}}}$, 
this counter has period $2^{2^{l+n-2}}+1$.

\section{\label{section.machines} Machines layer}
In this section, we present the implementation of 
Turing machines.

The support of this layer is the 
bottom face of three-dimensional cells 
having order $qp$ for some $q \ge 0$, 
according to direction $\vec{e}^3$. 

In order 
to preserve minimality, simulate each possible 
degenerated behavior of the machines, we use 
an adaptation of 
the Turing machine model 
as follows.
The bottom line 
of the face is initialized with 
symbols in $\mathcal{A} \times \mathcal{Q}$
(we allow multiple heads). The 
sides of the face are ''initialized'' 
with elements of $\mathcal{Q}$ (we allow 
machine heads to enter at each step on 
the two sides).
As usual in this type of constructions, 
the tape is not connected. 
Between two computation positions, the information 
is transported. In our model, each 
computation position takes as 
input up to four symbols coming 
from bottom and the sides, and outputs 
up to two symbols to the top and sides.
Moreover, we add special states 
to the definition of Turing machine, 
in order to manage the presence 
of multiple machine heads.
We describe this model in 
Section~\ref{subsec.machines.min}, 
and then show how to implement 
it with local rules in Section~\ref{subsec.machines.local.rules}.

In Section~\ref{sec.computation.active.lines} 
we describe signals which 
activate or deactivate lines and columns 
of the computation areas. 
These lines and columns are used by the 
machine if and only if they are active.
These signals are determined by the 
value of the linear counter.

The machine has to take into 
account only computations starting 
on well initialized tape and no machine 
head entering during 
computation. For this purpose, 
we use error signals, described in 
Section~\ref{sec.error.signals.min}.

\subsubsection{\label{sec.computation.active.lines} Computation-active 
lines and columns}

In this section we describe the 
first sublayer.

\noindent \textbf{\textit{Symbols:}} \bigskip

Elements of $\{\texttt{on},\texttt{off}\}^2$, 
of $\{\texttt{on},\texttt{off}\}$ 
and a blank symbol. \bigskip

\noindent \textbf{\textit{Local rules:}} \bigskip

\begin{itemize}
\item \textbf{Localization rules:}

\begin{itemize}
\item 
the non-blank symbols are superimposed 
on active lines and active columns positions 
on a the bottom face according to direction 
$\vec{e}^3$, with $p$-counter equal 
to $\overline{0}$ 
and border bit equal to $1$.
\item the couples are superimposed on 
intersections of an active line 
and an active column, the simple 
symbols are superimposed on the other 
positions.
\end{itemize}
\item \textbf{Transmission rule:}
the symbol is transmitted 
along lines/columns. On 
the intersections the second symbol 
is equal to the symbol on the column. 
The first one is equal 
to the symbol on the line.
\item \textbf{Connection rule:}
Across the line 
connecting type $6,7$ (resp. 
$2,3$) face and the machine 
face, and on positions where the bottom 
line intersects with active columns, 
the symbol in $\{\texttt{on},
\texttt{off}\}$ is equal to the first (resp. second)
element of the couple in 
$\{\texttt{on},
\texttt{off}\}$ in this layer.
\end{itemize}

\noindent \textbf{\textit{Global 
behavior:}} \bigskip

On the machine face of any order $qp$ 
three-dimensional cell, the active 
columns and lines are colored 
with a symbol in $\{\texttt{on},\texttt{off}\}$
which is determined by the value of the 
counter on this cell. We call 
columns (resp.lines) colored 
with $\texttt{on}$ \textbf{computation-active}
columns (resp. lines). 

\subsubsection{\label{subsec.machines.min} Adaptation 
of computing machines 
model to minimality property}

In this section we present the 
way computing machines work in our construction. 
The model we use is adapted in order 
to have the minimality property, 
and is defined as follows: 

\begin{definition}
A \textbf{computing machine} $\mathcal{M}$
is some tuple $=(\mathcal{Q}, \mathcal{A}, 
\delta, q_0, q_e,q_s, \#)$, 
where $\mathcal{Q}$ is the state set, $\mathcal{A}$ 
the alphabet, $q_0$ the initial state, and $\#$ 
is the blank symbol, 
and $$\delta : \mathcal{A} \times \mathcal{Q} 
\rightarrow
\mathcal{A} \times \mathcal{Q}   
\times \{\leftarrow,\rightarrow,\uparrow\}.$$

The other elements $q_e,q_s$ are 
states in $\mathcal{Q}$, such that 
for all $q \in \{q_e,q_s\}$, 
and for all $a$ in $\mathcal{A}$, 
$\delta(a,q) = (a,q,\uparrow)$.

\end{definition}

The special states $q_e,q_s$ in this definition 
have the following meaning: 

\begin{itemize}
\item the error state $q_e$: a machine head 
enters this state when 
it detects an error, 
or when it collides with another 
machine head.

This state is not forbidden 
in the subshift, but this is replaced 
by the sending of an error signal, and forbidding the coexistence of the error 
signal with a well initialized tape. 
The machine stops moving when 
it enters this state.

\item shadow state $q_s$: 
because of multiple heads, we need to specify 
some state which does not act on the 
tape and does not interact with 
the other heads (acting thus as a blank 
symbol). The initial tape will 
have a head in initial state on 
the leftmost position and shadow 
states on the other ones.
\end{itemize}

Any Turing machine can be transformed 
in such a machine by adding some 
state $q_s$ verifying 
the corresponding properties 
listed above. 

Moreover, we add elements to the 
alphabet which interact 
trivially with the machine states. 
This means that for any 
added letter $a$ and any state $q$, 
$\delta(a,q) = (a,q,\uparrow)$,
and then machines states which interact 
trivially with the new alphabet, 
so that the cardinality of 
the state set and the alphabet 
are $2^{2^l}$.

When the machine is well initialized, 
none of these states and letters will be 
reached. Hence the computations 
are the ones of the initial machine.
As a consequence, one can consider that the 
machine we used has these properties.

In this section, we use a machine 
which does the following operations for 
all $n$ 
\begin{itemize}
\item write $1$ on position 
$p_n$ if $n= 2^k$ for some $k$ 
and $0$ if not.
\item write $a^{(n)}_k$ 
on positions $p_k$, $k = 1 ... n$.
\end{itemize}
The sequence $a$ is the $\Pi_1$-computable 
sequence defined at the beginning 
of the construction.
The sequence $(a^{(n)}_k)$ 
is a computable sequence such that 
for all $k$, $a_n = \inf_n a^{(n)}_k$. 
For all $n$, the position $p_n$ is 
defined to be the number of the first a
active
column from left to right which 
is just on the right of an order $n$ 
two dimensional cell on a face, amongst 
active columns.

\subsubsection{\label{subsec.machines.local.rules}Implementation 
of the machines}

In this section, we describe the second 
sublayer of this layer.

\noindent \textbf{\textit{Symbols:}}

The symbols are elements of the sets 
$\mathcal{A} \times \mathcal{Q}$,
in $\mathcal{A}$, 
$\mathcal{Q}^2$, 
and a blank symbol. \bigskip

\noindent \textbf{\textit{Local rules:}} \bigskip

\begin{itemize}
\item \textbf{Localization:}
the non-blank symbols are superimposed on 
the bottom faces of the three-dimensional cells, according to direction $\vec{e}^3$.
On these faces, they are superimposed 
on positions of 
computation-active columns and rows 
with $p$-counter value equal to $\overline{0}$ 
and border bit equal to $1$.
More precisely: 
\begin{itemize}
\item the possible symbols for 
computation active columns are elements 
of the sets $\mathcal{A}$, $\mathcal{A} \times \mathcal{Q}$
and elements of 
$\mathcal{A} \times \mathcal{Q}$ are on 
the intersection with computation-active 
rows.
\item other positions are 
superimposed with an element of 
$\mathcal{Q}^2$. See an illustration on 
Figure~\ref{fig.loc.machine}.
\end{itemize}
\item along the rows and columns, 
the symbol is transmitted while not 
on intersections of computation-active 
columns and rows.

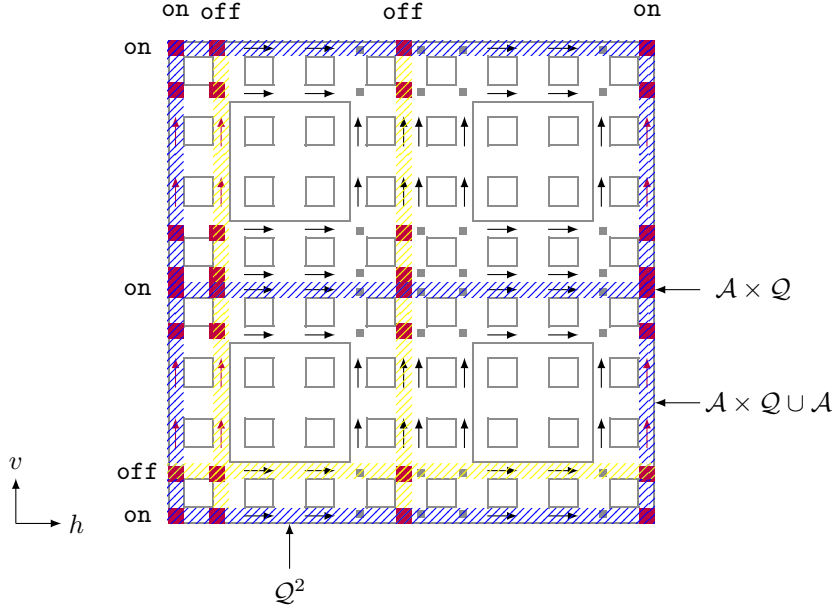
\begin{figure}[ht]
\[\begin{tikzpicture}[scale=0.05]
\fill[gray!90] (0,0) rectangle (128,128);
\fill[white] (0.5,0.5) rectangle (127.5,127.5);
\begin{scope}
\fill[gray!90] (16,16) rectangle (48,48); 
\fill[white] (16.5,16.5) rectangle (47.5,47.5);

\fill[gray!90] (4,4) rectangle (4.5,12); 
\fill[gray!90] (4,4) rectangle (12,4.5); 
\fill[gray!90] (4.5,11.5) rectangle (12,12); 
\fill[gray!90] (11.5,4.5) rectangle (12,12); 

\fill[gray!90] (4,20) rectangle (4.5,28); 
\fill[gray!90] (4,20) rectangle (12,20.5); 
\fill[gray!90] (4.5,27.5) rectangle (12,28); 
\fill[gray!90] (11.5,20.5) rectangle (12,28); 

\fill[gray!90] (4,36) rectangle (4.5,44); 
\fill[gray!90] (4,36) rectangle (12,36.5); 
\fill[gray!90] (4.5,43.5) rectangle (12,44); 
\fill[gray!90] (11.5,36.5) rectangle (12,44);

\fill[gray!90] (4,52) rectangle (4.5,60); 
\fill[gray!90] (4,52) rectangle (12,52.5); 
\fill[gray!90] (4.5,59.5) rectangle (12,60); 
\fill[gray!90] (11.5,52.5) rectangle (12,60); 

\fill[purple] (0,0) rectangle (4,4);
\fill[purple] (0,0+11) rectangle (4,4+11);
\fill[purple] (11,0) rectangle (15,4);
\fill[purple] (11,11) rectangle (15,4+11);

\fill[purple] (0,64) rectangle (4,64-4);
\fill[purple] (0,64-11) rectangle (4,64-4-11);
\fill[purple] (11,64) rectangle (15,64-4);
\fill[purple] (11,64-11) rectangle (15,64-4-11);

\fill[purple] (64,64) rectangle (60,60);
\fill[purple] (64,64-11) rectangle (60,60-11);
\fill[gray!90] (64-12.5,64-1.5) rectangle (64-14.5,64-3.5);
\fill[gray!90] (64-12.5,64-1.5-11) rectangle (64-14.5,64-3.5-11);

\fill[purple] (64,0) rectangle (60,4);
\fill[purple] (64,11) rectangle (60,4+11);
\fill[gray!90] (64-12.5,1.5) rectangle (64-14.5,3.5);
\fill[gray!90] (64-12.5,1.5+11) rectangle (64-14.5,3.5+11);

\fill[gray!90] (20,4) rectangle (20.5,12); 
\fill[gray!90] (20,4) rectangle (28,4.5); 
\fill[gray!90] (20.5,11.5) rectangle (28,12); 
\fill[gray!90] (27.5,4.5) rectangle (28,12); 

\fill[gray!90] (20,20) rectangle (20.5,28); 
\fill[gray!90] (20,20) rectangle (28,20.5); 
\fill[gray!90] (20.5,27.5) rectangle (28,28); 
\fill[gray!90] (27.5,20.5) rectangle (28,28); 

\fill[gray!90] (20,36) rectangle (20.5,44); 
\fill[gray!90] (20,36) rectangle (28,36.5); 
\fill[gray!90] (20.5,43.5) rectangle (28,44); 
\fill[gray!90] (27.5,36.5) rectangle (28,44);

\fill[gray!90] (20,52) rectangle (20.5,60); 
\fill[gray!90] (20,52) rectangle (28,52.5); 
\fill[gray!90] (20.5,59.5) rectangle (28,60); 
\fill[gray!90] (27.5,52.5) rectangle (28,60);

\fill[gray!90] (36,4) rectangle (36.5,12); 
\fill[gray!90] (36,4) rectangle (44,4.5); 
\fill[gray!90] (36.5,11.5) rectangle (44,12); 
\fill[gray!90] (43.5,4.5) rectangle (44,12); 

\fill[gray!90] (36,20) rectangle (36.5,28); 
\fill[gray!90] (36,20) rectangle (44,20.5); 
\fill[gray!90] (36.5,27.5) rectangle (44,28); 
\fill[gray!90] (43.5,20.5) rectangle (44,28); 

\fill[gray!90] (36,36) rectangle (36.5,44); 
\fill[gray!90] (36,36) rectangle (44,36.5); 
\fill[gray!90] (36.5,43.5) rectangle (44,44); 
\fill[gray!90] (43.5,36.5) rectangle (44,44);

\fill[gray!90] (36,52) rectangle (36.5,60); 
\fill[gray!90] (36,52) rectangle (44,52.5); 
\fill[gray!90] (36.5,59.5) rectangle (44,60); 
\fill[gray!90] (43.5,52.5) rectangle (44,60);

\fill[gray!90] (52,4) rectangle (52.5,12); 
\fill[gray!90] (52,4) rectangle (60,4.5); 
\fill[gray!90] (52.5,11.5) rectangle (60,12); 
\fill[gray!90] (59.5,4.5) rectangle (60,12); 

\fill[gray!90] (52,20) rectangle (52.5,28); 
\fill[gray!90] (52,20) rectangle (60,20.5); 
\fill[gray!90] (52.5,27.5) rectangle (60,28); 
\fill[gray!90] (59.5,20.5) rectangle (60,28); 

\fill[gray!90] (52,36) rectangle (52.5,44); 
\fill[gray!90] (52,36) rectangle (60,36.5); 
\fill[gray!90] (52.5,43.5) rectangle (60,44); 
\fill[gray!90] (59.5,36.5) rectangle (60,44);

\fill[gray!90] (52,52) rectangle (52.5,60); 
\fill[gray!90] (52,52) rectangle (60,52.5); 
\fill[gray!90] (52.5,59.5) rectangle (60,60); 
\fill[gray!90] (59.5,52.5) rectangle (60,60);

\draw[-latex] (20,62) -- (28,62);
\draw[-latex] (36,62) -- (44,62);
\draw[-latex] (20,2) -- (28,2);
\draw[-latex] (36,2) -- (44,2);

\draw[-latex] (20,50) -- (28,50);
\draw[-latex] (36,50) -- (44,50);
\draw[-latex] (20,14) -- (28,14);
\draw[-latex] (36,14) -- (44,14);

\draw[-latex] (62,20) -- (62,28);
\draw[-latex] (62,36) -- (62,44);
\draw[-latex,color=purple] (2,20) -- (2,28);
\draw[-latex,color=purple] (2,36) -- (2,44);

\draw[-latex] (50,20) -- (50,28);
\draw[-latex] (50,36) -- (50,44);
\draw[-latex,color=purple] (14,20) -- (14,28);
\draw[-latex,color=purple] (14,36) -- (14,44);

\draw[-latex] (-40,0) -- (-28,0);
\draw[-latex] (-40,0) -- (-40,12);
\node at (-40,16) {$v$};
\node at (-24,0) {$h$};

\node at (-8,2) {$\texttt{on}$};
\node at (-8,14) {$\texttt{off}$};
\node at (-8,62) {$\texttt{on}$};
\node at (-8,126) {$\texttt{on}$};
\node at (2,136) {$\texttt{on}$};
\node at (14,136) {$\texttt{off}$};
\node at (62,136) {$\texttt{off}$};
\node at (126,136) {$\texttt{on}$};
\end{scope}

\begin{scope}[yshift=64cm]
\fill[gray!90] (16,16) rectangle (48,48); 
\fill[white] (16.5,16.5) rectangle (47.5,47.5);

\fill[gray!90] (4,4) rectangle (4.5,12); 
\fill[gray!90] (4,4) rectangle (12,4.5); 
\fill[gray!90] (4.5,11.5) rectangle (12,12); 
\fill[gray!90] (11.5,4.5) rectangle (12,12); 

\fill[gray!90] (4,20) rectangle (4.5,28); 
\fill[gray!90] (4,20) rectangle (12,20.5); 
\fill[gray!90] (4.5,27.5) rectangle (12,28); 
\fill[gray!90] (11.5,20.5) rectangle (12,28); 

\fill[gray!90] (4,36) rectangle (4.5,44); 
\fill[gray!90] (4,36) rectangle (12,36.5); 
\fill[gray!90] (4.5,43.5) rectangle (12,44); 
\fill[gray!90] (11.5,36.5) rectangle (12,44);

\fill[gray!90] (4,52) rectangle (4.5,60); 
\fill[gray!90] (4,52) rectangle (12,52.5); 
\fill[gray!90] (4.5,59.5) rectangle (12,60); 
\fill[gray!90] (11.5,52.5) rectangle (12,60); 

\fill[purple] (0,0) rectangle (4,4);
\fill[purple] (0,0+11) rectangle (4,4+11);
\fill[purple] (11,0) rectangle (15,4);
\fill[purple] (11,11) rectangle (15,4+11);

\fill[purple] (0,64) rectangle (4,64-4);
\fill[purple] (0,64-11) rectangle (4,64-4-11);
\fill[purple] (11,64) rectangle (15,64-4);
\fill[purple] (11,64-11) rectangle (15,64-4-11);

\fill[purple] (64,64) rectangle (60,60);
\fill[purple] (64,64-11) rectangle (60,60-11);
\fill[gray!90] (64-12.5,64-1.5) rectangle (64-14.5,64-3.5);
\fill[gray!90] (64-12.5,64-1.5-11) rectangle (64-14.5,64-3.5-11);

\fill[purple] (64,0) rectangle (60,4);
\fill[purple] (64,11) rectangle (60,4+11);
\fill[gray!90] (64-12.5,1.5) rectangle (64-14.5,3.5);
\fill[gray!90] (64-12.5,1.5+11) rectangle (64-14.5,3.5+11);

\fill[gray!90] (20,4) rectangle (20.5,12); 
\fill[gray!90] (20,4) rectangle (28,4.5); 
\fill[gray!90] (20.5,11.5) rectangle (28,12); 
\fill[gray!90] (27.5,4.5) rectangle (28,12); 

\fill[gray!90] (20,20) rectangle (20.5,28); 
\fill[gray!90] (20,20) rectangle (28,20.5); 
\fill[gray!90] (20.5,27.5) rectangle (28,28); 
\fill[gray!90] (27.5,20.5) rectangle (28,28); 

\fill[gray!90] (20,36) rectangle (20.5,44); 
\fill[gray!90] (20,36) rectangle (28,36.5); 
\fill[gray!90] (20.5,43.5) rectangle (28,44); 
\fill[gray!90] (27.5,36.5) rectangle (28,44);

\fill[gray!90] (20,52) rectangle (20.5,60); 
\fill[gray!90] (20,52) rectangle (28,52.5); 
\fill[gray!90] (20.5,59.5) rectangle (28,60); 
\fill[gray!90] (27.5,52.5) rectangle (28,60);

\fill[gray!90] (36,4) rectangle (36.5,12); 
\fill[gray!90] (36,4) rectangle (44,4.5); 
\fill[gray!90] (36.5,11.5) rectangle (44,12); 
\fill[gray!90] (43.5,4.5) rectangle (44,12); 

\fill[gray!90] (36,20) rectangle (36.5,28); 
\fill[gray!90] (36,20) rectangle (44,20.5); 
\fill[gray!90] (36.5,27.5) rectangle (44,28); 
\fill[gray!90] (43.5,20.5) rectangle (44,28); 

\fill[gray!90] (36,36) rectangle (36.5,44); 
\fill[gray!90] (36,36) rectangle (44,36.5); 
\fill[gray!90] (36.5,43.5) rectangle (44,44); 
\fill[gray!90] (43.5,36.5) rectangle (44,44);

\fill[gray!90] (36,52) rectangle (36.5,60); 
\fill[gray!90] (36,52) rectangle (44,52.5); 
\fill[gray!90] (36.5,59.5) rectangle (44,60); 
\fill[gray!90] (43.5,52.5) rectangle (44,60);

\fill[gray!90] (52,4) rectangle (52.5,12); 
\fill[gray!90] (52,4) rectangle (60,4.5); 
\fill[gray!90] (52.5,11.5) rectangle (60,12); 
\fill[gray!90] (59.5,4.5) rectangle (60,12); 

\fill[gray!90] (52,20) rectangle (52.5,28); 
\fill[gray!90] (52,20) rectangle (60,20.5); 
\fill[gray!90] (52.5,27.5) rectangle (60,28); 
\fill[gray!90] (59.5,20.5) rectangle (60,28); 

\fill[gray!90] (52,36) rectangle (52.5,44); 
\fill[gray!90] (52,36) rectangle (60,36.5); 
\fill[gray!90] (52.5,43.5) rectangle (60,44); 
\fill[gray!90] (59.5,36.5) rectangle (60,44);

\fill[gray!90] (52,52) rectangle (52.5,60); 
\fill[gray!90] (52,52) rectangle (60,52.5); 
\fill[gray!90] (52.5,59.5) rectangle (60,60); 
\fill[gray!90] (59.5,52.5) rectangle (60,60);

\draw[-latex] (20,62) -- (28,62);
\draw[-latex] (36,62) -- (44,62);
\draw[-latex] (20,2) -- (28,2);
\draw[-latex] (36,2) -- (44,2);

\draw[-latex] (20,50) -- (28,50);
\draw[-latex] (36,50) -- (44,50);
\draw[-latex] (20,14) -- (28,14);
\draw[-latex] (36,14) -- (44,14);

\draw[-latex] (62,20) -- (62,28);
\draw[-latex] (62,36) -- (62,44);
\draw[-latex,color=purple] (2,20) -- (2,28);
\draw[-latex,color=purple] (2,36) -- (2,44);

\draw[-latex] (50,20) -- (50,28);
\draw[-latex] (50,36) -- (50,44);
\draw[-latex,color=purple] (14,20) -- (14,28);
\draw[-latex,color=purple] (14,36) -- (14,44);
\end{scope}

\begin{scope}[xshift=64cm]
\fill[gray!90] (16,16) rectangle (48,48); 
\fill[white] (16.5,16.5) rectangle (47.5,47.5);

\fill[gray!90] (4,4) rectangle (4.5,12); 
\fill[gray!90] (4,4) rectangle (12,4.5); 
\fill[gray!90] (4.5,11.5) rectangle (12,12); 
\fill[gray!90] (11.5,4.5) rectangle (12,12); 

\fill[gray!90] (4,20) rectangle (4.5,28); 
\fill[gray!90] (4,20) rectangle (12,20.5); 
\fill[gray!90] (4.5,27.5) rectangle (12,28); 
\fill[gray!90] (11.5,20.5) rectangle (12,28); 

\fill[gray!90] (4,36) rectangle (4.5,44); 
\fill[gray!90] (4,36) rectangle (12,36.5); 
\fill[gray!90] (4.5,43.5) rectangle (12,44); 
\fill[gray!90] (11.5,36.5) rectangle (12,44);

\fill[gray!90] (4,52) rectangle (4.5,60); 
\fill[gray!90] (4,52) rectangle (12,52.5); 
\fill[gray!90] (4.5,59.5) rectangle (12,60); 
\fill[gray!90] (11.5,52.5) rectangle (12,60); 

\fill[gray!90] (1.5,1.5) rectangle (3.5,3.5);
\fill[gray!90] (1.5,1.5+11) rectangle (3.5,3.5+11);
\fill[gray!90] (12.5,1.5) rectangle (14.5,3.5);
\fill[gray!90] (12.5,1.5+11) rectangle (14.5,3.5+11);

\fill[gray!90] (1.5,64-1.5) rectangle (3.5,64-3.5);
\fill[gray!90] (1.5,64-1.5-11) rectangle (3.5,64-3.5-11);
\fill[gray!90] (12.5,64-1.5) rectangle (14.5,64-3.5);
\fill[gray!90] (12.5,64-1.5-11) rectangle (14.5,64-3.5-11);

\fill[purple] (64,64) rectangle (64-4,64-4);
\fill[purple] (64,64-11) rectangle (64-4,64-4-11);
\fill[gray!90] (64-12.5,64-1.5) rectangle (64-14.5,64-3.5);
\fill[gray!90] (64-12.5,64-1.5-11) rectangle (64-14.5,64-3.5-11);

\fill[purple] (64,0) rectangle (64-4,4);
\fill[purple] (64,+11) rectangle (64-4,4+11);
\fill[gray!90] (64-12.5,1.5) rectangle (64-14.5,3.5);
\fill[gray!90] (64-12.5,1.5+11) rectangle (64-14.5,3.5+11);

\fill[gray!90] (20,4) rectangle (20.5,12); 
\fill[gray!90] (20,4) rectangle (28,4.5); 
\fill[gray!90] (20.5,11.5) rectangle (28,12); 
\fill[gray!90] (27.5,4.5) rectangle (28,12); 

\fill[gray!90] (20,20) rectangle (20.5,28); 
\fill[gray!90] (20,20) rectangle (28,20.5); 
\fill[gray!90] (20.5,27.5) rectangle (28,28); 
\fill[gray!90] (27.5,20.5) rectangle (28,28); 

\fill[gray!90] (20,36) rectangle (20.5,44); 
\fill[gray!90] (20,36) rectangle (28,36.5); 
\fill[gray!90] (20.5,43.5) rectangle (28,44); 
\fill[gray!90] (27.5,36.5) rectangle (28,44);

\fill[gray!90] (20,52) rectangle (20.5,60); 
\fill[gray!90] (20,52) rectangle (28,52.5); 
\fill[gray!90] (20.5,59.5) rectangle (28,60); 
\fill[gray!90] (27.5,52.5) rectangle (28,60);

\fill[gray!90] (36,4) rectangle (36.5,12); 
\fill[gray!90] (36,4) rectangle (44,4.5); 
\fill[gray!90] (36.5,11.5) rectangle (44,12); 
\fill[gray!90] (43.5,4.5) rectangle (44,12); 

\fill[gray!90] (36,20) rectangle (36.5,28); 
\fill[gray!90] (36,20) rectangle (44,20.5); 
\fill[gray!90] (36.5,27.5) rectangle (44,28); 
\fill[gray!90] (43.5,20.5) rectangle (44,28); 

\fill[gray!90] (36,36) rectangle (36.5,44); 
\fill[gray!90] (36,36) rectangle (44,36.5); 
\fill[gray!90] (36.5,43.5) rectangle (44,44); 
\fill[gray!90] (43.5,36.5) rectangle (44,44);

\fill[gray!90] (36,52) rectangle (36.5,60); 
\fill[gray!90] (36,52) rectangle (44,52.5); 
\fill[gray!90] (36.5,59.5) rectangle (44,60); 
\fill[gray!90] (43.5,52.5) rectangle (44,60);

\fill[gray!90] (52,4) rectangle (52.5,12); 
\fill[gray!90] (52,4) rectangle (60,4.5); 
\fill[gray!90] (52.5,11.5) rectangle (60,12); 
\fill[gray!90] (59.5,4.5) rectangle (60,12); 

\fill[gray!90] (52,20) rectangle (52.5,28); 
\fill[gray!90] (52,20) rectangle (60,20.5); 
\fill[gray!90] (52.5,27.5) rectangle (60,28); 
\fill[gray!90] (59.5,20.5) rectangle (60,28); 

\fill[gray!90] (52,36) rectangle (52.5,44); 
\fill[gray!90] (52,36) rectangle (60,36.5); 
\fill[gray!90] (52.5,43.5) rectangle (60,44); 
\fill[gray!90] (59.5,36.5) rectangle (60,44);

\fill[gray!90] (52,52) rectangle (52.5,60); 
\fill[gray!90] (52,52) rectangle (60,52.5); 
\fill[gray!90] (52.5,59.5) rectangle (60,60); 
\fill[gray!90] (59.5,52.5) rectangle (60,60);

\draw[-latex] (20,62) -- (28,62);
\draw[-latex] (36,62) -- (44,62);
\draw[-latex] (20,2) -- (28,2);
\draw[-latex] (36,2) -- (44,2);

\draw[-latex] (20,50) -- (28,50);
\draw[-latex] (36,50) -- (44,50);
\draw[-latex] (20,14) -- (28,14);
\draw[-latex] (36,14) -- (44,14);

\draw[-latex,color=purple] (62,20) -- (62,28);
\draw[-latex,color=purple] (62,36) -- (62,44);
\draw[-latex] (2,20) -- (2,28);
\draw[-latex] (2,36) -- (2,44);

\draw[-latex] (50,20) -- (50,28);
\draw[-latex] (50,36) -- (50,44);
\draw[-latex] (14,20) -- (14,28);
\draw[-latex] (14,36) -- (14,44);
\end{scope}

\begin{scope}[xshift=64cm,yshift=64cm]
\fill[gray!90] (16,16) rectangle (48,48); 
\fill[white] (16.5,16.5) rectangle (47.5,47.5);

\fill[gray!90] (4,4) rectangle (4.5,12); 
\fill[gray!90] (4,4) rectangle (12,4.5); 
\fill[gray!90] (4.5,11.5) rectangle (12,12); 
\fill[gray!90] (11.5,4.5) rectangle (12,12); 

\fill[gray!90] (4,20) rectangle (4.5,28); 
\fill[gray!90] (4,20) rectangle (12,20.5); 
\fill[gray!90] (4.5,27.5) rectangle (12,28); 
\fill[gray!90] (11.5,20.5) rectangle (12,28); 

\fill[gray!90] (4,36) rectangle (4.5,44); 
\fill[gray!90] (4,36) rectangle (12,36.5); 
\fill[gray!90] (4.5,43.5) rectangle (12,44); 
\fill[gray!90] (11.5,36.5) rectangle (12,44);

\fill[gray!90] (4,52) rectangle (4.5,60); 
\fill[gray!90] (4,52) rectangle (12,52.5); 
\fill[gray!90] (4.5,59.5) rectangle (12,60); 
\fill[gray!90] (11.5,52.5) rectangle (12,60); 

\fill[gray!90] (1.5,1.5) rectangle (3.5,3.5);
\fill[gray!90] (1.5,1.5+11) rectangle (3.5,3.5+11);
\fill[gray!90] (12.5,1.5) rectangle (14.5,3.5);
\fill[gray!90] (12.5,1.5+11) rectangle (14.5,3.5+11);

\fill[gray!90] (1.5,64-1.5) rectangle (3.5,64-3.5);
\fill[gray!90] (1.5,64-1.5-11) rectangle (3.5,64-3.5-11);
\fill[gray!90] (12.5,64-1.5) rectangle (14.5,64-3.5);
\fill[gray!90] (12.5,64-1.5-11) rectangle (14.5,64-3.5-11);

\fill[purple] (64,64) rectangle (64-4,64-4);
\fill[purple] (64,64-11) rectangle (64-4,64-4-11);
\fill[gray!90] (64-12.5,64-1.5) rectangle (64-14.5,64-3.5);
\fill[gray!90] (64-12.5,64-1.5-11) rectangle (64-14.5,64-3.5-11);

\fill[purple] (64,0) rectangle (64-4,4);
\fill[purple] (64,11) rectangle (64-4,4+11);
\fill[gray!90] (64-12.5,1.5) rectangle (64-14.5,3.5);
\fill[gray!90] (64-12.5,1.5+11) rectangle (64-14.5,3.5+11);

\fill[gray!90] (20,4) rectangle (20.5,12); 
\fill[gray!90] (20,4) rectangle (28,4.5); 
\fill[gray!90] (20.5,11.5) rectangle (28,12); 
\fill[gray!90] (27.5,4.5) rectangle (28,12); 

\fill[gray!90] (20,20) rectangle (20.5,28); 
\fill[gray!90] (20,20) rectangle (28,20.5); 
\fill[gray!90] (20.5,27.5) rectangle (28,28); 
\fill[gray!90] (27.5,20.5) rectangle (28,28); 

\fill[gray!90] (20,36) rectangle (20.5,44); 
\fill[gray!90] (20,36) rectangle (28,36.5); 
\fill[gray!90] (20.5,43.5) rectangle (28,44); 
\fill[gray!90] (27.5,36.5) rectangle (28,44);

\fill[gray!90] (20,52) rectangle (20.5,60); 
\fill[gray!90] (20,52) rectangle (28,52.5); 
\fill[gray!90] (20.5,59.5) rectangle (28,60); 
\fill[gray!90] (27.5,52.5) rectangle (28,60);

\fill[gray!90] (36,4) rectangle (36.5,12); 
\fill[gray!90] (36,4) rectangle (44,4.5); 
\fill[gray!90] (36.5,11.5) rectangle (44,12); 
\fill[gray!90] (43.5,4.5) rectangle (44,12); 

\fill[gray!90] (36,20) rectangle (36.5,28); 
\fill[gray!90] (36,20) rectangle (44,20.5); 
\fill[gray!90] (36.5,27.5) rectangle (44,28); 
\fill[gray!90] (43.5,20.5) rectangle (44,28); 

\fill[gray!90] (36,36) rectangle (36.5,44); 
\fill[gray!90] (36,36) rectangle (44,36.5); 
\fill[gray!90] (36.5,43.5) rectangle (44,44); 
\fill[gray!90] (43.5,36.5) rectangle (44,44);

\fill[gray!90] (36,52) rectangle (36.5,60); 
\fill[gray!90] (36,52) rectangle (44,52.5); 
\fill[gray!90] (36.5,59.5) rectangle (44,60); 
\fill[gray!90] (43.5,52.5) rectangle (44,60);

\fill[gray!90] (52,4) rectangle (52.5,12); 
\fill[gray!90] (52,4) rectangle (60,4.5); 
\fill[gray!90] (52.5,11.5) rectangle (60,12); 
\fill[gray!90] (59.5,4.5) rectangle (60,12); 

\fill[gray!90] (52,20) rectangle (52.5,28); 
\fill[gray!90] (52,20) rectangle (60,20.5); 
\fill[gray!90] (52.5,27.5) rectangle (60,28); 
\fill[gray!90] (59.5,20.5) rectangle (60,28); 

\fill[gray!90] (52,36) rectangle (52.5,44); 
\fill[gray!90] (52,36) rectangle (60,36.5); 
\fill[gray!90] (52.5,43.5) rectangle (60,44); 
\fill[gray!90] (59.5,36.5) rectangle (60,44);

\fill[gray!90] (52,52) rectangle (52.5,60); 
\fill[gray!90] (52,52) rectangle (60,52.5); 
\fill[gray!90] (52.5,59.5) rectangle (60,60); 
\fill[gray!90] (59.5,52.5) rectangle (60,60);

\draw[-latex] (20,62) -- (28,62);
\draw[-latex] (36,62) -- (44,62);
\draw[-latex] (20,2) -- (28,2);
\draw[-latex] (36,2) -- (44,2);

\draw[-latex] (20,50) -- (28,50);
\draw[-latex] (36,50) -- (44,50);
\draw[-latex] (20,14) -- (28,14);
\draw[-latex] (36,14) -- (44,14);

\draw[-latex,color=purple] (62,20) -- (62,28);
\draw[-latex,color=purple] (62,36) -- (62,44);
\draw[-latex] (2,20) -- (2,28);
\draw[-latex] (2,36) -- (2,44);

\draw[-latex] (50,20) -- (50,28);
\draw[-latex] (50,36) -- (50,44);
\draw[-latex] (14,20) -- (14,28);
\draw[-latex] (14,36) -- (14,44);
\end{scope}

\fill[pattern=north east lines,pattern color=blue] (0,4) rectangle (4,128);
\fill[pattern=north east lines,pattern color=yellow] (12,4) rectangle (16,128);
\fill[pattern=north east lines,pattern color=blue] (124,4) rectangle (128,128);
\fill[pattern=north east lines,pattern color=yellow] (64,0) rectangle (60,128);

\fill[pattern = north east lines, pattern color = blue] (0,0) rectangle (128,4);
\fill[pattern = north east lines, pattern color =yellow] (0,12) rectangle (128,16);
\fill[pattern = north east lines, pattern color = blue] (0,60) rectangle (128,64);
\fill[pattern = north east lines, pattern color = blue] (0,124) rectangle (128,128);

\node at (154,62) {$\mathcal{A} \times \mathcal{Q}$};
\draw[-latex] (140,62) -- (128,62);
\draw[-latex] (140,32) -- (128,32);
\node at (158,32) {$
\mathcal{A} \times \mathcal{Q} \cup \mathcal{A}$};

\node at (32,-18) {$\mathcal{Q}^2 $};
\draw[-latex] (32,-12) -- (32,0);
\end{tikzpicture}\]
\caption{\label{fig.loc.machine} Localization 
of the machine symbols on the bottom faces 
of the cubes, according to the direction $\vec{e}^3$. Blue columns (resp. rows) symbolize 
computation-active columns (resp. rows).}
\end{figure}

\item 
\textbf{Connection with the counter:} 
On active computation positions that 
are in the bottom
line of this area, 
the symbols 
are equal to the corresponding 
subsymbol in $\mathcal{A} \times 
\mathcal{Q}$ on face 2 of the 
linear counter. On the leftmost 
(resp. rightmost) column of the area that 
are in a computation-active line, the symbol 
is equal to the corresponding 
symbol on face 7 (resp. 6) of the 
linear counter.

\item \textbf{Computation positions rules:}

Consider 
some computation position 
which is the intersection of a 
computation-active row and a 
computation-active column.

For such a position, 
the \textbf{inputs} include: 

\begin{enumerate}
\item the symbols written 
on the south position (or 
on the corresponding position on 
face $2$ when on the bottom line), 
\item 
the first symbol written on the west position 
(or the symbol on the corresponding 
position on face $7$ when on the west border
of the machine face), 
\item and the second symbol 
on the east position (or the 
symbol on the corresponding 
position on face $6$ when on the west border
of the machine face). 
\end{enumerate} 

The \textbf{outputs}
include: 

\begin{enumerate}
\item the symbols written on the 
north position (when not in 
the topmost row), 
\item the second symbol of the west position (when 
not in the leftmost column), 
\item and the first symbol on 
the east position (when not on the 
rightmost column).
\end{enumerate}

Moreover, on the bottom line, the inputs 
from inside the area are always the 
shadow state $q_s$.

See Figure~\ref{fig.schema.inputs.outputs} 
for an illustration.  

\begin{figure}[ht]
\[\begin{tikzpicture}[scale=0.2]
\begin{scope}
\fill[purple] (-1,-1) rectangle (1,1);
\draw (-1,-1) rectangle (1,1);
\fill[pattern = north east lines, pattern color = blue] 
(-9,-1) rectangle (9,1);
\fill[pattern = north east lines, pattern color = blue] 
(-1,-9) rectangle (1,9);
\draw[-latex] (-9,-0.4) -- (-1,-0.4);
\draw[-latex,dashed] (-1,0.4) -- (-9,0.4);
\draw[-latex] (9,-0.4) -- (1,-0.4);
\draw[-latex,dashed] (1,0.4) -- (9,0.4);
\draw[-latex] (0,-9) -- (0,-1);
\draw[-latex] (0,1) -- (0,9);

\node[scale=1.5] at (-11,0) {$1$};
\end{scope}

\begin{scope}[xshift=30cm]
\fill[purple] (-1,-1) rectangle (1,1);
\draw (-1,-1) rectangle (1,1);
\fill[pattern = north east lines, pattern color = blue] 
(0,-1) rectangle (9,1);
\fill[pattern = north east lines, pattern color = blue] 
(-1,-9) rectangle (1,9);
\draw[-latex] (-2,0) -- (-1,-0);
\node at (-5,0) {face 7};
\draw[-latex] (9,-0.4) -- (1,-0.4);
\draw[-latex,dashed] (1,0.4) -- (9,0.4);
\draw[-latex] (0,-9) -- (0,-1);
\draw[-latex,dashed] (0,1) -- (0,9);

\node[scale=1.5] at (-11,0) {$2$};

\node at (16,3) {Outputs};
\node at (16,6) {Inputs};
\draw (10,3) -- (12,3); 
\draw[dashed] (10,6) -- (12,6);
\end{scope}

\begin{scope}[yshift=-25cm]
\fill[purple] (-1,-1) rectangle (1,1);
\draw (-1,-1) rectangle (1,1);
\fill[pattern = north east lines, pattern color = blue] 
(0,-1) rectangle (9,1);
\fill[pattern = north east lines, pattern color = blue] 
(-1,-1) rectangle (1,9);
\draw[-latex] (-2,0) -- (-1,-0);
\node at (-5,0) {face 7};
\draw[-latex] (9,-0.4) -- (1,-0.4);
\draw[-latex,dashed] (1,0.4) -- (9,0.4);
\draw[-latex] (0,-2) -- (0,-1);
\node at (0,-3) {face 2};
\draw[-latex,dashed] (0,1) -- (0,9);

\node[scale=1.5] at (-11,0) {$3$};
\end{scope}

\begin{scope}[yshift=-25cm,xshift=30cm]
\fill[purple] (-1,-1) rectangle (1,1);
\draw (-1,-1) rectangle (1,1);
\fill[pattern = north east lines, pattern color = blue] 
(-9,-1) rectangle (9,1);
\fill[pattern = north east lines, pattern color = blue] 
(-1,-1) rectangle (1,9);
\draw[-latex,dashed] (-1,0.4) -- (-9,0.4);
\draw[-latex] (-9,-0.4) -- (-1,-0.4);
\draw[-latex] (9,-0.4) -- (1,-0.4);
\draw[-latex,dashed] (1,0.4) -- (9,0.4);
\draw[-latex] (0,-2) -- (0,-1);
\node at (0,-3) {face 2};
\draw[-latex,dashed] (0,1) -- (0,9);
\node[scale=1.5] at (-11,0) {$4$};
\end{scope}

\begin{scope}[yshift=-37cm]
\fill[purple] (-1,-1) rectangle (1,1);
\draw (-1,-1) rectangle (1,1);
\fill[pattern = north east lines, pattern color = blue] 
(0,-1) rectangle (9,1);
\fill[pattern = north east lines, pattern color = blue] 
(1,1) rectangle (-1,-9);
\draw[-latex] (-2,0) -- (-1,-0);
\node at (-5,0) {face 7};
\draw[-latex] (9,-0.4) -- (1,-0.4);
\draw[-latex,dashed] (1,0.4) -- (9,0.4);
\draw[-latex] (0,-9) -- (0,-1);

\node[scale=1.5] at (-11,0) {$5$};
\end{scope}

\begin{scope}[yshift=-37cm,xshift=30cm]
\fill[purple] (-1,-1) rectangle (1,1);
\draw (-1,-1) rectangle (1,1);
\fill[pattern = north east lines, pattern color = blue] 
(-9,-1) rectangle (9,1);
\fill[pattern = north east lines, pattern color = blue] 
(1,1) rectangle (-1,-9);
\draw[-latex,dashed] (-1,0.4) -- (-9,0.4);
\draw[-latex] (-9,-0.4) -- (-1,-0.4);
\draw[-latex] (9,-0.4) -- (1,-0.4);
\draw[-latex,dashed] (1,0.4) -- (9,0.4);
\node[scale=1.5] at (-11,0) {$6$};
\draw[-latex] (-9,-0.4) -- (-1,-0.4);
\draw[-latex,dashed] (-1,0.4) -- (-9,0.4);
\draw[-latex] (0,-9) -- (0,-1);
\end{scope}

\end{tikzpicture}\]
\caption{\label{fig.schema.inputs.outputs}
Schema of the inputs and outputs 
directions when inside the area (1) 
and on the border of the area (2,3,4,5,6).}
\end{figure}
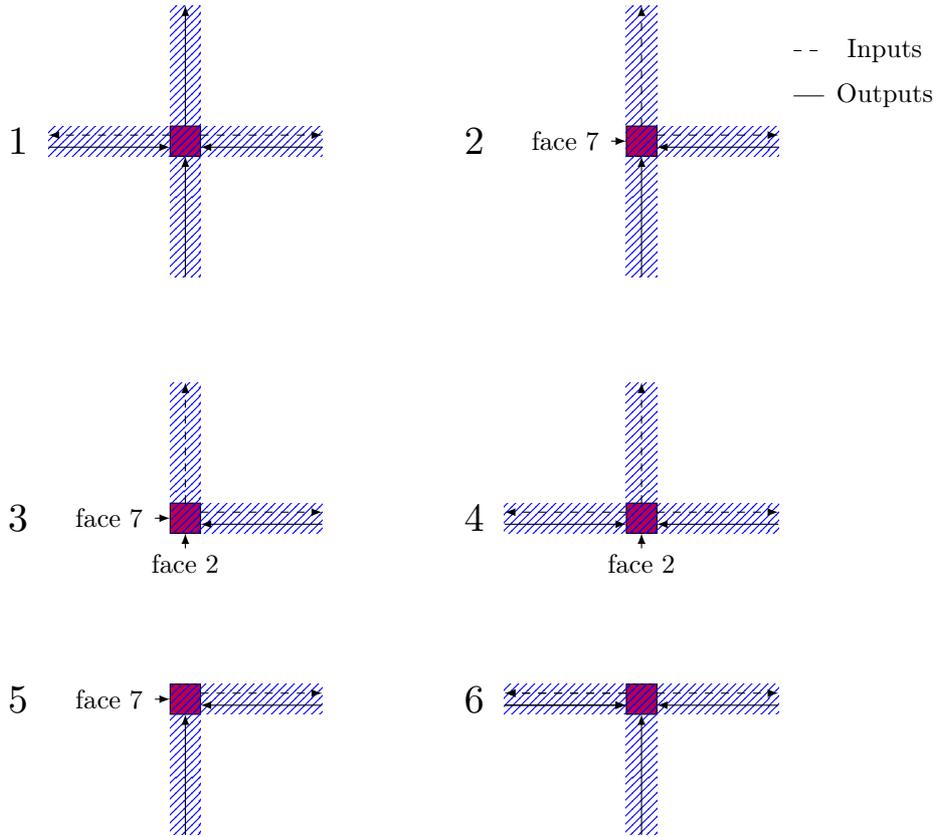

On the first row, all the inputs 
are determined by the counter and 
by the above rule. Then, each 
computation-active row is 
determined from the adjacent one 
on the bottom and the value of the 
linear counter on faces $6$ and $7$, by 
the following rules. These rules determine, on 
each computation position, the 
outputs from the inputs: 

\begin{enumerate}
\item \textbf{Collision between machine heads:} 
if there are at least 
two elements of $\mathcal{Q} \backslash \{q_s\}$ 
in the inputs, then the 
computation position is superimposed 
with $(a,q_e)$. The output on the top 
(when this exists)
is $(a,q_e)$, where $a$ is the letter input below. 
The outputs on the sides are $q_s$.
When there is a unique symbol in 
$\mathcal{Q} \backslash \{q_s\}$ in the inputs, 
this symbol is called the machine head 
state (the symbol $q_s$ is not considered 
as representing a machine head).
\item \textbf{Standard rule:}
\begin{enumerate}
\item when the head input comes from a side, 
then the functional position is superimposed with 
$(a,q)$. The above output is the couple $(a,q)$, 
where $a$ is the letter input under, and $q$ the head input.
The other outputs are $q_s$.
See Figure~\ref{fig.standard.rule.0} for 
an illustration of this rule.

\begin{figure}[ht]
\[\begin{tikzpicture}[scale=0.2]
\fill[purple] (-1,-1) rectangle (1,1);
\draw (-1,-1) rectangle (1,1);
\fill[pattern = north east lines, pattern color = blue] 
(-9,-1) rectangle (9,1);
\fill[pattern = north east lines, pattern color = blue] 
(-1,-9) rectangle (1,9);
\draw[-latex] (-9,0) -- (-1,0);
\draw[-latex] (0,-9) -- (0,-1);
\draw[-latex] (0,1) -- (0,9);
\node at (-10,0) {$q$};
\node at (0,-10) {$a$};
\node at (0,10) {$(a,q)$};
\end{tikzpicture}\]
\caption{\label{fig.standard.rule.0} Illustration 
of the standard rules (1).}
\end{figure}
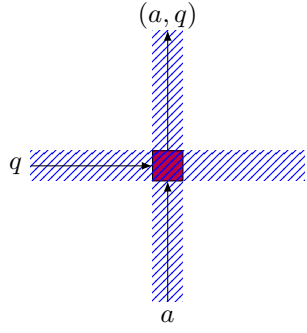

\item when the head input comes from under, 
the above output is:

\begin{itemize}
\item $\delta_1 (a,q)$ when 
$\delta_3 (a,q)$ is in 
$\{\rightarrow,\leftarrow\}$ 
\item and 
$(\delta_1(a,q),\delta_2 (a,q))$ when 
$\delta_3 (a,q) = \uparrow$. 
\end{itemize}

The head output is in the 
direction of $\delta_3 (a,q)$ (when 
this output direction exists) and equal to 
$\delta_2 (a,q)$
when it is in $\{\rightarrow,\leftarrow\}$. 
The other outputs are $q_s$.
See Figure~\ref{fig.standard.rule.1} 
for an illustration.

\begin{figure}[ht]
\[\begin{tikzpicture}[scale=0.175]
\begin{scope}
\fill[purple] (-1,-1) rectangle (1,1);
\draw (-1,-1) rectangle (1,1);
\fill[pattern = north east lines, pattern color = blue] 
(-9,-1) rectangle (9,1);
\fill[pattern = north east lines, pattern color = blue] 
(-1,-9) rectangle (1,9);
\draw[-latex] (-1,0) -- (-9,0);
\draw[-latex] (0,-9) -- (0,-1);
\draw[-latex] (0,1) -- (0,9);
\node at (-13,0) {$\delta_2 (a,q)$};
\node at (0,-10) {$(a,q)$};
\node at (0,10) {$\delta_1 (a,q)$};
\node at (0,-14) {$\delta_3 (a,q) = \leftarrow$};
\end{scope}

\begin{scope}[xshift=23cm]
\fill[purple] (-1,-1) rectangle (1,1);
\draw (-1,-1) rectangle (1,1);
\fill[pattern = north east lines, pattern color = blue] 
(-9,-1) rectangle (9,1);
\fill[pattern = north east lines, pattern color = blue] 
(-1,-9) rectangle (1,9);
\draw[-latex] (1,0) -- (9,0);
\draw[-latex] (0,-9) -- (0,-1);
\draw[-latex] (0,1) -- (0,9);
\node at (13,0) {$\delta_2 (a,q)$};
\node at (0,-10) {$(a,q)$};
\node at (0,10) {$\delta_1 (a,q)$};
\node at (0,-14) {$\delta_3 (a,q) = \rightarrow$};
\end{scope}

\begin{scope}[xshift=50cm]
\fill[purple] (-1,-1) rectangle (1,1);
\draw (-1,-1) rectangle (1,1);
\fill[pattern = north east lines, pattern color = blue] 
(-9,-1) rectangle (9,1);
\fill[pattern = north east lines, pattern color = blue] 
(-1,-9) rectangle (1,9);
\draw[-latex] (0,-9) -- (0,-1);
\draw[-latex] (0,1) -- (0,9);
\node at (0,10) {$(\delta_1(a,q),\delta_2 (a,q))$};
\node at (0,-10) {$(a,q)$};
\node at (0,-14) {$\delta_3 (a,q) = \uparrow$};
\end{scope}
\end{tikzpicture}\]
\caption{\label{fig.standard.rule.1}
Illustration of the standard rules (2).}
\end{figure}
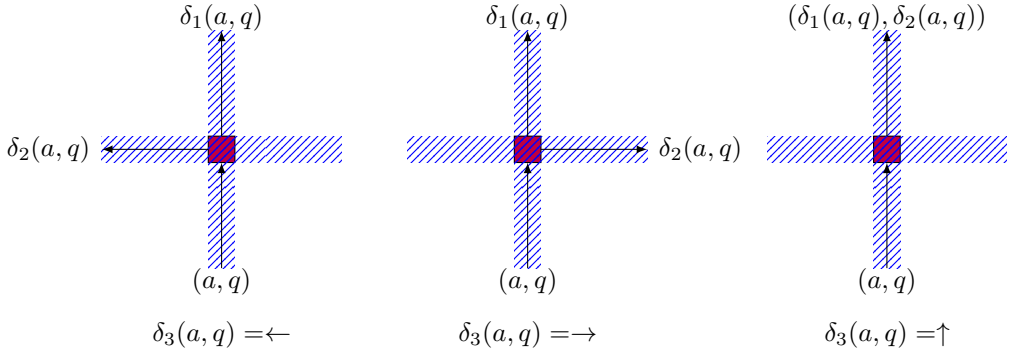
\end{enumerate}
\item \textbf{Collision with border:} 
When the output 
direction does not exist, the output 
is $(a,q_e)$ on the top, and the outputs 
on the side is $q_s$.
The computation position is 
superimposed with $(a,q)$.  
\item \textbf{No machine head:} when all 
the inputs in $\mathcal{Q}$ are 
$q_s$, and the above output is 
in $\mathcal{A}$ and equal to its input $a$.
\end{enumerate} 

\end{itemize}

\noindent {\textbf{\textit{Global behavior:}} \bigskip

On the bottom faces according to $\vec{e}^3$ 
of order $qp$ three-dimensional cells, we implemented some computations using our
modified Turing machine model. This model 
allows multiple 
machine heads on the initial tape and 
entering in each row. When there 
is a unique machine head on the leftmost 
position of the bottom line and only blank 
letters on the initial tape, and 
all the lines and columns are computation-active, then the computations 
are as intended. This means that a the 
machine write successively the 
bits $a^{(n)}_k$ 
on the $p_k$th column of its tape (in 
order to impose the value of the frequency bits),
 and moreover 
writes $1$ if $k$ is a power of two, and $0$ 
if not, in order to impose the value of 
the grouping bits. It enters in the error 
state $q_e$ when it detects an error.

When this is not the case, the computations 
are determined by the rules 
giving the outputs on computation positions 
from the inputs. When there is a collision 
of a machine head with the border, 
it enters in state $q_e$. When 
heads collide, they fusion into a unique 
head in state $q_e$.
In Section~\ref{sec.error.signals.min}, 
we describe signal errors that helps us
to take into account the computations 
only when the initial tape is 
empty, all the lines and columns are computation-active, and there is no machine 
head entering 
on the sides of the area.

\subsection{\label{sec.error.signals.min} 
Error signals}

In order to simulate any behavior that happens in infinite areas in 
finite ones, we need error signals. This means that when 
the machine detects an error (enters a halting state), it sends a signal 
to the initialization line to verify it was well initialized: that 
the tape was empty, that 
no machine head enters on the sides, 
and that the machine was initially in the leftmost position 
of the line, and in initialization state. Moreover, for the reason 
that we need to compute precisely the number of possible 
initial tape contents, we allow initialization of multiple heads. 
The first error signal will detect the first position from left 
to right in the top row of the area 
where there is a machine head in error state 
or the active column is $\texttt{off}$.
This position only will trigger an 
error signal (described 
in Section~\ref{subsec.error.signal}, according 
to the direction specified just above when it in the top line of the area 
(the word of arrows specifying the direction is a part of the counter). The empty 
tape signal detects if the 
initial tape was empty, and that there 
was a unique machine head on the leftmost 
position in initialization state $q_0$. 
The empty tape and first error
signals are described in Section~\ref{subsec.empty.tape.first.machine}. 
The empty sides signal, described in 
Section~\ref{subsec.empty.sides} detects if 
there is no machine head entering 
on the sides, and that the $\texttt{on}/
\texttt{off}$ signals on the sides 
are all equal to $\texttt{on}$.
The error signal is taken into account 
(meaning forbidden) when the empty tape, 
and empty sides signals are detecting 
an error.

\subsubsection{ \label{subsec.empty.tape.first.machine}
Empty tape, first error signals}

\noindent \textbf{\textit{Symbols:}} \bigskip

The first sublayer has the 
following symbols: 

$\begin{tikzpicture}[scale=0.3]
\fill[Salmon] (0,0) rectangle (1,1);
\draw (0,0) rectangle (1,1);
\end{tikzpicture}, \begin{tikzpicture}[scale=0.3]
\fill[YellowGreen] (0,0) rectangle (1,1);
\draw (0,0) rectangle (1,1);
\end{tikzpicture}$, symbols in 
$\left\{\begin{tikzpicture}[scale=0.3]
\fill[Salmon] (0,0) rectangle (1,1);
\draw (0,0) rectangle (1,1);
\end{tikzpicture}, \begin{tikzpicture}[scale=0.3]
\fill[YellowGreen] (0,0) rectangle (1,1);
\draw (0,0) rectangle (1,1);
\end{tikzpicture}\right\} ^2 $, and 
a blank symbol $
\begin{tikzpicture}[scale=0.3]
\draw (0,0) rectangle (1,1);
\end{tikzpicture}$. \bigskip

\noindent \textbf{\textit{Local rules:}}  

\begin{itemize}
\item \textbf{Localization:} 
non blank symbols are superimposed on the 
top line and bottom line of the
border of the machine face as a two-dimensional cell.
\item \textbf{First error signal:} this signal 
detects the first error on the top 
of the functional area, from 
the left to the right, 
where an error means a symbol 
$\texttt{off}$ or $q_e$. The rules are: 
\begin{itemize}
\item the topmost leftmost position of the top 
line of the cell is 
marked with $\begin{tikzpicture}[scale=0.3]
\fill[YellowGreen] (0,0) rectangle (1,1);
\draw (0,0) rectangle (1,1);
\end{tikzpicture}$. 
\item the symbol $\begin{tikzpicture}[scale=0.3]
\fill[YellowGreen] (0,0) rectangle (1,1);
\draw (0,0) rectangle (1,1);
\end{tikzpicture}$ propagates the the left, 
and propagates to the right
while the position under is not in error 
state $q_e$ and the symbol in $\{\texttt{on},
\texttt{off}\}$ is $\texttt{on}$.
\item when on position in the top row with an 
error, 
the position on the top right is colored 
$\begin{tikzpicture}[scale=0.3]
\fill[Salmon] (0,0) rectangle (1,1);
\draw (0,0) rectangle (1,1);
\end{tikzpicture}$.
\item the symbol $\begin{tikzpicture}[scale=0.3]
\fill[Salmon] (0,0) rectangle (1,1);
\draw (0,0) rectangle (1,1);
\end{tikzpicture}$ propagates to the right, 
and propagates to the left while 
the positions under is not in error state.
\end{itemize} 
\textbf{Empty tape signal:} this signal
detects if the initial tape of the machine is empty. 
This means that it is 
filled with the symbol $(\#,q_s)$ except 
on the leftmost position 
where it has to be $(\#,q_0)$. The signal detects 
the first symbol which is different from $(\#,q_s)$ or $(\#,q_0)$ 
when on the left, 
from left to right (first color), and from left to right 
(second color).
Concerning the first color: 
\begin{itemize}
\item on the bottom row, the leftmost position is colored with 
$\begin{tikzpicture}[scale=0.3]
\fill[YellowGreen] (0,0) rectangle (1,1);
\draw (0,0) rectangle (1,1);
\end{tikzpicture}$.
\item The symbol 
$\begin{tikzpicture}[scale=0.3]
\fill[YellowGreen] (0,0) rectangle (1,1);
\draw (0,0) rectangle (1,1);
\end{tikzpicture}$
propagates to the right unless when 
on a position under a symbol different from:
\begin{itemize}
\item 
$(\#,q_0)$ when on the leftmost functional position, 
\item $(\#,q_s)$ on another functional position.
\end{itemize}
\item When on these positions, the symbol on the right is 
$\begin{tikzpicture}[scale=0.3]
\fill[Salmon] (0,0) rectangle (1,1);
\draw (0,0) rectangle (1,1);
\end{tikzpicture}$.
\item the symbol 
$\begin{tikzpicture}[scale=0.3]
\fill[Salmon] (0,0) rectangle (1,1);
\draw (0,0) rectangle (1,1);
\end{tikzpicture}$ 
propagates to the right.
\end{itemize}
for the second one: 
\begin{itemize}
\item on the bottom row, the rightmost position is colored with 
$\begin{tikzpicture}[scale=0.3]
\fill[YellowGreen] (0,0) rectangle (1,1);
\draw (0,0) rectangle (1,1);
\end{tikzpicture}$.
\item The symbol 
$\begin{tikzpicture}[scale=0.3]
\fill[YellowGreen] (0,0) rectangle (1,1);
\draw (0,0) rectangle (1,1);
\end{tikzpicture}$
propagates to the left except when
on a position under a symbol different from:
\begin{itemize}
\item  
$(\#,q_0)$ when on the leftmost functional position, 
\item $(\#,q_s)$ on another computation 
position.
\end{itemize}
\item When on these positions, the symbol on the right is 
$\begin{tikzpicture}[scale=0.3]
\fill[Salmon] (0,0) rectangle (1,1);
\draw (0,0) rectangle (1,1);
\end{tikzpicture}$.
\item the symbol 
$\begin{tikzpicture}[scale=0.3]
\fill[Salmon] (0,0) rectangle (1,1);
\draw (0,0) rectangle (1,1);
\end{tikzpicture}$ 
propagates to the left.
\end{itemize}
\end{itemize}

\noindent \textbf{\textit{Global behavior:}} \bigskip

The top row is separated in 
two parts: before and after (from left to right) 
the first error. 
The left part is colored $\begin{tikzpicture}[scale=0.3]
\fill[YellowGreen] (0,0) rectangle (1,1);
\draw (0,0) rectangle (1,1);
\end{tikzpicture}$, and the right part 
$\begin{tikzpicture}[scale=0.3]
\fill[Salmon] (0,0) rectangle (1,1);
\draw (0,0) rectangle (1,1);
\end{tikzpicture}$. The bottom row is colored 
with a couple of color. The first one 
separates the row in two parts. The limit 
between the two parts is the first occurrence 
\textit{from left to right} 
of a symbol different from $(\#,q_s)$ or $(\#,q_0)$ when on 
the leftmost computation position above. The 
second color of the couple separates 
two similar parts \textit{from 
right to left}.

See Figure~\ref{fig.error.signal.dim.entropique}
for an illustration.

\subsubsection{\label{subsec.empty.sides}
Empty sides signals}

This second sublayer has the same 
symbols as the first sublayer.
The principle of the local rules 
is similar: the leftmost (resp. rightmost) 
column is splitted 
in two parts, the top one colored 
$\begin{tikzpicture}[scale=0.3]
\fill[YellowGreen] (0,0) rectangle (1,1);
\draw (0,0) rectangle (1,1);
\end{tikzpicture}$ and 
the bottom one colored 
$\begin{tikzpicture}[scale=0.3]
\fill[Salmon] (0,0) rectangle (1,1);
\draw (0,0) rectangle (1,1);
\end{tikzpicture}$.
The limit is the first position 
from top to bottom where 
the corresponding symbol across 
the limit with face $7$ (resp. face $6$) 
is $(q_s,\texttt{on})$ (resp. $q_s$).
Moreover, the bottom row
of the machine face is colored with 
a couple of colors. This couple 
is constant over the row. 
The first one of the colors is equal to the color 
at the bottom of the leftmost column. 
The second one is the color at the bottom 
of the rightmost one. 

See an illustration on Figure~\ref{fig.error.signal.dim.entropique}.

\subsubsection{\label{subsec.error.signal} 
Error signals}

\noindent \textbf{\textit{Symbols:}} \bigskip

$\begin{tikzpicture}[scale=0.3]
\fill[purple] (0,0) rectangle (1,1);
\draw (0,0) rectangle (1,1);
\end{tikzpicture}$ (error signal), $\begin{tikzpicture}[scale=0.3]
\draw (0,0) rectangle (1,1);
\end{tikzpicture}$. \bigskip

\noindent \textbf{\textit{Local rules:}} 

\begin{itemize}
\item  \textbf{Localization:} 
the non-blank symbols are superimposed 
on the right, left and top sides of the border of 
machine faces in three-dimensional cells.
\item \textbf{Propagation:} each of the two symbols 
propagates when inside one of these two areas: 
\begin{itemize}
\item the union of the left side of the face 
and positions colored $\begin{tikzpicture}[scale=0.3]
\fill[YellowGreen] (0,0) rectangle (1,1);
\draw (0,0) rectangle (1,1);
\end{tikzpicture}$ in the top side of the face.
\item the union of the right side of the face 
and positions colored $\begin{tikzpicture}[scale=0.3]
\fill[Salmon] (0,0) rectangle (1,1);
\draw (0,0) rectangle (1,1);
\end{tikzpicture}$ in the top side of the face.
\end{itemize}
\item \textbf{Induction:} 
\begin{itemize} 
\item on a position of the top side of the face 
which is colored $\begin{tikzpicture}[scale=0.3]
\fill[YellowGreen] (0,0) rectangle (1,1);
\draw (0,0) rectangle (1,1);
\end{tikzpicture}$ and the position on the right is 
$\begin{tikzpicture}[scale=0.3]
\fill[Salmon] (0,0) rectangle (1,1);
\draw (0,0) rectangle (1,1);
\end{tikzpicture}$, if the symbol 
above in the information transfers layer is $\rightarrow$ (resp. 
$\leftarrow$), then 
there is an error signal $\begin{tikzpicture}[scale=0.3]
\fill[purple] (0,0) rectangle (1,1);
\draw (0,0) rectangle (1,1);
\end{tikzpicture}$ on this position and none on the right (resp. 
there is no error signal on this position and there is one on the right).
\item on the rightmost topmost position of the face, if the 
first machine signal is $\begin{tikzpicture}[scale=0.3]
\fill[YellowGreen] (0,0) rectangle (1,1);
\draw (0,0) rectangle (1,1);
\end{tikzpicture}$, then 
there is no error signal.
\end{itemize} 
\item \textbf{Forbidding wrong configurations:} 

there can not be four 
symbols $\begin{tikzpicture}[scale=0.3]
\fill[YellowGreen] (0,0) rectangle (1,1);
\draw (0,0) rectangle (1,1);
\end{tikzpicture}$ and an error signal $\begin{tikzpicture}[scale=0.3]
\fill[purple] (0,0) rectangle (1,1);
\draw (0,0) rectangle (1,1);
\end{tikzpicture}$ on the same position.
\end{itemize}

\noindent \textbf{\textit{Global behavior:}} \bigskip

When there is a machine head in 
error state in the top row, the first one 
(from left to right) 
sends an error signal to the bottom row (see 
Figure~\ref{fig.error.signal.dim.entropique}) 
in the direction indicated by the arrow 
on the corresponding position on face $3$.
This signal is forbidden 
if the machine is well initialized. This means 
that the 
working tape of the machine is 
empty in the bottom row. Moreover, there is 
a unique machine in state $q_0$ in the leftmost position of the bottom row, 
all the lines and columns are $\texttt{on}$ 
and there is no machine entering on the sides. 
This means that the error signal is taken 
into account only when the computations 
have the intended behavior.

\begin{figure}[ht]
\[\begin{tikzpicture}[scale=0.5]
\begin{scope}
\fill[YellowGreen] (0,8) rectangle (3,9);
\fill[Salmon] (3,8) rectangle (9,9);
\fill[Salmon] (0,0) rectangle (6,0.5);
\fill[YellowGreen] (6,0) rectangle (9,0.5);
\fill[YellowGreen] (0,0.5) rectangle (4,1);
\fill[Salmon] (4,0.5) rectangle (9,1);
\fill[purple] (0,0) rectangle (1,9);
\fill[purple] (0,8) rectangle (2,9);
\draw (2,9) rectangle (3,10);
\draw[-latex] (3,9.5) -- (2,9.5);
\node at (2.5,8.5) {$q_e$};
\draw (0,0) grid (9,9); 
\draw[dashed] (0,0.5) -- (9,0.5);
\draw[-latex] (-1,5) -- (0,5);
\node at (-3,5) {Error signal};
\draw[-latex] (5,-1) -- (5,0);
\node at (5,-2) {Empty tape signal};
\draw[-latex] (7,10) -- (7,9) ;
\node at (7,11) {First machine signal};
\end{scope}

\begin{scope}[xshift=12cm]
\fill[Salmon] (0,0) rectangle (9,1);
\fill[Salmon] (0,0) rectangle (1,5);
\fill[YellowGreen] (0,5) rectangle (1,9);

\fill[Salmon] (8,0) rectangle (9,3);
\fill[YellowGreen] (8,3) rectangle (9,9);
\draw (0,0) grid (9,9); 

\draw[dashed] (0,0.5) -- (9,0.5);

\draw[-latex] (5,-1) -- (5,0);
\node at (5,-2) {Empty sides signal};
\end{scope}
\end{tikzpicture}\]
\caption{\label{fig.error.signal.dim.entropique} Illustration 
of the propagation of an error signal, 
where are represented the empty tape, 
first machine and empty sides signals.}
\end{figure}
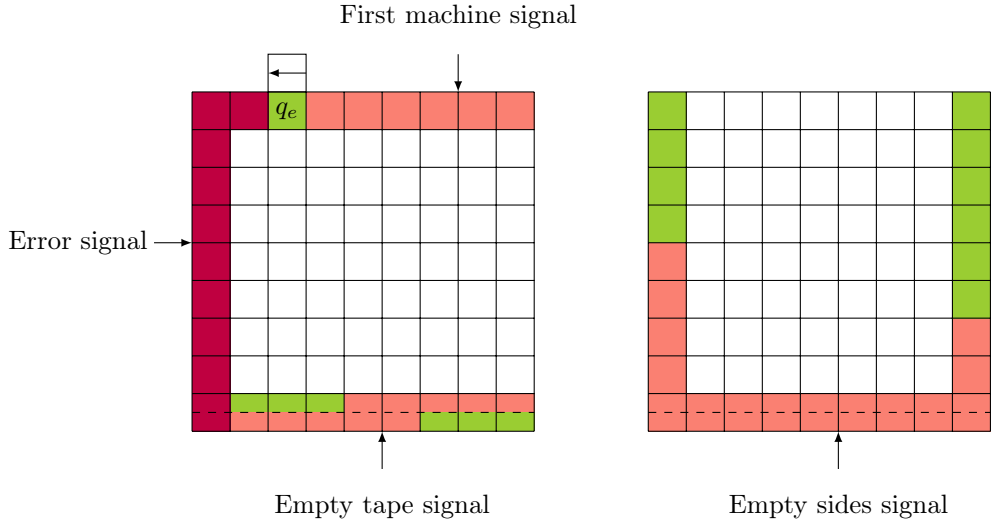

Because for any $n$ and any configuration, there 
exists some three-dimensional cell in which 
the machine is well initialized 
(because of the presence of counters). 
For any $k$, there 
exists some $n$ such that 
the machine has enough time to 
check the $k$th frequency bit and grouping 
bit.
This means that in any configuration 
of the subshift, 
the $k$th frequency bit is equal to $a_k$, 
and the $2^k$th grouping bit is $1$, 
and the other ones are $0$.

\section{\label{section.completing.machines} Completing the machines areas}

When completing the machine face, 
there are two types of difficulties. 
The first 
one is managing the various signals : machine signals, 
first error, empty tape, empty side 
signals, and error signals. 
The second one is 
managing the space-time diagram 
of the machine. When 
the machine face is all known, there 
is no completion to make. Hence 
we describe the completion 
according to the parts of the 
face that are known (meaning that 
they appear in the initial pattern).
Since there is strictly less difficulty 
to complete knowing only a part 
inside the face than on the border 
(since the difficulties come from 
completing the space-time diagram 
of the machine, in a similar way than 
for the border), we describe 
the completion only when a part 
of the border is known.

\begin{itemize}
\item When knowing the top right corner of the machine face: 

\begin{figure}[ht]
\[\begin{tikzpicture}[scale=0.04]

\draw[latex-] (16,136) -- (48,136);
\node at (32,144) {Transport of information}; 
\draw[-latex] (136,48) -- (136,16);

\begin{scope}

\fill[gray!90] (4,4) rectangle (4.5,12); 
\fill[gray!90] (4,4) rectangle (12,4.5); 
\fill[gray!90] (4.5,11.5) rectangle (12,12); 
\fill[gray!90] (11.5,4.5) rectangle (12,12); 

\fill[gray!90] (4,20) rectangle (4.5,28); 
\fill[gray!90] (4,20) rectangle (12,20.5); 
\fill[gray!90] (4.5,27.5) rectangle (12,28); 
\fill[gray!90] (11.5,20.5) rectangle (12,28); 

\fill[gray!90] (4,36) rectangle (4.5,44); 
\fill[gray!90] (4,36) rectangle (12,36.5); 
\fill[gray!90] (4.5,43.5) rectangle (12,44); 
\fill[gray!90] (11.5,36.5) rectangle (12,44);

\fill[gray!90] (4,52) rectangle (4.5,60); 
\fill[gray!90] (4,52) rectangle (12,52.5); 
\fill[gray!90] (4.5,59.5) rectangle (12,60); 
\fill[gray!90] (11.5,52.5) rectangle (12,60); 

\fill[purple] (0,0) rectangle (4,4);
\fill[purple] (0,0+11) rectangle (4,4+11);
\fill[purple] (11,0) rectangle (15,4);
\fill[purple] (11,11) rectangle (15,4+11);

\fill[purple] (0,64) rectangle (4,64-4);
\fill[purple] (0,64-11) rectangle (4,64-4-11);
\fill[purple] (11,64) rectangle (15,64-4);
\fill[purple] (11,64-11) rectangle (15,64-4-11);

\fill[purple] (64,64) rectangle (60,60);
\fill[purple] (64,64-11) rectangle (60,60-11);
\fill[gray!90] (64-12.5,64-1.5) rectangle (64-14.5,64-3.5);
\fill[gray!90] (64-12.5,64-1.5-11) rectangle (64-14.5,64-3.5-11);

\fill[purple] (64,0) rectangle (60,4);
\fill[purple] (64,11) rectangle (60,4+11);
\fill[gray!90] (64-12.5,1.5) rectangle (64-14.5,3.5);
\fill[gray!90] (64-12.5,1.5+11) rectangle (64-14.5,3.5+11);

\fill[gray!90] (20,4) rectangle (20.5,12); 
\fill[gray!90] (20,4) rectangle (28,4.5); 
\fill[gray!90] (20.5,11.5) rectangle (28,12); 
\fill[gray!90] (27.5,4.5) rectangle (28,12); 

\fill[gray!90] (20,20) rectangle (20.5,28); 
\fill[gray!90] (20,20) rectangle (28,20.5); 
\fill[gray!90] (20.5,27.5) rectangle (28,28); 
\fill[gray!90] (27.5,20.5) rectangle (28,28); 

\fill[gray!90] (20,36) rectangle (20.5,44); 
\fill[gray!90] (20,36) rectangle (28,36.5); 
\fill[gray!90] (20.5,43.5) rectangle (28,44); 
\fill[gray!90] (27.5,36.5) rectangle (28,44);

\fill[gray!90] (20,52) rectangle (20.5,60); 
\fill[gray!90] (20,52) rectangle (28,52.5); 
\fill[gray!90] (20.5,59.5) rectangle (28,60); 
\fill[gray!90] (27.5,52.5) rectangle (28,60);

\fill[gray!90] (36,4) rectangle (36.5,12); 
\fill[gray!90] (36,4) rectangle (44,4.5); 
\fill[gray!90] (36.5,11.5) rectangle (44,12); 
\fill[gray!90] (43.5,4.5) rectangle (44,12); 

\fill[gray!90] (36,20) rectangle (36.5,28); 
\fill[gray!90] (36,20) rectangle (44,20.5); 
\fill[gray!90] (36.5,27.5) rectangle (44,28); 
\fill[gray!90] (43.5,20.5) rectangle (44,28); 

\fill[gray!90] (36,36) rectangle (36.5,44); 
\fill[gray!90] (36,36) rectangle (44,36.5); 
\fill[gray!90] (36.5,43.5) rectangle (44,44); 
\fill[gray!90] (43.5,36.5) rectangle (44,44);

\fill[gray!90] (36,52) rectangle (36.5,60); 
\fill[gray!90] (36,52) rectangle (44,52.5); 
\fill[gray!90] (36.5,59.5) rectangle (44,60); 
\fill[gray!90] (43.5,52.5) rectangle (44,60);

\fill[gray!90] (52,4) rectangle (52.5,12); 
\fill[gray!90] (52,4) rectangle (60,4.5); 
\fill[gray!90] (52.5,11.5) rectangle (60,12); 
\fill[gray!90] (59.5,4.5) rectangle (60,12); 

\fill[gray!90] (52,20) rectangle (52.5,28); 
\fill[gray!90] (52,20) rectangle (60,20.5); 
\fill[gray!90] (52.5,27.5) rectangle (60,28); 
\fill[gray!90] (59.5,20.5) rectangle (60,28); 

\fill[gray!90] (52,36) rectangle (52.5,44); 
\fill[gray!90] (52,36) rectangle (60,36.5); 
\fill[gray!90] (52.5,43.5) rectangle (60,44); 
\fill[gray!90] (59.5,36.5) rectangle (60,44);

\fill[gray!90] (52,52) rectangle (52.5,60); 
\fill[gray!90] (52,52) rectangle (60,52.5); 
\fill[gray!90] (52.5,59.5) rectangle (60,60); 
\fill[gray!90] (59.5,52.5) rectangle (60,60);

\draw[-latex] (20,62) -- (28,62);
\draw[-latex] (36,62) -- (44,62);
\draw[-latex] (20,2) -- (28,2);
\draw[-latex] (36,2) -- (44,2);

\draw[-latex] (20,50) -- (28,50);
\draw[-latex] (36,50) -- (44,50);
\draw[-latex] (20,14) -- (28,14);
\draw[-latex] (36,14) -- (44,14);

\draw[-latex] (62,20) -- (62,28);
\draw[-latex] (62,36) -- (62,44);
\draw[-latex,color=purple] (2,20) -- (2,28);
\draw[-latex,color=purple] (2,36) -- (2,44);

\draw[-latex] (50,20) -- (50,28);
\draw[-latex] (50,36) -- (50,44);
\draw[-latex,color=purple] (14,20) -- (14,28);
\draw[-latex,color=purple] (14,36) -- (14,44);

\draw[-latex] (-40,0) -- (-28,0);
\draw[-latex] (-40,0) -- (-40,12);
\node at (-40,16) {$v$};
\node at (-24,0) {$h$};
\end{scope}

\begin{scope}[yshift=64cm]

\fill[gray!90] (4,4) rectangle (4.5,12); 
\fill[gray!90] (4,4) rectangle (12,4.5); 
\fill[gray!90] (4.5,11.5) rectangle (12,12); 
\fill[gray!90] (11.5,4.5) rectangle (12,12); 

\fill[gray!90] (4,20) rectangle (4.5,28); 
\fill[gray!90] (4,20) rectangle (12,20.5); 
\fill[gray!90] (4.5,27.5) rectangle (12,28); 
\fill[gray!90] (11.5,20.5) rectangle (12,28); 

\fill[gray!90] (4,36) rectangle (4.5,44); 
\fill[gray!90] (4,36) rectangle (12,36.5); 
\fill[gray!90] (4.5,43.5) rectangle (12,44); 
\fill[gray!90] (11.5,36.5) rectangle (12,44);

\fill[gray!90] (4,52) rectangle (4.5,60); 
\fill[gray!90] (4,52) rectangle (12,52.5); 
\fill[gray!90] (4.5,59.5) rectangle (12,60); 
\fill[gray!90] (11.5,52.5) rectangle (12,60); 

\fill[purple] (0,0) rectangle (4,4);
\fill[purple] (0,0+11) rectangle (4,4+11);
\fill[purple] (11,0) rectangle (15,4);
\fill[purple] (11,11) rectangle (15,4+11);

\fill[purple] (0,64) rectangle (4,64-4);
\fill[purple] (0,64-11) rectangle (4,64-4-11);
\fill[purple] (11,64) rectangle (15,64-4);
\fill[purple] (11,64-11) rectangle (15,64-4-11);

\fill[purple] (64,64) rectangle (60,60);
\fill[purple] (64,64-11) rectangle (60,60-11);
\fill[gray!90] (64-12.5,64-1.5) rectangle (64-14.5,64-3.5);
\fill[gray!90] (64-12.5,64-1.5-11) rectangle (64-14.5,64-3.5-11);

\fill[purple] (64,0) rectangle (60,4);
\fill[purple] (64,11) rectangle (60,4+11);
\fill[gray!90] (64-12.5,1.5) rectangle (64-14.5,3.5);
\fill[gray!90] (64-12.5,1.5+11) rectangle (64-14.5,3.5+11);

\fill[gray!90] (20,4) rectangle (20.5,12); 
\fill[gray!90] (20,4) rectangle (28,4.5); 
\fill[gray!90] (20.5,11.5) rectangle (28,12); 
\fill[gray!90] (27.5,4.5) rectangle (28,12); 

\fill[gray!90] (20,20) rectangle (20.5,28); 
\fill[gray!90] (20,20) rectangle (28,20.5); 
\fill[gray!90] (20.5,27.5) rectangle (28,28); 
\fill[gray!90] (27.5,20.5) rectangle (28,28); 

\fill[gray!90] (20,36) rectangle (20.5,44); 
\fill[gray!90] (20,36) rectangle (28,36.5); 
\fill[gray!90] (20.5,43.5) rectangle (28,44); 
\fill[gray!90] (27.5,36.5) rectangle (28,44);

\fill[gray!90] (20,52) rectangle (20.5,60); 
\fill[gray!90] (20,52) rectangle (28,52.5); 
\fill[gray!90] (20.5,59.5) rectangle (28,60); 
\fill[gray!90] (27.5,52.5) rectangle (28,60);

\fill[gray!90] (36,4) rectangle (36.5,12); 
\fill[gray!90] (36,4) rectangle (44,4.5); 
\fill[gray!90] (36.5,11.5) rectangle (44,12); 
\fill[gray!90] (43.5,4.5) rectangle (44,12); 

\fill[gray!90] (36,20) rectangle (36.5,28); 
\fill[gray!90] (36,20) rectangle (44,20.5); 
\fill[gray!90] (36.5,27.5) rectangle (44,28); 
\fill[gray!90] (43.5,20.5) rectangle (44,28); 

\fill[gray!90] (36,36) rectangle (36.5,44); 
\fill[gray!90] (36,36) rectangle (44,36.5); 
\fill[gray!90] (36.5,43.5) rectangle (44,44); 
\fill[gray!90] (43.5,36.5) rectangle (44,44);

\fill[gray!90] (36,52) rectangle (36.5,60); 
\fill[gray!90] (36,52) rectangle (44,52.5); 
\fill[gray!90] (36.5,59.5) rectangle (44,60); 
\fill[gray!90] (43.5,52.5) rectangle (44,60);

\fill[gray!90] (52,4) rectangle (52.5,12); 
\fill[gray!90] (52,4) rectangle (60,4.5); 
\fill[gray!90] (52.5,11.5) rectangle (60,12); 
\fill[gray!90] (59.5,4.5) rectangle (60,12); 

\fill[gray!90] (52,20) rectangle (52.5,28); 
\fill[gray!90] (52,20) rectangle (60,20.5); 
\fill[gray!90] (52.5,27.5) rectangle (60,28); 
\fill[gray!90] (59.5,20.5) rectangle (60,28); 

\fill[gray!90] (52,36) rectangle (52.5,44); 
\fill[gray!90] (52,36) rectangle (60,36.5); 
\fill[gray!90] (52.5,43.5) rectangle (60,44); 
\fill[gray!90] (59.5,36.5) rectangle (60,44);

\fill[gray!90] (52,52) rectangle (52.5,60); 
\fill[gray!90] (52,52) rectangle (60,52.5); 
\fill[gray!90] (52.5,59.5) rectangle (60,60); 
\fill[gray!90] (59.5,52.5) rectangle (60,60);

\draw[-latex] (20,62) -- (28,62);
\draw[-latex] (36,62) -- (44,62);
\draw[-latex] (20,2) -- (28,2);
\draw[-latex] (36,2) -- (44,2);

\draw[-latex] (20,50) -- (28,50);
\draw[-latex] (36,50) -- (44,50);
\draw[-latex] (20,14) -- (28,14);
\draw[-latex] (36,14) -- (44,14);

\draw[-latex] (62,20) -- (62,28);
\draw[-latex] (62,36) -- (62,44);
\draw[-latex,color=purple] (2,20) -- (2,28);
\draw[-latex,color=purple] (2,36) -- (2,44);

\draw[-latex] (50,20) -- (50,28);
\draw[-latex] (50,36) -- (50,44);
\draw[-latex,color=purple] (14,20) -- (14,28);
\draw[-latex,color=purple] (14,36) -- (14,44);
\end{scope}

\begin{scope}[xshift=64cm]

\fill[gray!90] (4,4) rectangle (4.5,12); 
\fill[gray!90] (4,4) rectangle (12,4.5); 
\fill[gray!90] (4.5,11.5) rectangle (12,12); 
\fill[gray!90] (11.5,4.5) rectangle (12,12); 

\fill[gray!90] (4,20) rectangle (4.5,28); 
\fill[gray!90] (4,20) rectangle (12,20.5); 
\fill[gray!90] (4.5,27.5) rectangle (12,28); 
\fill[gray!90] (11.5,20.5) rectangle (12,28); 

\fill[gray!90] (4,36) rectangle (4.5,44); 
\fill[gray!90] (4,36) rectangle (12,36.5); 
\fill[gray!90] (4.5,43.5) rectangle (12,44); 
\fill[gray!90] (11.5,36.5) rectangle (12,44);

\fill[gray!90] (4,52) rectangle (4.5,60); 
\fill[gray!90] (4,52) rectangle (12,52.5); 
\fill[gray!90] (4.5,59.5) rectangle (12,60); 
\fill[gray!90] (11.5,52.5) rectangle (12,60); 

\fill[gray!90] (1.5,1.5) rectangle (3.5,3.5);
\fill[gray!90] (1.5,1.5+11) rectangle (3.5,3.5+11);
\fill[gray!90] (12.5,1.5) rectangle (14.5,3.5);
\fill[gray!90] (12.5,1.5+11) rectangle (14.5,3.5+11);

\fill[gray!90] (1.5,64-1.5) rectangle (3.5,64-3.5);
\fill[gray!90] (1.5,64-1.5-11) rectangle (3.5,64-3.5-11);
\fill[gray!90] (12.5,64-1.5) rectangle (14.5,64-3.5);
\fill[gray!90] (12.5,64-1.5-11) rectangle (14.5,64-3.5-11);

\fill[purple] (64,64) rectangle (64-4,64-4);
\fill[purple] (64,64-11) rectangle (64-4,64-4-11);
\fill[gray!90] (64-12.5,64-1.5) rectangle (64-14.5,64-3.5);
\fill[gray!90] (64-12.5,64-1.5-11) rectangle (64-14.5,64-3.5-11);

\fill[purple] (64,0) rectangle (64-4,4);
\fill[purple] (64,+11) rectangle (64-4,4+11);
\fill[gray!90] (64-12.5,1.5) rectangle (64-14.5,3.5);
\fill[gray!90] (64-12.5,1.5+11) rectangle (64-14.5,3.5+11);

\fill[gray!90] (20,4) rectangle (20.5,12); 
\fill[gray!90] (20,4) rectangle (28,4.5); 
\fill[gray!90] (20.5,11.5) rectangle (28,12); 
\fill[gray!90] (27.5,4.5) rectangle (28,12); 

\fill[gray!90] (20,20) rectangle (20.5,28); 
\fill[gray!90] (20,20) rectangle (28,20.5); 
\fill[gray!90] (20.5,27.5) rectangle (28,28); 
\fill[gray!90] (27.5,20.5) rectangle (28,28); 

\fill[gray!90] (20,36) rectangle (20.5,44); 
\fill[gray!90] (20,36) rectangle (28,36.5); 
\fill[gray!90] (20.5,43.5) rectangle (28,44); 
\fill[gray!90] (27.5,36.5) rectangle (28,44);

\fill[gray!90] (20,52) rectangle (20.5,60); 
\fill[gray!90] (20,52) rectangle (28,52.5); 
\fill[gray!90] (20.5,59.5) rectangle (28,60); 
\fill[gray!90] (27.5,52.5) rectangle (28,60);

\fill[gray!90] (36,4) rectangle (36.5,12); 
\fill[gray!90] (36,4) rectangle (44,4.5); 
\fill[gray!90] (36.5,11.5) rectangle (44,12); 
\fill[gray!90] (43.5,4.5) rectangle (44,12); 

\fill[gray!90] (36,20) rectangle (36.5,28); 
\fill[gray!90] (36,20) rectangle (44,20.5); 
\fill[gray!90] (36.5,27.5) rectangle (44,28); 
\fill[gray!90] (43.5,20.5) rectangle (44,28); 

\fill[gray!90] (36,36) rectangle (36.5,44); 
\fill[gray!90] (36,36) rectangle (44,36.5); 
\fill[gray!90] (36.5,43.5) rectangle (44,44); 
\fill[gray!90] (43.5,36.5) rectangle (44,44);

\fill[gray!90] (36,52) rectangle (36.5,60); 
\fill[gray!90] (36,52) rectangle (44,52.5); 
\fill[gray!90] (36.5,59.5) rectangle (44,60); 
\fill[gray!90] (43.5,52.5) rectangle (44,60);

\fill[gray!90] (52,4) rectangle (52.5,12); 
\fill[gray!90] (52,4) rectangle (60,4.5); 
\fill[gray!90] (52.5,11.5) rectangle (60,12); 
\fill[gray!90] (59.5,4.5) rectangle (60,12); 

\fill[gray!90] (52,20) rectangle (52.5,28); 
\fill[gray!90] (52,20) rectangle (60,20.5); 
\fill[gray!90] (52.5,27.5) rectangle (60,28); 
\fill[gray!90] (59.5,20.5) rectangle (60,28); 

\fill[gray!90] (52,36) rectangle (52.5,44); 
\fill[gray!90] (52,36) rectangle (60,36.5); 
\fill[gray!90] (52.5,43.5) rectangle (60,44); 
\fill[gray!90] (59.5,36.5) rectangle (60,44);

\fill[gray!90] (52,52) rectangle (52.5,60); 
\fill[gray!90] (52,52) rectangle (60,52.5); 
\fill[gray!90] (52.5,59.5) rectangle (60,60); 
\fill[gray!90] (59.5,52.5) rectangle (60,60);

\draw[-latex] (20,62) -- (28,62);
\draw[-latex] (36,62) -- (44,62);
\draw[-latex] (20,2) -- (28,2);
\draw[-latex] (36,2) -- (44,2);

\draw[-latex] (20,50) -- (28,50);
\draw[-latex] (36,50) -- (44,50);
\draw[-latex] (20,14) -- (28,14);
\draw[-latex] (36,14) -- (44,14);

\draw[-latex,color=purple] (62,20) -- (62,28);
\draw[-latex,color=purple] (62,36) -- (62,44);
\draw[-latex] (2,20) -- (2,28);
\draw[-latex] (2,36) -- (2,44);

\draw[-latex] (50,20) -- (50,28);
\draw[-latex] (50,36) -- (50,44);
\draw[-latex] (14,20) -- (14,28);
\draw[-latex] (14,36) -- (14,44);
\end{scope}

\begin{scope}[xshift=64cm,yshift=64cm]

\fill[gray!90] (4,4) rectangle (4.5,12); 
\fill[gray!90] (4,4) rectangle (12,4.5); 
\fill[gray!90] (4.5,11.5) rectangle (12,12); 
\fill[gray!90] (11.5,4.5) rectangle (12,12); 

\fill[gray!90] (4,20) rectangle (4.5,28); 
\fill[gray!90] (4,20) rectangle (12,20.5); 
\fill[gray!90] (4.5,27.5) rectangle (12,28); 
\fill[gray!90] (11.5,20.5) rectangle (12,28); 

\fill[gray!90] (4,36) rectangle (4.5,44); 
\fill[gray!90] (4,36) rectangle (12,36.5); 
\fill[gray!90] (4.5,43.5) rectangle (12,44); 
\fill[gray!90] (11.5,36.5) rectangle (12,44);

\fill[gray!90] (4,52) rectangle (4.5,60); 
\fill[gray!90] (4,52) rectangle (12,52.5); 
\fill[gray!90] (4.5,59.5) rectangle (12,60); 
\fill[gray!90] (11.5,52.5) rectangle (12,60); 

\fill[gray!90] (1.5,1.5) rectangle (3.5,3.5);
\fill[gray!90] (1.5,1.5+11) rectangle (3.5,3.5+11);
\fill[gray!90] (12.5,1.5) rectangle (14.5,3.5);
\fill[gray!90] (12.5,1.5+11) rectangle (14.5,3.5+11);

\fill[gray!90] (1.5,64-1.5) rectangle (3.5,64-3.5);
\fill[gray!90] (1.5,64-1.5-11) rectangle (3.5,64-3.5-11);
\fill[gray!90] (12.5,64-1.5) rectangle (14.5,64-3.5);
\fill[gray!90] (12.5,64-1.5-11) rectangle (14.5,64-3.5-11);

\fill[purple] (64,64) rectangle (64-4,64-4);
\fill[purple] (64,64-11) rectangle (64-4,64-4-11);
\fill[gray!90] (64-12.5,64-1.5) rectangle (64-14.5,64-3.5);
\fill[gray!90] (64-12.5,64-1.5-11) rectangle (64-14.5,64-3.5-11);

\fill[purple] (64,0) rectangle (64-4,4);
\fill[purple] (64,11) rectangle (64-4,4+11);
\fill[gray!90] (64-12.5,1.5) rectangle (64-14.5,3.5);
\fill[gray!90] (64-12.5,1.5+11) rectangle (64-14.5,3.5+11);

\fill[gray!90] (20,4) rectangle (20.5,12); 
\fill[gray!90] (20,4) rectangle (28,4.5); 
\fill[gray!90] (20.5,11.5) rectangle (28,12); 
\fill[gray!90] (27.5,4.5) rectangle (28,12); 

\fill[gray!90] (20,20) rectangle (20.5,28); 
\fill[gray!90] (20,20) rectangle (28,20.5); 
\fill[gray!90] (20.5,27.5) rectangle (28,28); 
\fill[gray!90] (27.5,20.5) rectangle (28,28); 

\fill[gray!90] (20,36) rectangle (20.5,44); 
\fill[gray!90] (20,36) rectangle (28,36.5); 
\fill[gray!90] (20.5,43.5) rectangle (28,44); 
\fill[gray!90] (27.5,36.5) rectangle (28,44);

\fill[gray!90] (20,52) rectangle (20.5,60); 
\fill[gray!90] (20,52) rectangle (28,52.5); 
\fill[gray!90] (20.5,59.5) rectangle (28,60); 
\fill[gray!90] (27.5,52.5) rectangle (28,60);

\fill[gray!90] (36,4) rectangle (36.5,12); 
\fill[gray!90] (36,4) rectangle (44,4.5); 
\fill[gray!90] (36.5,11.5) rectangle (44,12); 
\fill[gray!90] (43.5,4.5) rectangle (44,12); 

\fill[gray!90] (36,20) rectangle (36.5,28); 
\fill[gray!90] (36,20) rectangle (44,20.5); 
\fill[gray!90] (36.5,27.5) rectangle (44,28); 
\fill[gray!90] (43.5,20.5) rectangle (44,28); 

\fill[gray!90] (36,36) rectangle (36.5,44); 
\fill[gray!90] (36,36) rectangle (44,36.5); 
\fill[gray!90] (36.5,43.5) rectangle (44,44); 
\fill[gray!90] (43.5,36.5) rectangle (44,44);

\fill[gray!90] (36,52) rectangle (36.5,60); 
\fill[gray!90] (36,52) rectangle (44,52.5); 
\fill[gray!90] (36.5,59.5) rectangle (44,60); 
\fill[gray!90] (43.5,52.5) rectangle (44,60);

\fill[gray!90] (52,4) rectangle (52.5,12); 
\fill[gray!90] (52,4) rectangle (60,4.5); 
\fill[gray!90] (52.5,11.5) rectangle (60,12); 
\fill[gray!90] (59.5,4.5) rectangle (60,12); 

\fill[gray!90] (52,20) rectangle (52.5,28); 
\fill[gray!90] (52,20) rectangle (60,20.5); 
\fill[gray!90] (52.5,27.5) rectangle (60,28); 
\fill[gray!90] (59.5,20.5) rectangle (60,28); 

\fill[gray!90] (52,36) rectangle (52.5,44); 
\fill[gray!90] (52,36) rectangle (60,36.5); 
\fill[gray!90] (52.5,43.5) rectangle (60,44); 
\fill[gray!90] (59.5,36.5) rectangle (60,44);

\fill[gray!90] (52,52) rectangle (52.5,60); 
\fill[gray!90] (52,52) rectangle (60,52.5); 
\fill[gray!90] (52.5,59.5) rectangle (60,60); 
\fill[gray!90] (59.5,52.5) rectangle (60,60);

\draw[-latex] (20,62) -- (28,62);
\draw[-latex] (36,62) -- (44,62);
\draw[-latex] (20,2) -- (28,2);
\draw[-latex] (36,2) -- (44,2);

\draw[-latex] (20,50) -- (28,50);
\draw[-latex] (36,50) -- (44,50);
\draw[-latex] (20,14) -- (28,14);
\draw[-latex] (36,14) -- (44,14);

\draw[-latex,color=purple] (62,20) -- (62,28);
\draw[-latex,color=purple] (62,36) -- (62,44);
\draw[-latex] (2,20) -- (2,28);
\draw[-latex] (2,36) -- (2,44);

\draw[-latex] (50,20) -- (50,28);
\draw[-latex] (50,36) -- (50,44);
\draw[-latex] (14,20) -- (14,28);
\draw[-latex] (14,36) -- (14,44);
\end{scope}

\fill[pattern=north east lines,pattern color=yellow] (0,4) rectangle (4,128);
\fill[pattern=north east lines,pattern color=yellow] (12,4) rectangle (16,128);
\fill[pattern=north east lines,pattern color=blue] (124,4) rectangle (128,128);
\fill[pattern=north east lines,pattern color=blue] (64,0) rectangle (60,128);

\fill[pattern = north east lines, pattern color = yellow] (0,0) rectangle (128,4);
\fill[pattern = north east lines, pattern color =yellow] (0,12) rectangle (128,16);
\fill[pattern = north east lines, pattern color = blue] (0,60) rectangle (128,64);
\fill[pattern = north east lines, pattern color = blue] (0,124) rectangle (128,128);

\draw[line width = 0.5mm] (58,130) 
rectangle (130,58);

\end{tikzpicture}\]
\caption{\label{fig.completing.off} Illustration 
of the completion 
of the $\texttt{on}/\texttt{off}$ 
signals and the space-time diagram 
of the machine. The known part is 
surrounded by a black square.}
\end{figure}
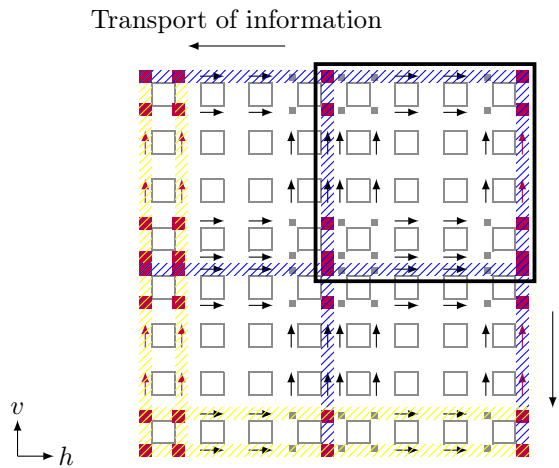

\begin{enumerate}
\item if the signal detects the first machine in error state from left to right 
(becomes \begin{tikzpicture}[scale=0.3]
\fill[Salmon] (0,0) rectangle (1,1);
\draw (0,0) rectangle (1,1);
\end{tikzpicture} after being \begin{tikzpicture}[scale=0.3]
\fill[YellowGreen] (0,0) rectangle (1,1);
\draw (0,0) rectangle (1,1);
\end{tikzpicture}), then 
we already know the propagation 
direction of 
the error signal. Then we complete 
the first machine signal and the error 
signal according to what is known.
For completing the 
space-time diagram of the machine, 
the difficulty comes from 
the fact that this is possible that when 
completing the trajectory of two machine 
heads according to the local rules, 
they have to collide reversely in time. 
This is not possible in our model.
This is where we use the 
$\texttt{on}/\texttt{off}$ signals.
We complete first the non already determined 
signals using only the symbol $\texttt{off}$, 
as illustrated on 
Figure~\ref{fig.completing.off}.

Then the space-time diagram is completed 
by only transporting the information.

In the end, we completed with any symbols 
in $\mathcal{Q}$ or $\mathcal{A} \times 
\mathcal{Q}$ where the symbols are not 
determined. We can do this since 
they do not interact with the known 
part of the space-time diagram. 
Then we complete empty tape and empty side 
signals according to the determined 
symbols.

\item if the signal is all 
\begin{tikzpicture}[scale=0.3]
\fill[YellowGreen] (0,0) rectangle (1,1);
\draw (0,0) rectangle (1,1);
\end{tikzpicture}, then one can 
complete the face without 
wondering about the error signal.
\item if the signal is all \begin{tikzpicture}[scale=0.3]
\fill[Salmon] (0,0) rectangle (1,1);
\draw (0,0) rectangle (1,1);
\end{tikzpicture}, we complete in the same way as for the first point, and 
the $\texttt{off}$ signals contribute 
to the first machine signal being 
all \begin{tikzpicture}[scale=0.3]
\fill[Salmon] (0,0) rectangle (1,1);
\draw (0,0) rectangle (1,1);
\end{tikzpicture}. If 
there is no error signal on the right side, 
one can complete so that 
the first machine in error state has above the arrow $\leftarrow$ indicating 
that the error signal has to propagate to the left. If there is an error signal 
on the right, then we set this arrow as $\rightarrow$. This is illustrated on 
Figure~\ref{fig.completion.signal.direction}.

\begin{figure}[ht]
\[\begin{tikzpicture}[scale=0.4]
\begin{scope}
\fill[purple] (6,0) rectangle (10,1);
\draw (6,0) -- (10,0);
\draw (6,1) -- (10,1);
\draw[dashed] (7,-0.5) rectangle (11,1.5);

\draw[-latex] (13,0.5) -- (18,0.5);
\end{scope}

\begin{scope}[xshift=20cm]
\draw (4,1) rectangle (5,2);
\draw[-latex] (4,1.5) -- (5,1.5);
\fill[purple] (4,0) rectangle (10,1);
\fill[Salmon] (4,0) rectangle (0,1);
\draw (0,0) -- (10,0);
\draw (0,1) -- (10,1);
\draw (4,0) -- (4,1) ; 
\draw (5,0) -- (5,1);
\draw[dashed] (7,-0.5) rectangle (11,1.5);
\end{scope}

\begin{scope}[yshift=8cm]
\fill[Salmon] (6,0) rectangle (10,1);
\draw (6,0) -- (10,0);
\draw (6,1) -- (10,1);
\draw[dashed] (7,-0.5) rectangle (11,1.5);

\draw[-latex] (13,0.5) -- (18,0.5);
\node at (15.5,1.5) {Completing};
\end{scope}

\begin{scope}[yshift=8cm,xshift=20cm]
\draw (4,1) rectangle (5,2);
\draw[latex-] (4,1.5) -- (5,1.5);
\fill[Salmon] (4,0) rectangle (10,1);
\fill[purple] (4,0) rectangle (0,1);
\draw (0,0) -- (10,0);
\draw (0,1) -- (10,1);
\draw (4,0) -- (4,1) ; 
\draw (5,0) -- (5,1);
\draw[dashed] (7,-0.5) rectangle (11,1.5);
\end{scope}

\end{tikzpicture}\]
\caption{\label{fig.completion.signal.direction}
Illustration of the completion 
of the arrows according to the error 
signal in the known part of the area, 
designated by a dashed rectangle.}
\end{figure}
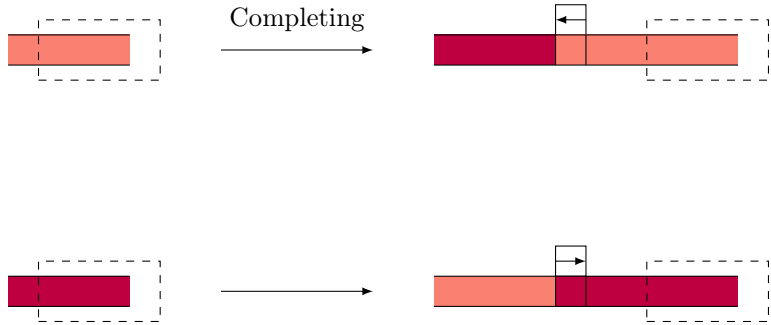

\end{enumerate}
\item when knowing the top left corner, the difficulty comes from the direction of 
propagation of the error signal. This is ruled in a similar way as point 3 
of the last case. 
\item when knowing the bottom right corner or bottom left corner: the completion 
is similar as in the last points, except 
that we have to manage the empty 
tape and sides signals. The 
difficult point comes from the signal, 
whose propagation direction 
is towards the known corner. 
If this signal detects an error before entering 
in the known area, we complete 
so that the added symbols in 
$\{\texttt{on},\texttt{off}\}$
are all $\texttt{off}$: this induces the error. 
When this signal enters without detecting an error, 
we complete all the symbols so that 
they do not introduce any error.
A particular difficulty comes from the case when 
the bottom right corner is known. Indeed, 
when the signal enters without having detecting 
an error, this means we have to complete the 
pattern so that a machine head in initial state 
is initialized in the leftmost position 
of the bottom row. Since the pattern 
is completed in the structure layer into 
a great enough cell, this head can not enter 
in the known pattern.
\item when the pattern crosses only a edge or 
the center of the face, 
the completion is similar (but easier since 
these parts have less information, 
then we need to add less to the pattern).
\end{itemize}

\end{document}